\documentclass[12pt,centertags,oneside]{amsart}
\usepackage{lgrind}
\usepackage{cmap}
\usepackage{amsmath}
\usepackage{amsthm}
\usepackage{amssymb}
\usepackage{graphicx}
\usepackage{tikz}
\usepackage{tikz-cd}
\usepackage{tocvsec2}
\usepackage{mathrsfs}
\usepackage{pifont}
\usepackage[utf8]{inputenc}
\usepackage{hyperref}
\hypersetup{
    colorlinks=true,
    linkcolor=red,
    filecolor=magenta,      
    urlcolor=cyan,
}

\definecolor{labelkey}{rgb}{1,0,0}

\usetikzlibrary{arrows}
\theoremstyle{plain}
\newtheorem{definition}{Definition}[section]
\newtheorem{theorem}[definition]{Theorem}
\newtheorem{remark}[definition]{Remark}
\newtheorem{proposition}[definition]{Proposition}

\newtheorem{lemma}[definition]{Lemma}

\newcommand{\product}[1]{\langle #1\rangle}
\newcommand{\derivative}[1]{\frac{\partial}{\partial #1}}
\newcommand{\cc}{\mathbb{C}}
\newcommand{\zz}{\mathbb{Z}}
\newcommand{\rr}{\mathbb{R}}
\newcommand{\nn}{\mathbb{N}}

\newcommand{\trace}{\mathrm{Tr}}
\newcommand{\ENDO}{\mathrm{End}}

\newcommand{\dbar}{{\overline{\partial}^X}}

\newcommand{\dbarstar}{\overline{\partial}^{X,*}}
\newcommand{\bisconn}{\mathbb{B}}

\newcommand{\chern}{\mathrm{ch}}

\newcommand{\bktv}{\mathcal{B}_{K,t,v}}
\newcommand{\bkttv}{\mathcal{B}_{K,t,tv}}
\newcommand{\lk}{\mathcal{L}_K}

\newcommand{\midarrow}{\tikz \draw[-triangle 90] (0,0)-- +(.1,0);}
\newcommand{\str}{\mathrm{Tr}_{\text{s}}}
\newcommand{\kahler}{\mathrm{K\Ddot{a}hler}}
\newcommand{\td}{\mathrm{Td}}
\newcommand{\kongge}{\mathrm{\;}}
\newcommand{\ch}{\mathrm{ch}}
\newcommand{\cvt}{\mathcal{A}_{v,t}}
\newcommand{\ctvt}{\mathcal{A}_{tv,t}}
\newcommand{\dabar}{\overline{da}}
\newcommand{\scrbku}{\mathscr{B}_{K,u}}
\newcommand{\dba}{\overline{\partial}}
\newcommand{\uprime}{\mathcal{U}_{\epsilon}'}
\newcommand{\uneighbor}{\mathcal{U}_{\epsilon}}
\newcommand{\vneighbor}{\mathcal{V}_{\epsilon}}
\newcommand{\ccvt}{\mathcal{C}_{v,t}}
\newcommand{\akt}{A_{K,t}}

\newcommand{\bigtwo}{\uppercase\expandafter{\romannumeral2}}
\newcommand{\ft}{\widetilde{F}_t}
\newcommand{\Op}{{\mathcal{O}p}}
\newcommand{\lzktv}{L_{Z_0,K}^{3,(t,v)}}
\newcommand{\cinfty}{{\mathscr{C}}^{\infty}}
\newcommand{\Tr}{\mathrm{Tr}_{\mathrm{s}}}

\title[Equivariant Analytic Torsion Forms]{Equivariant Analytic Torsion Forms}

\author{Yuxiang Li}

\address{Universit\'e Paris Cit\'e, CNRS, IMJ-PRG, B\^atiment Sophie Germain, UFR de Math\'ematiques, Case 7012, 75205 Paris Cedex 13, France}

\email{yuxiang.li@imj-prg.fr}

\usepackage{tikz-cd}

\begin{document}
\begin{abstract}
In this paper, we first construct infinitesimal equivariant analytic torsion forms for K\"ahler fibrations. We then compare them with Ma's equivariant analytic torsion forms, modulo $\partial$- and $\overline{\partial}$-exact forms, in terms of equivariant Bott--Chern currents and the $I$-genus. Our comparison formula extends that of Bismut and Goette to families.
\end{abstract}
\maketitle
\numberwithin{equation}{section}
\date{\today}
\tableofcontents
\setcounter{section}{-1}
\settocdepth{section}

\section{Introduction}
The holomorphic analytic torsion of Ray-Singer \cite{RaySinger} is a famous spectral invariant in complex geometry, as a complex analogue of the analytic torsion for flat vector bundle \cite{RS71}. Since the 1980s, the  theory of holomorphic analytic torsion was fully developed by Bismut and his co-authors, and they have found many applications.  In particular, Bismut, Gillet and Soul\'e established a curvature formula \cite{BGS3} of the determinant line bundle, which refines the Riemann-Roch-Grothendieck theorem at the differential form level in degree 2. In \cite{BL}, Bismut and Lebeau studied the behaviors of analytic torsions under holomorphic immersions, in terms of the Bott-Chern currents \cite{BGS1988} and  the $R$-genus of Gillet-Soul\'{e} \cite{GS91}. 
The analytic torsions arise naturally  in the Arakelov geometry as the analytic part of direct images (cf. \cite{soulebourbaki}, \cite{Soulebook}). By using Bismut-Lebeau's immersion formula \cite{BL}, Gillet and Soul\'{e} \cite{GS92} proved the arithmetic Riemmann-Roch theorem, where their $R$-genus appears as a correction term.

To extend the above results to the family case, the first question is to construct a family version of the analytic torsion, this tough task is completed by Bismut and K\"ohler in \cite{BK1992}. In their article, they follow the previous work of Bismut-Gillet-Soul\'e \cite{BGS2}, and use essentially the Bismut superconnections \cite{B1986} to construct the holomorphic analytic torsion form, whose $0$-degree part coincides with the Ray-Singer analytic torsion. The $\partial\overline{\partial}$ of Bismut-K\"ohler's analytic torsion form is the difference of the terms which appear in the Riemann-Roch-Grothendieck theorem at the form level. In particular, it gives a refinement of the Riemann-Roch-Grothendieck theorem in full degrees.
 In \cite{BK1992}, Bismut and K\"ohler also extend the anomaly formula in \cite{BGS3} to analytic torsion forms. The immersion formula for analytic torsion forms is established by Bismut in \cite{B1997}, which extends the result of \cite{BL}. Then in \cite{GRS}, Gillet, R\"ossler and Soul\'e used Bismut's immersion formula \cite{B1997} to obtain the arithmetic Riemann-Roch theorem in full degrees, we refer to \cite{Faltings92} for related works. These results motivate a lot to study the functoriality and diverse properties of the torsion forms. For example, in \cite{ma99}, \cite{ma00}, Ma studied the behavior of analytic torsion forms under the composition of two submersions. 
See also  \cite{BFL1402}, \cite{BFL14} for recent works on analytic torsion class and their applications to arithmetic Riemann-Roch theorem for general projective morphisms.

As the Atiyah-Singer index theorem has its equivariant version, naturally we try  to extend the above  story to the equivariant situation. In fact K\"ohler and R\"ossler proved a Lefschetz fixed point formula in Arakelov geometry in \cite{KR01} by using Bismut's immersion formula \cite{B95} for equivariant analytic torsion.
Their fixed point formula have various applications in number theory and in arithmetic geometry, which could be found in \cite{KaK}, \cite{KR2}, \cite{KR03}, \cite{MaR}, \cite{MaR04}, \cite{MaR09}. 

In \cite{Ma2000}, Ma extended the holomorphic analytic torsion forms to the equivariant case, and he proved the corresponding anomaly formula. The immersion formula in the equivariant case is established by Bismut and Ma in \cite{BismutMa02}, by combining the various techniques developed in \cite{BL}, \cite{B95}, \cite{B1997}. Their formula is important in extending K\"ohler-R\"ossler's' arithmetic Lefschetz fixed point theorem to full degrees,  we refer to \cite{Tangshun10}, \cite{tang12}, \cite{Tang21} for some recent progresses on this subject.

The equivariant analytic torsion is not necessarily continuous as the function on the group, thus we hope to study the analytic properties of this function. It turns out to be difficult to work on the equivariant analytic torsion directly, thus in \cite{BismutGoette2000}, Bismut and Goette  constructed an infinitesimal version of equivariant  analytic torsion, which is an analytic function in a neighborhood of $0$ in the Lie algebra. However, the operators used in their construction do not have lower bounds, thus Bismut and Goette developped new techniques to go through this difficulty. Subsequently, they compared the two versions of equivariant analytic torsion in terms of equivariant Bott-Chern current of Bismut \cite{B1990} and their $I$-genus, from which the analytic properties of equivariant analytic torsions are revealed.  As an application in Arakelov geometry, in \cite{KR2}, K\"ohler and R\"ossler used Bismut-Goette's comparison formula to obtain a residue formula in Arakelov geometry.

This article aims to generalize the result of Bismut-Goette to the family case. We will first introduce an extension of Bismut-Goette's infinitesimal equivariant analytic torsion to higher degrees.  Then we will  prove a comparison formula between the equivariant analytic torsion form and its infinitesimal version, from which we could express the singularity of the  equivariant analytic torsion form module $\partial$- and $\overline{\partial}$-exact forms, as locally computable differential forms.  Our work is an analogue of \cite{BG04} for Bismut-Lott torsion forms for flat vector bundles, and of \cite{Go00}, \cite{LiuMa2022} for Bismut-Cheeger $\eta$-forms.

We briefly comment on the main challenge in the proof. The comparison formula is obtained by computing the limit of a contour integral. In \cite{BismutGoette2000}, this integral vanishes identically. However, in our case, the contour integral equals a sum of $\partial$ and $\overline{\partial}$ of certain double integrals. If the base manifold is compact and K\"ahler, the space of $\partial$- and $\overline{\partial}$-exact forms is closed under limits, so we would not need to analyze these double integrals. In general, however, a detailed study of these double integrals is necessary. Our situation is similar to Bismut's proof of the immersion formula for analytic torsion forms \cite{B1997}, and his methods provide valuable guidance.

Let's descibe our results in more detail.
Let $\pi: W\rightarrow S$ be a submersion of complex manifolds with compact fiber $X$ of complex dimension ${\ell}$. Let $TX=TW/S$ be the holomorphic relative tangent bundle. Let $\xi$ be a holomorphic vector bundle over $W$. 

Let $R^{\bullet}\pi_{*}\xi=\oplus_{i=0}^{\ell}R^{i}\pi_{*}\xi$ be the direct image of $\xi$. We make the basic \textbf{assumption} that  $R^{\bullet}\pi_{*}\xi$ is locally free. Let $H^{\bullet}(X,\xi|_{X})$ be the cohomology of $\xi|_{X}$, then we have the canonical identification of holomorphic vector bundles on $S$, 
\begin{align}\label{8q}
R^{\bullet}\pi_{*}\xi\cong H^{\bullet}(X,\xi|_{X}).
\end{align}

Let $(\Omega^{0,\bullet}(X,\xi|_{X}),\overline{\partial}^X)$ be the Dolbeault complex of smooth sections of $\Lambda(T^{*(0,1)}X)\otimes \xi$ on the fiber of $\pi$, then its cohomology coincides with $H^{\bullet}(X,\xi|_{X})$.


We suppose that the submersion $\pi: W\rightarrow S$ is a $\kahler$ fibration in the sense of \cite{BGS2}. Namely, there is a real closed  $(1,1)$ form $\omega$ on $W$, which restricts to the K\"ahler form $\omega^{TX}$ on the relative tangent bundle $TX$. Let $h^{TX}$ be the Hermitian metric on $TX$ induced by $\omega^{TX}$. Let $h^{\xi}$ be a Hermitian metric on $\xi$.  We  equip $\Omega^{0,\bullet}(X,\xi|_{X})$ with the $L^{2}$-metric (cf. \eqref{L2 metric}) induced by $h^{TX}$ and $h^{\xi}$. Let $\overline{\partial}^{X,*}$ be the formal adjoint of $\overline{\partial}^X$ with respect to the $L^{2}$-metric. Set
\begin{align}\label{8r}
D^{X}=\overline{\partial}^X+\overline{\partial}^{X,*}
\end{align}
 acting on $\Omega^{0,\bullet}(X,\xi|_{X})$.
 By using the Hodge theory, we can identify $H^{\bullet}(X,\xi|_X)$ to the kernel of $D^{X}$, so that $H^{\bullet}(X, \xi|_X)$ inherits a Hermitian metric $h^{H^{\bullet}(X,  \xi|_X)}$ from the $L^{2}$-metric.

Now let $G$ be a compact Lie group. Let $\mathfrak{g}$ be the Lie algebra of $G$ and $\mathfrak{g}^{*}$ be its dual space. We assume that $G$ acts holomorphically, fiberwise on $W$, and its action lifts holomorphically to $\xi$. We suppose that $\omega$ and $h^{\xi}$ are $G$-invariant.

For $g\in G$, let $W_g=\left\{ x\in W: gx=x \right\}$ be the fixed point set of $g$. Then the restriction of $\pi$ on $W_g$, is a submersion $\pi|_{W_g}:W_g\to S$ with compact fiber $X_g$.
Let  $\ch_g(\xi, h^{\xi})\in \Omega^{\bullet}(W_{g},\mathbb{C})$ be the equivariant Chern character form (cf. \eqref{genus-g}) on $W_{g}$ associated to $\nabla^{\xi}$. Similarly, let $\td_{g}(TX,h^{TX})\in \Omega^{\bullet}(W_{g},\mathbb{C})$ be the equivariant Todd form of $(TX, h^{TX})$. Put 
\begin{align}\label{8u}
    \begin{split}
 \ch_{g}(H^{\bullet}(X,\xi|_{X}),h^{H^{\bullet}(X,\xi|_{X})})=\sum_{k=0}^{\ell}(-1)^{k}\ch_{g}(H^{k}(X,\xi|_{X}),h^{H^{\bullet}(X,\xi|_{X})}),\\
        \ch_{g}'(H^{\bullet}(X,\xi|_{X}),h^{H^{\bullet}(X,\xi|_{X})})=\sum_{k=0}^{\ell}(-1)^{k}k\ch_{g}(H^{k}(X,\xi|_{X}),h^{H^{\bullet}(X,\xi|_{X})}).
    \end{split}
\end{align}

Let $P^S$ be the space of  sums of smooth $(p,p)$ forms on $S$.  Let $P^{S,0}$ be the subspace of $P^S$ consisting of $\gamma\in P^S$ of the form $\partial \alpha+\overline{\partial}\beta$, where $\alpha, \beta$ are smooth forms on $S$.

Let $T_g(\omega, h^{\xi})\in P^S$ be the equivariant analytic torsion form constructed by Ma in \cite[Def. 2.11]{Ma2000}, which is an equivariant extension of Bismut-K\"ohler's analytic torsion form \cite{BK1992}. It satisfies 
the following transgression formula, 
    \begin{equation}\label{intro-02}
        \frac{\overline{\partial}\partial}{2i\pi}T_g(\omega,h^{\xi})=\chern_g(H^{\bullet}(X,\xi|_{\xi}),h^{H^{\bullet}(X,\xi{|_X})})-\int_{X_g}\td_g(TX,h^{TX})\ch_g(\xi,h^{\xi}).
    \end{equation}

Recall that the fixed point formula has two versions, the Lefschetz fixed point formula and the Kirillov like formula \cite[Theorem 8.2]{BGV}, we expect that the same is true for equivariant analytic torsion forms. The first version is the $T_g(\omega,h^\xi)$ in \eqref{intro-02}. Once we consider \eqref{intro-02} in the de Rham cohomology, it reduces to an equivariant version of Riemann-Roch-Grothendieck theorem. Thus it gives a refinement of the Lefschetz fixed point formula at the differential form level. Now we will explain the second version of  equivariant analytic torsion forms,  which extend Bismut-Goette's  equivariant infinitesimal analytic torsion to higher degrees.

Let $\bisconn_{u}, u>0$ be the rescaled Bismut superconnection (cf. \eqref{Bismut superconnection}) constructed by Bismut in \cite{B1986}. The operators $\bisconn_{u},u>0$ commute with the action of $G$.

Set $\Phi: P^{S}\to P^{S}$ as multiplication by $(2i\pi)^{-p}$ on $(p,p)$-forms.
Let $N_{V}$ be the number operator of $\Omega^{0,\bullet}(X,\xi|_{X})$, which acts as multiplication by $p$ on $\Omega^{0,p}(X,\xi|_{X})$.

 We denote by $\omega^{H}$ the horizontal part of $\omega$ (cf. \eqref{95}).  For $u>0$, set 
\begin{align}\label{8t}
N_{u}=N_{V}+i\frac{\omega^{H}}{u}.
\end{align}

For $K\in \mathfrak{g}$,  let $\mathcal{L}_K$ be the Lie derivative of $K$, and let $K^{X}\in \cinfty(W,T_{\mathbb{R}}X)$ be the corresponding smooth vector field on $W$ (see \eqref{96}, \eqref{definition lk}).

\textbf{Assumption(*)}: There is a moment map $\mu\in \cinfty(W,\mathfrak{g}^{*})$ associated to the $G$-action and real form $\omega$, that is, for any $g\in G$, $x\in W$, $K\in\mathfrak{g}$, 
\begin{align}\label{97}
&\mu(gx)=\text{Ad}_{g}^{*} \mu(x), &&d\product{\mu,K}+i_{K^{X}}\omega=0.
\end{align}

For $g\in G$, we denote by $\mathfrak{z}(g)$ the Lie algebra of the centralizer $Z(g)\subset G$ of $g$.
For $K\in \mathfrak{z}(g)$, $u>0$, put
\begin{align}\label{8s}
\gamma_{g,K}(u)=\Tr\left[ g\left(N_{u}-i\frac{\left\langle \mu,K \right\rangle}{u}\right)\exp\left(-\mathcal{L}_K-\left( \bisconn_{u}+\frac{c(K^{X})}{2\sqrt{2u}} \right)^{2}\right) \right]\in P^{S}.
\end{align}

 For $g\in G$, $K\in \mathfrak{z}(g)$, with $|K|$ small enough, let $\chern_{g,K}(\xi,h^{\xi})$, $\td_{g,K}(TX, h^{TX})$, $\td_{g,K}'(TX, h^{TX})\in \Omega^{\bullet}(W_{g},\mathbb{C})$ be the equivariant infinitesimal characteristic forms (cf. \eqref{equivariant curvature}, \cite[(2.30)]{BismutGoette2000}). 
Set 
\begin{equation}
\begin{split}
     &C_{-1,g,K}=\int_{X_g}\left(\frac{\omega}{2\pi}-i\product{\mu,K}\right)\td_{g,K}(TX,h^{TX})\ch_{g,K}(\xi,h^{\xi}),\\
        &C_{0,g,K}=\int_{X_g}\left(-\td_{g,K}'(TX,h^{TX})+\ell\cdot \td_{g,K}(TX,h^{TX})\right)\ch_{g,K}(\xi,h^{\xi}).
\end{split}
\end{equation}

    \begin{theorem}\label{8x}
         Under \textbf{Assumption(*)}, for $g\in G$, for any compact subset $S'\subset S$, there exist $\beta>0$, $C_{0,g,K}'\in P^{S'}$ with $C_{0,g,K}'=C_{0,g,K}$ in $P^{S'}/P^{S',0}$, such that the following integral is well defined for any $K\in \mathfrak{z}(g)$, $|K|<\beta$, 
         \begin{equation}
\begin{split}
    T_{g,K}(\omega, h^{\xi})&=-\int_{0}^1\left(\Phi\gamma_{g,K}(u)-\frac{C_{-1,g,K}}{u}-C_{0,g,K}'\right)\frac{du}{u}\\
    &-\int_{1}^{+\infty}
    \left(\Phi\gamma_{g,K}(u)-
    \chern'_{ge^{K}}( H^{\bullet}(X,\xi|_{X}),h^{H^{\bullet}(X,\xi|_{X})}) \right) \frac{du}{u}\\
    &+C_{-1,g,K}+\Gamma'(1)\left(C_{0,g,K}'-\chern'_{ge^{K}}(H^{\bullet}(X,\xi|_{X}),h^{H^{\bullet}(X,\xi|_{X})})\right).
\end{split}
\end{equation}
       The differential form $T_{g,K}(\omega,h^{\xi})\in P^{S'}$ satisfies
   \begin{multline}
  \frac{\Bar{\partial}\partial}{2i\pi}T_{g,K}(\omega,h^{\xi})=\chern_{ge^K}\left(H^{\bullet}(X,\xi|_{X}),h^{H^{\bullet}(X,\xi|_{X})}\right)\\ -\int_{X_{g}}\td_{g,K}(TX,h^{TX})\ch_{g,K}(\xi,h^{\xi}).
 \end{multline}

For fixed $K_{0}\in \mathfrak{z}(g)$, the form $T_{g,zK_{0}}(\omega,h^{\xi})$ is an analytic function in $z\in \mathbb{C}$ near $0$. 
    \end{theorem}

    The differential form $T_{g,K}(\omega,h^{\xi})$ is called the equivariant infinitesimal analytic torsion form, whose zero degree part coincides with the  equivariant infinitesimal analytic torsion of  Bismut-Goette \cite[(2.75)]{BismutGoette2000}. As an element of $P^{S'}/ P^{S',0}$, $T_{g,K}(\omega,h^{\xi})$ does not depend on the choice of $\mu$ (cf. \S \ref{93}).

For $K=0$, the torsion form $T_{g,K}(\omega,h^{\xi})$ is exactly the equivariant analytic torsion form $T_{g}(\omega,h^{\xi})$ in \eqref{intro-02}. For $g=e$, the unit element, we simply denote $T_{g,K}(\omega,h^{\xi})$ as $T_{K}(\omega,h^{\xi})$.

One may ask how \textbf{Assumption(*)} arises naturally.
 By following their spirit in \cite[\S 2.5]{BismutGoette2000}, we may understand  $T_{K}(\omega, h^\xi)$, at the formal power series level, as universal analytic torsion forms of some further fibrations. \textbf{Assumption(*)} then implies that this further fibration is a K\"ahler fibration. To be explicit, let $G_{\cc}$ be the complexification of the Lie group $G$, by \cite{GuS}, $G_{\cc}$ acts holomorphically fiberwise on $W$ and its action lifts to $\xi$. Let $B$ be a complex manifold. Let $p: P\rightarrow B$ be a holomorphic $G_{\cc}$-principal  bundle, and we choose a $G$-reduction of $P$, $q:Q\rightarrow B$ , then the $G$-principal bundle $Q$ is canonically equipped with a connection which lifts to a holomorphic connection on $P$ (cf. \eqref{9a}), let $\Omega$ be the associated curvature, this is a differential form on $P$ with value in $\mathfrak{g}$. With \textbf{Assumption (*)}, the natural fibration $\widetilde{\pi}:P\times_{G_{\cc}}W\rightarrow B\times S$ is a $\kahler$ fibration with fiber $X$ (cf. Proposition \ref{9e}). And $P\times_{G_{\mathbb{C}}}\xi$ is a holomorphic vector bundle on $P\times_{G_{\cc}}W$. For this fibration, the squre of Bismut superconnection has the same structure as the operator used to define $T_{K}(\omega, h^\xi)$, with $K$ replaced by $\Omega$ (cf. Proposition \ref{9r}). Therefore, if $\widetilde{T}$ is the corresponding Bismut-$\mathrm{K\Ddot{o}hler}$'s analytic torsion form of the fibration $\widetilde{p}$, we would have
\begin{equation}
    \widetilde{T}=T_{\frac{i}{2\pi}\Omega}(\omega, h^\xi).
\end{equation}
Similar results hold for $T_{g,K}(\omega, h^\xi)$ with general $g\in G$, if we replace $G_{\mathbb{C}}$ by $Z(g)_{\mathbb{C}}$ in the arguments above.

Note that we aim to study the analytic properties of $T_{g}(\omega,h^{\xi})$ as a function of $g\in G$, which is difficult to solve directly. So we will follow Bismut-Goette \cite{BismutGoette2000} to establish a comparison formula between $T_{ge^{zK_{0}}}(\omega, h^\xi)$ and $T_{g,zK_{0}}(\omega, h^{\xi})$.  By Theorem \ref{8x}, the latter one is  analytic  in $z$, so the comparison formula enables us to study $T_{ge^{zK_{0}}}(\omega, h^{\xi})$ in the variable $z$ near $z=0$. Now we try to explain several terms appearing in the final formula.

 For $g\in G$, $K\in \mathfrak{z}(g)$, put 
 \begin{align}\label{90}
 W_{K}=\left\{ x\in W: K^{X}(x)=0 \right\}.
 \end{align}
 The restriction of $\pi$ to $W_{K}$ is K\"ahler fibration with fiber $X_{K}$. Put $W_{g,K}=W_{g}\cap W_{K}$, $X_{g,K}=X_{g}\cap X_{K}$.

In \cite[Def. 2.9]{B1992Inter}, Bismut introduced the equivariant Bott-Chern currents which, in complex geometry, refine the localization formulas of Berline and Vergne \cite[Theorem 7.13]{BGV}. In our case, the equivariant Bott-Chern current $S_{K}(W_{g},\omega)$ is a current on $W_{g}$ which satisfies
\begin{equation}
    \frac{\bar{\partial}_K\partial_K}{2i\pi}{S}_K(W_{g},\omega)=1-\frac{\delta_{W_{g,K}}}{c_{\mathrm{max},K}(N_{X_{g,K}/X_{g}},h^{N_{X_{g,K}/X_{g}}})}
\end{equation}
for $K\in\mathfrak{z}(g)$ with $|K|$ small enough. Here $\dba_K$, $\partial_K$ are the anti-holomorphic and holomorphic parts of $d_{K}=d-2i\pi i_{K^{X}}$, $N_{X_{g,K}/X_{g}}$ is the normal bundle to $X_{g,K}$ in $X_{g}$, and $c_{\text{max},K}$ denotes the equivariant Eular class.

For $\theta,\theta'\in \mathbb{R}$, $x\in \mathbb{C}$,  $\theta\notin 2\pi \zz$,  $|\theta'|$, $|x|$ small enough, or for $\theta\in 2\pi\zz$, $|\theta'|$, $|x|$ small, with $x+i\theta'\neq 0$, set 
\begin{equation}
    I({\theta,\theta',x})=\sum_{k\in\zz, 2k\pi+\theta\neq 0}\frac{\log(1+\frac{\theta'}{2k\pi+\theta})}{i(2k\pi +\theta+\theta')+x}.
\end{equation}
 Let $N_{X_K/X}$ be the normal bundle to $X_K$ in $X$, Bismut and Goette \cite[\S 4]{BismutGoette2000} associated an additive genus $I_{g,K}(N_{X_K/X})$ on $X_{g,K}$  to the function $I(\theta,\theta',x)$, as is explicitly explained in Section 4.3. The $I$-genus is closely related to the $R(\theta,\cdot)$-genus (cf. \cite[Theorem 4.2]{BismutGoette2000}), and in particular, it is analytic in $K\in\mathfrak{z}(g)$ in a neighborhood of $0$. 

Now we can state our main result as follows, 
\begin{theorem}\label{8o}
  With \textbf{Assumption (*)},  for $g\in G$ and $K_0\in \mathfrak{z}(g)$, for any compact subset $S'\subset S$, there exists $\beta>0$ such that for $z\in \rr$ with $0<|z|<\beta$, for $K=zK_0$, we have
    \begin{align}\label{91}
        \begin{split}
        T_{g,K}(\omega,h^{\xi})=T_{ge^K}(\omega, h^{\xi})+\int_{X_g}\td_{g,K}(TX,h^{TX})\ch_{g,K}(\xi,h^\xi)S_K(W_g,\omega)\qquad\\ 
        -\int_{X_{g,K}}\td_{ge^K}(TX, h^{TX})I_{g,K}(N_{X_K/X}, h^{N_{X_{K}/X}})\ch_{ge^K}(\xi,h^{\xi})\quad\text{ in } P^{S'}/P^{S',0}.
        \end{split}
    \end{align}
\end{theorem}

 If $A$ is an invertible $(q,q)$ matrix, put 
\begin{align}\label{8n}
c_{\text{max}}^{-1,\prime}(A)=\frac{\partial}{\partial b}\left[ \det (A+b) \right]^{-1}|_{b=0}.
\end{align}

In \cite{KR2}, K\"ohler and R\"ossler studied the behavior of equivariant Bott-Chern currents $S_{zK_{0}}(X,\omega)$ near $z=0$, when they prove a "Bott"-like residue formula in Arakelov geometry. In \cite{KK23}, K\"ohler studies the equivariant Bott-Chern current in more detail, in particular by \cite[Cor. 4.2]{KK23}, if $S$ is compact, we have the following identity module $\overline{\partial}_K$ and $\partial_K$-exact currents,
\begin{align}\label{8m}
z^{{N'}/2+1}S_{zK_{0}}(W_{g},\omega)-S_{K_{0}}(W_{g},\omega)=\log z^{2} c_{\text{max},K_{0}}^{-1,\prime}(N_{X_{g,K}/X_{g}})\cdot \delta_{W_{g,K}},
\end{align}
where $N'$ denotes the number operator on differential forms.

Now as a corollary to Theorem \ref{8o}, we have the following result.
\begin{theorem}
     With \textbf{Assumption (*)},
  if $S$ is compact and K\"ahler, for $g\in G$, $K_{0}\in \mathfrak{z}(g)$,  there exist $C_{j}, D_{j}\in P^{S}, -(\dim X_{g}+\dim S)\leqslant j\leqslant 0$, such that for $z\in \mathbb{R}$, $z\neq 0$, as $z\to 0$,
\begin{align}\label{8p}
T_{ge^{zK_{0}}}(\omega,h^{\xi})=\sum_{j=-\dim X_{g}-\dim S}^{0}z^{j-1}\left( C_j+ \log z^2 \cdot D_j\right)+\mathcal{O}(z^{0})\text{ in } P^{S}/P^{S,0}.
\end{align}
\end{theorem}

Our paper is organized as follows.

In Section 1, we recall the construction of equivariant analytic torsion forms of Ma \cite{Ma2000}.
In Section 2, we construct the infinitesimal version of equivariant analytic torsion forms. We will use some estimates on several supertraces involving $\exp\left(-\mathcal{L}_K-\bisconn_{K,u}^2\right)$, whose proof will be deferred to Section 6. 
In Section 3, we recall the construction of equivariant Bott-Chern current $S_{K}(W_{g},\omega)$ of \cite{B1992Inter}.

In Section 4, we study the infinitesimal equivariant analytic torsion form when the fibration $\pi$ is a vector bundle. The result of this section will be in particular useful in Sections 5.6-5.7.

In Section 5,  we will construct a two parameter operator $\mathcal{A}_{v,u}, u,v>0$, which links the $\bisconn_{u}$ with $\bisconn_{K,u}$. With some estimates on the heat operator $\exp\left(-\mathcal{L}_K-\mathcal{A}_{v,u}^{2}\right)$, we will express the difference of the two sides of \eqref{91} as the $\partial$ and $\overline{\partial}$ of some integrals on a domain in $\mathbb{R}^{2}$, by following the steps of \cite[\S 6]{BismutGoette2000}. If $S$ is K\"ahler and compact, the subspace $P^{S,0}\subset P^{S}$ is closed, thus we obtain the comparison formula \eqref{91} directly. In the general case, we have to take care of the two-parameter integrals carefully, our proof is inspired by \cite[\S 6.6-6.8]{B1997}.

In Section 6, we study the asymptotic of the supertraces of $\exp\left(-\mathcal{L}_K-\bisconn_{K,u}^2\right)$. The technique developped by Bismut-Goette in \cite[\S 7]{BismutGoette2000} plays a key role in our analysis. The results in this Section will complete the proof of Theorem \ref{8x}.
In Sections 7-8, we prove the intermediate lemmas in Section 5.

\section{Equivairant Analytic Torsion Forms}
In this section, we recall the construction of the equivariant analytic torsion forms in \cite[Section 2]{Ma2000}. In Section 1.1, we recall elementary results on Clifford algebras and complex vector spaces. In Section 1.2, we explain the $\kahler$ fibration. In Section 1.3, we introduce the Bismut superconnection. In Section 1.4, we recall the double tansgression of certain superconnection forms. Finally in Section 1.5, we construct the equivariant analytic torsion forms and state some basic properties.

\subsection{The Clifford action on exterior algebras}
Let $V$ be a complex Hermitian vector space of dimension $\ell$, with underlying real vector space $V_{\rr}$. Let $\overline{V}$ be the conjugate vector space of $V$. We denote by $\langle\,,\,\rangle$ the scalar product on $V_{\mathbb{R}}$, and we extend it to a $\mathbb{C}$-bilinear form on $V_{\mathbb{R}}\otimes \mathbb{C}$. Let $c(V_{\mathbb{R}})$ be the Clifford algebra of $V_{\mathbb{R}}$, that is, the algebra over $\mathbb{R}$ spanned by $1\in \mathbb{R}$, $X\in V_{\mathbb{R}}$ with the relation 
\begin{align}\label{98}
XY+YX=-2\left\langle X,Y \right\rangle.
\end{align}

If $X\in V$, let $\overline{X}\in \overline{V}$ be the conjugate element of $X$, and  $X^*\in \overline{V}^*$ be the metric dual of $X$. Set
\begin{equation}\label{clifford action}
    c(X)=\sqrt{2}X^{*}\wedge,\quad c(\overline{X})=-\sqrt{2}i_{\overline{X}}.
\end{equation}
The actions in \eqref{clifford action} induce a representation of Clifford algebras:
\begin{equation}
    c: c(V_{\rr})\otimes \mathbb{C}\rightarrow \ENDO(\Lambda({V}^{*(0,1)})).
\end{equation}

If $A\in \ENDO(V)$, then $A$ acts naturally on $\Lambda({V}^{*(0,1)})$. Let $e_1,... e_{2\ell}$ be an orthonormal basis of $V_{\rr}$, then we have
\begin{equation}\label{action on exterior algebra}
    A|_{\Lambda(V^{*(0,1)})}=\frac{1}{4}\langle A e_i,e_j\rangle c(e_i)c(e_j)+\frac{1}{2}\trace[A].
\end{equation}

\subsection{Bismut superconnection}\label{bl}
Let $\pi: W\rightarrow S$ be a holomorphic submersion of complex manifolds, with $\ell$-dimensional compact fibers $X$. Let $TX=TW/S$  be the relative tangent bundle to the fibers $X$. Let $T_{\rr}X$ be the corresponding real tangent bundle and $J^{TX}$ be the complex structure on it.

We suppose that there is a real closed $(1,1)$-form $\omega$ on $W$, which restricts to be a K\"ahler form $\omega^{TX}$ on $TX$. For 
 $U,V\in T_{\rr} X$, set 
 \begin{align}\label{92}
 g^{TX}( U, V)=\omega(J^{TX}U,V).
 \end{align}
We denote by $h^{TX}$ the Hermitian metric on $TX$ induced by $g^{TX}$.

Put  
\begin{align}
 T^{H}W=\{U\in TW: \omega(U,{\overline{V}})=0 \text{ for any }  V\in TX\}.
 \end{align}
Then $T^{H}W$ is a horizontal subbundle of $TW$. We have the decomposition \begin{equation}\label{splitting}
    TW=T^{H}W\oplus TX.
\end{equation}
as smooth vector bundles. By \cite[Definition 1.4]{BGS2}, the triple $(\pi, h^{TX}, T^{H}W)$ defines a K\"ahler fibration.

From the splitting \eqref{splitting}, we have the following identification, 
\begin{align}\label{94}
T^{H}W \simeq \pi^{*}TS.
\end{align}
If $U$ is a vector in $T_{\mathbb{R}}S$, let $U^H\in T_{\rr}^HW$ be its horizontal lift in $T_{\rr}^HW$, so that $\pi_{*}U^{H}=U$.

If $U,V\in T_{\mathbb{R}}S$, set 
\begin{equation}\label{95}
    \omega^{H}(U,V)=\omega(U^{H},V^{H}).
\end{equation}

 Let $\xi$ be a holomorphic vector bundle on $W$. Let $h^\xi$ be a Hermitian metric on $\xi$. Let $\nabla^{TX}$, $\nabla^{\xi}$ be  the holomorphic Hermitian connections on $(TX,h^{TX})$, $(\xi, h^\xi)$, with curvatures $R^{TX}$, $R^{\xi}$.  Let $\nabla^{\Lambda(T^{*(0,1)}X)}$ be the connection on $\Lambda(T^{*(0,1)}X)$ induced by $\nabla^{TX}$. Set 
\begin{equation}
    \nabla^{\Lambda(T^{*(0,1)}X)\otimes\xi}=\nabla^{\Lambda(T^{*(0,1)}X)}\otimes 1+1\otimes \nabla^{\xi}.
\end{equation}

Let $g^{TS}$ be a Riemannian metric on $T_{\rr}S$. We equip $T_{\rr}W$ with the metric
\begin{equation}\label{metric on W}
    g^{TW}=\pi^{*}g^{TS}\oplus g^{TX}.
\end{equation}
Let $\nabla^{T_{\rr}W,L}$ (resp.$\nabla^{T_{\rr}S}$) be the Levi-Civita connection on ($T_{\rr}W,g^{TW}$) (resp. ($T_{\rr}S,g^{TS}$)).

 Let $\nabla^{T_{\rr}W}=\pi^{*}\nabla^{T_{\rr}S}\oplus\nabla^{T_{\rr}X}$ be the  connection on $T_{\rr}W=T_{\mathbb{R}}^HW\oplus T_{\mathbb{R}}X$.
 Set 
\begin{equation}
    S=\nabla^{T_{\rr}W,L}-\nabla^{T_{\rr}W}.
\end{equation}
Then $S$ is a 1-form on $W$ which takes value in antisymmetric elements of $\ENDO(T_{\rr}W)$.

Let $T$ be the torsion of $\nabla^{T_{\rr}W}$. By \cite[\S 1c]{B1986}, $T$ is a two form on $W$ with value in $T_{\mathbb{R}}X$, and if $U,V,Z\in T_{\rr}W$,
\begin{equation}\label{S and T}
    \begin{split}
       & S(U)V-S(V)U+T(U,V)=0,\\
       & 2\langle S(U)V,Z\rangle+\product{T(U,V),Z}+\product{T(Z,U),V}-\product{T(V,Z),U}=0.
    \end{split}
\end{equation}

 Let $\mathcal{L}_{U^H}$ be the Lie derivative operator associated with the vector field $U^H$. Then $\mathcal{L}_{U^{H}}$ acts on the tensor algebra of $T_{\rr}X$. In particular, if $U\in T_{\rr}S,$ $(g^{TX})^{-1}\mathcal{L}_{U^{H}}g^{TX}$ defines a self-adjoint endomorphism of $T_{\rr}X$. We recall some properties in \cite[\S 1.1, 2.1]{B1997}.
\begin{proposition}\label{prop T}
    1)\, The tensors $T$ and $\product{S(\cdot)\cdot,\cdot}$ do not depend on $g^{T_{\rr}S}$. And the following identities hold:

    -If $U,V\in T_{\rr}S$, 
    \begin{equation}\label{9n}
        T(U^{H},V^{H})=-P^{TX}[U^{H},V^{H}].
    \end{equation}

    -If $U\in T_{\rr}S, Z\in T_{\rr}X$,
    \begin{equation}
        T(U^H,Z)=\frac{1}{2}(g^{TX})^{-1}\mathcal{L}_{U^{H}}g^{TX}Z.
    \end{equation}

    -If $Z,Z'\in T_{\rr}X$,
    \begin{equation}
        T(Z,Z')=0.
\end{equation}

2)\, For $Z\in T_{\mathbb{R}}X$, the 2-forms $\left\langle S(Z)\cdot, \cdot \right\rangle$ and $T$ are of complex type $(1,1)$.
\end{proposition}

\begin{definition}\label{9g}
    For $0\leq p\leq \ell$, $s\in S$, let $\mathbb{E}_s^p$ be the vector space of smooth sections of $(\Lambda(T^{*(0,1)}X)\otimes\xi)$ $|_{X_s}$ over the fiber $X_s=\pi^{-1}(s)$. Set
    \begin{equation}
        \mathbb{E}_s=\bigoplus_{p=0}^{\ell}\mathbb{E}_s^{p}\mathrm{\;},\mathrm{\;} \mathbb{E}_{s,+}=\bigoplus_{p\mathrm{\, even}}\mathbb{E}_s^p\mathrm{\;},\mathrm{\;} \mathbb{E}_{s,-}=\bigoplus_{p\mathrm{\, odd}}\mathbb{E}_s^p.
    \end{equation}
\end{definition}

As in \cite[\S 1f)]{B1986}, we can view $\mathbb{E}_s, s\in S$ as the fibers of a smooth $\zz$-graded infinite dimensional vector bundle over $S$. And the smooth sections of $\mathbb{E}$ over $S$ are identified to smooth sections of $\Lambda(T^{*(0,1)}X)\otimes\xi$ over $W$.

Let $dv_{X_s}$ be the Riemannian volume form on the fiber $(X_s,g^{TX})$. Let $h^{\Lambda(T^{*(0,1)}X)\otimes  \xi}$ be the Hermitian metric on $\Lambda(T^{*(0,1)}X)\otimes\xi$ induced by $h^{TX}$, $h^{\xi}$. We equip $\mathbb{E}_s$ with the Hermitian product
\begin{equation}\label{L2 metric}
    \alpha,\alpha'\in \mathbb{E}_s\rightarrow\product{\alpha,\alpha'}_{0}=\Big(\frac{1}{2\pi}\Big)^\ell\int_{X_s}\product{\alpha,\alpha'}_{h^{\Lambda(T^{*(0,1)}X)}\otimes h^\xi}dv_{X_s}.
\end{equation}

Let $\dbar$ be the fiberwise Dolbeault operator acting on $\mathbb{E}_s$ and let $\dbarstar$ be its formal adjoint with respect to the $L^2$ metric (\ref{L2 metric}). Set 
\begin{equation}
    D^{X}=\dbar+\dbarstar.
\end{equation}
Since for each $s\in S$, the 2-form $\omega$ restricts to be a $\kahler$ form on $X_s$, therefore $\sqrt{2}D^X$ is  a fiberwise Dirac operator.

For any $U\in T_{\rr}S$, if $s$ is a smooth section of $\mathbb{E}$ over $s$, set
\begin{equation}\label{9i}
    \nabla_{U}^{\mathbb{E}}s=\nabla_{U^H}^{\Lambda(T^{*(0,1)}X)\otimes\xi} s. 
\end{equation}
By \cite[Theorem 1.14]{BGS2}, $\nabla^{\mathbb{E}}$ defines a Hermitian connection with respect to the $L^2$ metric (\ref{L2 metric}).  Let $\nabla^{\mathbb{E}'}$ and $\nabla^{\mathbb{E}''}$ be the holomorphic and antiholomorphic part of $\nabla^{\mathbb{E}}$.

Let $\{f_p\}$ be a local frame of $T_{\rr}S$ and $\{f^p\}$ be its dual. Set
\begin{equation}\label{cth}
    c(T)=\frac{1}{2}f^p f^q c(T(f_p^{H},f_q^{H})).
\end{equation}
Then $c(T)$ is a section of $\pi^{*}\Lambda^2(T_{\rr}^*S)\widehat{\otimes}\ENDO(\Lambda(T^{*(0,1)}X)\otimes\xi)$. Let $T^{(1,0)}$, $T^{(0,1)}$ be the components of $T$ in $T^{(1,0)}X$, $T^{(0,1)}X$, we define $c(T^{(1,0)})$ an $c(T^{(0,1)})$ as in (\ref{cth}) so that
\begin{equation}
    c(T)=c(T^{(1,0)})+c(T^{(0,1)}).
\end{equation}

\begin{definition}
    For $u>0$, let $\bisconn_u$ be the Bismut superconnection constructed in \cite[\S 3]{B1986} and \cite[\S 2a]{BGS2},
    \begin{equation}\label{Bismut superconnection}
        \begin{split}
            &\bisconn_u''=\sqrt{u}\,\dbar+\nabla^{\mathbb{E}''}-\frac{c(T^{(1,0)})}{2\sqrt{2u}},\\
              &\bisconn_u'=\sqrt{u}\,\dbarstar+\nabla^{\mathbb{E}'}-\frac{c(T^{(0,1)})}{2\sqrt{2u}},\\
              &\bisconn_u=\bisconn_u''+\bisconn_u'.
        \end{split}
    \end{equation}
\end{definition}

Let $G$ be a compact Lie group, which acts holomorphically and fiberwisely on $W$. We assume that the action of $G$ lifts to a holomorphic action on $\xi$.
We suppose that the action of $G$ also preserves the $(1,1)$-form $\omega$, and the Hermitian metric $h^{\xi}$. Then the Hermitian metric $h^{TX}$  and the subbundle $T^HW\subset TW$ also $G$-invariant.

The group $G$ acts on $\mathbb{E}=\cinfty(X,\Lambda(T^{*(0,1)}X)\otimes\xi)$ by
\begin{equation}\label{g on section}
    (g\cdot s)(x)=g\cdot s(g^{-1}x)\quad \mathrm{for\;any\;} g\in G.
\end{equation}
Then the Bismut superconnection $\bisconn_{u}, u>0$ commute with the action of $G$.

\subsection{Superconnection forms and transgression formulas}

Let $N_V$ be the operator which defines the $\zz$-grading on $\mathbb{E}$, i.e. $N_V$ acts by multiplication by $p$ on $\mathbb{E}^{p}$.
\begin{definition}
    Set
    \begin{equation}\label{number operator-g}
        N_u=N_V+\frac{i\omega^H}{u}.
    \end{equation}
    Then $N_u$ is a section of $\pi^{*}\Lambda^2(T_{\rr}^{*}S)\widehat{\otimes}\ENDO(\Lambda(T^{*(0,1)}X)\otimes\xi)$.
\end{definition}

We recall some results in \cite[Theorem 2.6]{BGS2}.
\begin{theorem}\label{identities-g}
    The following identities hold,
    \begin{equation}
        \begin{split}
            &[\bisconn_u'',N_u]=-2u\frac{\partial}{\partial u}\bisconn_u''\mathrm{\;},\mathrm{\;}[\bisconn_u',N_u]=2u\frac{\partial}{\partial u}\bisconn_u'.\\
            &\bisconn_u'^{2}=0\kongge,\kongge \bisconn_u''^{2}=0.\\
            &\bisconn_u^{2}=[\bisconn_u'',\bisconn_u'].
        \end{split}
    \end{equation}
\end{theorem}

\begin{definition}
    Let $P^{S}$ be the vector space of smooth forms on $S$, which are sums of forms of type $(p,p)$. Let $P^{S,0}$ be subspace of the $\omega\in P^{S}$ of the form $\partial\alpha+\overline{\partial}\beta$, where $\alpha,\beta$ are smooth forms on $S$. The vector spaces $P^{W}$, $P^{W,0}$ are defined in a similar way.
\end{definition}

Set $\Phi:\Lambda^{\text{even}}(T_{\rr}^*S)\rightarrow \Lambda^{\text{even}}(T_{\rr}^*S) $ be the map $\alpha\rightarrow (2i\pi)^{-\mathrm{deg}\alpha/2}\alpha$. Then we recall the double transgression formula in \cite[Theorem 2.9]{Ma2000}, which extends the result of \cite[Theorem 2.9]{BGS2}. Note that in our case, $g$ acts trivially on the base space $S$, so $S_g=S$,
\begin{theorem}\label{double transgression}
    For $u>0$, the forms $\Phi \Tr[g\exp(-\bisconn_u^2)]$ and $\Phi\Tr[gN_u\exp(-\bisconn_u^{2})]$ lie in $P^{S}$. The forms $\Phi \Tr[g\exp(-\bisconn_u^2)]$ are closed and their cohomology classes remain constant. Moreover,
    \begin{equation}
        \frac{\partial}{\partial u}\Phi \Tr[g\exp(-\bisconn_u^2)]=-\frac{1}{u}\frac{\overline{\partial}\partial}{2i\pi}\Phi \Tr[gN_u\exp(-\bisconn_u^2)].
    \end{equation}
\end{theorem}

\subsection{Equivariant analytic torsion forms}

Fix a $g\in G$, set $W_g=\{x\in W: gx=x\}$ the fixed point set of $g$. Then we have a holomorphic submersion $\pi|_{W_{g}}: W_{g}\to S$ with compact fiber $X_{g}$.

Let $N_{W_g/W}$ be the normal bundle to $W_{g}$ in $W$, thus
\begin{equation}
    N_{W_g/W}:=\frac{TW}{TW_g}=\frac{TX}{TX_g}=:N_{X_g/X}.
\end{equation}


Then $g$ acts on $N_{X_g/X}$, let $e^{i\theta_1},...,e^{i\theta_q}$ be the locally constant distinct eigenvalues of $g$ acting on $N_{X_g/X}$, and let $N_{X_g/X}^{\theta_1},...,N_{X_g/X}^{\theta^q}$ be the corresponding eigenbundles. Then $TX$ has a holomorphic and orthogonal splitting:
\begin{equation}\label{splitting of tangent space-g}
    TX=TX_g\oplus N_{X_g/X}=TX_g\oplus_{1\leq j\leq q}N_{X_g/X}^{\theta_j}.
\end{equation}
Let $h^{TX_g}, h^{N_{X_g/X}^{\theta_1}}\dots.$ be the  metrics on $TX_g,N_{X_g/X}^{\theta_1}\dots.$ induced by the Hermitian metric $g^{TX}$. And let $\nabla^{TX_g},\nabla^{N_{X_g/X}^{\theta_1}}\dots.$ be the holomorphic Hermitian connection on $(TX_g,h^{TX_g})$, $\Big(N_{X_g/X}^{\theta_1},h^{N_{X_g/X}^{\theta_1}}\Big)\dots.$ Let $R^{TX_g},R^{N_{X_g/X}^{\theta_1}}\dots.$ be the corresponding curvatures.

If $A$ is a $(q,q)$ matrix, set
\begin{equation}
    \td(A)=\mathrm{det}\left(\frac{A}{1-e^{-A}}\right)\kongge,\kongge \ch(A)=\trace[\exp(A)]\kongge,\kongge c_{\text{max}}(A)=\mathrm{det}(A).
\end{equation}
\begin{definition}
    Set
    \begin{equation}\label{genus-g}
        \begin{split}
            &\td_g(TX,h^{TX})=\td\left(\frac{-R^{TX_g}}{2i\pi}\right)\prod_{j=1}^{q} \Big(\frac{\td}{c_{\mathrm{max}}}\Big)\Big(-\frac{R^{N_{X_g/X}^{\theta_j}}}{2i\pi}+i\theta_j\Big),\\
             &\td_g'(TX,h^{TX})=\frac{\partial}{\partial b}\Big|_{{b=0}}\Big[\td\Big(\frac{-R^{TX_g}}{2i\pi}+b\Big)\prod_{j=1}^{q} \Big(\frac{\td}{c_{\mathrm{max}}}\Big)\Big(-\frac{R^{N_{X_g/X}^{\theta_j}}}{2i\pi}+i\theta_j+b\Big)\Big],\\
              &\ch_g(\xi,h^\xi)=\trace\left[g\exp\left(\frac{-R^{\xi}|_{X_g}}{2i\pi}\right)\right].
        \end{split}
    \end{equation}
\end{definition}
Then the forms in (\ref{genus-g}) are closed forms on $W_g$ and lie in $P^{W_g}$.

For $s\in S$, let $H^p(X_s,\xi|_{X})$ be the $p$-th cohomology of the sheaf of holomorphic sections of $\xi|_{X_s}$.

We make the basic \textbf{assumption} that for $0\leq p \leq \ell$, the dimension of $H^{p}(X_s,\xi|_{X_s})$ is locally constant. Then these vector spaces form a $\zz$-graded holomorphic vector bundle $H^{\bullet}(X,\xi|_{X})$ over $S$. By Hodge theory,  we have an identification of $G$-spaces,
\begin{equation}\label{Hodge}
    H^p(X_s,\xi|_{X_s})\cong \mathrm{Ker}(D|_{\mathbb{E}_s^p})=\{f\in \mathbb{E}_s^p: \dbar f=0, \kongge \dbarstar f=0\}.
\end{equation}
As a $G$-invariant subspace of $\mathbb{E}_s$, $\mathrm{Ker}(D)$ inherits a $G$-invariant metric from the $L^{2}$-metric \eqref{L2 metric}, let $h^{H^{\bullet}(X,\xi|_{X})}$ be the corresponding $G$-invariant metric on $H^{\bullet}(X,\xi|_{X})$ by the identification (\ref{Hodge}). 

Set
\begin{equation}\label{D formula}
    \begin{split}
        &\ch_g(H^{\bullet}(X,\xi|_{X}),h^{H(X,\xi|_{X})})=\sum_{k=0}^{\ell}(-1)^{k}\ch_g(H^{k}(X,\xi|_{X}),h^{H^{\bullet}(X,\xi|_{X})}),\\
        &\ch_g'(H^{\bullet}(X,\xi|_{X}),h^{H^{\bullet}(X,\xi|_{X})})=\sum_{k=0}^{\ell}(-1)^{k}k\ch_g(H^{k}(X,\xi|_{X}),h^{H^{\bullet}(X,\xi|_{X})}),\\
        &D_{-1,g}=\int_{X_g}\frac{\omega}{2\pi}\td_g(TX,h^{TX})\ch_g(\xi,h^\xi),\\
        &D_{0,g}=\int_{X_g}\big(-\td_{g}'(TX,h^{TX})+\ell\cdot \td_g(TX,h^{TX})\big)\ch_g(\xi,h^\xi).
    \end{split}
\end{equation}

Now we state a result proved in \cite[Theorem 2.10]{Ma2000}.
\begin{theorem}\label{thm asymptotic-g}
There are forms $A_{j,g}, D_{j,g}'\in P^{S}, j\geq -1$ such that for any $k\in\nn$, as $u\rightarrow 0$,
\begin{equation}\label{asymptotic-gn}
    \begin{split}
        &\Phi\Tr[g\exp(-\bisconn_u^2)]=\sum_{j=0}^{k}A_{j,g}u^{j}+\mathcal{O}(u^{k+1}),\\
         & \Phi\Tr[gN_u\exp(-\bisconn_u^2)]=\sum_{j=-1}^{k}D_{j,g}'u^{j}+\mathcal{O}(u^{k+1}).
    \end{split}
\end{equation}
Moreover we have
\begin{equation}\label{DDprime}
    \begin{split}
        &D_{-1,g}'=D_{-1,g},\\
        &A_{0,g}=\int_{X_g}\td_{g}(TX,h^{TX})\ch_{g}(\xi,h^\xi),\\ 
    &\kongge D_{0,g}'=D_{0,g}\kongge\qquad\quad\quad\text{in }\kongge P^S/P^{S,0},\\ 
        &\frac{\dba\partial}{2i\pi}D_{j,g}=-jA_{j,g}\quad\text{ for any }j\geq 0.
    \end{split}
\end{equation}
As $u\rightarrow +\infty$,
\begin{equation}\label{asymptotic-infty-g}
    \begin{split}
        &  \Phi\Tr[g\exp(-\bisconn_u^2)]=\ch_g(H^{\bullet}(X,\xi|_{X}),h^{H^{\bullet}(X,\xi|_{X})})+\mathcal{O}(\frac{1}{\sqrt{u}}),\\
         &  \Phi\Tr[gN_u\exp(-\bisconn_u^2)]=\ch_g'(H^{\bullet}(X,\xi|_{X}),h^{H^{\bullet}(X,\xi|_{X})})+\mathcal{O}(\frac{1}{\sqrt{u}}).
    \end{split}
\end{equation}
\end{theorem}

We denote by $\Gamma(s), s\in \mathbb{C}$ the Gamma function. For $s\in \cc,$ $\mathrm{Re}(s)>1$, set
       \begin{equation}
        \begin{split}
    \zeta_{1,g}(s)=-\frac{1}{\Gamma(s)}\int_{0}^{1} u^{s-1} 
    \Big(\Phi\Tr[gN_u\exp(-\bisconn_u^2)]\qquad\qquad \\ \qquad\qquad-
    \chern'_g(H^{\bullet}(X,\xi|_{X}),h^{H^{\bullet}(X,\xi|_{X})})\Big)du.
        \end{split}
\end{equation}
By using (\ref{asymptotic-gn}), we see that $\zeta_{1,g}(s)$ extends to a holomorphic function of $s\in\cc$ near $0$.

For $s\in \cc,$ $\mathrm{Re}(s)<\frac{1}{2}$, set
    \begin{equation}
        \begin{split}
    \zeta_{2,g}(s)=-\frac{1}{\Gamma(s)}\int_{1}^{+\infty} u^{s-1} 
    \Big(\Phi\Tr[gN_u\exp(-\bisconn_u^2)]\qquad\qquad \\ \qquad\qquad-
    \chern'_g(H^{\bullet}(X,\xi|_{X}),h^{H^{\bullet}(X,\xi|_{X})})\Big)du.
        \end{split}
\end{equation}
By using (\ref{asymptotic-infty-g}), we see that $\zeta_{2,g}(s)$ is a holomorphic function of $s\in\cc$ near $0$.

\begin{definition}
    Set 
    \begin{equation}\label{analytic torsion forms-g}
          T_g(\omega,h^\xi)=\frac{\partial}{\partial s}(\zeta_{1,g}+\zeta_{2,g})(0).
    \end{equation}
\end{definition}

Then $T_g(\omega,h^\xi)$ is a smooth form on $S$, and it will be called the equivariant analytic torsion forms. Using (\ref{asymptotic-gn}), (\ref{asymptotic-infty-g}), we get
\begin{equation}\label{analytic torsion expression-g}
\begin{split}
    &T_g(\omega, h^\xi)=-\int_{0}^1\Big(\Phi\Tr[gN_u\exp(-\bisconn_u^{2})]-\frac{D_{-1,g}'}{u}-D_{0,g}'\Big)\frac{du}{u}\\
    &-\int_{1}^{+\infty}
    \Big(\Phi\Tr[ gN_u\exp(-\bisconn_u^2)]-
    \chern'_g(H^{\bullet}(X,\xi{|_X}),h^{H^{\bullet}(X,\xi{|_X})})\Big)\frac{du}{u}\\
    &+D_{-1,g}'+\Gamma'(1)\Big(D_{0,g}'-\chern'_g(H^{\bullet}(X,\xi{|_X}),h^{H^{\bullet}(X,\xi{|_X})})\Big).
\end{split}
\end{equation}

Using Theorem \ref{double transgression} and (\ref{analytic torsion expression-g}), we obtain \cite[Theorem 2.12]{Ma2000} :
\begin{theorem}\label{ah}
    The form $T_g(\omega,h^\xi)$ lie in $P^{S}$. Moreover,
    \begin{equation}\label{double transgression global}
        \frac{\overline{\partial}\partial}{2i\pi}T_g(\omega,h^{\xi})=\chern_g(H^{\bullet}(X,\xi|_{X}),h^{H^{\bullet}(X,\xi{|_X})})-\int_{X_g}\td_g(TX,h^{TX})\ch_g(\xi,h^{\xi}).
    \end{equation}
\end{theorem}
\section{Equivariant Infinitesimal Analytic Torsion Forms}
In this section, we aim to extend the infinitesimal equivariant analytic torsion of Bismut and Goette \cite{BismutGoette2000} to higher degrees. In Section \ref{az}, we state the family Kirillov formula for equivariant Riemann-Roch theorem. In Section \ref{b0},  following the idea of \cite[\S 2.5]{BismutGoette2000} , we explain how to relate the family Kirillov formula to the equivariant Riemann-Roch theorem for a further fibration. In Section \ref{93}, we define the equivariant infinitesimal analytic torsion forms, and establish Theorem \ref{8x}.

We use the same assumptions and notation as in Section 1. In particular, the compact Lie group $G$ acts fiberwisely on $W$.

\subsection{Family Kirillov formula}\label{az}
  Let $\mathfrak{g}$ be the Lie algebra of $G$.
 For $K\in  \mathfrak{g}$, we denote by $K^X\in \cinfty(W,T_{\mathbb{R}}X)$ given by
\begin{equation}\label{96}
    K^{X}(x)=\frac{\partial}{\partial t}\Big|_{t=0}e^{tK}\cdot x \kongge\kongge \mathrm{for}\kongge x\in W.
\end{equation}
For any $K,K'\in \mathfrak{g}$, we have
\begin{equation}
    [K^{X},K'^{X}]=-[K,K']^{X}.
\end{equation}

For $K\in \mathfrak{g}$, let $\lk$ be the corresponding Lie derivative given by 
\begin{equation}\label{definition lk}
    \lk s=\frac{\partial}{\partial t}\Big|_{t=0}(e^{-tK}\cdot s),
\end{equation}
for smooth sections $s$.

As in  \cite[Definition 7.5]{BGV}, for $K\in \mathfrak{g}$, we set
\begin{equation}\label{moment}
\begin{split}
      &  m^{TX}(K):=\nabla_{K^X}^{TX}-\lk^{TX}\in \cinfty(W,\ENDO(TX)),\\
       &  m^{\xi}(K):=\nabla_{K^X}^{\xi}-\lk^{\xi}\in \cinfty(W,\ENDO(\xi)).
\end{split}
\end{equation}
Then $m^{TX}(\cdot)$ and $m^{\xi}(\cdot)$ are the associated moment maps on $TX$ and $\xi$.

Since the action of $G$ preserves the metrics $h^{TX}$ and $h^{\xi}$ on $TX$ and $\xi$, and the connections $\nabla^{TX}$ and $\nabla^{\xi}$ are Hermitian, the moment maps $m^{TX}(K)$ and $m^{\xi}(K)$ takes value in skew-adjoint sections of $\ENDO(TX)$ and $\ENDO(\xi)$. Moreover since the connections $\nabla^{TX}$ and $\nabla^{\xi}$ are $G$-invariant, from $(\ref{moment})$ (cf. \cite[(7.4)]{BGV}),
\begin{equation}\label{dmoment map}
\begin{split}
   & \nabla^{TX}_{\cdot}m^{TX}(K)+i_{K^{X}}R^{TX}=0,\\
   & \nabla^{\xi}_{\cdot}m^{\xi}(K)+i_{K^{X}}R^{\xi}=0.
\end{split}
\end{equation}

Let $d$ be the exterior differential on $W$, it splits as $d=\partial+\overline{\partial}$. As in \cite[(2.25)]{BismutGoette2000},  we set 
\begin{equation}\label{dk}
    \begin{split}
        & d_{K}=d-2i\pi i_{K^X},\\
        & \partial_{K}=\partial-2i\pi i_{K^{X(0,1)}},\\
        & \overline{\partial}_K=\overline{\partial}-2i\pi i_{K^{X(1,0)}}.
    \end{split}
\end{equation}
Clearly,
\begin{equation}
    \begin{split}
        d_K^{2}=-2i\pi \mathcal{L}_{K}.
    \end{split}
\end{equation}
As the action of $G$ is holomorphic,
\begin{equation}\label{dkpingfangdengyu0}
    \partial_K^2=0,\kongge \overline{\partial}_K^{2}=0.
\end{equation}
By (\ref{dk}) and (\ref{dkpingfangdengyu0}),  we get
\begin{equation}
    \begin{split}
       & [\partial_K,\overline{\partial}_K]=-2i\pi \mathcal{L}_{K},\\
       [\mathcal{L}_{K}&,\partial_K]=0,\kongge [\mathcal{L}_{K},\overline{\partial}_K]=0.
    \end{split}
\end{equation}
In particular when acting on $G$-invariant forms, the operators $\partial_k$ and $\overline{\partial}_K$ anticommute.

Put
\begin{equation}\label{equivariant curvature}
    R_K^{TX}=R^{TX}-2i\pi m^{TX}(K),\kongge \kongge R_K^{\xi}=R^{\xi}-2i\pi m^{\xi}(K).
\end{equation}
Then $R_K^{TX}$ (resp. $R_K^{\xi}$) is  the equivariant curvature of $TX$ (resp. of $\xi$).

Take $g\in G$, let $Z(g)\subset G$ be the centralizer of $g$ with Lie algebra $\mathfrak{z}(g)$, then
\begin{equation}
    \mathfrak{z}(g)=\{K\in\mathfrak{g}: \text{Ad}_{g} K=K\}.
\end{equation}

In the sequel, we fix $g\in G$ and we always take $K\in \mathfrak{z}(g)$. Put
\begin{equation}
    W_K=\{x\in W: K^X(x)=0\}.
\end{equation}
Then $W_K$ is a totally geodesic submanifold along each fiber $X$ and $\pi|_{W_K}:W_K\rightarrow S$ is a submersion with compact fiber $X_{K}$. Set
\begin{equation}
    W_{g,K}=W_g\cap W_K.
\end{equation}
If $K_0\in\mathfrak{z}(g)$, for $z\in \rr$ small enough, we have
\begin{equation}
    W_{g,zK_0}=W_{ge^{zK_0}}.
\end{equation}

Observe that $m^{TX}(K)$ in \eqref{moment} preserves the splitting (\ref{splitting of tangent space-g}). We set $m^{TX_g}(K)$, $m^{N_{X_g/X}^{\theta_1}}(K).\dots$ be the restriction of $m^{TX}(K)$ to $TX_g$, $N_{X_g/X}^{\theta_1}\dots.$ We define the corresponding equivariant curvatures $R_K^{TX_g}$, $R_{K}^{N_{X_g/X}^{\theta_1}}\dots.$ as in (\ref{equivariant curvature}).

\begin{definition}
    For $K\in\mathfrak{z}(g)$ with $|K|$ small enough, set
     \begin{equation}\label{genus-gK}
        \begin{split}
            &\td_{g,K}(TX,h^{TX})=\td\Big(\frac{-R_K^{TX_g}}{2i\pi}\Big)\cdot\prod_{j=1}^{q} \Big(\frac{\td}{c_{\mathrm{max}}}\Big)\Big(-\frac{R_K^{N_{X_g/X}^{\theta_j}}}{2i\pi}+i\theta_j\Big),\\
             &\td_{g,K}'(TX,h^{TX})=\frac{\partial}{\partial b}\Big|_{{b=0}}\Big[\td\Big(\frac{-R_K^{TX_g}}{2i\pi}+b\Big)\\ & \qquad\qquad\qquad\qquad\qquad\times\prod_{j=1}^{q} \Big(\frac{\td}{c_{\mathrm{max}}}\Big)\Big(-\frac{R_K^{N_{X_g/X}^{\theta_j}}}{2i\pi}+i\theta_j+b\Big)\Big],\\
              &\ch_{g,K}(\xi,h^{\xi})=\trace\Big[g\exp\Big(\frac{-R_K^{\xi}|_{X_g}}{2i\pi}\Big)\Big].
        \end{split}
    \end{equation}\end{definition}

    The forms in (\ref{genus-gK})  lie in $P^{W_g}$. Moreover by \cite[Theorem 7.7]{BGV},
    \begin{align}
       &d_K \td_{g,K}(TX,h^{TX})=0, &&d_K\ch_{g,K}(\xi,h^{\xi})=0.
    \end{align}

Recall that $\bisconn_u$ is the rescaled Bismut superconnection in (\ref{Bismut superconnection}). Set
 \begin{equation}\label{bku}
        \begin{split}
            &\bisconn_{K,u}''=\bisconn_u''+\frac{c(K^{X(1,0)})}{2\sqrt{2u}},\\
              &\bisconn_{K,u}'=\bisconn_u'+\frac{c(K^{X(0,1)})}{2\sqrt{2u}},\\
              &\bisconn_{K,u}=\bisconn_{K,u}''+\bisconn_{K,u}'.
        \end{split}
    \end{equation}

Then we state the Kirillov formula for the equivariant Riemann-Roch-Grothendieck theorem.
\begin{theorem}\label{thm gK asym bku}
    For any $K\in\mathfrak{z}(g)\otimes_{\mathbb{R}}\mathbb{C}$,  for $|K|$ sufficiently small, for $u>0$, the form
\begin{equation}\label{phi gK Bku Ps}
    \Phi\Tr[g\exp(-\lk-\bisconn_{K,u}^2)]\in P^{S}.
\end{equation}
     is closed, its cohomology class  does not depend on $u$ and
     \begin{equation}\label{Kirillov formula-gK}
     \begin{split}
         &\lim_{u\rightarrow 0}\Phi\trace_{\mathrm{s}}[g\exp(-\lk-\bisconn_{K,u}^2)]=\int_{X_g}\td_{g,K}(TX,h^{TX})\ch_{g,K}(\xi,h^{\xi}),\\
          &\lim_{u\rightarrow +\infty}\Phi\Tr[g\exp(-\lk-\bisconn_{K,u}^2)]=\ch_{ge^K}(H^{\bullet}(X,\xi|_{X}),h^{H^{\bullet}(X,\xi|_{X})}).
      \end{split}
    \end{equation}
\end{theorem}

The first identity in (\ref{Kirillov formula-gK}) is a special case of \cite[Theorem 2.1]{LiuMa2022}, and the second identity is a special case of \cite[(2.28)]{LiuMa2022}. Indeed, if we apply $\mathcal{E}=\Lambda(T^{*(0,1)}X)\otimes \xi$ to their formulas, the Bismut superconnection $\bisconn_{K,t}$ in their notation coincides with $\bisconn_{K,2t}$ of ours (cf. \cite[p.96]{BGS2}). The notation $\mathcal{E}/S$ in \cite{LiuMa2022} is formally $\det(T^{(1,0)}X)^{1/2}\otimes \xi$ here, thus we get $\td_{g,K}(TX,h^{TX})$ in our formula instead of their $\widehat{A}_{g,K}(TX,\nabla^{TX})$.

 For $g=e$, (\ref{Kirillov formula-gK}) is obtained in \cite{y.wang}.

\subsection{Moment maps and K\"ahler fibrations} \label{b0}\noindent{}

\textbf{Assumption (*)}: From now on, we assume that there is  a moment map $\mu$ associated to the $2$-form $\omega$ and the $G$-action on $W$, that is,
\begin{itemize}
    \item For $g\in G, x\in W$,
    \begin{equation}\label{9c}
        \mu(gx)=\text{Ad}_{g}^{*} \mu(x).
    \end{equation}
    \item  For $K\in\mathfrak{g}$, then
    \begin{equation}\label{dmuk}
        d\product{\mu,K}+i_{K^{X}}\omega=0.
    \end{equation}
\end{itemize}
Then $\mu$ restricts to the moment map for the action of $G$ on each fiber $(X_s,\omega|_{X_s}), s\in S$, and $\mu$ is constant in the horizontal direction $T^{H}W$.

Now following the idea of \cite[\S 2.5]{BismutGoette2000}, we will show how to extend our fibration by a fibration with structure group 
G, so that $\bisconn_{K,u}$ is  interpreted as the Bismut superconnection on this extended fibration. In this process, \textbf{Assumption (*)} plays a role analogues to that of Kähler fibrations.

Assume that $G$ is connected, and denote by $G_{\mathbb{C}}$ its unique complexification. According to \cite[Theorem 4.4]{GuS}, the holomorphic action of $G$ on $W$ extends uniquely to a holomorphic action of $G_{\mathbb{C}}$, and this action lifts to $\xi$.

Let $B$ be a complex manifold, and let $p: P\to B$ be a holomorphic principal bundle over $B$ with structure group $G_{\mathbb{C}}$. Since $G_{\mathbb{C}}/G$ is contractible, the $G_{\mathbb{C}}$-principal bundle $P$ can be reduced to a $G$-principal bundle $q: Q\to B$.

Let $J^{TP}$ denote the complex structure of $T_{\mathbb{R}}P$, for $x\in Q$, define
\begin{align}\label{9a}
T^{H}_{\mathbb{R},x}Q=\left\{ U\in T_{\mathbb{R},x}Q: J^{TP}U\in T_{\mathbb{R},x}Q \right\}.
\end{align}
The splitting $T_{\mathbb{R}}Q=T_{\mathbb{R}}^{H}Q\oplus \mathfrak{g}$ induce the unique connection form $\theta$ on $Q$ which lifts to a holomorphic connection on $P$. Let $\Omega\in \Omega^{2}(Q,\mathfrak{g})^{G}$ be its curvature.

Set 
\begin{align}\label{9b}
\widetilde{\pi}: \widetilde{W}=P\times_{G_{\mathbb{C}}}W=Q\times_{G}W\to B\times S.
\end{align}
Then 
 $\widetilde{\pi}$ is a holomorphic submersion with compact fiber $X$.

The splitting $T_{\mathbb{R}}Q=T_{\mathbb{R}}^{H}Q\oplus \mathfrak{g}$ induce a corresponding splitting $T_{\mathbb{R}}\widetilde{W}=T_{\mathbb{R}}^{H}Q\oplus T_{\mathbb{R}}W$. We extend the form $\omega\in \Omega^{2}(W)$ to a differential form on $\widetilde{W}$ by requiring that $i_{U}\omega=0$ for $U\in T_{\mathbb{R}}^{H}Q$.

Note that the moment map $\mu\in \cinfty(W, \mathfrak{g}^{*})$ is $G$-invariant (cf. \eqref{9c}), thus $\left\langle \mu,  \Omega \right\rangle$ is a well-defined differential form on $\widetilde{W}$.

Set   \begin{align}\label{9d}
\widetilde{\omega}=\omega+\left\langle \mu, \Omega \right\rangle\in \Omega^{2}(\widetilde{W}).
\end{align}

We now state the following analogue of \cite[Proposition 2.11]{BismutGoette2000}.

\begin{proposition}\label{9e}
The differential form $\widetilde{\omega}$ is a real closed $(1,1)$-form on $\widetilde{W}$ which restricts to $\omega^{TX}$ on each fiber of $\widetilde{\pi}$. 
\end{proposition}
\begin{proof}
By using \eqref{dmuk} in \textbf{Assumption (*)}, we get $d\widetilde{\omega}=0$. Since $\omega$ is a $(1,1)$-form and $p: P\to B$ is holomorphic, we see that $\widetilde{\omega}$ is of complex type $(1,1)$. The curvature $\Omega\in \Omega^{2}(Q,\mathfrak{g})^{G}$ is horizontal, thus $\widetilde{\omega}|_{T_{\mathbb{R}}X}=\omega|_{T_{\mathbb{R}}X}=\omega^{TX}$. 
\end{proof}

Note that the decomposition $T\widetilde{W}=T^HQ\oplus TW=T^{H}Q\oplus T^H W\oplus TX$ is orthogonal with respect to $\widetilde{\omega}$, thus the triple $(\widetilde{\pi}, h^{TX}, T^{H}Q\oplus T^{H}W)$ is a K\"ahler fibration in the sense of Section \ref{bl}. 

We denote by $\widetilde{\omega}^{H}$ the horizontal part of $\widetilde{\omega}$ as in \eqref{95}, so we get 
\begin{align}\label{9f}
\widetilde{\omega}^H=\omega^{H}+\left\langle \mu, \Omega \right\rangle.
\end{align}

Observe that $\widetilde{\xi}=P\times_{G_{\mathbb{C}}}\xi$ forms a holomorphic vector bundle over $\widetilde{W}$, whose Hermitian metric is induced by $h^{\xi}$. 
Let $\widetilde{\mathbb{E}}$ be the analogue of $\mathbb{E}$ associated with $(\widetilde{W},\widetilde{\xi})$ (cf. Definition \ref{9g}).  Regarding $\widetilde{\mathbb{E}}$ as an infinite dimensional vector bundle over $B\times S$,  we have 
\begin{align}\label{9h}
\widetilde{\mathbb{E}}=P\times_{G_{\mathbb{C}}}\mathbb{E}=Q\times_{G}\mathbb{E}.
\end{align}
If $U\in T_{\mathbb{R}}B$, we denote by $U^{H}\in T_{\mathbb{R}}^{H}Q$ the horizontal lift of $U$. For $U\in T_{\mathbb{R}}B$, we define $\nabla_{U}^{\widetilde{\mathbb{E}}}$ in a similar way as in \eqref{9i}.

Put 
\begin{align}\label{ax}
\mathscr{E}=P\times_{G_{\mathbb{C}}}\cinfty(S,\,{\mathbb{E}})=Q\times_G \cinfty(S,\, {\mathbb{E}}).
\end{align}
Then $\mathscr{E}$ form an infinite dimensional vector bundle over $B$. Let $\nabla^{\mathscr{E}}$ be the connection on $\mathscr{E}$ induced by the connection form $\theta$ on $Q$. It is clear that for $U\in T_{\mathbb{R}}B$, $\nabla_U^{\mathscr{E}}$ and $\nabla_U^{\widetilde{\mathbb{E}}}$ coincide as operators on $\cinfty(B,\mathscr{E})\simeq \cinfty(B\times S,\, \widetilde{\mathbb{E}})$.

Let $\left\{ K_i \right\}$ be a basis of $\mathfrak{g}$. We write $\Omega\in \Omega^2(Q,\mathfrak{g})$ in the form 
\begin{align}\label{av}
\Omega=\sum \Omega_i  K_i,
\end{align}
where $\Omega_i\in \Omega^2(Q)$. Then the curvature of $\nabla^{\mathscr{E}}$ is
\begin{align}\label{aw}
\big(\nabla^{\mathscr{E}}\big)^2=\mathcal{L}_{-\Omega}=\sum \Omega_i \mathcal{L}_{-K_i}.
\end{align}

Let $\widetilde{T}$ be the $(1,1)$-form on $\widetilde{W}$ with values in $T_{\mathbb{R}}X$, which corresponds to the form $T$ on $W$ in \eqref{S and T}. Put $\Omega^X=\sum \Omega_i K_i^X\in \Omega^2(Q,T_{\mathbb{R}}X)$.

By \eqref{9n}, we have the proposition as follows.

\begin{proposition}\label{9q}
  The following identity holds, 
   \begin{align}\label{ad}
   \widetilde{T}(U^{H},V^{H})=\begin{cases}
   \Omega^{X}(U^{H},V^{H})  &\text{if}\quad U,V\in T_{\mathbb{R}}B,\\ T(U^{H},V^H)  &\text{if}\quad U,V\in T_{\mathbb{R}}S,\\0  &\text{if}\quad U\in T_{\mathbb{R}}B, V\in T_{\mathbb{R}}S.
   \end{cases}
   \end{align}
\end{proposition}

Recall that the Bismut superconnection $\bisconn_{u}, u>0$ (see \eqref{Bismut superconnection}) commutes with the action of $G$. We continue to denote by $\bisconn_{u}$ the induced operator on $\widetilde{W}$ via \eqref{9h}.

For $u>0$, let $\widetilde{\bisconn}_u$ be the Bismut superconnection on $\widetilde{W}$. By Proposition \ref{9q} and \eqref{Bismut superconnection},  we get 
\begin{align}\label{9p}
\widetilde{\bisconn}_u=\bisconn_{u}+\nabla^{\mathscr{E}}-\frac{c(\Omega^{X})}{2\sqrt{2u}}.
\end{align}

In the same manner as in \eqref{aw}, we define the operator $\mathcal{L}_{-\Omega}$, 

\begin{proposition}\label{9r}
We have the following identity, 
\begin{align}\label{9s}
\widetilde{\bisconn}_u^{2}=\mathcal{L}_{-\Omega}+\left( \bisconn_{u}-\frac{c(\Omega^{X})}{2\sqrt{2u}} \right)^2.
\end{align}
\end{proposition}
\begin{proof}
We proceed as in the proof of \cite[Theorem 10.39]{BGV}. By using the fact that $\bisconn_u$ and $\Omega$ are $G$-invariant, and the Bianchi identity on $W$, we see that the commutators  $[\bisconn_{u}, \nabla^{\mathscr{E}}]$ and $[c(\Omega^{X}),\nabla^{\mathscr{E}}]$ vanish. By using \eqref{aw}, we complete the proof of Proposition \ref{9r}.
\end{proof}

From the above, we notice that by taking $g=e$, $K=-\Omega$ in the Kirillov formula  \eqref{Kirillov formula-gK}, it indeed leads to the Riemann-Roch-Grothendieck theorem for the fibration $\widetilde{\pi}: \widetilde{W}\to B\times S$. This correspondence plays a key role throughout this article.

For any $g\in G$, if we replace $G$ by $Z(g)$ in the settings above, we will obtain an equivariant Riemann-Roch-Grothendieck theorem.

\subsection{Equivariant infinitesimal analytic torsion forms}\label{93}
Recall that the number operator $N_u$ is defined in (\ref{number operator-g}).
\begin{definition}
    Set 
    \begin{equation}\label{definition nkt}
        N_{K,u}=N_u-i\frac{\product{\mu,K}}{u}=N_V+i\frac{\omega^H}{u}-i\frac{\product{\mu,K}}{u}.
    \end{equation}
\end{definition}

\begin{proposition}\label{a2}
    The following  identities hold,
    \begin{equation}\label{bku-algebraic identity}
        \begin{split}
        &[\bisconn_{K,u}'',N_{K,u}]=-2u\frac{\partial}{\partial u}\bisconn_{K,u}''\mathrm{\;},\mathrm{\;}[\bisconn_{K,u}',N_{K,u}]=2u\frac{\partial}{\partial u}\bisconn_{K,u}'.\\
        &\bisconn_{K,u}'^{2}=0\kongge,\kongge \bisconn_{K,u}''^{2}=0.\\
        &\bisconn_{K,u}^{2}=[\bisconn_{K,u}'',\bisconn_{K,u}'].
        \end{split}
    \end{equation}
\end{proposition}
\begin{proof}
   When $K=0$, this is exactly the result of Theorem \ref{identities-g}. We will now deal with the additional terms. 
    By \eqref{clifford action}, (\ref{dmuk}),  we obtain
\begin{align}\label{af}
    &[\bisconn_{K,u}'',i\product{\mu,K}]=-\frac{\sqrt{u}c(K^{X(1,0)})}{\sqrt{2}}, &&[\bisconn_{K,u}',i\product{\mu,K}]=\frac{\sqrt{u}c(K^{X(0,1)})}{\sqrt{2}}.
\end{align}
From \eqref{af}, we get the first line in \eqref{bku-algebraic identity}. Since the action of $G$ is holomorphic, by \eqref{bku}, we have 
\begin{align}\label{ag}
&[\bisconn_{K,u}'', c(K^{X(1,0)})]=0, &&[\bisconn_{K,u}', c(K^{X(0,1)})]=0.
\end{align}
From \eqref{ag}, we obtain the remaining parts of \eqref{bku-algebraic identity}.
\end{proof}

\begin{remark}
    Note that by \eqref{9f}, we have the following identity, 
\begin{align}\label{a1}
N_{-\Omega,u}=N_{V}+i\frac{\widetilde{\omega}^H}{u}.
\end{align}
We could thus understand  Proposition \ref{a2} as a universal form of Theorem \ref{identities-g}.
\end{remark}

For $u>0$, $K\in \mathfrak{z}(g)\otimes_{\mathbb{R}}\mathbb{C}$, put   
  \begin{equation}\label{gama-ugk}
         \gamma_{u,g,K}=\Phi\Tr[gN_{K,u}\exp(-\lk-\bisconn_{K,u}^{2})].
    \end{equation}
   As in Theorem \ref{double transgression}, $\gamma_{u,g,K}\in P^{S}$.

Set \begin{equation}\label{C formula}
\begin{split}
     &C_{-1,g,K}=\int_{X_g}\left(\frac{\omega}{2\pi}-i\product{\mu,K}\right)\td_{g,K}(TX,h^{TX})\ch_{g,K}(\xi,h^{\xi}),\\
        &C_{0,g,K}=\int_{X_g}\left(-\td_{g,K}'(TX,h^{TX})+\ell\cdot \td_{g,K}(TX,h^{TX})\right)\ch_{g,K}(\xi,h^{\xi}).
\end{split}
\end{equation}
\begin{theorem}\label{thm asymptotic-gK}
    There exist $\beta>0$, $C, C'>0$, $\gamma \in\, ]0,1]$, such that for $K\in \mathfrak{z}(g)\otimes_{\mathbb{R}}\mathbb{C}$, $|K|<\beta$,  
\begin{itemize}
    \item for any $u>1$,
    \begin{equation}\label{asymptotic-infty-gK}
        \big|\gamma_{u,g,K}-\ch_{ge^K}'(H^{\bullet}(X,\xi|_{X}),h^{H^{\bullet}(X,\xi|_{X})})\big|\leq \frac{C'}{u^{\gamma}};
    \end{equation}
    \item there exist $C_{-1,g,K}',$ $C_{0,g,K}' \in P^S$, which is analytic for $|K|<\beta$, for any $u\in\, ]0,1]$,
    \begin{equation}\label{asymptotic-gK}
    \begin{split}
       & \Big|\gamma_{u,g,K}-\frac{C_{-1,g,K}'}{u}-C_{0,g,K}'\Big|\leq  C u,\\ & C_{-1,g,K}'=C_{-1,g,K},\\ 
         & C_{0,g,K}'=C_{0,g,K}\kongge\qquad\qquad\qquad\qquad\qquad \text{in }\kongge P^S/P^{S,0}.
         \end{split}
    \end{equation}
    
\end{itemize}
\end{theorem}
\begin{proof}
    Theorem \ref{thm asymptotic-gK} will be proved in Section 6.
\end{proof}
For $|K|<\beta$, $s\in \cc,$ $\mathrm{Re}(s)>1$, set
    \begin{equation}\label{9x}
        \begin{split}
    \zeta_{1,g,K}(s)=-\frac{1}{\Gamma(s)}\int_{0}^1 u^{s-1} 
    \Big(\gamma_{u,g,K}-
    \chern'_{ge^{K}}(H^{\bullet}(X,\xi|_{X}),h^{H^{\bullet}(X,\xi|_{X})})\Big)du.
        \end{split}
\end{equation}
By using (\ref{asymptotic-gK}), we see that $\zeta_{1,g,K}(s)$ extends to a holomorphic function of $s\in\cc$ near $0$.

For $|K|<\beta$, $s\in \cc,$ $\mathrm{Re}(s)<\gamma$, by (\ref{asymptotic-infty-gK})
    \begin{equation}
        \begin{split}
    \zeta_{2,g,K}(s)=-\frac{1}{\Gamma(s)}\int_{1}^{+\infty} u^{s-1}\Big(\gamma_{u,g,K}-
    \chern'_{ge^{K}}(H^{\bullet}(X,\xi|_{X}),h^{H^{\bullet}(X,\xi|_{X})})\Big)du
        \end{split}
\end{equation}
is a holomorphic function of $s\in\cc$ near $0$.

\begin{definition}
   As in (\ref{analytic torsion forms-g}), we define  the equivariant infinitesimal analytic torsion form $T_{g,K}(\omega,h^{\xi})$ as
    \begin{equation}\label{a0}
          T_{g,K}(\omega,h^{\xi})=\frac{\partial}{\partial s}(\zeta_{1,g,K}+\zeta_{2,g,K})(0).
    \end{equation}
\end{definition}

 As in (\ref{analytic torsion expression-g}), by \eqref{asymptotic-infty-gK} and \eqref{asymptotic-gK}, we have
\begin{align}\label{explicit tgk}
\begin{split}
    T_{g,K}(\omega, h^{\xi})&=-\int_{0}^1\Big(\gamma_{u,g,K}-\frac{C_{-1,g,K}'}{u}-C_{0,g,K}'\Big)\frac{du}{u}\quad\\
    &-\int_{1}^{+\infty}
    \Big(\gamma_{u,g,K}-
    \chern'_{ge^{K}}(H^{\bullet}(X,\xi|_{X}),h^{H^{\bullet}(X,\xi|_{X})})\Big)\frac{du}{u}\\
    &+C_{-1,g,K}'+\Gamma'(1)\Big(C_{0,g,K}'-\chern'_{ge^{K}}(H^{\bullet}(X,\xi|_{X}),h^{H^{\bullet}(X,\xi|_{X})})\Big).
\end{split}
\end{align}
 
\begin{theorem}\label{ai}
    We have the following identity, 
     \begin{equation}
        \begin{split}
        \frac{\overline{\partial}\partial}{2i\pi}T_{g,K}(\omega,h^{\xi})=\chern_{ge^K}(H^{\bullet}(X,\xi|_{X}),h^{H^{\bullet}(X,\xi|_{X})})\qquad\qquad\\ -\int_{X_g}\td_{g,K}(TX,h^{TX})\ch_{g,K}(\xi,h^{\xi}).
        \end{split}
    \end{equation}
\end{theorem}
\begin{proof}
The proof is the same as the proof of Theorem \ref{ah}. By using Proposition \ref{a2}, we have the transgression formula 
\begin{align}\label{ak}
        \frac{\partial}{\partial u}\Phi \Tr[g\exp(-\bisconn_{K,u}^2)]=-\frac{1}{u}\frac{\overline{\partial}\partial}{2i\pi}\Phi \Tr[gN_{K,u}\exp(-\bisconn_{K,u}^2)].
\end{align} By  \eqref{explicit tgk} and  \eqref{ak}, we obtain Theorem \ref{ai}.
\end{proof}

\begin{proposition}\label{9t}
As an element of $P^{S}/P^{S,0}$, the equivariant infinitesimal analytic torsion form $T_{g,K}(\omega,h^{\xi})$ does not depend on the choice of moment map $\mu$.
\end{proposition}
\begin{proof}
Let $\mu'$ be another moment map associated to $\omega$, then $C_{K}=\left\langle \mu'-\mu,K \right\rangle$ is a locally constant function on $W$. When we replace $\mu$ by $\mu'$,  a superscript $'$ is added to the corresponding objects. Since the operator $\mathcal{L}_K+\bisconn_{K,u}^{2}$ does not depend on $\mu$, by \eqref{definition nkt},  we have
\begin{align}\label{9w}
\gamma_{u,g,K}'-\gamma_{u,g,K}=\frac{C_{K}}{u}\Tr[g\exp(-\lk-\bisconn_{K,u}^{2})].
\end{align} 

By Theorem \ref{thm gK asym bku}, $\Tr[g\exp(-\lk-\bisconn_{K,u}^{2})]$ remain constant in the cohomology class of $S$ as $u$ varies. By \cite[\S 22, Theorem 17]{GdeRham}, $d\Omega(S,\mathbb{C})$ is closed under the uniform convergence. By \eqref{9x}, \eqref{9w}, for $s\in \mathbb{C}$, $\mathrm{Re}(s)> 1$, we get 
\begin{align}\label{9y}
\zeta_{1,g,K}'(s)-\zeta_{1,g,K}(s)=\frac{C_{K}}{\Gamma(s)(s-1)}\Tr[g\exp(-\lk-\bisconn_{K,1}^{2})]\qquad \text{mod }\, d\Omega(S,\mathbb{C}).
\end{align}

Similarly, we have 
\begin{align}\label{9z}
\zeta_{2,g,K}'(s)-\zeta_{2,g,K}(s)=-\frac{C_{K}}{\Gamma(s)(s-1)}\Tr[g\exp(-\lk-\bisconn_{K,1}^{2})]\qquad \text{mod }\, d\Omega(S,\mathbb{C})
\end{align}
on its domain of definition. 
By \eqref{a0}, \eqref{9y}, \eqref{9z}, we finish the proof of Proposition \ref{9t}.
\end{proof}

By using Duhamel's formula (see e.g. \cite[Theorem 2.48]{BGV}), 
\begin{align}\label{b1}
    \begin{split}
\frac{\partial}{\partial \bar{z}}\gamma_{u,g,zK}=\Phi\frac{\partial}{\partial b} \Tr\Big[gN_{zK,u}\exp\big(-\mathcal{L}_{zK}-\bisconn_{zK,u}^2-b\frac{\partial}{\partial \bar{z} }\bisconn_{zK,u}^2\big)\Big]_{b=0}\\
        +\Phi \Tr[g\frac{\partial}{\partial \bar{z}}N_{zK,u}\exp(-\lk-\bisconn_{zK,u}^{2})]=0.
    \end{split}
\end{align}
 Then by using the dominated convergence theorem and Theorem \ref{thm asymptotic-gK}, we see that in their domain of definitions, the functions $\zeta_{1,g,zK}(s)$, $\zeta_{2,g,zK}(s)$ are holomorphic  in the variable $z$ near $0$. This implies that $T_{g,zK}(\omega, h^{\xi})$ is holomorphic near $0$, which completes the last part of Theorem \ref{8x}.

\section{Equivariant Bott-Chern Currents}
In \cite{B1990}, \cite{B1992Inter}, Bismut constructed the equivariant Bott-Chern currents on K\"ahler manifolds, which refine the localization formulas of Berline and Vergne \cite[Theorem 7.13]{BGV}. In this section, we will apply Bismut's construction to the total space $(W,\omega)$. Although $(W,\omega)$ is not K\"ahler in general, his methods can be extended to this setting, since $G$ acts fiberwisely on $W$ and the form $\omega$ restricts to a K\"ahler form on each fiber.

This section is organized as follows. We first state the construction of the equivariant Bott-Chern currents $S_K(W,\omega)$. And we describe the variation of $S_{zK}(W,\omega)$ in terms of a current $\rho_1$ on $W$, as $z$ varies. Then we will give a formula concerning the current $\rho_1$. This formula gives a  proof of \eqref{8m}, and it also plays an important role in the proof of our main result. Finally, we follow \cite[\S 3.3]{BismutGoette2000} to give an alternative construction of $S_K(W,\omega)$ under \textbf{Assumption (*)}.

For $t>0$, we set 
\begin{align}\label{btdt}
 &b_{K,t}=\exp\Big(\frac{\overline{\partial}_K\partial_K}{2i\pi t}\frac{\omega}{2\pi}\Big), &&d_{K,t}=-\frac{\omega}{2\pi t}\exp\Big(\frac{\overline{\partial}_K\partial_K}{2i\pi t}\frac{\omega}{2\pi}\Big).
\end{align}

From (\ref{btdt}), we obtain
\begin{proposition}\label{transgression bt}
    The following identity holds
    \begin{equation}
        \frac{\partial}{\partial t}b_{K,t}=\frac{\overline{\partial}_K\partial_K}{2i\pi t}{d_{K,t}}.
    \end{equation}
\end{proposition}

\begin{definition}\label{b2}
    Let $P_{K,W_K}^{W}$ be the set of $K^X$-invariant currents on $W$, which are sums of currents of type $(p,p)$, whose wave-front set is included in $N_{X_K/X,\rr}^{*}=N_{W_K/W,\rr}^{*}$. 
\end{definition}

We identify $N_{X_K/X}$ with the orthogonal bundle to $TX_K$ in $TX$. Let $h^{N_{X_K/X}} $ be the Hermitian metric on $N_{X_K/X}$ induced by $h^{TX}$.
We apply the result of \cite[Theorem 1.3]{B1986Local}, \cite[Theorem 2.5-2.7]{B1992Inter} to $(W,\omega)$, then there exist currents $\rho_1,...,\rho_k,...$ in $P_{K,W_K}^W$ such that if $\eta$ is a smooth form on $W$, as $t\to 0$,
\begin{equation}\label{aymptopic-bt}
    \int_W \eta b_{K,t}=\int_{W_K}\frac{\eta}{c_{\mathrm{max},K}(N_{X_K/X},h^{N_{X_K/X}})}+\sum_{j=1}^{k}\int_{W}\eta \rho_j t^j+\mathcal{O}(t^{k+1}).
\end{equation}

By (\ref{aymptopic-bt}), we see that as $t\rightarrow 0$, 
\begin{equation}\label{expansion dt}
    \int_{W}\eta {d_{K,t}}=-\int_{W_K}\frac{{\omega}/2\pi \cdot\eta}{c_{\mathrm{\mathrm{max}},K}(N_{X_K/X},h^{N_{X_K/X}})}\frac{1}{t}-\sum_{j=1}^{k}\int_{W}\frac{\omega}{2\pi}\eta \rho_j t^{j-1}+\mathcal{O}(t^k).
\end{equation}
By (\ref{expansion dt}), for $s\in \cc, \mathrm{Re}(s)>1$, set
\begin{equation}\label{feta1}
    F_{\eta}^1(s)=\frac{1}{\Gamma(s)}\int_0^1t^{s-1}\Big(
       \int_{W}\eta {d_{K,t}} \Big)dt.
\end{equation}
It extends to a holomorphic function near $s=0$. For $s\in\cc, \mathrm{Re}(s)<1$, the function
\begin{equation}\label{feta2}
    F_{\eta}^2(s)=\frac{1}{\Gamma(s)}\int_1^{\infty}t^{s-1}\Big( \int_{W}\eta {d_{K,t}}\Big) dt
\end{equation}
is holomorphic.

\begin{definition}\label{aa}
    Let $S_K(W,\omega)$ be the current on $W$, such that
    \begin{equation}
        \int_{W} \eta {S}_K(W,\omega)=\frac{\partial}{\partial s}(F_{\eta}^1+F_{\eta}^2)(0).
    \end{equation}
\end{definition}

Let $\delta_{W_K}$ be the current of integration on $W_K$. Now we state the result of \cite[(3.16)]{BismutGoette2000}.
\begin{proposition}\label{explicit sk}
    The following identity holds,
    \begin{equation}\label{a6}
     \begin{split}
           \int_{W}\eta {S}_K&(W,\omega)= 
           \int_{0}^{1}\frac{dt}{t}\int_{W}\eta\Big({d_{K,t}}+\frac{{\omega}/2\pi \cdot \delta_{W_K}}{c_{\mathrm{max},K}(N_{X_K/X},h^{N_{X_K/X}})}\frac{1}{t}+\frac{\omega}{2\pi}\rho_1\Big)\\
           +&\int_1^{+\infty}\frac{dt}{t}\int_{W}\eta {d_{K,t}}+\int_{W_K}\frac{{\omega}/{2\pi}\cdot\eta}{c_{\mathrm{max},K}(N_{X_K/X},h^{N_{X_K/X}})}+\Gamma'(1)\int_{W} \frac{\omega}{2\pi}\eta \rho_1.
     \end{split}
    \end{equation}
\end{proposition}

The following result is established in \cite[Theorem 6]{B1990}, \cite[Theorem 2.12]{B1992Inter}.
\begin{theorem}
    The current ${S}_K(W,\omega)$ lies in $P_{K,W_K}^W$. Moreover
    \begin{equation}\label{ddbarsk}
        \frac{\overline{\partial}_K\partial_K}{2i\pi}{S}_K(W,\omega)=1-\frac{\delta_{W_K}}{c_{\mathrm{max},K}(N_{X_K/X},h^{N_{X_K/X}})}.
    \end{equation} 
\end{theorem}

Set $N'$ as the multiplication by $p$ on $\Lambda^{p}(T^{*}W)$.

\begin{proposition}
    For $z\in \mathbb{R}$, $z\neq 0$, we have the following identity, 
    \begin{align}\label{a7}
z^{N'/2+1}S_{zK}(W,\omega)=S_{K}(W,\omega)-\log z^{2}\frac{\omega}{2\pi}\rho_{1}.
\end{align}
\end{proposition}
\begin{proof}
From \eqref{btdt}, clearly we have
\begin{align}\label{a4}
z^{N'/2+1}d_{zK,t}=d_{K,t/z^{2}}.
\end{align}

By \eqref{a6}, \eqref{a4}, we see that if $\eta$ is a smooth form on $W$,  then
\begin{align}\label{a5}
   \begin{split}
        \int_{W}\eta   & z^{N'/2+1}{S}_{zK}(W,\omega)=\\ 
           &\int_{0}^{1}\frac{dt}{t}\int_{W}\eta\Big({d_{K,t/z^{2}}}+\frac{z^{2} {\omega}/2\pi \cdot \delta_{W_K}}{c_{\mathrm{max},K}(N_{X_K/X},h^{N_{X_K/X}})}\frac{1}{t}+\frac{\omega}{2\pi}\rho_1\Big)\\
           +&\int_1^{+\infty}\frac{dt}{t}\int_{W}\eta {d_{K,t/z^{2}}}+\int_{X_K}\frac{z^{2} {\omega}/{2\pi}\cdot\eta}{c_{\mathrm{max},K}(N_{X_K/X},h^{N_{X_K/X}})}+\Gamma'(1)\int_{W} \frac{\omega}{2\pi}\eta \rho_1.
     \end{split}
\end{align}
From \eqref{a6}, \eqref{a5}, we get \eqref{a7}.
\end{proof}

If $A$ is an invertible $(q,q)$ matrix, put
\begin{equation}
    c_{\mathrm{max}}^{-1,\prime}(A)=\frac{\partial}{\partial b}[\mathrm{det}(A+b)]^{-1}\Big|_{b=0}.
\end{equation}

If $\eta$ is a smooth form on $W$, then for $\rho_{1}$ in \eqref{expansion dt}, by Definition \ref{b2}, $\int_X -\frac{\omega}{2\pi}\eta\rho_1$ is a smooth form on $S$.  Now we state an analogue of \cite[(2.23)]{B1992Inter} and \cite[Theorem 3.17]{BismutGoette2000}.
\begin{proposition}\label{prop rho1}
    If $\eta\in P^{W}$ is $d_K$-closed form on $W$, then we have
    \begin{equation}\label{rho1}
        \int_{X} -\frac{\omega}{2\pi}\eta\rho_1=\int_{X_K}\eta\, c_{\mathrm{\mathrm{max}},K}^{-1,\prime}(N_{X_K/X},h^{N_{X_K/X}})\qquad \kongge\text{ in } \kongge P^{S}/P^{S,0}.
    \end{equation}
\end{proposition}
\begin{proof}
  By  \cite[(2.23)]{B1992Inter}, 
   \begin{equation}\label{b3}
        \int_{W} -\frac{\omega}{2\pi}\eta\rho_1=\int_{W_K}\eta\, c_{\mathrm{\mathrm{max}},K}^{-1,\prime}(N_{X_K/X},h^{N_{X_K/X}}).
    \end{equation}
    If $S$ is a compact K\"ahler manifold, our formula \eqref{rho1} follows from \eqref{b3} and the $\partial\overline{\partial}$-lemma. In general, by (\ref{btdt}), (\ref{aymptopic-bt}), we have
    \begin{equation}\label{calculation rho1}
        \begin{split}
        \int_{X} -\frac{\omega}{2\pi}\eta\rho_1=&\lim_{t\rightarrow 0}\derivative{t}\Big(\int_{X} -\frac{\omega}{2\pi}\eta \exp\Big(\frac{\overline{\partial}_K\partial_K}{2\pi i t}\frac{\omega}{2\pi}\Big)\Big)\\
          =&  \lim_{t\rightarrow 0}\int_{X} \frac{\omega}{2\pi t}\Big(\frac{\overline{\partial}_K\partial_K}{2\pi i t}\frac{\omega}{2\pi}\Big)\eta \exp\Big(\frac{\overline{\partial}_K\partial_K}{2\pi i t}\frac{\omega}{2\pi}\Big).
        \end{split}
    \end{equation}
    
If $\alpha$ is a smooth form on $W$, since $K^X\in T_{\rr}X$, 
\begin{equation}\label{int dbar}
    \int_{X} \overline{\partial}_K\alpha=\int_{X} \overline{\partial}\alpha-2i\pi\int_{X} i_{K^{X(1,0)}}\alpha=\int_{X} \overline{\partial}\alpha=\overline{\partial}^S \int_{X} \alpha.
\end{equation}

Let $K^{\prime X}$ be the one form on $W$ which is dual to $K^{X}$ with respect to $g^{TW}$, then we have
\begin{equation}\label{k' ikw}
    K^{\prime X}=i(i_{K^{X(1,0)}}-i_{K^{X(0,1)}})\omega.
\end{equation}
By  (\ref{k' ikw}), we obtain
\begin{align}\label{partial muk}
    &\frac{1}{2\pi}\partial_K{\omega}= K'^{X(1,0)},  &&\frac{1}{2\pi}\overline{\partial}_K{\omega}=- K'^{X(0,1)}.
\end{align}

Using the fact that $\eta$ is $\overline{\partial}_K$-closed, by (\ref{dkpingfangdengyu0}), (\ref{int dbar}) and (\ref{partial muk}) we see that
\begin{equation}\label{split of rho1}
    \begin{split}
        \int_{X} \frac{\omega}{2\pi t}\Big(\frac{\overline{\partial}_K\partial_K}{2\pi i t}\frac{\omega}{2\pi}\Big)\eta \exp\Big(\frac{\overline{\partial}_K\partial_K}{2\pi i t}\frac{\omega}{2\pi}\Big)\qquad\qquad\qquad\\ 
        =  \frac{1}{2i\pi}\int_{X} \frac{K'^{X(0,1)}}{{t}}\frac{K'^{X(1,0)}}{{t}}\eta \exp\Big(\frac{\overline{\partial}_K\partial_K}{2\pi i t}\frac{\omega}{2\pi}\Big)\quad\\ 
        -\dba^S \frac{1}{2i\pi}\int_{X} \frac{\omega}{2\pi t}\frac{K'^{X(1,0)}}{  t}\eta \exp\Big(\frac{\overline{\partial}_K\partial_K}{2\pi i t}\frac{\omega}{2\pi}\Big).
    \end{split}
\end{equation}

Now we explain how to calculate the limit of  the right hand side of \eqref{split of rho1}.
By \eqref{partial muk}, 
    \begin{equation}\label{dkkprime}
        \frac{\overline{\partial}_K\partial_K}{2i\pi t}\frac{\omega}{2\pi}=-\frac{|K^{X}|^2}{2t}+\frac{dK'^X}{4i\pi t}.
    \end{equation}
    Therefore for $x\notin W_K$, when $t\rightarrow 0$, the integral term in (\ref{aymptopic-bt}) decays exponentially. So the study of the integral in (\ref{aymptopic-bt}) can be localized on a neighborhood of $W_K$.

    Given $\varepsilon>0$, let $U_{\varepsilon}$ be the $\varepsilon$-neighborhood of $W^{K}$ in $N_{X_K/X,\rr}$. There exists $\varepsilon_0$ such that for $0<\varepsilon<\varepsilon_0$, the exponential map $(y,z)\in N_{X_K/X}\rightarrow \exp_{y}^X(z)$ is a diffeomorphism from $U_{\varepsilon}$ to a neighborhood $V_{\varepsilon}$ of $W_K$ in $W$. We denote by $\widetilde{V}_{\varepsilon}$ the fiber of the fibration $V_{\varepsilon}\rightarrow S$. With this identification, let $\overline{k}(y,z)$ be the function such that 
    \begin{equation}
        dv_{X}(y,Z)=\overline{k}(y,Z)dv_{X_K}(y)dv_{N_{X_K/X,\rr}}(Z).
    \end{equation} 
    Here $dv_X$, $dv_{X_K}$ are the volume forms on $X$ and $X_K$, and $dv_{N_{X_K/X,\rr}}$ is the Euclidean volume form on $N_{X_K/X,\rr}$. Then $\overline{k}(y,0)=1$ for any $y\in W_K$.
    
    Let $e^1,...,e^{2\ell}$ be a locally orthonormal frame of $T_{\rr}^*X$, for $\beta\in \Omega^{\bullet}(W)$, let $[\beta]^{\mathrm{\mathrm{max}}}$ denote the coefficient of $e^1\wedge...\wedge e^{2\ell}$ of $\beta$.  For $t>0$, let $\delta_t$ be the dilation of $N_{X_K/X,\rr}$, defined by $\delta_t(y,Z)=(y,\sqrt{t}Z).$ Then we have
    \begin{equation}\label{calculation of bt}
        \begin{split}
            &\int_{\widetilde{V}_{\varepsilon}} \frac{K'^{X(0,1)}}{{t}}\frac{K'^{X(1,0)}}{{t}}\eta \exp\Big(\frac{\overline{\partial}_K\partial_K}{2\pi i t}\frac{\omega}{2\pi}\Big)\\
            =&\int_{X_K}\int_{\substack{|Z|<\varepsilon/\sqrt{t}\\ Z\in N_{X_K/X,\rr}}}\Big[\delta_t^*\Big(\frac{K'^{X(0,1)}}{{t}}\frac{K'^{X(1,0)}}{{t}}\eta \exp\Big(-\frac{|K^{X}|^2}{2t}+\frac{dK'^X}{4i\pi t}\Big)\Big)\Big]^{\mathrm{max}}(y,Z)\\
           &\qquad\qquad\qquad\qquad\qquad\qquad\qquad\qquad \cdot \overline{k}(y,\sqrt{t}Z)dv_{X_K}(y)dv_{N_{X_K/X,\rr}}(Z).
        \end{split}
    \end{equation}
    
    Let $\nabla^{N_{X_K/X,\rr}}$ be the connection on $N_{X_K/X,\rr}$ induced by $\nabla^{TX}$. Let $\pi_N: N_{X_K/X,\rr}\rightarrow W_K$ be the obvious projection.  With respect to $\nabla^{N_{X_K/X,\rr}}$, we have the canonical splitting of the tangent bundle of $N_{X_K/X,\rr}.$
\begin{equation}\label{split of N}
    \begin{split}
    &TN_{X_K/X,\rr}=T^HN_{X_K/X,\rr}\oplus \pi_N^*N_{X_K/X,\rr},\\
   & T^HN_{X_K/X,\rr}\simeq \pi_N^{*}T_{\rr}W_K.
    \end{split}
\end{equation}

For $y\in W_K$, we take $Y_1,Y_1'\in T_{\rr,y}W_K$, $Y^V,Y'^V\in N_{X_K/X,\rr}$. By the identification in (\ref{split of N}), we can lift $Y_1, Y_1'$ to $Y_1^H,Y_1'^H$, which are  sections of $T^HN_{X_K/X,\rr}$. Then $Y=Y_1^H+Y^V, Y'=Y_1'^H+Y'^V$ are sections of $TN_{X_K/X,\rr}$ along $N_{X_K/X,\rr,y}.$

For $z\in N_{X_K/X}$, we identify $Z=z+\bar{z}\in N_{X_K/X,\rr}$ with $z$.
Let $\theta_0$ be the one form on the total space $\mathcal{N}$ of $N_{X_K/X,\rr}$ given by 
\begin{equation}
    \theta_0(Y)_{(y,Z)}=\product{m^{TX}(K)Z,Y^V}_y \qquad \kongge \mathrm{for}\kongge Y=Y_1^H+Y^V.
\end{equation}
From the arguments before \cite[Lemma 7.16]{BGV}, we calculate that for $(y,Z)\in N_{X_K/X,\mathbb{R}}$,
\begin{equation}
    d\theta_0(Y,Y')_{(y,Z)}=2\product{m^{TX}(K)Y^V,Y'^V}_y-\product{R^{TX}(Y_1^H,Y_1'^H)(m^{TX}(K)Z),Z}_y.
\end{equation}
By \cite[Lemma 7.15]{BGV}, we have
\begin{equation}\label{expansion of K}
    \begin{split}
        &\frac{1}{t}\delta_t^*K'^X=\theta_0+\mathcal{O}(t^{1/2}),\\
        &\frac{1}{t}|K^X(y,\sqrt{t}Z)|^2=|m^{TX}(K)Z|^2+\mathcal{O}(t^{1/2}).
    \end{split}
\end{equation}
From the first equation in (\ref{expansion of K}), we also have
\begin{equation}\label{expansion of dk}
    \frac{1}{t}\delta_t^*d K'^X=d\theta_0+\mathcal{O}(t^{1/2}).
\end{equation}

From (\ref{calculation of bt}), (\ref{expansion of K}) and (\ref{expansion of dk}), we get
 \begin{equation}
    \begin{split}
    &\int_{X} \frac{K'^{X(0,1)}}{{t}}\frac{K'^{X(1,0)}}{{t}}\eta \exp\Big(\frac{\overline{\partial}_K\partial_K}{2\pi i t}\frac{\omega}{2\pi}\Big)\\ 
    =&\int_{X_K}\eta\int_{Z\in N_{X_K/X,\rr}}\theta_0^{(0,1)}\wedge \theta_0^{(1,0)}\wedge\exp\Big(\frac{d\theta_0}{4i\pi}\Big)\cdot \exp(-\frac{|m^{TX}(K)Z|^2}{2})+\mathcal{O}(t^{1/2}).
    \end{split}
 \end{equation}
Since $\theta_0^{(0,1)}\wedge \theta_0^{(1,0)}\wedge\exp(\frac{d\theta_0}{4i\pi})$ is an even polynomial in $Z$ and $m^{TX}(K)$ is non-degenerate on $N_{X_K/X,\rr}$, we can compute this Gaussian integral explicitly as in \cite[(47)-(50)]{B1990} to obtain
\begin{equation}\label{asy expansion 1}
    \begin{split}
   \lim_{t\to 0} \frac{1}{2i\pi}\int_{X} \frac{K'^{X(0,1)}}{{t}}\frac{K'^{X(1,0)}}{{t}}\eta \exp\Big(\frac{\overline{\partial}_K\partial_K}{2\pi i t}\frac{\omega}{2\pi}\Big)
        =\int_{X_K}\eta c_{\mathrm{\mathrm{max}},K}^{-1,\prime}(N_{X_K/X},h^{N_{X_K/X}}).
    \end{split}
\end{equation}
Similarly, by \eqref{expansion of K} and \eqref{expansion of dk},
\begin{equation}\label{aym expansion 2}
    \begin{split}
  &\lim_{t\to 0}\frac{t}{2i\pi}\int_{X} \frac{\omega}{2\pi t}\frac{K^{\prime X(1,0)}}{ t}\eta \exp\Big(\frac{\overline{\partial}_K\partial_K}{2\pi i t}\frac{\omega}{2\pi}\Big)\\ 
  =&\frac{1}{2i\pi}\int_{X_K}\frac{\omega}{2\pi}\eta\int_{Z\in N_{X_K/X,\rr}} \theta_0^{(1,0)}\wedge\exp(\frac{d\theta_0}{4i\pi})\cdot \exp(-\frac{|m^{TX}(K)Z|^2}{2})=0.
    \end{split}
\end{equation}
The limit vanishes because it equals to an integral of the product an odd polynomial with a Gaussian function. Moreover by expanding $\frac{1}{t}\delta_t^*K'^X$, $\frac{1}{t}|K^X(y,\sqrt{t}Z)|^2$ as in (\ref{expansion of K}) to higher degrees, in the variable $t^{1/2}$, we see that there exist smooth forms  $\sigma_{1/2}$, $\sigma_1$ on $S$ such that 
\begin{align}\label{b4}
\frac{t}{2i\pi}\int_{X} \frac{\omega}{2\pi t}\frac{K^{\prime X(1,0)}}{ t}\eta \exp\Big(\frac{\overline{\partial}_K\partial_K}{2\pi i t}\frac{\omega}{2\pi}\Big)=\sigma_{1/2}t^{1/2}+\sigma_1 t+\mathcal{O}(t^{3/2}).
\end{align}

Observe that the terms appearing in $\sigma_{1/2}$ are of the form
\begin{align}\label{aj}
\int_{X_{K}}\varphi\int_{Z\in N_{X_{K}/X,\mathbb{R}}} Z^{\alpha}\beta\cdot\exp\Big(-\frac{|m^{TX}(K)Z|^2}{2}\Big)
\end{align}
with $\varphi$ a smooth form on $W_{K}$,   $\beta\in \Lambda^{\bullet}(N_{X_{K}/X,\mathbb{R}}^{*})$, and $|\alpha|+\deg \beta$  odd. The dimension of $N_{X_{K}/X,\mathbb{R}}$ is even, but \eqref{aj} is either an integral of an odd function, or an integral of an odd form, so we get 
\begin{align}\label{b5}
\sigma_{1/2}=0.
\end{align}
By using (\ref{calculation rho1}), (\ref{split of rho1}), (\ref{asy expansion 1})-(\ref{b5}), we obtain
\begin{equation}\label{sigma tau}
    \begin{split}
        \int_{X} -\frac{\omega}{2\pi}\eta\rho_1=\int_{X_K}\eta c_{\mathrm{\mathrm{max}},K}^{-1,\prime}(N_{X_K/X},h^{N_{X_K/X}})-\dba\sigma_1.
    \end{split}
\end{equation}
This finishes the proof of Proposition \ref{prop rho1}.
\end{proof}

Now we make \textbf{Assumption (*)} as in Section 2.1. By \eqref{dmuk}, we have 
\begin{align}\label{a9}
d_{K}\left( \omega-2i\pi \left\langle \mu,K \right\rangle \right)=0.
\end{align} 

Using essentially the equation \eqref{a9}, the results in this section holds true if we replace $\omega$ by $2i \pi \left\langle \mu,  K \right\rangle$. In particular, 
 if $\eta\in P^{W}$ is $d_K$-closed form on $W$, then 
    \begin{equation}\label{bo}
        \int_{X} -i\left\langle \mu,K \right\rangle \eta {\rho}_1=\int_{X_K}\eta c_{\mathrm{\mathrm{max}},K}^{-1,\prime}(N_{X_K/X},h^{N_{X_K/X}})\qquad \kongge\text{ in } \kongge P^{S}/P^{S,0}.
    \end{equation}

For $t>0$, set 
\begin{align}\label{a8}
\widetilde{d}_{K,t}=-i \frac{\left\langle \mu,K \right\rangle}{t}\exp\left(\frac{\overline{\partial}_K\partial_K}{2\pi t}\left\langle \mu,K \right\rangle\right). 
\end{align}
Then by (\ref{aymptopic-bt}), we see that if $\eta$ is a smooth form on $W$, as $t\rightarrow 0$, 
\begin{equation}\label{bv}
    \int_{W}\eta \widetilde{d}_{K,t}=-\int_{W_K}\frac{i\left\langle \mu,K \right\rangle\eta}{c_{\mathrm{\mathrm{max}},K}(N_{X_K/X},h^{N_{X_K/X}})}\frac{1}{t}-\sum_{j=1}^{k}\int_{W}i\left\langle \mu,K \right\rangle\eta \rho_j t^{j-1}+\mathcal{O}(t^k).
\end{equation}
Indeed by using the same arguments as in the proof of Proposition \ref{prop rho1}, we have 
\begin{equation}\label{bw}
    \int_{X}\eta \widetilde{d}_{K,t}=-\int_{X_K}\frac{i\left\langle \mu,K \right\rangle\eta}{c_{\mathrm{\mathrm{max}},K}(N_{X_K/X},h^{N_{X_K/X}})}\frac{1}{t}-\sum_{j=1}^{k}\int_{X}i\left\langle \mu,K \right\rangle\eta \rho_j t^{j-1}+\mathcal{O}(t^k).
\end{equation}

Let $\widetilde{S}_K(W,\omega)$ be the current as in Definition \ref{aa}, but with $\widetilde{d}_{K,t}$ in the place of $d_{K,t}$, we have the following result analogue to \cite[Theorem 3.18]{BismutGoette2000}.

\begin{theorem}\label{ab}
If $\eta\in P^{W}$ is a $d_{K}$-closed form, the following identify holds, 
\begin{align}\label{ac}
\int_{X}\eta \widetilde{S}_{K}(W,\omega)=\int_{X}\eta {S}_{K}(W,\omega) \qquad\qquad\text{ in } P^{S}/P^{S,0}.
\end{align}
\end{theorem}
\begin{proof}
    By Proposition \ref{transgression bt} and \ref{a9}, The cohomology class of $\widetilde{d}_{K,t}-d_{K,t}$ does not depend on $t$. 
The proof of Theorem \ref{ab} proceeds as the proof of Proposition \ref{9t}.
\end{proof}

\section{The infinitesimal equivariant analytic torsion forms on the vector bundle}\label{em}
Inspired by \cite{B1990Kuszul}, \cite{B1994Equivariantshort}, \cite[\S 5]{B1997}, we will study the equivariant analytic torsion forms when the K\"ahler fibration  is a holomorphic Hermitian vector bundle $M$.   In Section \ref{b7}, we construct a family of Bismut superconnections $\mathscr{B}_{K,u}$ on the total space of $M$. In Section \ref{b9}, we introduce the $I$-genus of Bismut and Goette \cite{BismutGoette2000}. In Section \ref{b8}, we recall the Mehler's formula. In Section \ref{ba}, we introduce the generalized supertraces of the heat kernels of $\mathscr{B}_{K,u}$, and we  relate it to the $I$-genus. Finally, in Section \ref{bb}, we establish various identities on the generalized supertraces, which will be needed in Sections \ref{sec intermediate right}-\ref{b6}.

\subsection{Vector bundles and superconnections}\label{b7}
Let $B$ be a complex manifold and let $M$ be holomorphic vector bundle on $B$ of dimension $m$. Denote by $\pi$ the obvious projection from $M$ to $B$. Let $G$ be a compact Lie group with Lie algebra $\mathfrak{g}$. Let $G$ act holomorphically, and fiberwisely on the vector bundle $M$. Let $m^{M}\in C^{\infty}(B,\mathfrak{g}^*\otimes\ENDO(M))$ be the moment map (cf. \eqref{moment})  on $M$.

Let $h^{M}$ be a $G$-invariant Hermitian metric on $M$. Let $\nabla^{M}$ be the holomorphic Hermitian connection on $M$ with curvature $R^{M}$. We denote by $h^{\Lambda(\overline{M}^*)}$ and $\nabla^{\Lambda(\overline{M}^*)}$ the Hermitian metric and connection on $\Lambda(\overline{M}^*)$ induced by $h^{M}$ and $\nabla^M$. 

\begin{definition}
    For $y\in B$, let $I_y$ be the space of smooth sections of $\Lambda(\overline{M}^*)$ along the fiber $M_{\rr,y}$. 
    Let $I$ be the $\zz$-graded vector bundle on $B$, which is formed by $I_{y}$.
\end{definition}

 The group $G$ acts on $I$ via \eqref{g on section}. Let $dv_M$ be the standard Euclidean volume form on the fiber of $M$, we can equip $I$ with the $G$-invariant Hermitian metric
\begin{equation}\label{product I}
    f,g\in I_y\rightarrow \product{f,g}=\Big(\frac{1}{2\pi}\Big)^{m}\int_{M}\product{f,g}_{h^{\Lambda(\overline{M}^*)}}dv_M.
\end{equation}

Let $\dba^{M_y}$ be the fiberwise Dolbeaut operator acting on $I_y$ and let $\dba^{M_y,*}$ be its formal adjoint with respect to the product (\ref{product I}). Set
\begin{equation}\label{DM}
    D_y^{M}=\dba^{M_y}+\dba^{M_y,*}.
\end{equation}

The connection $\nabla^{M}$ defines a horizontal subbundle $T^HM$ of $TM$. For $U\in T_{\rr}B$, denote by $U^H\in T_{\rr}^HM$ the horizontal lift of $U$. For a smooth section $s$ of $I$, put 
\begin{equation}
    \nabla_{U}^{I}s=\nabla_{U^H}^{\Lambda(\overline{M}^*)}s.
\end{equation}
Then $\nabla^{I}$ is a Hermitian connection on $I$ with respect to the product (\ref{product I}).

For $z\in M$, set $Z=z+\bar{z}\in M_{\rr}$.
\begin{definition}
    For $K\in \mathfrak{g}$, for $u\geq 1$, let $\mathscr{B}_{K,u}$ be the superconnection on $I$ given by 
    \begin{equation}\label{bc}
        \mathscr{B}_{K,u}=D^M+\nabla^I-\frac{c(R^M Z)}{2\sqrt{2}}+\frac{c(m^M(K)Z)}{2\sqrt{2}}
        \Big(1-\frac{1}{u}\Big).
    \end{equation}
    As in \eqref{Bismut superconnection},  $\scrbku$ splits as
    \begin{equation}
        \scrbku=\scrbku''+\scrbku'.
    \end{equation}
\end{definition}
\begin{remark}
    For $K\in \mathfrak{g}$, the infinitesimal vector field $K^M$ on the total space of $M$ is given by 
        \begin{equation}
            K^{M}(x)=\frac{\partial}{\partial t}\Big|_{t=0}e^{tK}\cdot x \kongge\kongge \mathrm{for}\kongge x\in M.
        \end{equation}
        So we can compute directly to obtain that 
        \begin{equation}\label{km}
            K^M(Z)=m^{M}(K)Z.
        \end{equation}

        By proceeding as in \cite[\S 15]{B1997}, we see that in a toy model,  the superconnection $ \mathscr{B}_{K,u}$ is 
        just the operator $\mathcal{A}_{u,1}$ in \eqref{definition cvt}. In particular, the factor $c(R^MZ)$ in \eqref{bc} corresponds to $c(T)$ in \eqref{definition cvt}, and by \eqref{km}, $c(m^M(K))Z$ corresponds to $c(K^X)$.
\end{remark}

Let $J^M$ be the complex structure on $M$, set
\begin{equation}\label{def of tilde mu}
\widetilde{\mu}(K)(Z)=\frac{1}{2}\product{J^Mm^M(K)Z,Z}.
\end{equation}
Then $\widetilde{\mu}(K)$ defines a function on the total space of $M$. By using (\ref{dmoment map}), we have
\begin{equation}\label{dmutilde}
    d\widetilde{\mu}(K)=\langle J^{M}m^M(K)Z,\cdot\rangle.
\end{equation}

\begin{proposition}
    We have
    \begin{equation}\label{scrbku commutator}
        \begin{split}
            &\Big[\scrbku'',i\frac{\widetilde{\mu}(K)}{u}\Big]=-2u\frac{\partial}{\partial u}\scrbku''\mathrm,\quad\Big[\scrbku',i\frac{\widetilde{\mu}(K)}{u}\Big]=2u\frac{\partial}{\partial u}\scrbku'.\\
            &\scrbku'^{2}=0,\quad \scrbku''^{2}=0,\quad
            \scrbku^{2}=[\scrbku'',\scrbku'].\\
            &[\scrbku'',\scrbku^2]=[\scrbku',\scrbku^2]=0.
            \end{split}
    \end{equation}
\end{proposition}
\begin{proof}
    The first line of (\ref{scrbku commutator}) follows from (\ref{dmutilde}). The proof of the remaining part is the same as the proof of (\ref{bku-algebraic identity}) and \cite[(5.9)]{B1997}.
\end{proof}

Let $\lk$ be the Lie derivative of $K$ on $\Lambda(\overline{M}^*)$. We take $e_1,...,e_{2m}$ as an orthonormal basis of $M_{\rr}$. Let $\trace[R^M]\in \Lambda^2(T_{\rr}^*B)$ be the trace of $R^M\in \Lambda^2(T_{\rr}^*B)\otimes \ENDO(M)$. We have the following extension of \cite[Theorem 3.6]{B1990Kuszul}.
\begin{theorem}
For $u\geq 1$, we have the following identity, 
\begin{equation}\label{lichnerowicz scrbku}
    \begin{split}
        \lk+\scrbku^2=&-\frac{1}{2}\sum_{i=1}^{2m}\Big(\nabla_{e_i}+\frac{1}{2}\Big\langle \Big(R^M-\Big(1+\frac{1}{u}\Big)m^M(K)\Big)Z,e_i \Big\rangle\Big)^2 \\ +&\frac{1}{2u}|m^M(K)Z|^2
      -\frac{1}{4u}\product{m^M(K)e_i,e_j}c(e_i)c(e_j)\\+&\frac{1}{2u}\product{R^{M}m^M(K)Z,Z}+\frac{1}{2}\trace[R^M]-\frac{1}{2}\trace[m^M(K)].
    \end{split}
\end{equation}     
\end{theorem}
\begin{proof}
  From \cite[Theorem 3.6]{B1990Kuszul}, we have
  \begin{equation}\label{proof of lichnerowicz scrbku 1}
    \begin{split}
        \Big(D^M+\nabla^I-\frac{c(R^M Z)}{2\sqrt{2}}\Big)^2=-\frac{1}{2}\sum_{i=1}^{2m}\Big(\nabla_{e_i}+\frac{1}{2}\product{R^MZ,e_i}\Big)^2 +\frac{1}{2}\trace[R^M].
    \end{split}
  \end{equation}
So we get the $K=0$ part of \eqref{lichnerowicz scrbku}.
 By \eqref{98},
  \begin{equation}\label{bd}
    \begin{split}
         \big[D^M, c(m^M(K)Z)\big]=\frac{1}{\sqrt{2}}\sum_{i=1}^{2m}\big[c(e_i)\nabla_{e_i},c(m^M(K)Z)\big]\qquad\qquad\qquad\\ 
         =-\sqrt{2}\sum_{i=1}^{2m}\product{m^M(K)Z,e_i}\nabla_{e_i}+\frac{1}{\sqrt{2}}\product{m^M(K)e_i,e_j}c(e_i)c(e_j).
    \end{split}
  \end{equation}
 By (\ref{dmoment map}), and since $G$ act trivially on $B$,
  \begin{equation}
    \big[\nabla^I,c(m^M(K)Z)\big]=0.
  \end{equation}

Since $R^M$ takes values in skew symmetric endomorphisms of $M_{\rr}$, by \eqref{98},
\begin{equation}
    \begin{split}
         &\big[c(R^MZ),c(m^M(K)Z)\big]=2\product{R^{M}m^M(K)Z,Z}, \\ &\big[c(m^M(K)Z),c(m^M(K)Z)\big]=-2|m^M(K)Z|^2.
    \end{split}
\end{equation}

By (\ref{action on exterior algebra}), (\ref{moment}) and (\ref{km}), we have
\begin{equation}\label{proof of lichnerowicz scrbku 6}
    \begin{split}
         \lk&=\nabla_{K^M}-m^{\Lambda(\overline{M}^*)}(K)\\ &=\nabla_{m^M(K)Z}-\frac{1}{4}\product{m^M(K)e_i,e_j}c(e_i)c(e_j)-\frac{1}{2}\trace[m^M(K)].
    \end{split}
\end{equation}
From \eqref{bc}, (\ref{proof of lichnerowicz scrbku 1})-(\ref{proof of lichnerowicz scrbku 6}), we get the contribution of $\mathcal{L}_K+\mathscr{B}_{K,u}^2$ in the right side of \eqref{lichnerowicz scrbku}, in particular, by \eqref{bd} and \eqref{proof of lichnerowicz scrbku 6}, we get the factor $\Big(1+\frac{1}{u}\Big)m^M(K)$ in \eqref{lichnerowicz scrbku}.
\end{proof}

\subsection{The genus $I$, $\Phi_s$}\label{b9}

For $s\in \,]0,1]$, $\theta, \theta'\in\rr$, and $x\in\cc$. For $\theta\notin 2\pi \zz$,  $|\theta'|$, $|x|$ small enough, or for $\theta\in 2\pi\zz$, $|\theta'|$, $|x|$ small, with $x+i\theta'\neq 0$,   set
\begin{equation}\label{bi}
    \begin{split}
&\Phi_s(\theta,\theta',x)=\frac{1}{2}\Big[\coth\Big(\frac{i\theta+si\theta'}{2}\Big)-\coth\Big(\frac{i\theta+i\theta'+x}{2}\Big)\Big]\frac{si\theta'}{(1-s)i\theta'+x}.\\
 &   I({\theta,\theta',x})=\sum_{k\in\zz, 2k\pi+\theta\neq 0}\frac{\log(1+\frac{\theta'}{2k\pi+\theta})}{i(2k\pi +\theta+\theta')+x}.
    \end{split}
\end{equation}

We keep our notations in Section 4.1. Now we fix $g\in G$ and take $K\in \mathfrak{z}(g)$. Let $L=M_{g,K}=M_g\cap M_K$ be the $g$, $K$-invariant subbundle of $M$, and let  $N$ be the orthogonal bundle to $L$ with respect to $h^{M}$. Then $N$ inherits a $G$-invariant metric $h^{N}$. Let $N_0$ be the orthogonal bundle to $L$ in $M_g$, with the metric $h^{N_0}$.

Let $e^{i\theta_1},...,e^{i\theta_q}$, $0<\theta_j<2\pi$, be the locally constant distinct eigenvalue of the action of $g$ on $N$.  Let $i\theta_{1}',...,i\theta_{q'}',\theta_{j}\in\rr$, be the locally distinct eigenvalues of $m^M(K)$. Then $N$ and $N_0$ splits orthogonally as
\begin{equation}
    \begin{split}
        &N=\oplus_{j,j'}N^{\theta_j,\theta_{j'}'},\\ 
        &N_0=\oplus_{j'}N^{0,\theta_{j'}'}.
    \end{split}
\end{equation}
Then $N^{\theta_j,\theta_{j'}'}$ inherits a Hermitian metric from $M$. Let $\nabla^{N^{\theta_j,\theta_{j'}'}} $ be the corresponding holomorphic Hermitian connection on $N^{\theta_j,\theta_{j'}'}$, with curvature $R^{N^{\theta_j,\theta_{j'}'}}$.

\begin{definition}
    Put
    \begin{equation}
        \begin{split}
        \Phi_{s,g,K}(N,h^N)=\sum_{j,j'}\trace\Big[\Phi_s\Big(\theta_j,\theta_{j'}',-\frac{R^{N^{\theta_j,\theta_{j'}'}}}{2i\pi}\Big)\Big],\\ 
        I_{g,K}(N,h^N)=\sum_{j,j'}\trace\Big[I\Big(\theta_j,\theta_{j'}',-\frac{R^{N^{\theta_j,\theta_{j'}'}}}{2i\pi}\Big)\Big],\\ 
        \Big(\frac{c_{\mathrm{\mathrm{max}}}'}{c_{\mathrm{max}}}\Big)_K(N^0,h^{N^0})=\sum_{j'}\trace\Big[\Big(i\theta_{j'}'-\frac{R^{N^{0,\theta_{j'}'}}}{2i\pi}\Big)^{-1}\Big].
        \end{split}
    \end{equation}
\end{definition}

Now we state the result of \cite[Section 4.3]{BismutGoette2000}.
\begin{theorem}\label{thm Phi and I}
    As $s\in\, ]0,1]\rightarrow 0$,
    \begin{equation}\label{Phi and I}
        \Phi_{s,g,K}(N,h^N)=  \Big(\frac{c_{\mathrm{max}}'}{c_{\mathrm{max}}}\Big)_K(N^0,h^{N^0})+\mathcal{O}(s).
    \end{equation}
    And the following identity holds,
    \begin{equation}
        \int_{0}^1\Big(  \Phi_{s,g,K}(N,h^N)- \Big(\frac{c_{\mathrm{max}}'}{c_{\mathrm{max}}}\Big)_K(N^0,h^{N^0})\Big)\frac{ds}{s}=I_{g,K}(N,h^N).
    \end{equation}
\end{theorem}

\subsection{The heat Kernel of the Harmonic Oscillator}\label{b8}
Let $V$ be a finite dimensional complex Hermitian vector space. Let $F$, $H$ be commuting elements of $\ENDO(V)$, with $F$ self-adjoint and $H$ skew-adjoint. Let $e_1,...,e_{2l}$ be an orthonormal basis of $V_{\rr}$. Set
\begin{equation}
    \mathcal{L}=-\frac{1}{2}\Big(\nabla_{e_i}+\frac{1}{2}\product{HZ,e_i}\Big)^2+\frac{1}{2}\product{FZ,Z}.
\end{equation}
Put 
\begin{equation}
    P=H^2+4F.
\end{equation}
When $P\geq 0$, let $Q$ be the nonnegative self-adjoint square root of $P$. Let $p(Z,Z'), Z,Z'\in V_{\rr}$ be the smooth kernel of the operator $\exp(-\mathcal{L})$ with respect to $dv_{V}/(2\pi)^{\ell}$. By using the Mehler's formula, we have the following result which is obtained in \cite[(4.30)]{BismutGoette2000}.
\begin{equation}\label{mehler formula}
    \begin{split}
    p(Z,Z')=&\mathrm{det}\Big(\frac{Q/2}{\sinh(Q/2)}\Big)\exp\Big\{-\frac{1}{2}\Big\langle\frac{Q/2}{\tanh(Q/2)}Z,Z\Big\rangle\\ 
    &-\frac{1}{2}\Big\langle\frac{Q/2}{\tanh(Q/2)}Z',Z'\Big\rangle+\Big\langle\frac{Q/2}{\sinh(Q/2)}e^{H/2}Z,Z'\Big\rangle\Big\}.
    \end{split}
\end{equation}
From (\ref{mehler formula}), the kernel $p(Z,Z')$  is an analytic function of $Q^2=P$. Therefore it extends to a holomorphic function of the pair $(F,H)$ as long as for any $k\in \nn^{*}$, $-4k^2\pi^2$ is not an eigenvalue of $P$. We will still call this extension of $p(Z,Z')$ the smooth kernel of the operator $\exp(-\mathcal{L})$.

\subsection{The Heat Kernels of $\lk+\mathscr{B}_{K,u}^{2}$}\label{ba}\quad

Now we fix $g\in G$, and we take $K\in\mathfrak{z}(g)$ such that $|K|\neq 0$ small enough.

Let $dv_{M}$, $dv_{N}$ be the standard volume forms on the fibers of $M_{\rr}, N_{\rr}$ respectively. For $y\in B$, let $Q_u^y(Z,Z')$, $Z,Z'\in M_{\rr}$ be the smooth kernel of $\exp(-\lk-\scrbku^2)$ on the fiber $M_{\rr,y}$, with respect to the volume form $dv_{M}/(2\pi)^{\mathrm{dim} M}.$

Recall that $L=M_{g,K}$ is the $g,K$-invariant subbundle of $M$ and $N$ is the orthogonal bundle to $L$ in $M$. 
 By (\ref{lichnerowicz scrbku}), the coefficient of $\lk+\scrbku$ lie in $\ENDO(\Lambda(\overline{N}^*))$, so $Q_u^y(Z,Z')$ also takes value in $\ENDO(\Lambda(\overline{N}^*))$.
 
 We have the following analogue of \cite[Theorem 4.2]{B1990Kuszul}. 
\begin{theorem}\label{thm-transition invariance}
    For any $u\geq 1$, $y\in B$, $Z,Z'\in M_{\rr,y}$ and $Z_0\in L_{\rr,y}$, then we have
    \begin{equation}\label{transition invariant}
        Q_u^y(Z+Z_0,Z'+Z_0)=\exp\Big(-\frac{\product{R^MZ_0,Z}}{2}\Big)Q_u^y(Z,Z')\exp\Big(\frac{\product{R^MZ_0,Z'}}{2}\Big).
    \end{equation}
    In particular, we have 
    \begin{equation}\label{transition invariant 2}
        \begin{split}
        &\str^{\Lambda(\overline{N}^*)}\big[gQ_u(g^{-1}(Z+Z_0),Z+Z_0)\big]=\str^{\Lambda(\overline{N}^*)}\big[gQ_u(g^{-1}Z,Z)\big],\\ 
        &\str^{\Lambda(\overline{N}^*)}\Big[i\frac{\widetilde{\mu}(K)}{u}gQ_u(g^{-1}(Z+Z_0),Z+Z_0)\Big]=\str^{\Lambda(\overline{N}^*)}\Big[i\frac{\widetilde{\mu}(K)}{u}gQ_u(g^{-1}Z,Z)\Big].
        \end{split}
    \end{equation}
\end{theorem}
\begin{proof}
    If $K=0$, then \eqref{transition invariant} follows directly from \cite[Theorem 4.2]{B1990Kuszul}. So we only need to deal with the extra factors $m^M(K)Z$ and $\product{R^{M}m^M(K)Z,Z}$ in \eqref{lichnerowicz scrbku}. 
    By \eqref{km}, $m^M(K)Z_0=0$ for $Z_0\in L_{\mathbb{R},y}$, this implies that $m^M(K)Z$ is invariant under the translation by $Z_0$. Similarly, $\product{R^{M}m^M(K)Z,Z}$ is invariant under the translation, so we get \eqref{transition invariant}.

   Notice that $g^{-1}(Z+Z_0)=g^{-1}Z+Z_0$, and by \eqref{def of tilde mu},
   \begin{equation}\label{be}
    \widetilde{\mu}(K)(Z+Z_0)=\widetilde{\mu}(K)(Z).
   \end{equation}
 Then   (\ref{transition invariant 2}) follows from (\ref{transition invariant}) and \eqref{be}.
\end{proof}

By omitting the $da$, $\overline{da}$ terms in \cite[(4.45)]{BismutGoette2000}, we see that for any compact subset $K\subset B$, there exist $c,C>0$, such that for $y\in K$, $Z\in N_{\rr,y}$, $u\geq 1$, then 
    \begin{equation}\label{estimate scr kernel}
        |Q_u^y(g^{-1}Z,Z)|\leq c\exp(-C|Z|^2).
    \end{equation}

    \begin{definition}\label{bf}
    For $u\geq 1$, set 
    \begin{equation}\label{basic generalized supertraces}
    \begin{split}
       & \str\big[g\exp(-\lk-\scrbku^2)\big]=\int_{N_{\rr}}\str^{\Lambda(\overline{N}^*)}\big[gQ_u(g^{-1}Z,Z)\big]\frac{dv_N}{(2\pi)^{\mathrm{dim}N}},\\ 
        &\str\Big[i\frac{\widetilde{\mu}(K)}{u}g\exp(-\lk-\scrbku^2)\Big]=\int_{N_{\rr}}\str^{\Lambda(\overline{N}^*)}\Big[i\frac{\widetilde{\mu}(K)}{u}gQ_u(g^{-1}Z,Z)\Big]\frac{dv_N}{(2\pi)^{\mathrm{dim}N}}.
    \end{split}
\end{equation}
    \end{definition}

    Note that in general, the operator $g\exp\left(-\mathcal{L}_K-\mathscr{B}_{K,u}^2\right)$ is not a trace class (see \eqref{transition invariant 2}), $\str\big[g\exp(-\lk-\scrbku^2)\big]$ and $\str\big[i\frac{\widetilde{\mu}(K)}{u}g\exp(-\lk-\scrbku^2)\big]$ are generalized supertraces.

Set $\varphi:\Lambda(T_{\rr}^*B)\rightarrow \Lambda(T_{\rr}^*B) $ be the map $\alpha\rightarrow (2i\pi)^{-\mathrm{deg}\alpha/2}\alpha$. 
\begin{theorem}\label{thm kernel of scrbku}
   We have the following identities,
    \begin{equation}\label{bh}
        \begin{split}
            &\varphi\str\big[g\exp(-\lk-\scrbku^2)\big]=\td_{ge^K}(M,h^{M}),\\ 
            &\varphi\str\Big[i\frac{\widetilde{\mu}(K)}{u}g\exp(-\lk-\scrbku^2)\Big] =\td_{ge^K}(M,h^{M})\Phi_{\frac{1}{u},g,K}(N,h^N). 
        \end{split}
    \end{equation}
\end{theorem}
\begin{proof}
    Let $da, \overline{da}$ be the Grassmann variables. By proceeding as in the proof of \cite[(4.37)]{B1990Kuszul},   we get
    \begin{equation}
        \begin{split}
     \str\big[g\exp(-\lk-\scrbku^2)\big]+ da\dabar \str\Big[i\frac{\widetilde{\mu}(K)}{u}g\exp(-\lk-\scrbku^2) \Big]\\ =\str\Big[ g\exp\Big(-\lk-\scrbku^2+da\dabar \frac{i\widetilde{\mu}(K)}{u}\Big)\Big].
        \end{split}
    \end{equation}

    From (\ref{lichnerowicz scrbku}), we see that $\lk+\scrbku^2-da\dabar \frac{i\widetilde{\mu}(K)}{u}$ is just the operator $L_0^{1/u}$ in \cite[(4.40)]{BismutGoette2000} except for a $da$ term. So Theorem  \ref{thm kernel of scrbku} follows from \cite[Theorem 4.13]{BismutGoette2000}.
\end{proof}

Using Theorems  \ref{thm Phi and I} and \ref{thm kernel of scrbku}, it's easy to obtain the following result.
\begin{proposition}\label{fu} The following identity holds,
    \begin{equation}\label{integral of kernel of scrbku}
        \begin{split}
           \int_1^{+\infty}& \Big(\varphi\str\Big[i\frac{\widetilde{\mu}(K)}{u}g\exp(-\lk-\scrbku^2)\Big]\\&-\td_{ge^K}(M,h^M) \Big(\frac{c_{\mathrm{max}}'}{c_{\mathrm{max}}}\Big)_K(N^0,h^{N^0})\Big)\frac{du}{u} 
             =\td_{ge^K}(M,h^{M})I_{g,K}(N,h^N).
        \end{split}
    \end{equation}
\end{proposition}


\subsection{Some identities on generalized supertraces}\label{bb}
Let $da$ be the canonical generator of $\mathbb{C}^{*}$. Then $da$, $\dabar$ span $\rr^{2*}\otimes_{\rr}\cc$. If $\sigma\in\Lambda(T_{\rr}^*B)\widehat{\otimes}_{\rr}\Lambda(\rr^{2*})$, we write $\sigma$ in the form 
 \begin{equation}\label{notation dadabar}
     \sigma=\sigma_0+da\, \sigma^{da}+ \dabar \,\sigma^{\dabar}+da\dabar \,\sigma^{da\dabar}, \quad \sigma_0,\sigma^{da},\sigma^{\dabar},\sigma^{da\dabar}\in\Lambda(T_{\rr}B^*).
 \end{equation}

 \begin{definition}
    For $u\geq 1$, put
   \begin{equation}\label{many generalized supertraces}
        \begin{split}
            &\Theta_u=\derivative{b}\str\Big[g\exp\Bigl(-\Bigl(-\frac{1}{2}\Bigl(\nabla_{e_i}+\frac{1}{2}\Big\langle\Bigl(R^M-\Big(1+\frac{1}{u}\Big)m^M(K)-iJ^{M}da\dabar\Bigr)Z,e_i\Big\rangle\Bigr)^2\\ 
             +\frac{1}{2u}&|m^M(K)Z|^2-\frac{1}{4u}\product{m^M(K)e_i,e_j}c(e_i)c(e_j)+\frac{\dabar}{\sqrt{2}u}c(m^M(K)\bar{z})+\frac{da}{\sqrt{2}u}c(m^M(K)z)\\
             -da\dabar & \frac{m}{2}+\frac{1}{2u}\product{R^{M}m^M(K)Z,Z}+\frac{1}{2}\trace[R^M]-\frac{1}{2}\trace[m^M(K)]+b\frac{i\widetilde{\mu}(K)}{u}\Bigr)\Bigr)\Big]_{b=0}^{da\dabar},\\ 
             &\Lambda_{u}=\str\Big[g\exp\Bigl(-\lk-\scrbku^2-\frac{da}{\sqrt{2}u}c(m^M(K)z)\\ &\qquad\qquad\qquad\qquad\qquad\qquad\qquad-\frac{\dabar}{\sqrt{2}u}c(m^M(K)\bar{z}) 
            -da\dabar\frac{i\widetilde{\mu}(K)}{u}\Bigr) \Big]^{da\dabar},\\
            &\Pi_u'=\str\Big[i\frac{\widetilde{\mu}(K)}{u}g\exp\Bigl(-\lk-\scrbku^2+\frac{\dabar}{\sqrt{2}u}c(m^M(K)\bar{z})\Bigr)\Big]^{\dabar},\\ 
           & \Pi_u''=\str\Big[i\frac{\widetilde{\mu}(K)}{u}g\exp\Bigl(-\lk-\scrbku^2+\frac{da}{\sqrt{2}u}c(m^M(K){z})\Bigr)\Big]^{da}, \\
       &\mathrm{T}_u=\frac{1}{2}\frac{\partial^2}{\partial c^2}\str\Big[g\exp\Bigl(-\lk-\scrbku^2+c\frac{i\widetilde{\mu}(K)}{u}\Bigr)\Big]_{c=0}.
        \end{split}
    \end{equation}
    These generalized supertraces are well defined thanks to (\ref{estimate scr kernel}).
 \end{definition}

 As in (\ref{DM}), we can define the corresponding operators $D^N$, $D^L$ acting on the fibers of $N$ and $L$.  Let $P^N$, $P^L$ be the orthogonal projection from $M$ to $N,L$.
 Clearly, we have,
 \begin{equation}
    D^M=D^N+D^L.
 \end{equation}
  Then we have an analogue
   of \cite[Theorem 3.12]{B1990Kuszul}.
 \begin{lemma}\label{lemma scr}
    Let $\mathscr{L}_u$ be the operator definded as
    \begin{equation}
     \mathscr{L}_u=\nabla^{I}+D^{N}-\frac{c(P^N R^M Z)}{2\sqrt{2}}+\frac{c(m^M(K)Z)}{2\sqrt{2}}\Big(1-\frac{1}{u}\Big).
    \end{equation}
    We can split $\mathscr{L}_u$ into  $\mathscr{L}_u'+ \mathscr{L}_u''$ as usual.
  And we have the following identities.
\begin{equation}\label{bp}
    [\mathscr{L}_u,\lk+\scrbku^2]=[\mathscr{L}_u',\lk+\scrbku^2]=[\mathscr{L}_u'',\lk+\scrbku^2]=0.
\end{equation}
 \end{lemma}
 \begin{proof}
    We split $\scrbku$ into two parts,
    \begin{equation}\label{split of scrbku}
        \begin{split}
        \scrbku=\Big(\nabla^{I}+D^{N}-\frac{c(P^N R^M Z)}{2\sqrt{2}}+\frac{c(m^M(K)Z)}{2\sqrt{2}}\Big(1-\frac{1}{u}\Big)\Big)\\ 
        +\Big(D^L-\frac{c(P^L R^M Z)}{2\sqrt{2}}\Big).
        \end{split}
    \end{equation}

    So we can rewrite the Bianchi identity $[\scrbku,\lk+\scrbku^2]=0$ in the form
    \begin{equation}\label{rewrite bianchi}
        \begin{split}
            \Big[\nabla^{I}+D^{N}-\frac{c(P^N R^M Z)}{2\sqrt{2}}+\frac{c(m^M(K)Z)}{2\sqrt{2}}\Big(1-\frac{1}{u}\Big),\lk+\scrbku^2\Big]\\+\Big[D^L-\frac{c(P^L R^M Z)}{2\sqrt{2}},\lk+\scrbku^2\Big]=0.
        \end{split}
    \end{equation}
    Remember that $\Lambda(\overline{M}^*)=\Lambda(\overline{L}^*)\widehat{\otimes}\Lambda(\overline{N}^*).$ The first part in (\ref{split of scrbku}) preserves the degree of $\Lambda(\overline{L}^*)$ while the other part does not preserve. Also note that by (\ref{lichnerowicz scrbku}), the operator $\lk+\scrbku^2$ preserves the degree in $\Lambda(\overline{L}^*)$. So it is clear that both commutators in (\ref{rewrite bianchi}) vanish. Therefore we get the first result in \eqref{bp}. The same arguments work for $\scrbku'$ and $\scrbku''$, by using $[\scrbku',\lk+\scrbku^2]=0$ and $[\scrbku'',\lk+\scrbku^2]=0$ instead (cf.\,\eqref{scrbku commutator}).
 \end{proof}

 \begin{proposition}\label{prop, piu=0}
    For $u\geq 1$,
    \begin{equation}\label{piu=0}
        \begin{split}
            \Pi_u'=\partial \mathrm{T}_u=0,\\ 
            \Pi_u''=-\dba \mathrm{T}_u=0.
        \end{split}
    \end{equation}
 \end{proposition}
 \begin{proof}
By Definition \ref{bf},
\begin{equation}\label{piu=0 calculate}
    \begin{split}
    &\Pi_u'=\int_{N_{\rr}}\str^{\Lambda(\overline{N}^*)}\Big[g\frac{c(m^M(K)\bar{z})}{\sqrt{2}u}\exp\Big(-\lk-\scrbku^2 \\ &\qquad\qquad\qquad\qquad\qquad\qquad +da\dabar\frac{i\widetilde{\mu}(K)}{u}\Big) (g^{-1}Z,Z)\Big]^{da\dabar}\frac{dv_N}{(2\pi)^{\mathrm{dim}N}}.
    \end{split}
\end{equation}
The heat kernel $\exp(-\lk-\scrbku^2+da\dabar\frac{i\widetilde{\mu}(K)}{u}) (g^{-1}Z,Z)$ is calculated in \cite[Theorem 4.13]{BismutGoette2000}, which is an even function on $Z$ and can be controlled by $c\exp(-C|Z|^2)$. So the right-hand side of (\ref{piu=0 calculate}) is an integral of an odd function on $Z\in N_{\rr,y}$, which is zero.

By \eqref{scrbku commutator},
\begin{equation}\label{piu=0 proof}
        \begin{split}
        \Pi_u'=\str\Big[i\frac{\widetilde{\mu}(K)}{u}g\exp\big(-\lk-\scrbku^2+\dabar \big[\scrbku',i\frac{\widetilde{\mu}(K)}{u}\big]\big)\Big]^{\dabar}.
        \end{split}
    \end{equation}

If \eqref{piu=0 proof} is a supertrace in the normal sense, classically we have 
\begin{align}\label{piu=dsm}
 \Pi_u'=\frac{1}{2}\partial\, \str\Big[i\frac{\widetilde{\mu}(K)}{u}g\exp\big(-\lk-\scrbku^2+\dabar \frac{\widetilde{i\mu}(K)}{u}\big)\Big]^{\dabar}=  \partial \mathrm{T}_u.
\end{align}
Now $\Pi_u'$ is just a generalized supertrace, but \eqref{piu=dsm} remains valid by using the same arguments as in the proof of \cite[Theorem 4.6]{B1990Kuszul}, in particular Theorem \ref{thm-transition invariance} and Lemma \ref{lemma scr} are used as analogues of \cite[Theorems 3.12, 4.2]{B1990Kuszul}.

    Similar results hold true for $\Pi_u''$, which complete our proof.
 \end{proof}

\begin{proposition}\label{prop lambda theta}
    The following identity holds,
    \begin{equation}\label{lambdau thetau}
        \begin{split}
            &\Lambda_u=u\derivative{u} \str\Big[i\frac{\widetilde{\mu}(K)}{u}g\exp\big(-\lk-\scrbku^2\big)\Big].
        \end{split}
    \end{equation}
\end{proposition}
\begin{proof}
    By proceeding as in the proof of Proposition \ref{prop, piu=0}, we have
    \begin{equation}\label{bz}
        \begin{split}
            &u\derivative{u}\str\Big[i\frac{\widetilde{\mu}(K)}{u}g\exp\big(-\lk-\scrbku^2\big)\Big]\\ 
            =\derivative{b}&\str\Big[i\widetilde{\mu}(K)g\exp\Big(-\lk-\scrbku^2-b\big[\scrbku'',\derivative{u}\scrbku'\big]-b\big[\scrbku',\derivative{u}\scrbku''\big]\Big)\Big]_{b=0}\\ 
            &-\str\Big[i\frac{\widetilde{\mu}(K)}{u}g\exp\big(-\lk-\scrbku^2\big)\Big]\\ 
            =&\dba \derivative{b}\str\Big[i\widetilde{\mu}(K)g\exp\Big(-\lk-\scrbku^2-b\derivative{u}\scrbku'\Big)\Big]_{b=0}\\ 
           & +\partial \derivative{b}\str\Big[i\widetilde{\mu}(K)g\exp\Big(-\lk-\scrbku^2-b\derivative{u}\scrbku''\Big)\Big]_{b=0}\\ 
            &+\str\Big[\big[\scrbku'',i\widetilde{\mu}(K)\big]g\exp\Big(-\lk-\scrbku^2-\dabar\derivative{u}\scrbku'\Big)\Big]^{\dabar}\\
            &+\str\Big[\big[\scrbku',i\widetilde{\mu}(K)\big]g\exp\Big(-\lk-\scrbku^2-da\derivative{u}\scrbku''\Big)\Big]^{da}\\ 
           &-\str\Big[i\frac{\widetilde{\mu}(K)}{u}g\exp\big(-\lk-\scrbku^2\big)\Big]\\ 
          = &-\frac{1}{2u} \dba \Pi_u'- \frac{1}{2u}\partial\Pi_u''+\Lambda_u.
        \end{split}
    \end{equation}
    Using the result of Proposition \ref{prop, piu=0}, we obtain  (\ref{lambdau thetau}). 
\end{proof}

\begin{theorem}\label{bg}
    As $u\rightarrow +\infty$, we have
    \begin{equation}\label{phi Thetau phi Lambdau}
        \begin{split}
            \varphi\Lambda_u=\mathcal{O}\Big(\frac{1}{u}\Big).
        \end{split}
    \end{equation}
\end{theorem}
\begin{proof}
    By \eqref{bi} and \eqref{bh}, $\str\big[i\frac{\widetilde{\mu}(K)}{u}g\exp(-\lk-\scrbku^2)\big]$ is an analytic function of $u$ at $+\infty$. By \eqref{lambdau thetau}, as $u\to +\infty$, we have
     \begin{align}\label{bk}
     \Lambda_u=u\frac{\partial}{\partial u}\str\Big[i\frac{\widetilde{\mu}(K)}{u}g\exp(-\lk-\scrbku^2)\Big]=\mathcal{O}\Big(\frac{1}{u}\Big).
     \end{align}
     So we get  \eqref{phi Thetau phi Lambdau}. 
\end{proof}

\begin{proposition}\label{prop partial thetau}
    For $u\geq 1$, we have
    \begin{equation}\label{partial thetau}
        \begin{split}
            \dba \Theta_u=0,\quad \partial\Theta_u=0,\\ 
            \dba \Lambda_u=0,\quad \partial\Lambda_u=0.
        \end{split}
    \end{equation}
\end{proposition}
\begin{proof}
     From the proof of Theorem \ref{bg}, $\Theta_u$ is a combination of characteristic classes on $B$, so we get the first line of (\ref{partial thetau}). By \eqref{bh} and (\ref{lambdau thetau}), we obtain the second line.
\end{proof}
\section{A proof of Theorem \ref{8o}}\label{dz}
The aim of this section is to establish the main result, Theorem \ref{8o}, which is a family version of \cite[Theorem 5.1]{BismutGoette2000}. Our proof is based on the methods in \cite[\S 6]{B1997} and \cite[\S 6]{BismutGoette2000}.
 
Inspired by \cite[\S 6]{BismutGoette2000}, we will construct a differential form $\gamma_{v,t}$ on $\mathbb{R}_{+}^{*}\times \mathbb{R}_{+}^{*}\times S$ and compute its integral over the contour $\Gamma_{a,A}$ shown in the graph in Section \ref{br}. We will show that the contour integral $\int_{\Gamma_{a,A}}\gamma_{v,t}$ lies in $P^{S,0}$. More precisely, it has the following explicit expression as a double integral: $\int_{\Gamma_{a,A}}\gamma_{v,t}=\Phi \overline{\partial}\int_{\Delta_{a,A}}\alpha_{v,t}\,dv\wedge dt+\Phi \partial\int_{\Delta_{a,A}}\beta_{v,t}\,dv\wedge dt$. Here, the integrations are taken with respect to $(v,t)$, while $\partial$ and $\overline{\partial}$ act on $S$. Using some intermediate results to be proved in Sections \ref{ek}--\ref{el}, we study the asymptotic behavior of $\int_{\Gamma_{a,A}}\gamma_{v,t}$ as $A\to+\infty$, then $a\to0$. We show that the divergent terms in the asymptotic expansion of the contour integral belong to $P^{S,0}$, while the remaining terms are exactly those appearing in \eqref{91}. By passing to the limit, we conclude that \eqref{91} holds modulo the closure of $P^{S,0}$.  In the special case where $S$ is a closed K\"ahler manifold, $P^{S,0}$ is closed, so we obtain Theorem \ref{8o}.

The discussion above is similar to that in \cite{BismutGoette2000}, where $S=\left\{\mathrm{pt}\right\}$, but there are several differences. In \cite{BismutGoette2000}, the most difficult part is the estimates of a supertrace involving the moment map $\left\langle\mu,K\right\rangle$ (cf. \eqref{qs}). They introduce a splitting $\left\langle\mu,K\right\rangle=\mu_{1}+\mu_{2}$ and treat separately the supertraces corresponding to $\mu_{1}$ and $\mu_{2}$. In our case, we need to modify this splitting so that the resulting functions $\mu_{1}$ and $\mu_{2}$ are constant in the horizontal direction. We will explain this in detail in Section \ref{bq}. Moreover, when studying the supertrace corresponding to $\mu_{2}$, two extra terms arise in our case (cf. \eqref{hvt}), compared with \cite{BismutGoette2000}. We therefore need some additional analysis of these new terms, which is given in Theorem \ref{h,t01}.

In general, $P^{S,0}$ need not be closed under uniform convergence. To show that the limit still lies in $P^{S,0}$, we therefore need to study precisely the asymptotic behavior of the double integral $\Phi \overline{\partial}\int_{\Delta_{a,A}}\alpha_{v,t}\,dv\wedge dt+\Phi \partial\int_{\Delta_{a,A}}\beta_{v,t}\,dv\wedge dt$ as $A\to+\infty$, then $a\to 0$. This is the most challenging part of our proof. We divide the terms appearing in the double integral into two classes, corresponding respectively to $\mu_{1}$ and $\mu_{2}$ (cf. \eqref{sum of Ik0}). The terms corresponding to $\mu_{1}$ have structures similar to those on the right-hand side of \cite[(4.26)]{B1997}, and we treat them by following the approach of \cite{B1997}. In particular, the results in Section \ref{em} are important for this part and play a role similar to that of \cite[\S 5]{B1997} in the proof of their main result. By contrast, the terms involving $\mu_{2}$ are quite different. Motivated by the algebraic identity in \cite[Proposition 8.32]{BismutGoette2000}, we will establish several identities that reduce the analysis of the terms involving $\mu_{2}$ to that of certain expressions involving $d\mu_{2}$ (cf. \eqref{ddbarpartialt kappa1}, \eqref{kappa1va}, and \eqref{extension liuma}). We then use the key fact that $d\mu_{2}=0$ in a neighborhood of $W_{g,K}$ to simplify the estimates. In summary, we show that the divergent terms in the double integral agree exactly with those of $\int_{\Gamma_{a,A}}\gamma_{v,t}$. After removing these divergent parts, the limit of the double-integral expression belongs to $P^{S,0}$, while the remaining part of the contour integral is exactly the expression in \eqref{91}. This completes the proof of our main theorem.

This section is organized as follows. In Section \ref{bq}, we introduce the splitting $\left\langle\mu,K\right\rangle=\mu_{1}+\mu_{2}$. In Section \ref{br}, we introduce a two-parameter Bismut superconnection $\cvt$, which is equal to $\bisconn_{t}$ in \eqref{Bismut superconnection} when $v=t$ and deforms to $\bisconn_{K,t}$ in \eqref{bku} as $v\to+\infty$. Then we define the differential form $\gamma_{v,t}$ and corresponding contour integral. In Section \ref{bt}, we state a series of intermediate results, whose proofs are given in Sections \ref{ek}--\ref{el}. In Section \ref{bu}, we study the asymptotic behavior of the contour integral and prove Theorem \ref{8o} when $S$ is closed and K\"ahler. Finally, in Sections \ref{bs}--\ref{b6}, we treat the general case by studying the double integrals in much detail. The methods in \cite[\S 6.6--6.8]{B1997} will be the guidance of us.

\subsection{A splitting of $\product{\mu,K}$}\label{bq}
Before we start to prove the main theorem, we introduce a division of  the function $\product{\mu,K}$
 as in \cite[\S 8.3]{BismutGoette2000}.

For $K\in \mathfrak{z}(g)$, by \eqref{dmuk}, $\left\langle \mu,K \right\rangle$ is constant function on each component of $W_K=\left\{ x\in W: K^X(x)=0 \right\}$. Let $\mu_K^1, \cdots, \mu_K^m$ be the distinct values of $\left\langle \mu,K \right\rangle$ on $W_K$. We choose $\varepsilon>0$ small enough such that the intervals $[\mu_K^i-\varepsilon, \mu_K^i+\varepsilon]$ are pairwise disjoint.

Let $\gamma(s)$ be a smooth function from $\rr
$ to $[0,1]$ such that 
\begin{equation}\label{gammas}
    \begin{split}
        \gamma(s)&=1 \qquad \mathrm{if} \kongge |s|<\varepsilon/2,\\ 
        &=0 \qquad \mathrm{if} \kongge |s|\geq \varepsilon.
    \end{split}
\end{equation}

\begin{definition}
    Put 
    \begin{equation}\label{fy}
        \begin{split}
            &\mu_1=\product{\mu,K}-\mu_2,\\ 
            &\mu_2=\sum_{i=1}^m \mu_K^i \gamma(\left\langle \mu,K \right\rangle-\mu_K^i).
        \end{split}
    \end{equation}
\end{definition}

By (\ref{dmuk}) , (\ref{gammas}), we get the key features of $\mu_1$ and $\mu_2$, 
\begin{equation}\label{proposition of mu1mu2}
    \begin{split}
        &\mu_1=\mathcal{O}(|Z|^2) \kongge \mathrm{near} \kongge W_K,\\ 
        &\mu_2 \kongge\mathrm{is}\kongge \mathrm{locally} \kongge \mathrm{constant} \kongge \mathrm{near} \kongge W_K. 
    \end{split}
\end{equation}

By using \eqref{proposition of mu1mu2}, we can proceed as in the proof of  Proposition \ref{prop rho1} to get refinement of \eqref{bo},   if $\eta\in P^{W}$ is $d_K$-closed form on
    \begin{align}\label{obo}
        &\int_{X} -i\mu_1 \eta {\rho}_1=\int_{X_K}\eta c_{\mathrm{\mathrm{max}},K}^{-1,\prime}(N_{X_K/X},h^{N_{X_K/X}}),
&&\int_{X} -i\mu_2 \eta{\rho}_1\in  P^{S,0}.
    \end{align}

Recall that in \eqref{a8}, \eqref{ac}, we explain a construction of the equivariant Bott-Chern current using the moment map $\mu$. As an analogue of the splitting of $\mu$, for
 $k=1,2$, we set
\begin{equation}\label{def of dvt}
    \begin{split}
    &\widetilde{d}_{t}^k=-i\frac{\mu_k}{t}\exp\Big(\frac{\overline{\partial}_K\partial_K}{2\pi t}\product{\mu,K}\Big).
    \end{split}
\end{equation}
Therefore we get a splitting of $\widetilde{d}_{K,t}$ in \eqref{a8}.

Similarly, for $k=1,2$, let $K_k^X\in T_{\rr}X$ be the vertical vector field on $W$ given by
\begin{equation}
    d\mu_k+i_{K_k^X}\omega=0.
\end{equation}
So we get a splitting of $K^X$.

\subsection{A contour integral}\label{br}
\begin{definition}
    For $t>0, v>0$,
    set
    \begin{equation}\label{definition cvt}
        \begin{split}
            &\cvt''=\bisconn_t''+\frac{\sqrt{t}c(K^{X(1,0)})}{2\sqrt{2}}\Big(\frac{1}{t}-\frac{1}{v}\Big),\\
           &\cvt'=\bisconn_t'+\frac{\sqrt{t}c(K^{X(0,1)})}{2\sqrt{2}}\Big(\frac{1}{t}-\frac{1}{v}\Big),\\
           &\cvt=\cvt''+\cvt'.
        \end{split}
    \end{equation} 
\end{definition}
Let $\gamma_{v,t}$ be the smooth form on $\rr_+^{*2}\times S$ defined by
\begin{equation}\label{gammavt}
    \gamma_{v,t}=\frac{dt}{t}\trace_s\big[N_{K,t}g\exp\big(-\lk-\cvt^2\big)\big]+\frac{dv}{v^2}\trace_s\big[i\product{\mu,K}g\exp\big(-\lk-\cvt^2\big)\big].
\end{equation} 

Let $\alpha_{v,t},\beta_{v,t}$ be the forms on $\rr_+^{*2}\times S$, defined by
\begin{equation}\label{alphabeta}
\begin{split}
    \alpha_{v,t}=\derivative{b}\Big|_{b=0}\Big\{\trace_s\Big[\frac{N_{K,t}}{t}g\exp\Big(-\lk-\cvt^2-b\frac{\sqrt{t}c(K^{X(0,1)})}{2\sqrt{2}v^2}\Big)\Big]\\ 
    -\trace_s\Big[i\frac{\product{\mu,K}}{v^2}g\exp\Big(-\lk-\cvt^2-b\derivative{t}\cvt'\Big)\Big]\Big\},\\ 
    \beta_{v,t}=\derivative{b}\Big|_{b=0}\Big\{\trace_s\Big[\frac{N_{K,t}}{t}g\exp\Big(-\lk-\cvt^2-b\frac{\sqrt{t}c(K^{X(1,0)})}{2\sqrt{2}v^2}\Big)\Big]\\ 
    -\trace_s\Big[i\frac{\product{\mu,K}}{v^2}g\exp\Big(-\lk-\cvt^2-b\derivative{t}\cvt''\Big)\Big]\Big\}.
\end{split}
\end{equation}

\begin{theorem}\label{thm dclosed form}
    The following identity holds,
    \begin{equation}\label{dclosed form}
        (dv\wedge \frac{\partial}{\partial v}+dt\wedge\frac{\partial}{\partial t})\gamma_{v,t}=dv\wedge dt \wedge (\overline{\partial}^S\alpha_{v,t}+{\partial}^S\beta_{v,t}).
    \end{equation}
\end{theorem}
\begin{proof}
    By the analogue of (\ref{bku-algebraic identity}), we get
    \begin{equation}\label{cvt-algebraic identity}
        \begin{split}
        \cvt^2=[\cvt'',\cvt'],\kongge\quad [\cvt'',\cvt^2]=0,\kongge\quad [\cvt',\cvt^2]=0.
        \end{split}
    \end{equation}
From (\ref{cvt-algebraic identity}), we obtain
\begin{equation}\label{derivative of cvt2}
    \begin{split}
        &\derivative{v}\cvt^2=[\cvt'',\derivative{v}\cvt']+[\cvt',\derivative{v}\cvt''].
    \end{split}
\end{equation}
Using  (\ref{cvt-algebraic identity}), (\ref{derivative of cvt2}) and the fact that supertraces vanish on supercommutators, we have
\begin{equation}\label{calculation derivative v}
    \begin{split}
       \derivative{v}\trace_s[N_{K,t}g&\exp(-\lk- \cvt^2)]\\
       =\derivative{b}\Big|_{b=0}\trace_s[&N_{K,t}g \exp(-\lk-\cvt^2-b[\cvt'',\derivative{v}\cvt']-b[\cvt',\derivative{v}\cvt''])]\\ 
        =&\overline{\partial}^S \derivative{b}\Big|_{b=0}\trace_s\Big[N_{K,t}g\exp\Big(-\lk-\cvt^2-b\frac{\sqrt{t}c(K^{X(0,1)})}{2\sqrt{2}v^2}\Big)\Big]\\
        +&{\partial}^S \derivative{b}\Big|_{b=0}\trace_s\Big[N_{K,t}g\exp\Big(-\lk-\cvt^2-b\frac{\sqrt{t}c(K^{X(1,0)})}{2\sqrt{2}v^2}\Big)\Big]\\
        -&\derivative{b}\Big|_{b=0}\trace_s\Big[[\cvt'',N_{K,t}]g\exp\Big(-\lk-\cvt^2-b\frac{\sqrt{t}c(K^{X(0,1)})}{2\sqrt{2}v^2}\Big)\Big]\\
        -&\derivative{b}\Big|_{b=0}\trace_s\Big[[\cvt',N_{K,t}]g\exp\Big(-\lk-\cvt^2-b\frac{\sqrt{t}c(K^{X(1,0)})}{2\sqrt{2}v^2}\Big)\Big].
    \end{split}
\end{equation}
Then by (\ref{bku-algebraic identity}), (\ref{calculation derivative v}), we get
\begin{equation}\label{calculation2 derivative v}
    \begin{split}
        \derivative{v}\trace_s[\frac{N_{K,t}}{t}g&\exp(-\lk- \cvt^2)]=\\
        =&\overline{\partial}^S \derivative{b}\Big|_{b=0}\trace_s\Big[\frac{N_{K,t}}{t}g\exp\Big(-\lk-\cvt^2-b\frac{\sqrt{t}c(K^{X(0,1)})}{2\sqrt{2}v^2}\Big)\Big]\\
        +&{\partial}^S \derivative{b}\Big|_{b=0}\trace_s\Big[\frac{N_{K,t}}{t}g\exp\Big(-\lk-\cvt^2-b\frac{\sqrt{t}c(K^{X(1,0)})}{2\sqrt{2}v^2}\Big)\Big]\\
        +&\derivative{b}\Big|_{b=0}\trace_s\Big[2\derivative{t}\cvt''g\exp\Big(-\lk-\cvt^2-b\frac{\sqrt{t}c(K^{X(0,1)})}{2\sqrt{2}v^2}\Big)\Big]\\
        -&\derivative{b}\Big|_{b=0}\trace_s\Big[2\derivative{t}\cvt'g\exp\Big(-\lk-\cvt^2-b\frac{\sqrt{t}c(K^{X(1,0)})}{2\sqrt{2}v^2}\Big)\Big].
    \end{split}
\end{equation}

As in \eqref{af}, we have
\begin{equation}\label{commutator c mu}
    [\cvt'',i\product{\mu,K}]=-\frac{\sqrt{t}c(K^{X(1,0)})}{\sqrt{2}},\qquad  [\cvt',i\product{\mu,K}]=\frac{\sqrt{t}c(K^{X(0,1)})}{\sqrt{2}}.
\end{equation}
By proceeding as in (\ref{calculation derivative v}), (\ref{calculation2 derivative v}), we find that
\begin{equation}\label{calculation derivative t}
    \begin{split}
        \derivative{t}\frac{1}{v^2}\trace_s[i\product{\mu,K}g&\exp(-\lk- \cvt^2)]=\\
        =&\overline{\partial}^S \derivative{b}\Big|_{b=0}\trace_s\Big[i\frac{\product{\mu,K}}{v^2}g\exp\Big(-\lk-\cvt^2-b\derivative{t}\cvt'\Big)\Big]\\
        +&{\partial}^S \derivative{b}\Big|_{b=0}\trace_s\Big[i\frac{\product{\mu,K}}{v^2}g\exp\Big(-\lk-\cvt^2-b\derivative{t}\cvt''\Big)\Big]\\
        +&\derivative{b}\Big|_{b=0}\trace_s\Big[\frac{\sqrt{t}c(K^{X(1,0)})}{\sqrt{2}v^2}g\exp\Big(-\lk-\cvt^2-b\derivative{t}\cvt'\Big)\Big]\\
        -&\derivative{b}\Big|_{b=0}\trace_s\Big[\frac{\sqrt{t}c(K^{X(0,1)})}{\sqrt{2}v^2}g\exp\Big(-\lk-\cvt^2-b\derivative{t}\cvt''\Big)\Big].
    \end{split}
\end{equation}

From (\ref{calculation2 derivative v}), (\ref{calculation derivative t}), we get (\ref{dclosed form}).
\end{proof}

Take $0<a\leq 1\leq A<+\infty. $ Let $\Gamma=\Gamma_{a,A} $ be the oriented contour in $\rr_+^2$:

\begin{center}
    \begin{tikzpicture}
        \draw[thick,-latex] (-0.3,0) -- (5.5,0) node[anchor=north west] {$v$};
        \draw[thick,-latex] (0,-0.3) -- (0,5.5) node[anchor=south east] {$t$};
        \node[below=0.2cm, left=-0cm] at (-0,0) {$0\strut$};
        \node at (4.6,2.5){$\Gamma_1$};
        \node at (2.3,2.9){$\Gamma_2$};
        \node[below] at (1,0) {$a$};
        \node[below] at (4.2,0) {$A$};
        \node[left] at (0,1) {$a$};
        \node[left] at (0,4.2) {$A$};
        \node at (3.2,1.8) {$\Delta$};
    
        \begin{scope}[dashed]
            \draw (4.2,0) -- (4.2,1);
            \draw (1,0) --(1,1);
            \draw (0,1) --(1,1);
            \draw (0,4.2) --(4.2,4.2);
        \end{scope}
        \begin{scope}[thick, every node/.style={sloped,allow upside down}]
            \draw (1,1) -- node {\midarrow} node[below=0.1cm]{$\Gamma_3$} (4.2,1);
        \draw (4.2,1) --node {\midarrow}  (4.2,4.2);
        \draw (4.2,4.2) --node {\midarrow}  (1,1);
        \end{scope}
        
        \end{tikzpicture}
    \end{center}

    As indicated in the picture, $\Gamma$ is made of three oriented pieces: $\Gamma_1,\Gamma_2,\Gamma_3$. And $\Gamma$ bounds an oriented triangular domain $\Delta=\Delta_{a,A}.$ For $k=1,2,3$, we set  
        \begin{equation}\label{ft}
            I_k^0=\Phi\int_{\Gamma_k}\gamma_{v,t}.
        \end{equation}

        By Stokes' formula and Theorem \ref{thm dclosed form},
        \begin{equation}\label{0 sum of Ik0}
            \sum_{k=1}^3 I_k^0=\Phi\int_{\partial\Delta}\gamma_{v,t}=
            \Phi\int_{\Delta}(dv\wedge\frac{\partial}{\partial v}+dt\wedge\frac{\partial}{\partial t})\gamma_{v,t}
        =\Phi\overline{\partial}\alpha+\Phi{\partial}\beta.
        \end{equation}
        Here $\alpha,\beta$ are the forms on $S$ defined by 
        \begin{align}
            &\alpha=\int_{\Delta_{a,A}}\alpha_{v,t}dv\wedge dt, &&\beta=\int_{\Delta_{a,A}}\beta_{v,t}dv\wedge dt.
        \end{align}

        \subsection{Intermediate results on the left-hand side of (\ref{0 sum of Ik0})}\label{bt}
        We now state some intermediate results, which extend \cite[Section 6.3]{BismutGoette2000} to the family version. These resluts will be used to study the left hand side of (\ref{0 sum of Ik0}) as $A\rightarrow +\infty$, and then $a\rightarrow 0$.

        \begin{theorem}\label{main lemma1}
            For $K\in\mathfrak{z}(g)\otimes_{\mathbb{R}} \mathbb{C}$ with $|K|< 1$ and any $t_0>0$ , there exists $\gamma\in\, ]0,1]$, $C>0$, such that for $t_{0}\leq t\leq v$, we have
            \begin{equation}
                \big|\Phi\trace_s[N_{K,t}g\exp(-\lk-\cvt^2)]-\ch_{ge^K}'(H(X,\xi|_{X}),h^{H(X,\xi|_{X})})\big|\leq \frac{C}{t^{\gamma}}.
            \end{equation} 
        \end{theorem}
        \begin{proof}
             The proof will be delayed to Section 6.
        \end{proof}

        Recall that $\widetilde{d}_{K,t}$ is defined in (\ref{a8}),  set
        \begin{equation}\label{ev formula}
            \begin{split}
            &\widetilde{e}_v=\int_{X_g}\td_{g,K}(TX,h^{TX})\ch_{g,K}(\xi,h^{\xi})\widetilde{d}_{K,v},\\
            &\widetilde{D}_{-1,g,K}=\int_{X_{g,K}}\td_{ge^K}(TX,h^{TX})\ch_{ge^K}(\xi,h^{\xi})i\product{\mu,K},\\ 
            &E_{0,g,K}=\int_{X_{g,K}}\td_{ge^K}(TX,h^{TX})\Big(\frac{c_{\mathrm{max}}'}{c_{\mathrm{max}}}\Big)_K(N_{X_{g,K/X_g}},h^{X_{g,K/X_g}})\ch_{ge^K}(\xi,h^{\xi}). 
            \end{split}
        \end{equation}

        By applying  \eqref{bo} and (\ref{bw}) to $X_g$, we have the expansion: as $v\rightarrow 0$,
        \begin{equation}\label{expansion ev}
            \begin{split}
            &\widetilde{e}_v=-\frac{\widetilde{D}_{-1,g,K}}{v}-E_{0,g,K}'+\mathcal{O}(v),\\
            &E_{0,g,K}'=E_{0,g,K}\kongge\qquad\mathrm{in} \kongge P^S/P^{S,0}.
            \end{split}
        \end{equation}

 Now we state two intermediate lemmas concerning about the case when $t$ is small. We take $K_0\in\mathfrak{z}(g)$ and let $K=zK_0$.

Recall that we have the splitting $\left\langle \mu,K \right\rangle=\mu_1+\mu_2$, and that  $da$ is the canonical generator of $\mathbb{C}^{*}$.   For $0<t\leq v$,  set 
        \begin{equation}\label{hvt}
            \begin{split}
                \xi_{v,t}'=\frac{1}{2}\int_{0}^{t}\str\Big[\frac{i\mu_2}{v}g\exp\Big(-\lk-\mathcal{A}_{v,s}^2-2\dabar s\derivative{s}\mathcal{A}_{v,s}'\Big)\Big]^{\dabar} \frac{ds}{s},\\
                \xi_{v,s}''=\frac{1}{2}\int_{0}^{t}\str\Big[\frac{i\mu_2}{v}g\exp\Big(-\lk-\mathcal{A}_{v,s}^2-2da s\derivative{s}\mathcal{A}_{v,s}''\Big)\Big]^{da} \frac{ds}{s}.
            \end{split}
        \end{equation}
      By Theorem \ref{h,t01} and \eqref{pivt}, $\xi_{v,t}', \xi_{v,t}''$ are well-defined smooth forms on $S$, depending smoothly on $(v,t)$.

        \begin{theorem}\label{main lemma 1}
            For $z\in \rr$, $|z|\neq 0$ small enough,
           for any $v>0$, as $t\rightarrow 0$, we have
            \begin{equation}\label{limit of main}
                \Phi\str\Big[i\frac{\product{\mu,K}}{v}g\exp\Big(-\lk-\cvt^2\Big)\Big]\rightarrow -\widetilde{e}_v.
            \end{equation}
            Moreover, there exist $C>0$, $\gamma\in\, ]0,1]$, such that for $0<t\leq v\leq 1$,
            \begin{equation}\label{t01v01}
                \Big|\Phi\str\Big[i\frac{\product{\mu,K}}{v}g\exp\Big(-\lk-\cvt^2\Big)\Big]+\widetilde{e}_v-\Phi\partial \xi_{v,t}'-\Phi\dba \xi_{v,t}'' \Big|\leq C\Big(\frac{t}{v}\Big)^{\gamma}.
            \end{equation}  
            And for $0<t\leq 1\leq v$,
            \begin{equation}\label{t01v1}
                \Big|\Phi\str\Big[i\frac{\product{\mu,K}}{v}g\exp\Big(-\lk-\cvt^2\Big)\Big]\Big|\leq \frac{C}{v}.
            \end{equation}
        \end{theorem}
        \begin{proof}
        It will be proved in Section \ref{ek}.
        \end{proof}
 
        We apply the constructions in Section \ref{em} to $B=W_{g,K}$, $M=TX|_{W_{g,K}}$, therefore $\varphi\str\Big[i\frac{\widetilde{\mu}(K)}{v}g\exp\Big(-\lk-\mathscr{B}_{K,v}^2\Big)\Big]$ in \eqref{basic generalized supertraces} is a smooth form on $W_{g,K}$.  For $v\geq 1$, set
        \begin{equation}\label{i5}
            \widetilde{D}_{0,g,K}^v=\int_{X_{g,K}}\varphi\str\Big[i\frac{\widetilde{\mu}(K)}{v}g\exp\Big(-\lk-\mathscr{B}_{K,v}^2\Big)\Big]\ch_{ge^K}(\xi,h^{\xi}).
        \end{equation}

        \begin{theorem}\label{main lemma 2}
            For $z\in\rr$, $|z|\neq 0$ small enough, for any $v\geq 1$\,, when $t\in \,]0,1]\rightarrow 0$, we have 
            \begin{align}\label{expansion hvt}
                \begin{split}
                   \Phi\str\Big[i\frac{\product{\mu,K}}{tv}g\exp\Big(-\lk-\mathcal{A}_{tv,t}^2\Big)\Big]-\Phi\dba \xi_{tv,t}'-\Phi\partial \xi_{tv,t}''\\ =\frac{\widetilde{D}_{-1,g,K}}{tv}+\widetilde{D}_{0,g,K}^{v\prime}+o(1),
                   \\ 
                 \widetilde{D}_{0,g,K}^{v\prime}-\widetilde{D}_{0,g,K}^{v}=E_{0,g,K}'-E_{0,g,K}\in P^{S,0}.
                \end{split}
            \end{align}
        \end{theorem}
        \begin{proof}
            Theorem \ref{main lemma 2} will be proved in Section \ref{el}.
        \end{proof}
        \begin{remark}
            The proof of Theorem \ref{main lemma 1} and \ref{main lemma 2} will be separeted into two parts, by splitting $\product{\mu,K}$ into the sum of $\mu_1$ and $\mu_2$. In particular, we will prove that for any $v\geq 1$\,, when $t\in\, ]0,1]\rightarrow 0$, we have 
            \begin{equation}\label{mu1 part of main lemma 2}
                \Phi\str\Big[i\frac{\mu_1}{tv}g\exp\Big(-\lk-\mathcal{A}_{tv,t}^2\Big)\Big]=\widetilde{D}_{0,g,K}^v+o(1).
            \end{equation}
        \end{remark}
   
        \subsection{The asymptotics of the left hand side of \eqref{0 sum of Ik0}}\label{bu}
        We will study the asymptotic bahaviors of $I_k^0$ in \eqref{ft} in  for $k=1,2,3$, as $A\rightarrow +\infty,$ and then $a\rightarrow 0$, using the intermediate results in the last section. 

        \textbf{1.The term $I_1^0$.}
        By (\ref{gammavt}), it is clear that 
        \begin{equation}
            I_1^0=\int_a^A \Phi\trace_s\Big[N_{K,t}g\exp\Big(-\lk-\mathcal{A}_{A,t}^2\Big)\Big]\frac{dt}{t}.
        \end{equation}

        a)$\kongge A\rightarrow +\infty$. As $A\rightarrow +\infty$,
\begin{equation}
    \begin{split}
   &\Phi\trace_s\Big[N_{K,t}g\exp\Big(-\lk-\mathcal{A}_{A,t}^2\Big)\Big]\\ 
    &    \rightarrow \gamma_{t,g,K}:=\Phi\trace_s\Big[N_{K,t}g\exp\Big(-\lk-\bisconn_{K,t}^2\Big)\Big].
    \end{split}
\end{equation}

        From Theorem \ref{main lemma1}, for $|z|\leq 1$, there exists $\gamma\in\, ]0,1],$ $C>0$, such that for $0<a\leq t\leq A,$ 
        \begin{equation}
            \Big|\Phi\trace_s\Big[N_{K,t}g\exp\Big(-\lk-\mathcal{A}_{A,t}\Big)\Big]-\ch_{ge^K}'(H(X,\xi|_{X}),h^{H(X,\xi|_{X})})\Big|\leq \frac{C}{t^{\gamma}}.
        \end{equation} 
Therefore we find that as $A\rightarrow +\infty$, 
\begin{equation}
    \begin{split}\label{I11}
       & I_1^0-\ch_{ge^K}'(H(X,\xi|_{X}),h^{H(X,\xi|_{X})})\log(A)\\ 
       & \rightarrow I_1^1=\int_a^1 \gamma_{t,g,K}\frac{dt}{t}+\int_{1}^{+\infty}(\gamma_{t,g,K}-\gamma_{+\infty,g,K})\frac{dt}{t}.
    \end{split}
\end{equation}

b) $\kongge a\rightarrow 0$.  By Theorem \ref{thm asymptotic-gK}, as $a\rightarrow 0$,
\begin{equation}\label{I12}
    \begin{split}
       & I_1^1-\frac{C_{-1,g,K}'}{a}+C_{0,g,K}'\log(a)\\
       & \rightarrow I_1^2=\int_0^1 (\gamma_{t,g,K}-\frac{C_{-1,g,K}'}{t}-C_{0,g,K}')\frac{dt}{t}\\ &\qquad\qquad\qquad+\int_{1}^{+\infty}(\gamma_{t,g,K}-\gamma_{+\infty,g,K})\frac{dt}{t}-C_{-1,g,K}'.
    \end{split}
\end{equation}

c) $\kongge$ Evaluation of $I_1^2$. From (\ref{explicit tgk}), (\ref{I12}), we obtain
\begin{equation}\label{fv}
  I_1^2=-T_{g,K}(\omega,h^{\xi})+\Gamma'(1)(C_{0,g,K}'-\ch_{ge^K}'(H(X,\xi|_{X}),h^{H(X,\xi|_{X})}))  .
\end{equation}

\textbf{2.The term $I_2^0$.}
By (\ref{definition lk}), (\ref{definition nkt}), (\ref{definition cvt}), we have the following identity
\begin{equation}
    \begin{split}
    I_2^0=&-\int_a^A\Phi\trace_s\Big[\Big(N_{K,t}+i\frac{\product{\mu,K}}{t}\Big)g\exp\Big(-\lk-\mathcal{A}_{t,t}^2\Big)\Big]\frac{dt}{t}\\
    &=-\int_a^A\Phi\trace_s\Big[N_t ge^{K}\exp\Big(-\bisconn_{t}^2\Big)\Big]\frac{dt}{t}.
    \end{split}
\end{equation}

a) $\kongge A\rightarrow +\infty.\kongge$  By (\ref{asymptotic-infty-g}) in  Theorem \ref{thm asymptotic-g}, we have
\begin{equation}\label{I21}
    \begin{split}
 I_2^0+  & \ch_{ge^K}'(H(X,\xi|_{X}),h^{H(X,\xi|_{X})})\log(A)\\ 
  \rightarrow  & I_2^1=-\int_a^1 \Phi\trace_s\Big[N_t ge^{K}\exp\Big(-\bisconn_{t}^2\Big)\Big]\frac{dt}{t}\\
    & -\int_1^{+\infty}\Big(\Phi\trace_s\Big[N_t ge^{K}\exp\Big(-\bisconn_{t}^2\Big)\Big]-\ch_{ge^K}'(H(X,\xi|_{X}),h^{H(X,\xi|_{X})})\Big)\frac{dt}{t}.     
    \end{split}
\end{equation}

b) $\kongge a\rightarrow 0.\kongge$ By (\ref{asymptotic-gn}) in Theorem \ref{thm asymptotic-g}, we have
\begin{equation}\label{I22}
    \begin{split}
    &I_2^1+\frac{D_{-1,ge^K}'}{a}-D_{0,ge^{K}}'\log(a)\\
       & \rightarrow I_2^2=-\int_0^1 \Big(\Phi\trace_s\Big[N_t ge^{K}\exp\Big(-\bisconn_{t}^2\Big)\Big]-\frac{D_{-1,ge^K}'}{t}-D_{0,ge^K}'\Big)\frac{dt}{t}\\ 
        &-\int_1^{+\infty}\Big(\Phi\trace_s\Big[N_t ge^{K}\exp\Big(-\bisconn_{t}^2\Big)\Big]-\ch_{ge^K}'(H(X,\xi|_{X}),h^{H(X,\xi|_{X})})\Big)\frac{dt}{t}+D_{-1,ge^K}'.     
    \end{split}
\end{equation}
    
c) $\kongge$ Evaluation of $I_2^2$. From (\ref{analytic torsion expression-g}), (\ref{I22}), we obtain
\begin{equation}\label{fw}
    I_2^2=T_{ge^K}(\omega,h^{\xi})-\Gamma'(1)(D_{0,ge^K}'-\ch_{ge^K}'(H(X,\xi|_{X}),h^{H(X,\xi|_{X})})).
\end{equation}

\textbf{3.The term $I_3^0$.}
By (\ref{gammavt}), we have
\begin{equation}\label{qs}
    I_3^0=\int_{a}^{A}\Phi\str\Big[i\frac{\product{\mu,K}}{v}g\exp\Big(-\lk-\mathcal{A}_{v,a}^2\Big)\Big]\frac{dv}{v}.
\end{equation}

a) $\kongge A\rightarrow +\infty.\kongge $ As $A\rightarrow+\infty$, by (\ref{t01v1}),
\begin{equation}\label{I31}
    I_3^0\rightarrow I_3^1=\int_{a}^{+\infty}\Phi\str\Big[i\frac{\product{\mu,K}}{v}g\exp\Big(-\lk-\mathcal{A}_{v,a}^2\Big)\Big]\frac{dv}{v}.
\end{equation}

b) $\kongge a\rightarrow 0.\kongge$ We split $I_3^1$ into the sum of the following terms,
\begin{equation}\label{J1-J4}
    \begin{split}
        &J_1=-\int_a^1 \widetilde{e}_v\frac{dv}{v},\\ 
        &J_2=\int_{1}^{+\infty}\Phi\str\Big[i\frac{\product{\mu,K}}{v}g\exp\Big(-\lk-\mathcal{A}_{v,a}^2\Big)\Big]\frac{dv}{v},\\ 
        &J_3=\int_{1}^{1/a}\Big(\Phi\str\Big[i\frac{\product{\mu,K}}{av}g\exp\Big(-\lk-\mathcal{A}_{av,a}^2\Big)\Big]+\widetilde{e}_{av}\\ 
        &\hspace{6cm}-\Phi\dba \xi_{av,a}'-\Phi\partial \xi_{av,a}''\Big)\frac{dv}{v},\\ 
        &J_4=\Phi\dba\int_{a}^{1}\xi_{v,a}'\frac{dv}{v}+\Phi\partial\int_{a}^{1}\xi_{v,a}''\frac{dv}{v}.
    \end{split}
\end{equation}

By (\ref{expansion ev}), as $a\rightarrow 0$,
\begin{equation}\label{j1}
    \begin{split}
        J_1-\frac{\widetilde{D}_{-1,g,K}}{a}+E_{0,g,K}'\log(a)\qquad\qquad\qquad\qquad\qquad\\
        \rightarrow J_1^1=\int_{0}^{1}(-\widetilde{e}_v-\frac{\widetilde{D}_{-1,g,K}}{v}-E_{0,g,K}')\frac{dv}{v}-\widetilde{D}_{-1,g,K}.
    \end{split}
\end{equation}

By (\ref{limit of main}), (\ref{t01v1}), we see that as $a\rightarrow 0$,
\begin{equation}
    \begin{split}
        J_2\rightarrow J_2^1=\int_{1}^{+\infty}-\widetilde{e}_v\frac{dv}{v}.
    \end{split}
\end{equation}

By (\ref{expansion ev}), (\ref{expansion hvt}), as $a\rightarrow 0$,
\begin{equation}\label{limit of j3}
    \begin{split}
    \Phi\str\Big[i\frac{\product{\mu,K}}{av}g\exp\Big(-\lk-\mathcal{A}_{av,a}^2\Big)\Big]+\widetilde{e}_{av}-\Phi\dba \xi_{av,a}'-\Phi\partial \xi_{av,a}''\qquad\\ 
    \rightarrow \widetilde{D}_{0,g,K}^{v\prime}-E_{0,g,K}'=\widetilde{D}_{0,g,K}^{v}-E_{0,g,K}.
    \end{split}
\end{equation}

Then using Proposition \ref{fu} and (\ref{t01v01}), (\ref{limit of j3}) and , we obtain that as $a\rightarrow 0$,
\begin{equation}\label{j3}
    J_3\rightarrow J_3^1=\int_{X_{g,K}}\td_{ge^K}(TX,h^{TX})I_{g,K}(N_{X_{g,K}}/X, h^{N_{X_{g,K}/X}})\ch_{g,K}(\xi,h^{\xi}).
\end{equation}

Combining (\ref{j1})-(\ref{j3}), as $a\rightarrow 0$,
\begin{equation}\label{I32}
    \begin{split}
    I_3^1-\frac{\widetilde{D}_{-1,g,K}}{a}+E_{0,g,K}\log(a)-\Phi\dba\int_{a}^{1}\xi_{v,a}'\frac{dv}{v}-\Phi\partial\int_{a}^{1}\xi_{v,a}''\frac{dv}{v}\qquad\qquad\\ 
    \rightarrow I_3^2 :=\int_{0}^{1}(-\widetilde{e}_v-\frac{\widetilde{D}_{-1,g,K}}{v}-E_{0,g,K}')\frac{dv}{v}-\int_{1}^{+\infty}\widetilde{e}_v\frac{dv}{v}-\widetilde{D}_{-1,g,K}\\ 
    +\int_{X_{g,K}}\td_{ge^K}(TX,h^{TX})I_{g,K}(N_{X_{g,K}}/X, h^{N_{X_{g,K}/X}})\ch_{g,K}(\xi,h^{\xi}).
    \end{split}
\end{equation}

c) $\kongge$ Evaluation of $I_3^2$. By applying the analogue of Proposition \ref{explicit sk} to the form $\eta=\td_{g,K}(TX,h^{TX})\ch_{g,K}(\xi,h^{\xi})$ and $\widetilde{S}_{K}(X_g,\omega^{X_g})$, we have 
\begin{equation}\label{fx}
    \begin{split}
    I_3^2=-\int_{X_g}\td_{g,K}(TX,h^{TX})\ch_{g,K}(\xi,h^{\xi})\widetilde{S}_K(X_g,\omega^{X_g})+\Gamma'(1)E_{0,g,K}\\ 
    +\int_{X_{g,K}}\td_{ge^K}(TX,h^{TX})I_{g,K}(N_{X_{g,K}}/X, h^{N_{X_{g,K}/X}})\ch_{g,K}(\xi,h^{\xi}).
    \end{split}
\end{equation}

\textbf{4.The divergences of the left-hand side of (\ref{sum of Ik0})}\quad

We now summarize the  the diverging term of the left-hand side of $\sum_{k=1}^{3}I_{k}^{0}$ as $A\rightarrow 0$, and then $a\rightarrow 0$ .

\textbf{a)  $\kongge A\rightarrow +\infty.\kongge $} By (\ref{I11}), (\ref{I21}), (\ref{I31}), the diverging term of  $\sum_{k=1}^3 I_k^0$ is
\begin{equation}\label{left A}
    (\ch_{ge^K}'(H(X,\xi|_{X}),h^{H(X,\xi|_{X})})-\ch_{ge^K}'(H(X,\xi|_{X}),h^{H(X,\xi|_{X})}))\log(A),
\end{equation}
which clearly vanishes.

\textbf{b) $\kongge a\rightarrow 0.\kongge $} By (\ref{I12}), (\ref{I22}), (\ref{I32}), the form $\sum_{k=1}^3 I_k^1$ consists of three sorts of diverging terms.

The first sort of terms diverges in the speed of $\frac{1}{a}$,
\begin{equation}\label{diverge 1/a}
    (C_{-1,g,K}'-D_{-1,ge^K}'+\widetilde{D}_{-1,g,K})\frac{1}{a}.
\end{equation}

From \eqref{a9}, the form $\frac{\omega}{2\pi}-i\product{\mu,K}\in P^W$ is  $d_K$-closed. Therefore by using (\ref{D formula}), (\ref{C formula}), (\ref{ddbarsk}), (\ref{ev formula}), we obtain that (\ref{diverge 1/a}) lies in $P^{S,0}$.

Similarly, the second sort of terms are
\begin{equation}\label{diverge loga}
    (-C_{0,g,K}'+D_{0,ge^K}'-E_{0,g,K}')\log(a).
\end{equation}
By using \eqref{D formula}, \eqref{C formula}, (\ref{ddbarsk}), \eqref{ev formula},  the terms in (\ref{diverge loga}) lie in $P^{S,0}$.

The last sort of term is 
\begin{equation}\label{j4}
    J_4=\Phi\dba\int_{a}^{1}\xi_{v,a}'\frac{dv}{v}+\Phi\partial\int_{a}^{1}\xi_{v,a}''\frac{dv}{v}.
\end{equation}
And it clearly lies in $P^{S,0}$.

Therefore, if $S$ is compact and $\kahler$, $P^{S,0}$ is closed under uniform convergence. From (\ref{0 sum of Ik0}) and the above analysis, we conclude that $\sum_{k=1}^3 I_{k}^2\in P^{S,0}$. In this case, we have completed the proof of  Theorem \ref{8o}, by using \eqref{fv}, \eqref{fw}, \eqref{fx}.

But in general, when $P^{S,0}$ is not closed, one has to study in detail these diverging terms in $P^{S,0}$. This explains why we need to study the right-hand side of (\ref{0 sum of Ik0}) in the remaining part of this section.

\subsection{Some elementary identities}\label{bs}
 Recall that $da$ is the canonical generator of $\mathbb{C}^{*}$,  then $da\dabar$ span $\rr^{2*}\otimes_{\rr}\cc$. 

 For $0<t\leq v$, set
 \begin{equation}\label{pivt}
     \begin{split}
        &\pi_{v,t}^{0\prime}=\str\Big[\frac{i\product{\mu,K}}{v}g\exp\Big(-\lk-\cvt^2-2\dabar t\derivative{t}\cvt'\Big)\Big]^{\dabar},\\ 
         &\pi_{v,t}^{0\prime\prime}=\str\Big[\frac{i\product{\mu,K}}{v}g\exp\Big(-\lk-\cvt^2-2da t\derivative{t}\cvt''\Big)\Big]^{da},\\ 
         &\chi_{v,t}^{0\prime}=\derivative{b}\str\Big[\frac{{\sqrt{t}c(K^{X(0,1)})}}{2\sqrt{2}v} g\exp\Big(-\lk-\cvt^2-bN_{K,t}\Big)\Big]_{b=0},\\ 
         &\chi_{v,t}^{0\prime\prime}=\derivative{b}\str\Big[\frac{{\sqrt{t}c(K^{X(1,0)})}}{2\sqrt{2}v} g\exp\Big(-\lk-\cvt^2-bN_{K,t}\Big)\Big]_{b=0},\\ 
         &\kappa_{v,t}^{0}=\derivative{b}\str\Big[\frac{i\product{\mu,K}}{v}g\exp\Big(-\lk-\cvt^2-bN_{K,t}\Big)\Big]_{b=0}.
     \end{split}
 \end{equation}
 By replacing $\product{\mu,K}$ in (\ref{pivt}) by $\mu_1,\mu_2$ (cf. \eqref{fy}), and $K^{X}$ by $K_1^X, K_2^X$, we can define $\pi_{v,t}^{1,\prime/\prime\prime}$, $\pi_{v,t}^{2,\prime/\prime\prime}$, $\chi_{v,t}^{1,\prime/\prime\prime}$ , $\chi_{v,t}^{2,\prime/\prime\prime}$, $\kappa_{v,t}^1$ and $\kappa_{v,t}^2$ respectively.

 Clearly, if $h_{v,t}$ is one of the functions $\pi_{v,t}^{\prime/\prime\prime}, \chi_{v,t}^{\prime/\prime\prime}, \kappa$, then
 \begin{equation}
    h_{v,t}^0=h_{v,t}^1+h_{v,t}^2.
 \end{equation}
 
 \begin{proposition}
     For $k=0,1,2$, the following identities hold,
     \begin{equation}\label{pichikappa}
         \begin{split}
         \pi_{v,t}^{k\prime}+2\chi_{v,t}^{k\prime}=\partial \kappa_{v,t}^k,\\ 
         \pi_{v,t}^{k\prime\prime }+2\chi_{v,t}^{k\prime\prime }=-\dba \kappa_{v,t}^k.    
         \end{split}
     \end{equation}
 \end{proposition}
 \begin{proof}
     From (\ref{bku-algebraic identity}), we have 
     \begin{equation}\label{cvt-algebraic identity2}
         [\cvt'',N_{K,t}]=-2\derivative{t}\cvt'',\qquad [\cvt',N_{K,t}]=2t\derivative{t}\cvt'.
     \end{equation}
 
     Using (\ref{cvt-algebraic identity}), (\ref{commutator c mu}) and (\ref{cvt-algebraic identity2}), we have
     \begin{equation}\label{proof pichikappa}
         \begin{split}
             \partial \kappa_{v,t}^0 =&\derivative{b}\str\Big[\Big[\cvt',\frac{i\product{\mu,K}}{v}g\exp\Big(-\lk-\cvt^2-bN_{K,t}\Big)\Big]\Big]_{b=0}\\ 
             =&\derivative{b}\str\Big[\Big[\cvt',\frac{i\product{\mu,K}}{v}\Big]g\exp\Big(-\lk-\cvt^2-bN_{K,t}\Big)\Big]_{b=0}\\ 
             &+\derivative{b}\str\Big[\frac{i\product{\mu,K}}{v}g\exp\Big(-\lk-\cvt^2-b[\cvt',N_{K,t}]\Big)\Big]_{b=0}\\ 
            = &\derivative{b}\str\Big[\frac{\sqrt{t}c(K^{X(0,1)})}{\sqrt{2}}g\exp\Big(-\lk-\cvt^2-bN_{K,t}\Big)\Big]_{b=0}\\ 
             &+\derivative{b}\str\Big[\frac{i\product{\mu,K}}{v}g\exp\Big(-\lk-\cvt^2-2bt\derivative{t}\mathcal{A}_{v,t}'\Big)\Big]_{b=0}\\ 
             =&\pi_{v,t}^{0\prime}+2\chi_{v,t}^{0\prime}.
         \end{split}
     \end{equation}
     So we obtain the first identity in (\ref{pichikappa}) for the superscript $0$, for the other five cases in (\ref{pichikappa}), note that there is such  an  analogue of (\ref{commutator c mu}): for $k=1,2$,
     \begin{align}\label{commutator c muk}
        &[\cvt'',i\mu_k]=-\frac{\sqrt{t}c(K_k^{X(1,0)})}{\sqrt{2}}, &&[\cvt',i\mu_k]=\frac{\sqrt{t}c(K_k^{X(0,1)})}{\sqrt{2}}.
     \end{align}
     We can then proceed as in (\ref{proof pichikappa}) to complete the proof.
 \end{proof}
 
 As a direct consequence of (\ref{pichikappa}), for $k=0,1,2$, we have
 \begin{equation}\label{pichikappa2}
     \dba(\pi_{v,t}^{k\prime}+2\chi_{v,t}^{k\prime})=\partial(\pi_{v,t}^{k\prime\prime}+2\chi_{v,t}^{k\prime\prime})=\dba\partial\kappa_{v,t}^k.
 \end{equation}

 By (\ref{alphabeta}) and (\ref{pichikappa2}), we can rewrite the right-hand side of (\ref{0 sum of Ik0}) to obtain
 \begin{equation}\label{sum of Ik0}
    \begin{split}
    \sum_{k=1}^3I_k^0=\Phi\dba\int_{\Delta}(-\pi_{v,t}^{1\prime}-\frac{1}{2}\pi_{v,t}^{2\prime}+\chi_{v,t}^{2\prime})\frac{dv}{v}\frac{dt}{t}\qquad\qquad\qquad\\+\Phi\partial\int_{\Delta} (-\pi_{v,t}^{1\prime\prime}-\frac{1}{2}\pi_{v,t}^{2\prime\prime}+\chi_{v,t}^{2\prime\prime})\frac{dv}{v}\frac{dt}{t} +\Phi\dba\partial \int_{\Delta}\kappa_{v,t}^1 \frac{dv}{v}\frac{dt}{t}.
    \end{split}
 \end{equation}

\subsection{Intermediate results on the right-hand side of (\ref{0 sum of Ik0})}\label{sec intermediate right}
We will state some intermediate results concerning the right-hand side of (\ref{sum of Ik0}). Our situation is similar to that in \cite[\S 6.6-6.7]{B1997}, except for the terms containing $\mu_2$. 

Let $(v,t)$ be a coordinate system for $\rr^2$, for $\alpha\in\Lambda(T_{\rr}^*(\rr^2\times S))$, we write $\alpha$ in the form,
 \begin{equation}
    \alpha=\alpha^0+da \alpha^{dv}+ \dabar \alpha^{dt}+da\dabar \alpha^{dvdt}, \kongge \alpha^0,\alpha^{dv},\alpha^{dt},\alpha^{dvdt}\in\Lambda(T_{\rr}^*S).
 \end{equation}
And we set $\Phi_{\rr^2\times S}:\Lambda^{*}(T_{\rr}^*(\rr^2\times S))\rightarrow \Lambda^{*}(T_{\rr}^*(\rr^2\times S))$ to be the map $\alpha\rightarrow (2i\pi)^{-[\mathrm{deg}\alpha/2]}\alpha$.  

For $0<t\leq v$, put
\begin{equation}\label{very many hvt}
    \begin{split}
       &\theta_{v,t}^0=\derivative{b}\str\Big[g\exp\Big(-\lk-\cvt^2-2dat\derivative{t}\cvt'' -2\dabar t\derivative{t}\cvt'\\ 
        &\qquad\qquad\qquad +da\dabar N_V-b\frac{i\product{\mu,K}}{v}\Big)\Big]_{b=0}^{da\dabar},\\
       &\eta_{v,t}=\frac{1}{2}\frac{\partial^2}{\partial c^2}\str\Big[g\exp\Big(-\lk-\cvt^2+cN_{K,t}\Big)\Big]_{c=0},\\
       &\eta_{t}^{*}=\frac{1}{2}\frac{\partial^2}{\partial c^2}\str\Big[g\exp\Big(-\lk-\bisconn_{K,t}^2+cN_{K,t}\Big)\Big]_{c=0},\\
       &\tau_{v,t}=\frac{1}{2}\frac{\partial^2}{\partial c^2}\str\Big[g\exp\Big(-\lk-\cvt^2+c\frac{i\mu_1}{v}\Big)\Big]_{c=0},\\
      &\lambda_{v,t}'=\derivative{b}\trace_s\Big[g\exp\Big(-\lk-\cvt^2+2\dabar t\derivative{t}\cvt'-b\frac{\sqrt{t}c(K^{X(1,0)})}{2\sqrt{2}v}\Big)\Big]_{b=0}^{\dabar},\\
      &\lambda_{v,t}''=\derivative{b}\trace_s\Big[g\exp\Big(-\lk-\cvt^2+2da t\derivative{t}\cvt''-b\frac{\sqrt{t}c(K^{X(0,1)})}{2\sqrt{2}v}\Big)\Big]_{b=0}^{da},\\ 
      &\rho_{v,t}'=\derivative{b}\trace_s\Big[N_{K,t}g\exp\Big(-\lk-\cvt^2+2\dabar t\derivative{t}\cvt'-b\frac{\sqrt{t}c(K_2^{X(1,0)})}{2\sqrt{2}v}\Big)\Big]_{b=0}^{\dabar},\\
      &\rho_{v,t}''=\derivative{b}\trace_s\Big[N_{K,t}g\exp\Big(-\lk-\cvt^2+2da t\derivative{t}\cvt''-b\frac{\sqrt{t}c(K_2^{X(0,1)})}{2\sqrt{2}v}\Big)\Big]_{b=0}^{da},\\ 
      &\sigma_{v,t}'=\frac{\partial^3}{\partial^2 c\partial b}\str\Big[g\exp\Big(-\lk-\cvt^2+cN_{K,t}-b\frac{\sqrt{t}c(K_2^{X(0,1)})}{4\sqrt{2}v}\Big)\Big]_{c=0,b=0},\\ 
      &\sigma_{v,t}''=\frac{\partial^3}{\partial^2 c\partial b}\str\Big[g\exp\Big(-\lk-\cvt^2+cN_{K,t}-b\frac{\sqrt{t}c(K_2^{X(1,0)})}{4\sqrt{2}v}\Big)\Big]_{c=0,b=0},\\ 
      &\psi_{v,t}=\Big\{  \Phi_{\rr^2\times S}\str\Big[\Bigl(\frac{\sqrt{t}c(K_2^{X(0,1)})}{\sqrt{2}}-\frac{\sqrt{t}c(K_2^{X(1,0)})}{\sqrt{2}}\Bigr)g\exp\Big(-\lk-\cvt^2\\ &\qquad\qquad\qquad\qquad\qquad-dt\derivative{t}\cvt-dv\derivative{v}\cvt\Big)\Big]\Big\}^{dvdt}.
    \end{split}
\end{equation}
By replacing $\product{\mu,K}$ by $\mu_1,\mu_2$ in $\theta_{v,t}^{0}$, we can define $\theta_{v,t}^1, \theta_{v,t}^2$ respectively.

For a function $h_{v,t}$ on $(v,t)$, we denote by $h_{0,v}$ the limit of $h_{v,t}$ as $t\to 0$, if the limit exists.

\begin{proposition}\label{i2}
    The following identities hold,
        \begin{equation}\label{hv0}
        \begin{split}
        &\Phi\theta_{v,0}^{1}=\int_{X_g}(\mathrm{dim}X\,\td_{g,K}(TX,h^{TX})-\td_{g,K}'(TX,h^{TX}))\ch_{g,K}(\xi,h^{\xi})\widetilde{d}_v^1,\\ 
        &\Phi\theta_{v,0}^{2}=\int_{X_g}(\mathrm{dim}X\,\td_{g,K}(TX,h^{TX})-\td_{g,K}'(TX,h^{TX}))\ch_{g,K}(\xi,h^{\xi})\widetilde{d}_v^2,\\
        &\pi_{v,0}^{1'}=\pi_{v,0}^{2'}=\pi_{v,0}^{1''}=\pi_{v,0}^{2''}=0,\\
        &\lambda_{v,0}'=\lambda_{v,0}''=0.         
        \end{split}
    \end{equation}
\end{proposition}
\begin{proof}
    Proposition \ref{i2} will be proved in Section \ref{ek}.
\end{proof}

Now we use the notations of Section \ref{em} with respect to the vector filed $TX|_{W_{g,K}}$ over $W_{g,K}$. We put
\begin{equation}\label{def hv}
    \begin{split}
        &\theta_v^1=\Big(\frac{1}{2\pi i}\Big)^{\mathrm{dim}X_{g,K}}\int_{X_{g,K}}\Theta_v\,\ch_{g,K}(\xi,h^{\xi}),\\
        &\lambda_v=\Big(\frac{1}{2\pi i}\Big)^{\mathrm{dim}X_{g,K}}\int_{X_{g,K}}\Lambda_v\,\ch_{g,K}(\xi,h^{\xi}),\\
        &\lambda_v'=\frac{1}{2}\Big(\frac{1}{2\pi i}\Big)^{\mathrm{dim}X_{g,K}}\int_{X_{g,K}}\Lambda_v\,\ch_{g,K}(\xi,h^{\xi}),\\
        &\lambda_v'' =-\frac{1}{2}\Big(\frac{1}{2\pi i}\Big)^{\mathrm{dim}X_{g,K}}\int_{X_{g,K}}\Lambda_v\,\ch_{g,K}(\xi,h^{\xi}),\\
        &\pi_v^{1'}=\Big(\frac{1}{2\pi i}\Big)^{\mathrm{dim}X_{g,K}}\int_{X_{g,K}}\Pi_v'\,\ch_{g,K}(\xi,h^{\xi}),\\
        &\pi_v^{1''}=\Big(\frac{1}{2\pi i}\Big)^{\mathrm{dim}X_{g,K}}\int_{X_{g,K}}\Pi_v''\,\ch_{g,K}(\xi,h^{\xi}),\\ 
        &\tau_v=\Big(\frac{1}{2\pi i}\Big)^{\mathrm{dim}X_{g,K}}\int_{X_{g,K}}\mathrm{T}_v\,\ch_{g,K}(\xi,h^{\xi}).
    \end{split}
\end{equation}

Now we will start to state the intermediate results concerning about the right-hand side of (\ref{sum of Ik0}). In the following theorems, we fix $K_0\in \mathfrak{z}(g)$ and $K=zK_0$ with $|z|\neq 0$ is small enough.

\begin{theorem}\label{h,t1infty}
   For $k=0,1,2$, let $h_{v,t}$ be one of the functions $\pi_{v,t}^{k,\prime/\prime\prime}, \chi_{v,t}^{k,\prime/\prime\prime}, \kappa_{v,t}^k$. For any $t_0>0$, there exist $C>0$,  $\delta\in \,]0,1]$ such that for any  $t_0\leq t\leq v,$
    \begin{equation}
        |h_{v,t}|\leq  \frac{C}{v^{\delta}}.
    \end{equation}
\end{theorem}

\begin{theorem}\label{h,t01}
    For $k=0,1,2$, let $h_{v,t}$ be one of the functions $\theta_{v,t}^{k}, \pi_{v,t}^{k,\prime/\prime\prime}$, $\lambda_{v,t}^{\prime / \prime \prime}$. For any $v_0>0$, there exist $C>0$, $\gamma \in \,]0,1]$, such that for any $t\in \,]0,1]$, $v\geq v_0, v\geq t$,
    \begin{equation}\label{h,t01v1}
        |h_{v,t}-h_{v,0}|\leq C\frac{t^{\gamma}}{v},\qquad \Big|\derivative{t}(th_{v,t})-h_{v,0}\Big|\leq C\frac{t^{\gamma}}{v}.
    \end{equation}
   And if $h_{v,t}$ is one of the functions $\theta_{v,t}^{1}, \pi_{v,t}^{1,\prime/\prime\prime}, v\theta_{v,t}^2, v\pi_{v,t}^{2,\prime/\prime\prime}$, $\lambda_{v,t}^{\prime / \prime \prime}$, there exist $C>0,\gamma' \in \,]0,1]$, such that for any $0<t\leq v\leq 1$,
    \begin{equation}\label{h,t01v01}
        |h_{v,t}-h_{v,0}|\leq C\Big(\frac{t}{v}\Big)^{\gamma'},\qquad \Big|\derivative{t}(th_{v,t})-h_{v,0}\Big|\leq C\Big(\frac{t}{v}\Big)^{\gamma'}.
    \end{equation}
\end{theorem}

\begin{theorem}\label{thm h,vt}
    If $h_{v,t}$ is one of the functions $\theta_{v,t}^{1}, \lambda_{v,t}^{\prime/\prime\prime}, \pi_{v,t}^{1,\prime/\prime\prime},\tau_{v,t}$. There exist $C>0$ and $\alpha,\beta>0$ such that for any $0<t\leq 1 \leq v\leq 1/t$,
    \begin{equation}\label{h,vt}
        |h_{vt,t}-h_v|\leq Ct^{\alpha}v^{\beta}. 
    \end{equation}
\end{theorem}

\begin{theorem}\label{thm,h2,t01v01}
    If $h_{v,t}$ is one of the functions $\chi_{v,t}^{2,\prime/\prime\prime},\derivative{t}(t\chi_{v,t}^{2,\prime/\prime\prime}),\rho_{v,t}^{\prime/\prime\prime}, \sigma_{v,t}^{\prime/\prime\prime}$. Then for any $m\in \nn$, there is $C_m>0$, such that for any $0<t\leq v\leq 1$,
    \begin{equation}\label{h2,t01v01}
        |h_{v,t}|\leq C_m\frac{v^{m}}{t^2}.
    \end{equation}
And there exists $\delta\in \,]0,1[\,$ such that for any $m\in \nn^*$, there exists $C>0$, for any $0<t\leq v\leq 1$,
\begin{equation}\label{psi,t01v01}
    |\psi_{v,t}|\leq C\frac{v^m}{t^{\delta}}.
\end{equation}
\end{theorem}

Theorem \ref{h,t1infty} will be proved in Section \ref{e0}, Theorems \ref{h,t01} and \ref{thm,h2,t01v01} in Section \ref{ek}, Theorem \ref{thm h,vt} in Section \ref{el}.

\subsection{The asymptotics of the right-hand side of (\ref{0 sum of Ik0})}\label{b6}\quad
In this section,
we fix $g\in G$, $K_0\in \mathfrak{z}(g)$ and $K=zK_0$, such that the intermediate results in the last section are valid.
 For a function $h_{v,t}$ of $(v,t)$, we use the notation $\Delta_{a,A}(h)$ to denote the following double integral,
\begin{equation}
    \Delta_{a,A}(h)=\int_{\Delta_{a,A}}h_{v,t}\frac{dv}{v}\frac{dt}{t}=\int_{a}^{A}\frac{dt}{t}\int_{t}^{A}h_{v,t}\frac{dv}{v}.
\end{equation}

a) $\kongge$ $A\rightarrow +\infty \kongge$ If $h_{v,t}=\pi_{v,t}^{k,\prime/\prime\prime}, \chi_{v,t}^{k,\prime/\prime\prime} $ or $ \kappa_{v,t}^{k}$, $k=0,1,2$, by Theorem \ref{h,t1infty}, we have 
 \begin{equation}
    |h_{v,t}|\leq \frac{C}{(\sqrt{vt})^{\delta}}, \qquad\kongge \mathrm{for}\kongge \mathrm{any} \kongge a\leq t\leq v.
 \end{equation} 
Therefore as $ A\rightarrow +\infty$, we have 
\begin{equation}
   \Delta_{a,A}(h)=\int_{a}^{A}\frac{dt}{t}\int_{t}^{A}h_{v,t}\frac{dv}{v}\\ 
    \rightarrow \Delta_{a,\infty}(h)=\int_{a}^{+\infty}\frac{dt}{t}\int_{t}^{+\infty}h_{v,t}\frac{dv}{v}.
\end{equation}

So on the right-hand side of (\ref{sum of Ik0}), we have
\begin{equation}\label{right A}
   \begin{split}
   & \int_{\Delta_{a,A}}\Big(-\pi_{v,t}^{1\prime}-\frac{1}{2}\pi_{v,t}^{2\prime}+\chi_{v,t}^{2\prime}\Big)\frac{dv}{v}\frac{dt}{t}\rightarrow -\Delta_{a,\infty}(\pi^{1\prime})-\frac{1}{2}\Delta_{a,\infty}(\pi^{2\prime})+\Delta_{a,\infty}(\chi^{2\prime}),\\ 
    &\int_{\Delta_{a,A}}\Big(-\pi_{v,t}^{1\prime\prime}-\frac{1}{2}\pi_{v,t}^{2\prime\prime}+\chi_{v,t}^{2\prime\prime}\Big)\frac{dv}{v}\frac{dt}{t}\rightarrow -\Delta_{a,\infty}(\pi^{1\prime\prime})-\frac{1}{2}\Delta_{a,\infty}(\pi^{2\prime\prime})+\Delta_{a,\infty}(\chi^{2\prime\prime}),\\ 
    &\int_{\Delta_{a,A}}\kappa_{v,t}^1\frac{dv}{v}\frac{dt}{t}\rightarrow \Delta_{a,\infty}(\kappa^1).
   \end{split}
\end{equation}

b) $\kongge a\rightarrow 0\kongge$ Now we study terms by terms the right hand side of (\ref{right A}).

\textbf{\underline{1. The terms $\Delta_{a,\infty}(\pi^{1\prime})$, $\Delta_{a,\infty}(\pi^{1\prime\prime})$.}}
Clearly we have
\begin{equation}\label{Delta pi1'}
    \Delta_{a,\infty}(\pi^{1\prime})=\Delta_{1,\infty}(\pi^{1\prime})+\int_{a}^1\frac{dt}{t}\int_{1}^{+\infty} \pi_{v,t}^{1\prime}\frac{dv}{v}+\Delta_{a,1}(\pi^{1\prime}).
\end{equation}
The first term on the right-hand side of (\ref{Delta pi1'}) does not depend on $a$. As for the second term, by using (\ref{hv0}), (\ref{h,t01v1}), we have as $a\rightarrow 0$,
\begin{equation}
    \int_{a}^1\frac{dt}{t}\int_{1}^{+\infty} \pi_{v,t}^{1\prime}\frac{dv}{v}\rightarrow \int_{0}^1\frac{dt}{t}\int_{1}^{+\infty} \pi_{v,t}^{1\prime}\frac{dv}{v}.
\end{equation}
We can rewrite the third term in (\ref{Delta pi1'}) as the form,
\begin{equation}\label{deltaa1}
    \Delta_{a,1}(\pi^{1\prime})=\int_{a}^{1}\frac{dt}{t}\int_{1}^{1/t}\pi_{vt,t}^{1\prime}\frac{dv}{v}.
\end{equation}

By using (\ref{piu=0}), (\ref{h,t01v01}), (\ref{h,vt}), we see that for $t\in \,]0,1], v\in [1,1/t]$,
\begin{align}\label{pi1 two estimates}
        &| \pi_{vt,t}^{1\prime}|\leq \frac{C}{v^{\gamma}}, 
        &&| \pi_{vt,t}^{1\prime}|\leq C't^{\alpha}v^{\beta}.
\end{align}
Combining the two estimates in (\ref{pi1 two estimates}), we obtain that there exists $\alpha',\gamma'>0$ such that for $t\in \,]0,1], v\in [1,1/t]$,
\begin{equation}\label{pivtt1'leq}
    | \pi_{vt,t}^{1\prime}|\leq C\frac{t^{\alpha'}}{v^{\gamma'}}.
\end{equation}
By using (\ref{deltaa1}), (\ref{pivtt1'leq}), we thus have
\begin{equation}\label{Deltaa1pi1'rightarrow}
    \Delta_{a,1}(\pi^{1\prime})\rightarrow \int_{0}^{1}\frac{dt}{t}\int_{1}^{1/t}\pi_{vt,t}^{1\prime}\frac{dv}{v}=\Delta_{0,1}(\pi^{1\prime}).
\end{equation}

We can deal with $\pi_{v,t}^{1\prime\prime}$ in the same way. As a consequence, by combining (\ref{Delta pi1'})-(\ref{Deltaa1pi1'rightarrow}), we obtain that as $a\rightarrow 0$,
\begin{align}\label{1.term}
        &\Delta_{a,\infty}(\pi^{1\prime})\rightarrow \Delta_{0,\infty}(\pi^{1\prime}), 
        &&\Delta_{a,\infty}(\pi^{1\prime\prime})\rightarrow \Delta_{0,\infty}(\pi^{1\prime\prime}).
\end{align}

\textbf{\underline{2. The terms $\Delta_{a,\infty}(\kappa^1)$, $\Delta_{a,\infty}(\chi^{2\prime})$, $\Delta_{a,\infty}(\chi^{2\prime\prime})$.}}

We will deal with these terms by using the method in \cite[\S 6.6]{B1997}.
If $h_{v,t}= \kappa_{v,t}^1, \chi_{v,t}^{2\prime}$ or $\chi_{v,t}^{2\prime \prime}$, as $t\to 0$, $h_{v,t}$ has a singularity expansion of the type,
\begin{equation}
    h_{v,t}^1=\frac{A_v}{t}+B_v+\mathcal{O}(t).
\end{equation}

To deal with the singularity part, we note that
\begin{equation}\label{kappa1-integration by parts}
    \begin{split}
    \int_{a}^1\frac{dt}{t}\int_{t}^{+\infty}h_{v,t}\frac{dv}{v}= \int_{a}^1\frac{dt}{t}\int_{t}^{+\infty}\derivative{t}(t h_{v,t})\frac{dv}{v}-\int_{a}^{1}h_{t,t}\frac{dt}{t}\\ 
    +\int_{a}^{+\infty}h_{v,a}\frac{dv}{v}-\int_{1}^{+\infty}h_{v,1}\frac{dv}{v}.
    \end{split}
\end{equation}

The following result is an analogue of  \cite[Theorem 6.21]{B1997}, which deals with the first term on the right hand side of \eqref{kappa1-integration by parts}.
\begin{theorem}\label{thm-ddbarpartialt kappa1}
    The following identities hold,
    \begin{equation}\label{ddbarpartialt kappa1}
        \begin{split}
           & \dba\partial \derivative{t}(t\kappa_{v,t}^1)= \dba\partial \theta_{v,t}^1- v\derivative{v}(\dba\partial \eta_{v,t})-\dba\partial(\rho_{v,t}'-\rho_{v,t}''),\\
         &  2\derivative{t}(t\chi_{v,t}^{2\prime})-\partial(\rho_{v,t}'-\rho_{v,t}'')+\dba\partial\sigma_{v,t}'=\partial\theta_{v,t}^2-\derivative{t}(t\pi_{v,t}^{2\prime}),\\ 
         &2\derivative{t}(t\chi_{v,t}^{2\prime\prime})+\dba(\rho_{v,t}'-\rho_{v,t}'')+\dba\partial\sigma_{v,t}''=-\dba\theta_{v,t}^2-\derivative{t}(t \pi_{v,t}^{2\prime\prime}).
        \end{split}
    \end{equation}
\end{theorem}
\begin{proof}
   We use the notation "$\equiv$" to indicate that the equality holds modulo elements of the form $\overline{\partial}\alpha+\partial\beta$, where $\alpha$ and $\beta$ are smooth forms on $S$.
  We have
    \begin{equation}\label{partial t tkappa1 1}
       \begin{split}
        \derivative{t}\str\Big[tN_{K,t}g\exp\Big(-\lk-\cvt^2-b\frac{i\mu_1}{v}\Big)\Big]=\str\Big[N_V g\exp\Big(-\lk-\cvt^2-b\frac{i\mu_1}{v}\Big)\Big]\\ 
        +\derivative{c}\str\Big[tN_{K,t}g\exp\Big(-\lk-\cvt^2-c[\cvt'',\derivative{t}\cvt']
        -c[\cvt',\derivative{t}\cvt'']-b\frac{i\mu_1}{v}\Big)\Big]_{c=0}\\ 
        =\str\Big[N_V g\exp\Big(-\lk-\cvt^2-b\frac{i\mu_1}{v}\Big)\Big]\qquad\qquad\qquad\qquad\qquad\qquad\qquad\\
        +\dba\str\Big[tN_{K,t}g\exp\Big(-\lk-\cvt^2-\dabar \derivative{t}\cvt'-b\frac{i\mu_1}{v}\Big)\Big]^{\dabar}\qquad\qquad\qquad\\ 
        +\partial\str\Big[tN_{K,t}g\exp\Big(-\lk-\cvt^2-da \derivative{t}\cvt''-b\frac{i\mu_1}{v}\Big)\Big]^{da}\qquad\qquad\qquad\\ 
        +\str\Big[[\cvt'',tN_{K,t}]g\exp\Big(-\lk-\cvt^2-\dabar \derivative{t}\cvt'-b\frac{i\mu_1}{v}\Big)\Big]^{\dabar}\qquad\qquad\\ 
        +\str\Big[[\cvt',tN_{K,t}]g\exp\Big(-\lk-\cvt^2-da \derivative{t}\cvt''-b\frac{i\mu_1}{v}\Big)\Big]^{da}\qquad\qquad\\
        +b\derivative{c}\str\Big[tN_{K,t}g\exp\Big(-\lk-\cvt^2-\dabar \derivative{t}\cvt'-b\frac{i\mu_1}{v}-c[\cvt'',\frac{i\mu_1}{v}]\Big)\Big]_{c=0}^{\dabar}\\ 
        +b\derivative{c}\str\Big[tN_{K,t}g\exp\Big(-\lk-\cvt^2-da \derivative{t}\cvt''-b\frac{i\mu_1}{v}-c[\cvt',\frac{i\mu_1}{v}]\Big)\Big]_{c=0}^{da}.
       \end{split}
    \end{equation}
    By differentiating (\ref{partial t tkappa1 1}) at $b=0$, and using (\ref{bku-algebraic identity}), (\ref{commutator c muk}), \eqref{hvt}, we get 
    \begin{equation}\label{partial t tkappa1}
        \begin{split}
        \derivative{t}(t&\kappa_{v,t}^1)\equiv\theta_{v,t}^1\\
            +&\derivative{b}\str\Big[N_{K,t}g\exp\Big(-\lk-\cvt^2-\dabar t\derivative{t}\cvt'+b\frac{\sqrt{t}c(K_1^{X(1,0)})}{\sqrt{2}v}\Big)\Big]_{b=0}^{\dabar}\\ 
            +&\derivative{b}\str\Big[N_{K,t}g\exp\Big(-\lk-\cvt^2-da t\derivative{t}\cvt''-b\frac{\sqrt{t}c(K_1^{X(0,1)})}{\sqrt{2}v}\Big)\Big]_{b=0}^{da}.
        \end{split}
    \end{equation}

   Now we start to deal with $\eta_{v,t}$. By \eqref{hvt},
    \begin{equation}
        \begin{split}
           2v\derivative{v}\eta_{v,t}=\frac{\partial^3}{\partial b\partial c^2}\str\Big[g\exp\Big(-\lk-\cvt^2-bv[\cvt'',\derivative{v}\cvt']
          +cN_{K,t}\Big)\Big]_{b=0,c=0}\\ 
          +\frac{\partial^3}{\partial b\partial c^2}\str\Big[g\exp\Big(-\lk-\cvt^2-bv[\cvt',\derivative{v}\cvt'']
          +cN_{K,t}\Big)\Big]_{b=0,c=0}.
        \end{split}
    \end{equation}
By (\ref{bku-algebraic identity}) and proceeding as in (\ref{partial t tkappa1 1}), (\ref{partial t tkappa1}), we have
\begin{equation}
    \begin{split}
   & \frac{\partial^3}{\partial b\partial c^2}\str\Big[g\exp\Big(-\lk-\cvt^2-bv[\cvt'',\derivative{v}\cvt']
          +cN_{K,t}\Big)\Big]_{b=0,c=0}\\
        \equiv&2\frac{\partial}{\partial b}\str\Big[N_{K,t}g\exp\Big(-\lk-\cvt^2+da t\derivative{t}\cvt''-b\frac{\sqrt{t}c(K^{X(0,1)})}{\sqrt{2}v}\Big)\Big]_{b=0}^{da}.
    \end{split}
\end{equation}
Similarly, we have
\begin{equation}\label{proof of partialt kappa1 end}
    \begin{split}
   & \frac{\partial^3}{\partial b\partial c^2}\str\Big[g\exp\Big(-\lk-\cvt^2-bv[\cvt',\derivative{v}\cvt'']
          +cN_{K,t}\Big)\Big]_{b=0,c=0}\\
        \equiv&2\frac{\partial}{\partial b}\str\Big[N_{K,t}g\exp\Big(-\lk-\cvt^2+\dabar t\derivative{t}\cvt'+b\frac{\sqrt{t}c(K^{X(1,0)})}{\sqrt{2}v}\Big)\Big]_{b=0}^{\dabar}.
    \end{split}
\end{equation}

Combining (\ref{partial t tkappa1})-(\ref{proof of partialt kappa1 end}), we obtain the first equation in (\ref{ddbarpartialt kappa1}).

By (\ref{bku-algebraic identity}), (\ref{commutator c muk}), we have 
\begin{equation}\label{proof of partialt tkappa transgression1}
    \begin{split}
   & \str\Big[tN_{K,t}g\exp\Big(-\lk-\cvt^2-\dabar \derivative{t}\cvt'-b\frac{i\mu_2}{v}\Big)\Big]^{\dabar}\\
        =&\frac{1}{2}\str\Big[N_{K,t}g\exp\Big(-\lk-\cvt^2-\dabar [\cvt', N_{K,t}]-b\frac{i\mu_2}{v}\Big)\Big]^{\dabar}\\ 
        \equiv &-\frac{1}{4}b\frac{\partial^3}{\partial^2 c\partial b'}\str\Big[g\exp\Big(-\lk-\cvt^2+cN_{K,t}-b\frac{i\mu_2}{v}+b'\big[\cvt',\frac{i\mu_2}{v}\big]\Big)\Big]_{c=0,b'=0}.
    \end{split}
\end{equation}

By differentiating (\ref{proof of partialt tkappa transgression1}) at $b=0$, we get
\begin{equation}\label{proof of partialt tkappa transgression2}
    \derivative{b}\str\Big[tN_{K,t}g\exp\Big(-\lk-\cvt^2-\dabar \derivative{t}\cvt'-b\frac{i\mu_2}{v}\Big)\Big]_{b=0}^{\dabar}\equiv \sigma_{v,t}'.
\end{equation}
Similarly, we have
\begin{equation}\label{proof of partialt tkappa transgression3}
    \derivative{b}\str\Big[tN_{K,t}g\exp\Big(-\lk-\cvt^2-da\derivative{t}\cvt''-b\frac{i\mu_2}{v}\Big)\Big]_{b=0}^{da}\equiv \sigma_{v,t}''.
\end{equation}

Using (\ref{proof of partialt tkappa transgression2}), (\ref{proof of partialt tkappa transgression3}) and proceeding as in (\ref{partial t tkappa1 1}), (\ref{partial t tkappa1}), we obtain that 
\begin{equation}\label{partialt kappa2}
    \derivative{t}(t\kappa_{v,t}^2)=\theta_{v,t}^2+\rho_{v,t}'-\rho_{v,t}''+\dba \sigma_{v,t}'+\partial\sigma_{v,t}'' \qquad \mathrm{mod} \kongge\dba\partial \Omega^{\bullet}(S, \mathbb{C}).
\end{equation}
Applying $\partial$ and $\dba$ to (\ref{partialt kappa2}) and using (\ref{pichikappa}), we obtain the second and third equations in (\ref{ddbarpartialt kappa1}) respectively.
\end{proof}

By using Theorem \ref{thm-ddbarpartialt kappa1}, we have
\begin{equation}\label{integration by parts-sum}
    \begin{split}
    \dba\partial\int_{a}^{1}\frac{dt}{t}\int_{t}^{+\infty} \derivative{t}(t\kappa_{v,t}^{1})\frac{dv}{v}+\dba  \int_{a}^{1}\frac{dt}{t}\int_{t}^{+\infty} \derivative{t}(t\chi_{v,t}^{2\prime})\frac{dv}{v}
    +\partial\int_{a}^{1}\frac{dt}{t}\int_{t}^{+\infty} \derivative{t}(t\chi_{v,t}^{2\prime\prime})\frac{dv}{v}\\ 
    =\dba\partial\int_{a}^{1}\frac{dt}{t}\int_{t}^{+\infty}\theta_{v,t}^1\frac{dv}{v}+\frac{1}{2}\dba \int_{a}^{1}\frac{dt}{t}\int_{t}^{+\infty} \Big(\partial \theta_{v,t}^2-\derivative{t}(t\pi_{v,t}^{2\prime})\Big)\frac{dv}{v}\\ 
    +\frac{1}{2}\partial \int_{a}^{1}\frac{dt}{t}\int_{t}^{+\infty} \Big(-\dba \theta_{v,t}^2-\derivative{t}(t\pi_{v,t}^{2\prime\prime})\Big)\frac{dv}{v}+\dba\partial\int_{a}^1(\eta_{t,t}-\eta_t^{*})\frac{dt}{t}.
    \end{split}
\end{equation}

We now study each term in (\ref{kappa1-integration by parts}) and (\ref{integration by parts-sum}) individually.

\underline{2.1.The term $\theta_{v,t}^1$.}
Let $\theta_{0,0}^1$ be the limit of $\theta_{v,0}^1$ as $v\to 0$, by \eqref{bw}, \eqref{obo}, \eqref{hv0},  we get
\begin{align}\label{c0}
\begin{split}
    \Phi\theta_{0,0}^1=\int_{X_{g,K}}\Big[\td_{ge^K}'&(TX,h^{TX})-\mathrm{dim}X\,\td_{ge^K}(TX,h^{TX})\\
      +\td_{ge^K}&(TX,h^{TX})\Big(\frac{c_{\mathrm{max}}'}{c_{\mathrm{max}}}\Big)_K(N_{X_{g,K/X_g}},h^{X_{g,K/X_g}})  \Big]\\    & \cdot\Big(\frac{c_{\mathrm{max}}'}{c_{\mathrm{max}}}\Big)_K(N_{X_{g,K/X_g}},h^{X_{g,K/X_g}})\ch_{ge^K}(\xi,h^{\xi}).
\end{split}
\end{align}

By using (\ref{h,t01v1}), for $0<t\leq 1\leq v$, we have
\begin{equation}\label{theta1-t01v1-1}
    |\theta_{v,t}^1-\theta_{v,0}^1|\leq C\frac{t^{\gamma}}{v}.
\end{equation}
By \eqref{def of dvt}, (\ref{hv0}), we see that as $v\rightarrow +\infty$,
\begin{equation}\label{vinfty thetav0}
    \theta_{v,0}^1=\mathcal{O}\Big(\frac{1}{v}\Big).
\end{equation}
Therefore by (\ref{theta1-t01v1-1}), (\ref{vinfty thetav0}),  as $a\rightarrow 0$, we obtain that
\begin{equation}\label{theta1-t01v1}
    \begin{split}
    \int_{a}^{1}\frac{dt}{t}\int_{1}^{+\infty}\theta_{v,t}^1\frac{dv}{v}+\int_{1}^{+\infty}\theta_{v,0}^1\frac{dv}{v}\log(a)
        \rightarrow \int_{0}^{1}\frac{dt}{t}\int_{1}^{+\infty}(\theta_{v,t}^1-\theta_{v,0}^1)\frac{dv}{v}.
    \end{split}
\end{equation}

As for the integration for $v\in [t,1]$, we have
\begin{equation}\label{theta1-t01v01}
    \begin{split}
    \int_{a}^{1}\frac{dt}{t}\int_{t}^{1}\theta_{v,t}^1\frac{dv}{v}=\int_{a}^{1}\frac{dt}{t}\int_{1}^{1/t}(\theta_{vt,t}^1-\theta_v^1-\theta_{vt,0}^1+\theta_{0,0}^1)\frac{dv}{v}\qquad\qquad\qquad\\ 
    +\int_{a}^{1}\frac{dt}{t}\int_{1}^{1/t}(\theta_v^1-\theta_{0,0}^1)\frac{dv}{v}+\int_{a}^{1}\frac{dt}{t}\int_{t}^{1}(\theta_{v,0}^1-\theta_{0,0}^1)\frac{dv}{v}
    -\frac{1}{2}\log^2(a)\theta_{0,0}^1.
    \end{split}
\end{equation}
By  (\ref{h,t01v01}), \eqref{h,vt}, for $0<t\leq 1 \leq v\leq 1/t$,
\begin{align}\label{thetatv inequality1}
    &|\theta_{vt,t}^1-\theta_{vt,0}^1|\leq \frac{C}{v^{\gamma'}},
      &&  |\theta_{v}^1-\theta_{0,0}^1|\leq \frac{C}{v}.
\end{align}
By (\ref{bw}), (\ref{hv0}), (\ref{h,vt}), \eqref{c0}, for  $0<t\leq 1 \leq v\leq 1/t$,
\begin{align}\label{thetatv inequality2}
    &|\theta_{vt,t}^1-\theta_{v}^1|\leq Ct^{\alpha}v^{\beta},
    &&|\theta_{vt,0}^1-\theta_{0,0}^1|\leq C(tv)^{1/2}.
\end{align}
By (\ref{thetatv inequality1}), (\ref{thetatv inequality2}), we find that there exist $\alpha',\beta'\in \,]0,1]$ such that for $0<t\leq 1 \leq v\leq 1/t$,
\begin{equation}\label{thetatv inequality3}
    |\theta_{vt,t}^1-\theta_v^1-\theta_{vt,0}^1+\theta_{0,0}^1|\leq C\frac{t^{\alpha'}}{v^{\beta'}}.
\end{equation}
From (\ref{thetatv inequality1})-(\ref{thetatv inequality3}),  we see that as $a\rightarrow 0$,
\begin{equation}\label{three theta rightarrows}
    \begin{split}
    &\int_{a}^{1}\frac{dt}{t}\int_{1}^{1/t}(\theta_{vt,t}^1-\theta_v^1-\theta_{vt,0}^1+\theta_{0,0}^1)\frac{dv}{v}\\
    &\qquad\qquad\qquad\qquad\qquad\rightarrow\int_{0}^{1}\frac{dt}{t}\int_{1}^{1/t}(\theta_{vt,t}^1-\theta_v^1-\theta_{vt,0}^1+\theta_{0,0}^1)\frac{dv}{v},\\ 
    &\int_{a}^{1}\frac{dt}{t}\int_{1}^{1/t}(\theta_v^1-\theta_{0,0}^1)\frac{dv}{v}+\int_{1}^{+\infty}(\theta_v^1-\theta_{0,0}^1)\frac{dv}{v} \log(a)\\ 
&\qquad\quad\qquad\qquad\qquad\qquad\qquad\qquad\rightarrow -\int_{1}^{+\infty}(\theta_v^1-\theta_{0,0}^1)\log(v)\frac{dv}{v}\\
&\int_{a}^{1}\frac{dt}{t}\int_{t}^{1}(\theta_{v,0}^1-\theta_{0,0}^1)\frac{dv}{v}+\int_{0}^{1}(\theta_{v,0}^1-\theta_{0,0}^1)\frac{dv}{v}\log(a) \\
&\qquad\quad\qquad\qquad\qquad\qquad\qquad\qquad  \rightarrow\int_{0}^{1}(\theta_{v,0}^1-\theta_{0,0}^1)\log(v)\frac{dv}{v}.
    \end{split}
\end{equation}
By using (\ref{partial thetau}), \eqref{def hv}, \eqref{c0}, for $v\geq 1$, we have
\begin{align}\label{partial thetau 2}
    &\dba \theta_v^1=\partial\theta_v^1=0,
      &&\dba\theta_{0,0}^1=\partial\theta_{0,0}^1=0.
\end{align}

From (\ref{theta1-t01v1}), (\ref{theta1-t01v01}), (\ref{three theta rightarrows}), (\ref{partial thetau 2}), we find that as $a\rightarrow 0$,
\begin{equation}\label{2.1term}
    \begin{split}
   \dba\partial\int_{a}^{1}& \frac{dt}{t}\int_{t}^{+\infty}\theta_{v,t}^1\frac{dv}{v}\\ &+\dba\partial\Bigl(
    \int_{0}^{1}(\theta_{v,0}^1-\theta_{0,0}^1)\frac{dv}{v}+\int_{1}^{+\infty}\theta_{v,0}^1\frac{dv}{v}\Bigr)\log(a)\\
        \rightarrow \, \dba\partial \Bigl(&\int_{0}^{1}\frac{dt}{t}\int_{1}^{1/t}(\theta_{vt,t}^1-\theta_v^1-\theta_{vt,0}^1+\theta_{0,0}^1)\frac{dv}{v}\\
      &+\int_{0}^{1}\frac{dt}{t}\int_{1}^{+\infty}(\theta_{v,t}^1-\theta_{v,0}^1)\frac{dv}{v}
       + \int_{0}^{1}(\theta_{v,0}^1-\theta_{0,0}^1)\log(v)\frac{dv}{v}\Bigr).
    \end{split}
\end{equation}

\underline{2.2.The terms $\partial \theta_{v,t}^2-\derivative{t}(t\pi_{v,t}^{2\prime})$, $-\dba \theta_{v,t}^2-\derivative{t}(t\pi_{v,t}^{2\prime\prime})$.} By using \eqref{hv0}, (\ref{h,t01v1}), for $0<t\leq 1\leq v$, we have
\begin{equation}\label{theta2 t01v1 1}
    \begin{split}
        &\big|\partial \theta_{v,t}^2-\derivative{t}(t\pi_{v,t}^{2\prime})-\partial\theta_{v,0}^2\big|\leq C\frac{t^{\gamma}}{v},\\ 
        &\big|-\dba \theta_{v,t}^2-\derivative{t}(t\pi_{v,t}^{2\prime\prime})+\dba\theta_{v,0}^2\big|\leq C\frac{t^{\gamma}}{v}.
    \end{split}
\end{equation}
By (\ref{def of dvt}), (\ref{hv0}), we have as $v\rightarrow +\infty$,
\begin{equation}\label{theta2 t01v1 2}
  \theta_{v,0}^2=\mathcal{O}\Big(\frac{1}{v}\Big).
\end{equation}
Therefore by (\ref{theta2 t01v1 1}), (\ref{theta2 t01v1 2}), as $a\rightarrow 0$, we obtain that
\begin{equation}\label{theta2 t01v1 rightarrow}
\begin{split}
    \int_{a}^{1}\frac{dt}{t}\int_{1}^{+\infty}\Big(\partial \theta_{v,t}^2-\derivative{t}(t\pi_{v,t}^{2\prime})\Big)\frac{dv}{v}+ \int_{1}^{+\infty}\partial\theta_{v,0}^2\frac{dv}{v}\log(a)\\
        \rightarrow \int_{0}^{1}\frac{dt}{t}\int_{1}^{+\infty}\Big(\partial \theta_{v,t}^2-\derivative{t}(t\pi_{v,t}^{2\prime})-\partial\theta_{v,0}^2\Big)\frac{dv}{v},\\ 
        \int_{a}^{1}\frac{dt}{t}\int_{1}^{+\infty}\Big(-\dba\theta_{v,t}^2-\derivative{t}(t\pi_{v,t}^{2\prime\prime})\Big)\frac{dv}{v}- \int_{1}^{+\infty}\dba\theta_{v,0}^2\frac{dv}{v}\log(a)\\
        \rightarrow \int_{0}^{1}\frac{dt}{t}\int_{1}^{+\infty}\Big(-\dba \theta_{v,t}^2-\derivative{t}(t\pi_{v,t}^{2\prime\prime})+\dba\theta_{v,0}^2\Big)\frac{dv}{v}.
    \end{split}
\end{equation}

From \eqref{hv0}, (\ref{h,t01v01}), for $0<t\leq v\leq 1$,
\begin{equation}\label{theta2 t01v01-begin}
\begin{split}
&\big|\partial \theta_{v,t}^2-\derivative{t}(t\pi_{v,t}^{2\prime})-\partial\theta_{v,0}^2\big|\leq C\frac{1}{v}\Big(\frac{t}{v}\Big)^{\gamma'},\\
&\big|-\dba \theta_{v,t}^2-\derivative{t}(t\pi_{v,t}^{2\prime\prime})+\dba\theta_{v,0}^2\big|\leq C\frac{1}{v}\Big(\frac{t}{v}\Big)^{\gamma'}.
\end{split}
\end{equation}
Since $\mu_2$ is locally constant in a neighborhood $U$ of $W_{K}$, we have 
\begin{equation}\label{partial mu2}
\dba \mu_2=\partial\mu_2=0 \kongge\kongge \mathrm{in} \kongge U. 
\end{equation}
By \eqref{dkkprime},  \eqref{def of dvt}, \eqref{hv0}, \eqref{partial mu2},
we see that $\dba \theta_{v,0}^2$ and $\partial \theta_{v,0}^2$ decay exponentially as $v\rightarrow 0$. In particular, for any $m>0$, there exists $C_m>0$ such that for $v\in \,]0,1]$,
\begin{align}\label{partial theta2}
&|\dba \theta_{v,0}^2|\leq C_m v^{m}, &&|\partial\theta_{v,0}^2|\leq C_m v^m.
\end{align}
By Theorems \ref{thm,h2,t01v01} and \ref{thm-ddbarpartialt kappa1} , we obtain that for any $m>0$, there exists $C_m'>0$ such that for $0<t\leq v\leq 1$,
\begin{equation}\label{theta2 t01v01-end}
\begin{split}
&\big|\partial \theta_{v,t}^2-\derivative{t}(t\pi_{v,t}^{2\prime})\big|\leq C_{m}'\frac{v^m}{t^2},\\
&\big|-\dba \theta_{v,t}^2-\derivative{t}(t\pi_{v,t}^{2\prime\prime})\big|\leq C_{m}'\frac{v^m}{t^2}.
\end{split}
\end{equation}
We choose $m>0$ large enough , then by using (\ref{theta2 t01v01-begin})-(\ref{theta2 t01v01-end}), we see that there exists $\widetilde{\gamma}>0$, such that for $0<t\leq v\leq 1$,
\begin{equation}\label{theta2 t01v01 tot}
\begin{split}
&|\partial \theta_{v,t}^2-\derivative{t}(t\pi_{v,t}^{2\prime})-\partial\theta_{v,0}^2|\leq C(tv)^{\widetilde{\gamma}},\\
&|-\dba \theta_{v,t}^2-\derivative{t}(t\pi_{v,t}^{2\prime\prime})+\dba\theta_{v,0}^2|\leq C(tv)^{\widetilde{\gamma}}.
\end{split}
\end{equation}
So by (\ref{partial theta2}), (\ref{theta2 t01v01 tot}), as $a\rightarrow 0$, we get
\begin{equation}\label{theta2 t01v01 rightarrow}
\begin{split}
    \int_{a}^{1}\frac{dt}{t}\int_{t}^{1}\Big(\partial \theta_{v,t}^2-\derivative{t}(t\pi_{v,t}^{2\prime})\Big)\frac{dv}{v}+ \int_{0}^{1}\partial\theta_{v,0}^2\frac{dv}{v}\log(a)\qquad\quad\\
        \rightarrow \int_{0}^{1}\frac{dt}{t}\int_{0}^{1}\Big(\partial \theta_{v,t}^2-\derivative{t}(t\pi_{v,t}^{2\prime})-\partial\theta_{v,0}^2\Big)\frac{dv}{v},\\ 
        \int_{a}^{1}\frac{dt}{t}\int_{0}^{1}\Big(-\dba\theta_{v,t}^2-\derivative{t}(t\pi_{v,t}^{2\prime\prime})\Big)\frac{dv}{v}- \int_{0}^{1}\dba\theta_{v,0}^2\frac{dv}{v}\log(a)\qquad\quad\\
        \rightarrow \int_{0}^{1}\frac{dt}{t}\int_{0}^{1}\Big(-\dba \theta_{v,t}^2-\derivative{t}(t\pi_{v,t}^{2\prime\prime})+\dba\theta_{v,0}^2\Big)\frac{dv}{v}.
    \end{split}
\end{equation}
By (\ref{theta2 t01v1 rightarrow}), (\ref{theta2 t01v01 rightarrow}), we get
\begin{equation}\label{2.2term}
        \begin{split}
           \frac{1}{2}\dba \int_{a}^{1}\frac{dt}{t}\int_{t}^{+\infty} \Big(\partial \theta_{v,t}^2-\derivative{t}(t\pi_{v,t}^{2\prime})\Big)\frac{dv}{v}\qquad\qquad\qquad\qquad\qquad\qquad\qquad\quad\\
        +\frac{1}{2}\partial \int_{a}^{1}\frac{dt}{t}\int_{t}^{+\infty} \Big(-\dba \theta_{v,t}^2-\derivative{t}(t\pi_{v,t}^{2\prime\prime})\Big)\frac{dv}{v}
        +\int_{0}^{+\infty}\dba\partial\theta_{v,0}^2\frac{dv}{v}\log(a) \\ 
        \rightarrow  \frac{1}{2}\dba\int_{0}^{1}\frac{dt}{t}\int_{0}^{+\infty}\Big(\partial \theta_{v,t}^2-\derivative{t}(t\pi_{v,t}^{2\prime})-\partial\theta_{v,0}^2\Big)\frac{dv}{v} \qquad\qquad\qquad\qquad\qquad\\
         +\frac{1}{2}\partial\int_{a}^{1}\frac{dt}{t}\int_{0}^{+\infty}\Big(-\dba\theta_{v,t}^2-\derivative{t}(t\pi_{v,t}^{2\prime\prime})+\dba\theta_{v,0}^2\Big)\frac{dv}{v}\qquad\qquad\quad.
        \end{split}
        \end{equation}

\underline{2.3.The terms $\eta_{t,t}$, $\kappa_{t,t}^1$, $\eta_t^*$.} By (\ref{definition nkt}), \eqref{hvt}, we have
\begin{equation}\label{begin etatt kappatt}
    \begin{split}
    \eta&_{t,t}-\kappa_{t,t}^1=\frac{1}{2}\frac{\partial^2}{\partial c^2}\str\Big[ge^K\exp\Big(-\bisconn_t^2+c\Big(N_{K,t}+i\frac{\mu_1}{t}\Big)\Big)\Big]_{c=0}-\tau_{t,t}\\ 
=&\frac{1}{2}\frac{\partial^2}{\partial c^2}\str\Big[ge^K\exp\Big(-\bisconn_t^2+c\Big(N_{t}-i\frac{\mu_2}{t}\Big)\Big)\Big]_{c=0}-\tau_{t,t}\\ 
=&\frac{1}{2}\frac{\partial^2}{\partial c^2}\str\Big[ge^K\exp\Big(-\bisconn_t^2+cN_{t}\Big)\Big]_{c=0}
-\derivative{c}\str\Big[i\frac{\mu_2}{t}ge^K\exp\Big(-\bisconn_t^2+cN_t\Big)\Big]_{c=0}\\ 
&\qquad\qquad\qquad\qquad+\derivative{c}\str\Big[i\frac{\mu_2}{t}ge^K\exp\Big(-\bisconn_t^2+ci\frac{\mu_2}{t}\Big)\Big]_{c=0}-\tau_{t,t}.
    \end{split}
\end{equation}

By using the equivariant index theorem techniques and proceeding as in the proof of \cite[Theorems 2.14, 2.16]{BGS2}, \cite[Theorem 2.10]{Ma2000}, we know that  there are forms $H_{-2}, H_{-1}, H_{0}\in P^S$ such that as $t\rightarrow 0$,
\begin{equation}\label{i3}
    \Phi\frac{1}{2}\frac{\partial^2}{\partial c^2}\str\Big[ge^K\exp\Big(-\bisconn_t^2+cN_{t}\Big)\Big]_{c=0}=\frac{H_{-2}}{t^2}+\frac{H_{-1}}{t}+H_0+\mathcal{O}(t),
\end{equation}
and moreover,
\begin{equation}
    \begin{split}
        &\frac{\dba\partial}{2i\pi}H_{-2}=0, \qquad  \frac{\dba\partial}{2i\pi}H_{-1}=0,\\ 
       & \frac{\dba\partial}{2i\pi}H_0=D_{0,ge^K}-D_{0,ge^K}'.
    \end{split}
\end{equation}

\begin{proposition}\label{prop three asymptotics mu2}
   Let $g\in G$, $K_0\in \mathfrak{z}(g)$, $K=zK_0$ with $|z|\neq 0$ small enough.  For $j=-2,-1,0$, there are forms $R_{j}, S_{j}, T_{j}\in P^{S}$ such that as $t\rightarrow 0$,
\begin{equation}
    \begin{split}
        &\Phi \str\Big[i\frac{\mu_2}{t}ge^K\exp\Big(-\bisconn_t^2\Big)\Big]=\frac{R_{-1}}{t}+R_{0}+\mathcal{O}(t),\\
        &\Phi \derivative{c}\str\Big[i\frac{\mu_2}{t}ge^K\exp\Big(-\bisconn_t^2+ci\frac{\mu_2}{t}\Big)\Big]_{c=0}=\frac{S_{-2}}{t^2}+\frac{S_{-1}}{t}+S_{0}+\mathcal{O}(t),\\
         & \Phi\derivative{c}\str\Big[i\frac{\mu_2}{t}ge^K\exp\Big(-\bisconn_t^2+cN_t\Big)\Big]_{c=0}=\frac{T_{-2}}{t^2}+\frac{T_{-1}}{t}+T_{0}+\mathcal{O}(t).
    \end{split}
\end{equation}
Moreover we have
\begin{equation}\label{ddbas-2}
    \begin{split}   
        &\frac{\dba\partial}{2i\pi}S_{-2}=0, \qquad \frac{\dba\partial}{2i\pi}S_{-1}=0,\qquad \frac{\dba\partial}{2i\pi}S_0=0,\\ 
        &\frac{\dba\partial}{2i\pi}T_{-2}=0, \qquad \frac{\dba\partial}{2i\pi}T_{-1}=0,\qquad \frac{\dba\partial}{2i\pi}T_0=-R_{0}.
    \end{split}
\end{equation}
And
\begin{equation}
    R_{-1}=\widetilde{D}_{-1,g,K}=\int_{X_{g,K}}\td_{ge^K}(TX,h^{TX})\ch_{ge^K}(\xi,h^{\xi})i\product{\mu,K}.
\end{equation}
\end{proposition}
\begin{proof}
    Recall that the proof of Theorem \ref{thm asymptotic-g} is localized in a neighborhood of $W_{g,K}$. We can choose the neighborhood small enough so that $\mu_2$ is locally constant in this neighborhood. Then the same argument in proof of Theorem \ref{thm asymptotic-g} applies to Proposition \ref{prop three asymptotics mu2}.
\end{proof}

As for $\tau_{t,t}$, by (\ref{h,vt}), as $t\rightarrow 0$, we have
\begin{equation}
    |\tau_{t,t}-\tau_{1}|\leq C t^{\alpha}.
\end{equation}
By using (\ref{piu=0}) and \eqref{def hv}, we have
\begin{align}\label{end etatt kappatt}
   &\dba\tau_{1}=0, &&\partial\tau_{1}=0.
\end{align}

By combining (\ref{begin etatt kappatt})-(\ref{end etatt kappatt}), we find that there are smooth forms $K_{-2},K_{-1},K_0\in P^S$ such that as $t\rightarrow 0,$
\begin{equation}\label{total etatt kappatt}
    \Phi(\eta_{t,t}-\kappa_{t,t}^1)=\frac{K_{-2}}{t^2}+\frac{K_{-1}}{t}+K_0+\mathcal{O}(t^{\alpha}),
\end{equation}
and that
\begin{equation}\label{tot ddba etatt kappatt}
    \begin{split}
        \begin{split}
           \frac{\dba\partial}{2i\pi}K_{-2}=0,\qquad \frac{\dba\partial}{2i\pi}K_{-1}=0,\qquad
            \frac{\dba\partial}{2i\pi}K_0=D_{0,ge^K}-D_{0,ge^K}'+R_{0}.
        \end{split}
    \end{split}
\end{equation}

Recall that 
\begin{equation}
    \eta_{t}^{*}=\frac{1}{2}\frac{\partial^2}{\partial c^2}\str\Big[g\exp\Big(-\lk-\bisconn_{K,t}^2+cN_{K,t}\Big)\Big]_{c=0}.
\end{equation}
By the same arguments as  in the proof of Theorem \ref{thm asymptotic-gK} and \eqref{i3}, we know that  there are forms $L_{-2}, L_{-1}, L_{0}\in P^S$ such that as $t\rightarrow 0$,
\begin{equation}\label{etatstar}
    \Phi\frac{1}{2}\frac{\partial^2}{\partial c^2}\str\Big[g\exp\Big(-\lk-\bisconn_{K,t}^2+cN_{K,t}\Big)\Big]_{c=0}=\frac{L_{-2}}{t^2}+\frac{L_{-1}}{t}+L_0+\mathcal{O}(t),
\end{equation}
and moreover,
    \begin{equation}\label{ddba etatstar}
        \begin{split}
           \frac{\dba\partial}{2i\pi}L_{-2}=0, \qquad \frac{\dba\partial}{2i\pi}L_{-1}=0,\qquad
           \frac{\dba\partial}{2i\pi}L_0=C_{0,g,K}-C_{0,g,K}'.
        \end{split}
    \end{equation}
    By (\ref{total etatt kappatt})-(\ref{ddba etatstar}) we see that as $a\rightarrow 0$,
    \begin{equation}\label{2.3term}
    \begin{split}
        \frac{\dba\partial}{2i\pi}\Phi\Bigl(\int_{a}^{1}(\eta_{t,t}-\eta_t^*)\frac{dt}{t}-\int_{a}^{1}\kappa_{t,t}^1\frac{dt}{t}\Bigr)\qquad\qquad\qquad\\ 
        +(D_{0,ge^K}-D_{0,ge^K}'-C_{0,g,K}+C_{0,g,K}'+R_0)\log(a)\\ 
        \rightarrow  \frac{\dba\partial}{2i\pi}\Bigl(\int_{0}^{1}\Bigl(\Phi(\eta_{t,t}-\kappa_{t,t}^1)-\frac{K_{-2}}{t^2}-\frac{K_{-1}}{t}-K_0\Bigr)\frac{dt}{t}\\ 
        -\int_{0}^{1} \Bigl(\Phi\eta_t^*-\frac{L_{-2}}{t^2}-\frac{L_{-1}}{t}-L_0\Bigr)\frac{dt}{t}
        \Bigr).
    \end{split}
    \end{equation}

    \underline{2.4.The terms $\kappa_{v,a}^1$, $\chi_{v,a}^{2,\prime/\prime\prime}$.}
    By (\ref{calculation2 derivative v}) and \eqref{hvt}, we have
    \begin{equation}\label{i4}
        v\derivative{v}\str\Big[N_{K,t}g\exp\Big(-\lk-\cvt^2\Big)\Big]=\dba \chi_{v,t}^{0\prime}+\partial\chi_{v,t}^{0\prime\prime}+\lambda_{v,t}''-\lambda_{v,t}'.
    \end{equation}
    By using (\ref{pichikappa2}) and \eqref{i4}, we get 
    \begin{equation}\label{kappa1va}
        \begin{split}
        v\derivative{v}\str\Big[N_{K,t}g\exp\Big(-\lk-\cvt^2\Big)\Big]=\dba\partial \kappa_{v,t}^1+\dba \chi_{v,t}^{2\prime}+\partial\chi_{v,t}^{2\prime\prime}\\ 
        -\frac{1}{2}\dba \pi_{v,t}^{1\prime}-\frac{1}{2}\partial\pi_{v,t}^{1\prime\prime}+\lambda_{v,t}''-\lambda_{v,t}'.
        \end{split}
    \end{equation}

From \eqref{kappa1va}, we obtain
\begin{equation}\label{begin kappava}
    \begin{split}
       &\dba\partial  \int_{a}^{+\infty}\kappa_{v,a}^1\frac{dv}{v}+\dba\int_{a}^{+\infty}\chi_{v,a}^{2\prime}\frac{dv}{v}+\partial\int_{a}^{+\infty}\chi_{v,a}^{2\prime\prime}\frac{dv}{v}\\ 
        =&\str\Big[N_{K,a}g\exp\Big(-\lk-\bisconn_{K,a}^2\Big)\Big]-\str\Big[N_{K,a}ge^K\exp\Big(-\bisconn_{a}^2\Big)\Big]\\ 
        &+\frac{1}{2}\dba\int_{a}^{+\infty}\pi_{v,a}^{1\prime}\frac{dv}{v}+\frac{1}{2}\partial \int_{a}^{+\infty}\pi_{v,a}^{1\prime\prime}\frac{dv}{v}+\int_{a}^{+\infty}(\lambda_{v,a}'-\lambda_{v,a}'')\frac{dv}{v}.
    \end{split}
\end{equation}

By Theorem \ref{thm asymptotic-gK}, we see that as $a\rightarrow 0$, 
\begin{equation}
    \Phi\str\Big[N_{K,a}g\exp\Big(-\lk-\bisconn_{K,a}^2\Big)\Big]=C_{-1,g,K}'\frac{1}{a}+C_{0,g,K}'+\text{o}(1).
\end{equation}

For the second term on the right hand side of \eqref{begin kappava}, we note that
\begin{equation}\label{split str nka}
    \begin{split}
    \str\Big[N_{K,a}ge^K\exp\Big(-\bisconn_{a}^2\Big)\Big]=\str\Big[N_{a}ge^K\exp\Big(-\bisconn_a^2\Big)\Big]\qquad\qquad\\ 
    -\str\Big[i\frac{\mu_2}{a}ge^K\exp\Big(-\bisconn_a^2\Big)\Big]- \str\Big[i\frac{\mu_1}{a}g\exp\Big(-\lk-\mathcal{A}_{a,a}^2\Big)\Big].
    \end{split}
\end{equation}
Using Proposition \ref{prop three asymptotics mu2}, (\ref{asymptotic-gn}) and (\ref{mu1 part of main lemma 2}), we see that as $a\rightarrow 0$,
\begin{equation}
    \begin{split}
        &\Phi\str\Big[N_{a}ge^K\exp\Big(-\bisconn_a^2\Big)\Big]=D_{-1,ge^K}'\frac{1}{a}+D_{0,ge^K}'+\text{o}(1),\\ 
        &\Phi\str\Big[i\frac{\mu_2}{a}ge^K\exp\Big(-\bisconn_a^2\Big)\Big]=\widetilde{D}_{-1,g,K}\frac{1}{a}+R_0+\text{o}(1),\\ 
        &\Phi\str\Big[i\frac{\mu_1}{a}g\exp\Big(-\lk-\mathcal{A}_{a,a}^2\Big)\Big]=\widetilde{D}_{0,g,K}^1+\text{o}(1).
    \end{split}
\end{equation}

Using Propositions \ref{prop, piu=0}, \ref{i2}, Theorems \ref{h,t01} and  \ref{thm h,vt}, we proceed as in (\ref{J1-J4})-(\ref{j3}) and see that as $a\rightarrow 0$,
\begin{equation}
    \begin{split}
        &\int_{a}^{+\infty}(\lambda_{v,a}'-\lambda_{v,a}'' )\frac{dv}{v}\rightarrow \int_{1}^{+\infty}\lambda_{v}\frac{dv}{v},\\
        &\overline{\partial}\int_{a}^{+\infty}\pi_{v,a}^{1\prime}\frac{dv}{v}\rightarrow \overline{\partial}\int_{1}^{+\infty}\pi_{v}^{1\prime}\frac{dv}{v}=0,\\ 
        &\partial\int_{a}^{+\infty}\pi_{v,a}^{1\prime\prime}\frac{dv}{v}\rightarrow \partial\int_{1}^{+\infty}\pi_{v}^{1\prime\prime}\frac{dv}{v}=0.
    \end{split}
\end{equation}
From (\ref{Phi and I}), \eqref{bh}, (\ref{lambdau thetau}) \eqref{ev formula}, \eqref{i5} and (\ref{def hv}), we get
\begin{equation}\label{end kappava}
    \begin{split}
   \Phi \int_{1}^{+\infty}\lambda_v\frac{dv}{v}=E_{0,g,K}-\widetilde{D}_{0,g,K}^1.
    \end{split}
\end{equation}

By (\ref{begin kappava})-(\ref{end kappava}), we obtain that as $a\rightarrow 0$,
\begin{equation}\label{2.4term}
    \begin{split}
        \Phi\Bigl(\dba\partial \int_{a}^{+\infty}\kappa_{v,a}^1\frac{dv}{v}+\dba\int_{a}^{+\infty}\chi_{v,a}^{2\prime}\frac{dv}{v}+\partial\int_{a}^{+\infty}\chi_{v,a}^{2\prime\prime}\frac{dv}{v}\Bigr)\\
        -(C_{-1,g,K}'-D_{-1,ge^K}'+\widetilde{D}_{-1,g,K})\frac{1}{a}\qquad\\
        \rightarrow C_{0,g,K}'-D_{0,ge^K}'+E_{0,g,K}+R_0.\quad\quad
    \end{split}
\end{equation} 
Using (\ref{diverge loga}), (\ref{ddbas-2}),  we see that there is a smooth form $H_0\in P^S$ such that 
\begin{equation}\label{ddbaH_0}
    C_{0,g,K}'-D_{0,ge^K}'+E_{0,g,K}+R_0=\frac{\dba\partial}{2i\pi}H_0.
\end{equation}

\underline{2.5. A sum up.}
By using (\ref{transgression bt}), (\ref{hv0}), (\ref{partial thetau 2}), we have
\begin{equation}\label{cdenoprime}
    \begin{split}
   \Phi\Bigl({\dba\partial}
        \int_{0}^{1}(\theta_{v,0}^1-\theta_{0,0}^1)\frac{dv}{v}+{\dba\partial}\int_{1}^{+\infty}\theta_{v,0}^1\frac{dv}{v}+\int_{0}^{+\infty}\dba\partial\theta_{v,0}^2\frac{dv}{v}\Bigr)\\
 =C_{0,g,K}-D_{0,ge^K}+E_{0,g,K}.
    \end{split}
\end{equation}

Using (\ref{kappa1-integration by parts}), (\ref{integration by parts-sum}), (\ref{ddbaH_0}), (\ref{cdenoprime}), we take the sum of (\ref{2.1term}), (\ref{2.2term}), (\ref{2.3term}), (\ref{2.4term}) and see that there are smooth forms $T', T'' ,H$ on $S$, such that as $a\rightarrow 0$,
\begin{equation}\label{2.term}
    \begin{split}
        \Phi\Bigl(\dba\partial\Delta_{a,\infty}(\kappa^1)+\dba \Delta_{a,\infty}(\chi^{2\prime})+\partial\Delta_{a,\infty}(\chi^{2\prime\prime})\Bigr)\\ 
          -(-C_{0,g,K}'+D_{0,ge^K}'-E_{0,g,K}-R_0)\log(a)\\
        -\Bigl(C_{-1,g,K}'-D_{-1,ge^K}'+\widetilde{D}_{-1,g,K}\Bigr)\frac{1}{a}\\
        \rightarrow \Phi\dba T'+\Phi\partial T''+\frac{\dba\partial}{2i\pi}\Phi H.\qquad\qquad\qquad
    \end{split}
\end{equation}

\textbf{\underline{3. The terms $\Delta_{a,\infty}(\pi^{2\prime})$, $\Delta_{a,\infty}(\pi^{2\prime\prime})$.}}
By using (\ref{hv0}), (\ref{h,t01v1}), as $a\rightarrow 0$, we have
\begin{equation}\label{pi2 v1}
    \begin{split}
    \int_{a}^{1}\frac{dt}{t}\int_{1}^{+\infty}\pi_{v,t}^{2,\prime/\prime\prime}\frac{dv}{v}\rightarrow  \int_{0}^{1}\frac{dt}{t}\int_{1}^{+\infty}\pi_{v,t}^{2,\prime/\prime\prime}\frac{dv}{v}.
    \end{split}
\end{equation}

We now turn to the integration of $\pi_{v,t}^{2,\prime/\prime\prime}$ for $v\in [1, +\infty[$. This case is more subtle than that of $\pi_{v,t}^{1,\prime/\prime\prime}$ treated in (\ref{deltaa1})-(\ref{Deltaa1pi1'rightarrow}). Indeed, as seen in \eqref{h,t01v01}, the forms $\pi_{v,t}^{2,\prime/\prime\prime}$ behave worse  than $\pi_{v,t}^{1,\prime/\prime\prime}$ when $v$ is small. We need to exploit the fact that $\mu_2$ is locally constant near $W_K$ and use an analogue of \cite[Theorem 0.2]{LiuMa2022}.

For a  form $\alpha$ on $S$, set $\alpha^{(=)}$ to be the sum of complex $(p,p)$-type part of $\alpha$, and $\alpha^{(+1)}$ the sum of $(p,p+1)$ part, $\alpha^{(-1)}$ the sum of $(p+1,p)$ part. By using (\ref{cvt-algebraic identity}) and proceeding as in \cite[Proposition 1.8 and Theorem 1.9]{BGS1988}, we have
\begin{align}\label{p,p-1,p+1}
    &\pi_{v,t}^{2\prime}=(\pi_{v,t}^{2\prime})^{(-1)}, &&\pi_{v,t}^{2\prime\prime}=(\pi_{v,t}^{2\prime\prime})^{(+1)}.
\end{align}
For the convenience of using the language of \cite{LiuMa2022}, we will study $d(\pi_{v,t}^{2\prime}+\pi_{v,t}^{2\prime\prime })$ instead of $\dba \pi_{v,t}^{2\prime}+\partial \pi_{v,t}^{2\prime\prime }$. Clearly from (\ref{p,p-1,p+1}), we have
\begin{equation}\label{2,p,p-1,p+1}
    \dba \pi_{v,t}^{2\prime}+\partial \pi_{v,t}^{2\prime\prime }=\Bigl(d(\pi_{v,t}^{2\prime}+\pi_{v,t}^{2\prime\prime })\Bigr)^{(=)}.
\end{equation}

Recall that $K'^X$ is the one form on $W$ which is dual to $K^{X}\in T_{\rr}X$. Set 
\begin{equation}
    \widetilde{\mathfrak{e}}_{t}=-\int_{X_g}i\mu_2\td_{g,K}(TX,h^{TX})\ch_{g,K}(\xi,h^{\xi})\frac{K'^X}{4i\pi t}\exp\Bigl(\frac{d_K K'^X}{4i\pi t}\Bigr).
\end{equation}
By proceeding as in (\ref{aym expansion 2})-(\ref{sigma tau}), there is a smooth form $\sigma_1$ on $S$, such that as $t\rightarrow 0$, 
\begin{equation}\label{expansion frak et}
    \widetilde{\mathfrak{e}}_t=\sigma_1t+\mathcal{O}(t^2).
\end{equation}
And from (\ref{def of dvt}),  (\ref{ev formula}), \eqref{expansion ev}, we have
\begin{equation}\label{depsilon1/2}
    \begin{split}
    d\sigma_1=&-\derivative{t}\Bigl(t\int_{X_g} \td_{g,K}(TX,h^{TX})\ch_{g,K}(\xi,h^{\xi}) d_t^2\Bigr)_{t=0}\\ 
        =& E_{0,g,K}'-E_{0,g,K}.
    \end{split}
\end{equation}

Recall that the form $\psi_{v,t}$ is defined in (\ref{very many hvt}). For $v>0$, put
\begin{equation}\label{tilde psi}
    \widetilde{\psi}_{v}=\int_{0}^{v}\int_{t}^{v} \psi_{u,t}dudt.
\end{equation}
From (\ref{psi,t01v01}), we see that $\widetilde{\psi}_{t_0}$ is well defined and for any $m\in\nn^*$, there is $C>0$ such that for $v\leq 1$,
\begin{equation}\label{tilde psi estimate}
    |\widetilde{\psi}_v|\leq Cv^m.
\end{equation}

By \eqref{hv0}, (\ref{h,t01v1}), for $v>0$, the integral 
\begin{equation}\label{piprime + piprimeprime}
    \int_{0}^{v}\str\Big[i\mu_2 g\exp\Big(-\lk-\cvt^2-dt\derivative{t}\cvt\Big)\Big]^{dt}dt=\int_{0}^{v}\frac{v}{2}(\pi_{v,t}^{2\prime}+\pi_{v,t}^{2\prime\prime})\frac{dt}{t}
\end{equation}
is well defined. 

By using the same arguments as in the proof of Proposition \ref{prop three asymptotics mu2}, we have an analogue of \cite[Theorem 2.11]{BGS1988}, there is smooth odd form   $Q_0$ on $S$ such that as $t\rightarrow 0$,
\begin{equation}\label{expansion of mu2 transgression}
    \Big\{\Phi_{\rr^2\times S}\str\Big[i\mu_2 ge^K\exp\Big(-\bisconn_t^2-dt\derivative{t}\bisconn_t\Big)\Big]\Big\}^{dt}=Q_0+\mathcal{O}(t).
\end{equation}
And we have
\begin{equation}\label{dq0}
    dQ_0=R_0.
\end{equation}
From (\ref{expansion of mu2 transgression}), we see that for $v>0$, the integral 
\begin{equation}
\int_{0}^{v} \Big\{\Phi_{\rr^2\times S}\str\Big[i\mu_2 ge^K\exp\Big(-\bisconn_t^2-dt\derivative{t}\bisconn_t\Big)\Big]\Big\}^{dt}dt
\end{equation}
is well defined.

\begin{theorem}\label{analogue of liuma}
    Let $g\in G$, $K_0\in \mathfrak{z}(g)$, $K=zK_0$ with $|z|\neq 0$  small enough. Then for any $v>0$, module $d$-exact forms on $S$, we have
\begin{equation}\label{extension liuma}
    \begin{split}
        -\int_{0}^{v}\Big\{\Phi_{\rr^2\times S}\str\Big[i\mu_2 g\exp\Big(-\lk-\cvt^2-dt\derivative{t}\cvt\Big)\Big]\Big\}^{dt}dt\qquad\qquad\\ 
        +\int_{0}^{v}\Big\{\Phi_{\rr^2\times S}\str\Big[i\mu_2 ge^K\exp\Big(-\bisconn_t^2-dt\derivative{t}\bisconn_t\Big)\Big]\Big\}^{dt}dt=\int_{0}^{v}\widetilde{\mathfrak{e}}_t\frac{dt}{t}-\widetilde{\psi}_v.
    \end{split}
\end{equation}
\end{theorem}
\begin{remark}
Assuming that $\mu_2$ is constant on $W$, then $\widetilde{\psi}_v=0$. Moreover, when $v=+\infty$, we have $\cvt=\bisconn_{K,t}$, then Theorem \ref{analogue of liuma} coincides with \cite[Theorem 0.2]{LiuMa2022} up to a normalization.
\end{remark}
\begin{proof}
   We follow the steps in \cite[Section 4]{LiuMa2022}. For $t>0, v>0$, set 
   \begin{equation}
    \ccvt=\cvt+dt\wedge \derivative{t}+dv\wedge \derivative{v}.
   \end{equation}
   Then $\ccvt$ is a superconnection associated with the fiberation $\rr^{2*}\times W\rightarrow \rr^{2*}\times S$. So from (\ref{commutator c mu}), we have the analogue of \cite[(4.2)]{LiuMa2022},
   \begin{equation}\label{dcctv}
    \Big\{d^{\rr^2\times S}\Phi_{\rr^2\times S}\str\Big[i\mu_2 g\exp\Big(-\lk-\ccvt^2\Big)\Big]\Big\}^{dvdt}=\psi_{v,t}.
   \end{equation}
   Let $\beta_{g,K}$ be the part of $\Phi_{\rr^2\times S}\str\Big[i\mu_2 g\exp\Big(-\lk-\ccvt^2\Big)\Big]$ of degree one with respect to the coordinates $(v,t)$, and let $\alpha_{g,K}$ be the part of degree two. By (\ref{dcctv}), we get
   \begin{equation}
    \Bigl(dv\wedge \derivative{v}+dt\wedge\derivative{t}\Bigr)\beta_{g,K}=-dv\wedge dt\wedge (d^S \alpha_{g,K}+\psi_{v,t}).
   \end{equation}
   Recall that we have introduced the oriented contour $\Gamma$ and the triangular domain $\Delta$ in Section 5.2. So by Stokes' theorem as in (\ref{0 sum of Ik0}), we have
   \begin{equation}\label{sum liuma beta}
    \sum_{k=1}^3\int_{\Gamma_k}\beta_{g,K}=-d^S\Bigl(\int_{\Delta}\alpha_{g,K}dvdt\Bigr)-\int_{\Delta}\psi_{v,t}dvdt.
   \end{equation}
   Using the fact that $\mu_2$ is locally constant in a neighborhood of $W_{g,K}$, and the fact that the analysis in the proof of \cite[Theorem 4.3]{LiuMa2022} are localized near $W_{g,K}$, we can proceed as in \cite[(4.9)-(4.27)]{LiuMa2022} without letting $A\rightarrow +\infty$, and we obtain that as $a\rightarrow 0$,
   \begin{equation}
    \begin{split}
   \sum_{k=1}^3 &\int_{\Gamma_k}\beta_{g,K}\rightarrow\\ 
     &-\int_{0}^{A}\Big\{\Phi_{\rr^2\times S}\str\Big[i\mu_2 g\exp\Big(-\lk-\mathcal{A}_{A,t}^2-dt\derivative{t}\mathcal{A}_{A,t}\Big)\Big]\Big\}^{dt}dt\\ 
    &+\int_{0}^{A}\Big\{\Phi_{\rr^2\times S}\str\Big[i\mu_2 ge^K\exp\Big(-\bisconn_t^2-dt\derivative{t}\bisconn_t\Big)\Big]\Big\}^{dt}dt-\int_{0}^{A}\widetilde{\mathfrak{e}}_t\frac{dt}{t}.
    \end{split}
   \end{equation}
Also by (\ref{psi,t01v01}), (\ref{tilde psi}), we see that as $a\rightarrow 0$,
\begin{equation}\label{psi a rightarrow 0}
    \int_{\Delta}\psi_{v,t}dvdt\rightarrow \widetilde{\psi}_A.
\end{equation}
By using the fact that $d\Omega^{\bullet}(S,\cc)$ is closed under the uniformly convergence. Thus from (\ref{sum liuma beta})-(\ref{psi a rightarrow 0}), we get (\ref{extension liuma}).
\end{proof}

Now we return to the integration of $\pi_{v,t}^{2\prime}$ and $\pi_{v,t}^{2\prime\prime }$. From (\ref{piprime + piprimeprime}), (\ref{extension liuma}), we have
\begin{equation}\label{Phid-calcu-begin}
    \begin{split}
       &\Phi  d\int_{0}^{v}\frac{v}{2}(\pi_{v,t}^{2\prime}+\pi_{v,t}^{2\prime\prime})\frac{dt}{t}\\ 
       = &d\int_{0}^{v}\Big\{\Phi_{\rr^2\times S}\str\Big[i\mu_2 ge^K\exp\Big(-\bisconn_t^2-dt\derivative{t}\bisconn_t\Big)\Big]\Big\}^{dt}dt-d\int_{0}^{v}\widetilde{\mathfrak{e}}_t\frac{dt}{t}+d\widetilde{\psi}_v.
    \end{split}
\end{equation}
So we get
\begin{equation}
    \begin{split}
        &\Phi d \int_{a}^{1}\frac{dv}{v}\int_{0}^{v}\frac{1}{2}(\pi_{v,t}^{2\prime}+\pi_{v,t}^{2\prime\prime})\frac{dt}{t}\\
       =&d\int_{a}^{1}\frac{dv}{v^2}\int_{0}^{v}\Big\{\Phi_{\rr^2\times S}\str\Big[i\mu_2 ge^K\exp\Big(-\bisconn_t^2-dt\derivative{t}\bisconn_t\Big)\Big]\Big\}^{dt}dt\\ 
      & -d\int_{a}^{1}\frac{dv}{v^2}\int_{0}^{v}\widetilde{\mathfrak{e}}_t\frac{dt}{t}+d\int_{a}^{1}\widetilde{\psi}_v\frac{dv}{v^2}.
    \end{split}
\end{equation}
Also we have
\begin{equation}\label{calculate mu2 transgression}
    \begin{split}
    \int_{a}^{1}\frac{dv}{v^2}\int_{0}^{v}\Big\{\Phi_{\rr^2\times S}\str\Big[i\mu_2 ge^K\exp\Big(-\bisconn_t^2-dt\derivative{t}\bisconn_t\Big)\Big]\Big\}^{dt}dt\\
      = \int_{a}^{1}\Big\{\Phi_{\rr^2\times S}\str\Big[i\mu_2 ge^K\exp\Big(-\bisconn_t^2-dt\derivative{t}\bisconn_t\Big)\Big]\Big\}^{dt}\frac{dt}{t} \quad\\ 
     +\frac{1}{a}\int_{0}^{a}\Big\{\Phi_{\rr^2\times S}\str\Big[i\mu_2 ge^K\exp\Big(-\bisconn_t^2-dt\derivative{t}\bisconn_t\Big)\Big]\Big\}^{dt}dt \\ 
      -\int_{0}^{1}\Big\{\Phi_{\rr^2\times S}\str\Big[i\mu_2 ge^K\exp\Big(-\bisconn_t^2-dt\derivative{t}\bisconn_t\Big)\Big]\Big\}^{dt}dt.
    \end{split}
\end{equation}
So by (\ref{expansion of mu2 transgression}), (\ref{dq0}), (\ref{calculate mu2 transgression}), as $a\rightarrow 0$, we get
\begin{equation}
    \begin{split}
        d\int_{a}^{1}\frac{dv}{v^2}\int_{0}^{v}\Big\{\Phi_{\rr^2\times S}\str\Big[i\mu_2 ge^K\exp\Big(-\bisconn_t^2-dt\derivative{t}\bisconn_t\Big)\Big]\Big\}^{dt}dt+R_0\log(a)\\ 
        \rightarrow d \int_{0}^{1}\Bigl(\Big\{\Phi_{\rr^2\times S}\str\Big[i\mu_2 ge^K\exp\Big(-\bisconn_t^2-dt\derivative{t}\bisconn_t\Big)\Big]\Big\}^{dt}-Q_0\Bigr)\frac{dt}{t}\\ 
        -d\int_{0}^{1}\Big\{\Phi_{\rr^2\times S}\str\Big[i\mu_2 ge^K\exp\Big(-\bisconn_t^2-dt\derivative{t}\bisconn_t\Big)\Big]\Big\}^{dt}dt+ dQ_0.
    \end{split}
\end{equation}
Similarly by using (\ref{expansion frak et}), (\ref{depsilon1/2}), as $a\rightarrow 0$,
\begin{equation}
    \begin{split}
    d\int_{a}^{1}\frac{dv}{v^2}\int_{0}^{v}\widetilde{\mathfrak{e}}_t\frac{dt}{t}+(E_{0,g,K}'-E_{0,g,K})\log(a)\qquad\qquad\quad\\ 
    \rightarrow d\int_{0}^{1}(\widetilde{\mathfrak{e}}_t-\sigma_1t)\frac{dt}{t^2}-d\int_{0}^{1}\widetilde{\mathfrak{e}}_t\frac{dt}{t}.
    \end{split}
\end{equation}
And by (\ref{tilde psi estimate}), as $a\rightarrow 0$,
\begin{equation}\label{Phid-calcu-end}
    d\int_{a}^{1}\widetilde{\psi}_v\frac{dv}{v^2}\rightarrow d\int_{0}^{1}\widetilde{\psi}_v\frac{dv}{v^2}.
\end{equation}
From (\ref{Phid-calcu-begin})-(\ref{Phid-calcu-end}), we find that there is a smooth form $S_0$, such that as $a\rightarrow 0$,
\begin{equation}\label{Phid-calcu-tot}
    \begin{split}
        \Phi d \int_{a}^{1}\frac{dv}{v}\int_{0}^{v}\frac{1}{2}(\pi_{v,t}^{2\prime}+\pi_{v,t}^{2\prime\prime})\frac{dt}{t}+(-E_{0,g,K}'+E_{0,g,K}+R_0)\log(a)
        \rightarrow d S_0.
    \end{split}
\end{equation}
Using (\ref{hvt}), (\ref{J1-J4}) and (\ref{2,p,p-1,p+1}), by taking the sum of $(p,p)$-form parts of (\ref{Phid-calcu-tot}), as $a\rightarrow 0$, we have
\begin{equation}\label{v01 pi2 final}
    \begin{split}
    \frac{1}{2}\Phi\dba\int_{a}^{1}\frac{dt}{t}\int_{t}^{1}\pi_{v,t}^{2\prime}\frac{dv}{v}+ \frac{1}{2}\Phi\partial\int_{a}^{1}\frac{dt}{t}\int_{t}^{1}\pi_{v,t}^{2\prime\prime}\frac{dv}{v}+J_4\qquad\qquad\qquad\\ 
    +(-E_{0,g,K}'+E_{0,g,K}+R_0)\log(a)
    \rightarrow \dba S_0^{(-1)}+\partial S_0^{(+1)}.
    \end{split}
\end{equation}
Finally by (\ref{pi2 v1}), (\ref{v01 pi2 final}), there are explicit universal forms $S', S''$, such that as $a\rightarrow 0$,
\begin{equation}\label{3.term}
    \begin{split}
    \frac{1}{2}\Phi\dba\Delta_{a,\infty}(\pi^{2\prime})+\frac{1}{2}\Phi\partial\Delta_{a,\infty}(\pi^{2\prime\prime})+J_4\qquad\qquad\qquad\\ 
    +(-E_{0,g,K}'+E_{0,g,K}+R_0)\log(a)
        \rightarrow \Phi\dba S'+\Phi\partial S''.
    \end{split}
\end{equation}

c)$\kongge$ Maching the divergences.

As $A\rightarrow +\infty$, by (\ref{left A}), (\ref{right A}), the divergences on both sides of (\ref{sum of Ik0}) vanish.

From (\ref{1.term}), (\ref{2.term}) and (\ref{3.term}), we see that as $a\rightarrow 0$,
\begin{equation}\label{right a}
    \begin{split}
        &\Phi\dba\Bigl(-\Delta_{a,\infty}(\pi^{1\prime})-\frac{1}{2}\Delta_{a,\infty}\Delta(\pi^{2\prime})+\Delta_{a,\infty}(\chi^{2\prime})\Bigr)+\frac{\dba\partial}{2i\pi}\Delta_{a,\infty}(\kappa^1)\\
        &+\Phi\partial\Bigl(-\Delta_{a,\infty}(\pi^{1\prime})-\frac{1}{2}\Delta_{a,\infty}\Delta(\pi^{2\prime})+\Delta_{a,\infty}(\chi^{2\prime}) \Bigr)-J_4\\ 
        &-(-C_{0,g,K}'+D_{0,ge^K}'-E_{0,g,K}')\log(a)
        -\Bigl(C_{-1,g,K}'-D_{-1,ge^K}'+\widetilde{D}_{-1,g,K}\Bigr)\frac{1}{a} \\
        &\rightarrow \Phi\dba(-\Delta_{0,\infty}(\pi^{1\prime})+S'+T')+\Phi\partial(-\Delta_{0,\infty}(\pi^{1\prime\prime})+S''+T'')+\frac{\dba\partial}{2i\pi}\Phi H.
    \end{split}
\end{equation}
By (\ref{diverge 1/a}), (\ref{diverge loga}), (\ref{j4}) and (\ref{right a}), the divergences on the left-hand side of (\ref{sum of Ik0}) coincide with the divergences on the right. 

Ultimately, we conclude that 
\begin{equation}\label{sum of Ik2}
    \begin{split}
    \sum_{0}^{3}I_k^2=\Phi\dba(-\Delta_{0,\infty}(\pi^{1\prime})+S'+T')\qquad\qquad\quad\\ 
    +\Phi\partial(-\Delta_{0,\infty}(\pi^{1\prime\prime})+S''+T'')+\frac{\dba\partial}{2i\pi}\Phi H.
    \end{split}
\end{equation}
And each term on the right-hand side of (\ref{sum of Ik2}) has been explicitly computed in this section. This completes the proof of Theorem \ref{8o}.

\section{Construction of the equivariant infinitesimal analytic torsion forms}\label{e0}
In this section, we will give the proof of  Theorems \ref{thm gK asym bku}, \ref{thm asymptotic-gK} and  the intermediate results Theorems \ref{main lemma1}, \ref{h,t1infty}, \ref{h,t01} and \ref{thm,h2,t01v01}. These are the family versions of the results in \cite[\S 7, 8]{BismutGoette2000} and our situation is very similar to that in \cite[\S 5, 6]{LiuMa2022}, so we will constantly refer to \cite{BismutGoette2000} and \cite{LiuMa2022}.

\subsection{The proof of Theorems \ref{thm gK asym bku} and \ref{thm asymptotic-gK} as $u\to +\infty$.}\label{ej}
For $K=zK_0\in \mathfrak{z}(g)$, set
\begin{equation}\label{akt}
    \akt=\lk+\bisconn_{K,t}^2.
\end{equation}

This subsection is devoted to prove the following theorem, which will be shown to imply the $u\rightarrow +\infty$ part of Theorems \ref{thm gK asym bku}, \ref{thm asymptotic-gK}.
\begin{theorem}\label{thm infty main1}
    For any $\beta>0$, there exist $C>0$, $\delta>0$ such that for $K\in \mathfrak{z}(g)\otimes_{\mathbb{R}}\mathbb{C}$,  $|K|\leq \beta$, $t\geq 1$,
    \begin{equation}\label{infty main1}
        \begin{split}
       &\big|\str\big[g\exp(-\akt)\big]-\str\big[g\exp(-\lk-\bisconn_t^2)\big]\big|\leq \frac{C}{t^{\delta}}, \\
       &\big|\str\big[gN_{K,t}\exp(-\akt)\big]-\str\big[gN_t\exp(-\lk-\bisconn_t^2)\big]\big|\leq \frac{C}{t^{\delta}}.
        \end{split}
    \end{equation}
\end{theorem}

Now we explain how to obtain the $u\rightarrow+\infty$ part of Theorems \ref{thm gK asym bku} and \ref{thm asymptotic-gK} from Theorem \ref{thm infty main1}. Since $\bisconn_t$ is $G$-invariant, we have
\begin{equation}\label{equation two supertraces g gK}
    \str\big[g\exp(-\lk-\bisconn_t^2)\big]=\str\big[ge^{K}\exp(-\bisconn_t^2)\big].
\end{equation} 
From the proof of (\ref{asymptotic-infty-g}), we know that there is $C>0$ such that for any $|K|\leq \beta$, $t\geq 1$, we have
\begin{equation}\label{geK asym infty}
    \big|\Phi\str\big[ge^{K}\exp(-\bisconn_t^2)\big]-\ch_{ge^{K}}(H^{{\bullet}}(X,\xi),h^{H^{{\bullet}}(X,\xi)})\big|\leq \frac{C}{\sqrt{t}}.
\end{equation}
Using \eqref{infty main1}-\eqref{geK asym infty}, we get the second line of (\ref{Kirillov formula-gK}). The same arguments work for the proof of (\ref{asymptotic-infty-gK}).

For $s\in S$, recall that $\mathbb{E}_s$ is the space of the smooth sections of $\Lambda(T^{*(0,1)}X)\otimes \xi$ over $X_s$.
 For $k\in\zz$, let $\mathbb{E}_s^{k}$ be the Sobolev space of order $k$ of sections of $\Lambda(T^{*(0,1)}X)\otimes \xi$ over $X_s$. Then $\mathbb{E}_s^0$ is equipped with Hermitian product $\product{\,,}_0$  in (\ref{L2 metric}). Let $\|\cdot\|_0$ be the corresponding norm of $\mathbb{E}_s^0$. For $k\in\zz$, let $\|\cdot\|_{k}$ be the Sobolev norm on $\mathbb{E}_s^{k}$, which is induced by $\nabla^{TX}$ and $\nabla^{\xi}$.

Recall that $\mathrm{Ker}(D)$ form a subbundle of $\mathbb{E}$ over $S$. 
Let $P$ be the orthogonal projection from  $\mathbb{E}^0$ to $\mathrm{Ker}(D)$, and let $P^{\perp}=1-P$.

Recall that $P^{TX}: TW\rightarrow TX $ is the projection associated to the splitting in (\ref{splitting}). For $s,s'\in \mathbb{E}$, $t\geq 1$, we set
\begin{equation}\label{sobolev infty 01}
    \begin{split}
        &|s|_{t,0}^2=\|s\|_0^2,\\ 
        &|s|_{t,1}^2=\|Ps\|_0^2+t\|P^{\perp}s\|_0^2+t\|\nabla_{P^{TX}\cdot}^{\Lambda(T^{*(0,1)}X)\otimes \xi}P^{\perp}s\|_0^2.
    \end{split}
\end{equation} 
Set
\begin{equation}\label{sobolev infty -1}
    |s|_{t,-1}=\sup_{0\neq s'\in \mathbb{E}^1}\frac{|\product{s,s'}_0|}{|s'|_{t,1}}.
\end{equation}
Then (\ref{sobolev infty 01}) and (\ref{sobolev infty -1}) define Sobolev norms on $\mathbb{E}^{1}$ and $\mathbb{E}^{-1}$, which are equivalent to $\|\cdot\|_1$ and $\|\cdot\|_{-1}$. And we have
\begin{equation}\label{infty three sobolev}
    |s|_{t,-1}\leq |s|_{t,0}\leq |s|_{t,1}.
\end{equation}

Let $\akt^{(0)}$ be the piece  of $\akt$ which has degree zero in $\Lambda(T_{\rr}^*S)$, and let $\akt^{(>0)}$ be the piece  of $\akt$ which has positive degree in $\Lambda(T_{\rr}^*S)$. 

By omitting the $dt$ part terms in \cite[(5.6) and (5.15)]{LiuMa2022}, we get
\begin{lemma}\label{elliptic estimate infty}
    There exist $C_1,...,C_5>0$, such that for any $t\geq 1$, $K\in\mathfrak{z}(g)\otimes_{{\mathbb{R}}}\mathbb{C}$,  $|K|\leq \beta$, $s,s'\in \mathbb{E}$,
    \begin{equation}
        \begin{split}
           & \mathrm{Re}\product{\akt^{(0)}s,s}_0\geq C_1|s|_{t,1}^2-C_2|s|^{2}_{t,0},\\ 
            &|\mathrm{Im}\product{\akt^{(0)}s,s}_0|\leq C_3|s|_{t,1}|s|_{t,0},\\ 
            &|\product{\akt^{(0)}s,s'}|\leq C_4 |s|_{t,1}|s'|_{t,1},\\ 
            &|\product{\akt^{(>0)}s,s'}|\leq C_5(|s|_{t,0}|s'|_{t,1}+|s|_{t,1}|s'|_{t,0}).
        \end{split}
    \end{equation}
\end{lemma}

For $c>0$, set
\begin{equation}\label{U,V,Gamma}
    \begin{split}
    U_c:=\big\{\lambda\in\cc\,: \mathrm{Re}(\lambda)\leq \frac{\mathrm{Im}(\lambda)^2}{4c^2}-c^2\big\},\\ 
    V_c:=\big\{\lambda\in\cc\,: \mathrm{Re}(\lambda)\geq  \frac{\mathrm{Im}(\lambda)^2}{4c^2}-c^2\big\},\\ 
        \Gamma_c:=\big\{\lambda\in\cc\,: \mathrm{Re}(\lambda)= \frac{\mathrm{Im}(\lambda)^2}{4c^2}-c^2\big\}.
    \end{split}
\end{equation}
By using Lemma \ref{elliptic estimate infty} and the same arguments in \cite[Theorem 9.15]{B1997}, we get
\begin{lemma}\label{resolvent infty}
    There exist  $c, C>0$, $p\in\nn$, such that for $t\geq 1$, $K\in\mathfrak{z}(g)\otimes_{\mathbb{R}}\mathbb{C}$,  $|K|\leq \beta$, $\lambda\in U_c$, the resolvent $(\lambda-\akt)^{-1}$ exists, and for $s\in \mathbb{E}$,
    \begin{equation}
        |(\lambda-\akt)^{-1}s|_{t,1}\leq C(1+|\lambda|)^{p}|s|_{t,-1}.
    \end{equation}
\end{lemma}
By omitting the $c(K^{X})$ term, we know that Lemma \ref{resolvent infty} also holds for $\lk+\bisconn_{t}^2$.
\begin{proposition}\label{prop akt-bt}
    There exists $C>0$, such that for $t\geq 1$, $K\in\mathfrak{z}(g)\otimes_{\mathbb{R}}\mathbb{C}$,  $|K|\leq \beta$, $s\in \mathbb{E}$,
    \begin{equation}\label{estimate exp akt bt}
        \big\|\big(\exp(-\akt)-\exp(-\lk-\bisconn_t^2)\big)s\big\|_0\leq \frac{C}{\sqrt{t}}\|s\|_0.
    \end{equation}
\end{proposition}
\begin{proof}
    From (\ref{bku}), (\ref{akt}), we have
    \begin{equation}\label{akt-bt}
        \akt=\lk+\bisconn_t^2+\Big[\sqrt{t}D,\frac{c(K^{X})}{2\sqrt{2t}}\Big]+\Big[\nabla^E-\frac{c(T)}{2\sqrt{2t}},\frac{c(K^X)}{2\sqrt{2t}}\Big].
    \end{equation}
    Clearly we have 
    \begin{equation}
        P\Big[\sqrt{t}D,\frac{c(K^{X})}{2\sqrt{2t}}\Big]P=0.
    \end{equation}
    So by (\ref{sobolev infty 01}), (\ref{sobolev infty -1}), for $s\in \mathbb{E}$, we have
    \begin{equation}\label{estimate D c(K)}
        \Big\|\Big[\sqrt{t}D,\frac{c(K^{X})}{2\sqrt{2t}}\Big]s\Big\|_{t,-1}\leq \frac{C}{\sqrt{t}}\|s\|_{t,1}.
    \end{equation}
 And as $t\rightarrow +\infty$, we have
    \begin{equation}
        \Big[\nabla^E-\frac{c(T)}{2\sqrt{2t}},\frac{c(K^X)}{2\sqrt{2t}}\Big]=\mathcal{O}(t^{-1/2}).
    \end{equation}
    From (\ref{akt-bt}),
    \begin{equation}\label{difference resolvent akt bt}
        \begin{split}
        &(\lambda-\akt)^{-1}-(\lambda-\lk-\bisconn_t^2)^{-1}\\ 
            =&(\lambda-\akt)^{-1}\Bigl(\Big[\sqrt{t}D,\frac{c(K^{X})}{2\sqrt{2t}}\Big]+\Big[\nabla^E-\frac{c(T)}{2\sqrt{2t}},\frac{c(K^X)}{2\sqrt{2t}}\Big]\Bigr) (\lambda-\lk-\bisconn_t^2)^{-1}.
        \end{split}
    \end{equation}
    By using Lemma \ref{resolvent infty}, (\ref{infty three sobolev}) and (\ref{estimate D c(K)})-(\ref{difference resolvent akt bt}), we get
\begin{equation}\label{estimate difference resolvent akt bt}
    \big\|\big((\lambda-\akt)^{-1}-(\lambda-\lk-\bisconn_t^2)^{-1}\big)s\big\|_{0}\leq \frac{C}{\sqrt{t}}(1+|\lambda|)^{2p}\|s\|_0.
\end{equation}
Note that we have
\begin{equation}\label{exp resolvent akt}
    \exp(-\akt)=\frac{1}{2i\pi}\int_{\Gamma_c}e^{-\lambda}(\lambda-\akt)^{-1}d\lambda,
\end{equation}
and (\ref{exp resolvent akt}) also holds for $\lk+\bisconn_t^2$. Therefore by using (\ref{estimate difference resolvent akt bt}), we get (\ref{estimate exp akt bt}).
\end{proof}

Since $W$ is compact, we can choose a finite family of smooth sections $U_1,\dots,U_r\in T_{\rr}X$ which spans $T_{\rr,x}X$ for any $x\in W$. Let $\mathcal{D}$ be the family of operators acting on $\mathbb{E}$
\begin{equation}
    \mathcal{D}=\big\{ P^{\perp}\nabla_{U_i}^{\Lambda(T^{*(0,1)}X)\otimes \xi}P^{\perp} , i=1,\dots,r \big\}.
\end{equation}
And for $k\in\nn$, let $\mathcal{D}^k$ be the family operators $Q$ acting on $E$  of the form
\begin{equation}
    Q=Q_1\dots Q_k,\qquad Q_i\in \mathcal{D}.
\end{equation}
For $k\in\nn$, we equip $\mathbb{E}^k$ with the Hilbert norm $\|\cdot\|_{k}'$ such that if $s\in \mathbb{E}$,
\begin{equation}
    \|s\|_k'=\sum_{l=0}^{k}\sum_{Q\in \mathcal{D}^l}\|Qs\|_0.
\end{equation}


For $s\in S$, $x,x'\in X_s$, let $\exp(-\akt)(x,x')$, $\exp(-\lk-\bisconn_t^2)(x,x')$ be the smooth kernels of the operators $\exp(-\akt)$, $\exp(-\lk-\bisconn_t^2)$ with respect to $dv_X/(2\pi)^{\ell}$.

From the same arguments as in \cite[Lemmas 5.6, 5.7, Proposition 5.8]{LiuMa2022}, we have the following results.
\begin{lemma}
    For $k\in\nn$, there exists constants $C_k>0$, such that for $t\geq 1$, $K\in\mathfrak{z}(g)\otimes _{\mathbb{R}}\mathbb{C}$,  $|K|\leq \beta$, $Q_1,\dots,Q_k\in \mathcal{D}$, $s,s'\in \mathbb{E},$ we have
    \begin{equation}
        \big|\big\langle [Q_1,[Q_2,\dots[Q_k,\akt]\dots]]s,s'\big\rangle_0\big|\leq C_k\|s\|_{t,1}\|s'\|_{t,1}.
    \end{equation}
\end{lemma}

\begin{lemma}
    For $k\in\nn$, there exist constants $C_k'>0$, $p_k\in\nn$, such that for $t\geq 1$, $K\in\mathfrak{z}(g)\otimes_{\mathbb{R}}\mathbb{C}$,  $|K|\leq \beta$, $\lambda\in U_c$, $s\in \mathbb{E}$, we have
    \begin{equation}
        \|(\lambda-\akt)^{-1}s\|_{k+1}'\leq C_k (1+|\lambda|)^{p_k} \|s\|_k'.
    \end{equation}
\end{lemma}

\begin{proposition}\label{prop uniform estimate akt}
    For any $k\in\nn$, there exists $C>0$ such that for  $\alpha, \beta\in\nn^{2\ell}$, $|\alpha|+|\beta|\leq k$,  $t\geq 1$, $K\in\mathfrak{z}(g)\otimes_{\mathbb{R}}\mathbb{C}$,  $|K|\leq \beta$,  $s\in S$, $x,x'\in X_s$,  we have
    \begin{equation}
        \Big|\frac{\partial^{|\alpha|+|\beta|}}{\partial x^{\alpha}\partial x'^{\beta}}\exp(-\akt)(x,x')\Big|\leq C.
    \end{equation}
\end{proposition}
By omitting the $c(K^{X})$ term, we know that Proposition \ref{prop uniform estimate akt} also holds for $\exp(-\lk-\bisconn_{t}^2)$. Then by using Propositions \ref{prop akt-bt}, \ref{prop uniform estimate akt} and the same arguments as in \cite[\S 11 p)]{BL}, we get the following theorem.
\begin{theorem}\label{thm estimate kernel akt bt}
    There are $C>0$, $\delta>0$, such that for  $t\geq 1$, $K\in\mathfrak{z}(g)\otimes_{\mathbb{R}}\mathbb{C}$,  $|K|\leq \beta$,  $s\in S$, $x,x'\in X_s$,
    \begin{equation}
        \big|\exp(-\akt)(x,x')-\exp(-\lk-\bisconn_t^2)(x,x')\big|\leq \frac{C}{t^{\delta}}.
    \end{equation}
\end{theorem}
Moreover by (\ref{definition nkt}), we see that as $t\rightarrow +\infty$,
\begin{equation}\label{nkt-nt}
    N_{K,t}-N_{t}=\mathcal{O}(t^{-1}).
\end{equation}
From Theorems \ref{thm estimate kernel akt bt} and (\ref{nkt-nt}), we get Theorem \ref{thm infty main1}.

\subsection{The asymptotics of supertraces near $0$.} Recall that by Theorem \ref{thm gK asym bku}, we see that $\Phi\str\big[g\exp(-\akt)\big]$ has a limit 
 as $t\to 0$. This is done by Liu-Ma in \cite[\S 5]{LiuMa2022} by using the local index theory and Bismut-Goette's technique in \cite[\S 7]{BismutGoette2000}. Now we prove a higher degree expansion of $\Phi\str\big[g\exp(-\akt)\big]$, this kind of expansions will be needed in order to to get the expansion \eqref{asymptotic-infty-gK} in Theorem \ref{thm asymptotic-gK} (cf. \eqref{expansion gammatgk}).

\begin{theorem}\label{thm asymptotic higher akt}
There exists $\beta>0$, $A_{r,g,K}\in P^S$ $(r\in \mathbb{N})$,
such that for $k\in \nn$, there exists $C_k>0$, for any $K\in \mathfrak{z}(g)\otimes_{\mathbb{R}}\mathbb{C}$, $|K|<\beta$, $t\in \,]0,1]$,
\begin{equation}\label{asymptotic higher akt}
    \big|\Phi\str\big[g\exp(-\akt)\big]-\sum_{r=0}^{k}t^r A_{r,g,K}\big|\leq C_k t^{k+1}.
\end{equation}
In particular, 
\begin{equation}
    A_{0,g,K}=\int_{X_g}\td_{g,K}(TX,h^{TX})\ch_{g,K}(\xi,h^{\xi}),
\end{equation}
\end{theorem}
\begin{proof}
    For $k=0$, Theorem \ref{thm asymptotic higher akt} is established in \cite[Theorem 2.2]{LiuMa2022} by omitting its $dt$ part. And Section \ref{ed} is devoted to the proof of the  existence of the expansion (\ref{asymptotic higher akt}).
\end{proof}

Note that if $K=0$, it is a classical result that the kernel of $\exp(-\akt)$ has an asymptotic expansion as $t\to 0$. And the family  local index theorem (cf. \cite{B1986}) implies that the supertraces of negative degree part in the expansion vanishes. But now we have an extra $c(K^X)/2\sqrt{2t}$ term in $\akt$, so that we do not have an asymptotic expansion for its heat kernel  in general.  Thus we need to Bismut-Goette's technique as in \cite[\S 5]{LiuMa2022} to prove Theorem \ref{thm asymptotic higher akt}.

\subsection{Localization of the problem}\label{ed}
Let $a_X$ be the inf of the injectivity radius of the fibers $X$. Let $\alpha\in \,]0,a_X/8]$, the precise value of $\alpha$ will be fixed later. The constants $C>0$, $C'>0$... appearing later may depend on the choice of $\alpha$.

Let $f: \rr\rightarrow [0,1]$ be a smooth even function such that 
\begin{equation}\label{def of f}
    \begin{split}
        f(s)=&1\kongge\mathrm{for}\kongge |s|\leq \frac{\alpha}{2},\\ 
=&0 \kongge\mathrm{for}\kongge |s|\geq  {\alpha}.
    \end{split}
\end{equation}

For $t>0$, $a\in\cc$, put 
\begin{equation}\label{e2}
    \begin{split}
        &F_t(a)=\int_{-\infty}^{+\infty}\exp(is\sqrt{2}a)\exp\left(-\frac{s^2}{2}\right)f(\sqrt{t}s)\frac{ds}{\sqrt{2\pi}},\\ 
        &I_t(a)=\int_{-\infty}^{+\infty}\exp(is\sqrt{2}a/\sqrt{t})\exp\left(-\frac{s^2}{2}\right)(1-f(\sqrt{t}s))\frac{ds}{\sqrt{2\pi}}.
    \end{split}
\end{equation}
Then $F_t(a)$, $I_t(a)$ are even holomorphic functions of $a$ such that 
\begin{equation}\label{F+G}
    \exp(-a^2)=F_t(a)+I_t(\sqrt{t}a). 
\end{equation}
Moreover $F_t$ and $I_t$ both lie in the Schwartz space $\mathcal{S}(\rr)$. 
Let $\widetilde{F}_t(a)$,  $\widetilde{I}_t(a)$ be holomorphic functions such that
\begin{equation}\label{e3}
    F_t(a)=\widetilde{F}_t(a^2),\qquad\qquad I_t(a)=\widetilde{I}_t(a^2).
\end{equation}
By (\ref{F+G}), 
\begin{equation}\label{widetilde F,G,I}
    \exp(-a)=\widetilde{F}_t(a)+\widetilde{I}_t(ta).
\end{equation}
So from (\ref{widetilde F,G,I}), we have
\begin{equation}\label{F+I(akt)}
    \exp(-\akt)=\widetilde{F}_t(\akt)+\widetilde{I}_t(t\akt).
\end{equation}

Recall that for $s\in S$, the $0$-order Sobolev space $(\mathbb{E}_s^0, \product{\,,}_0)$ is a separable Hilbert space. We set the trace class 
\begin{equation}
    \mathcal{L}_1(\mathbb{E}_0^s)=\{A\in\mathcal{L}(\mathbb{E}_0^s)\,:\trace[(A^*A)^{1/2}]<+\infty\}.
\end{equation}
For $A\in \mathcal{L}_1(\mathbb{E}_0^s)$, set
\begin{equation}
    \|A\|_{(1)}=\trace[(A^*A)^{1/2}].
\end{equation}

By omitting the $dt$ part in \cite[Lemma 5.16]{LiuMa2022}, we get the following analogue of \cite[Theorem 7.15]{BismutGoette2000}.
\begin{lemma}
    There exist $\beta>0$, $c, C>0$ such that for $t\in \,]0,1]$, $K\in \mathfrak{g}\otimes_{\mathbb{R}}\mathbb{C}$, $|K|\leq \beta$, 
    \begin{equation}\label{I(akt)}
        \|\widetilde{I}_t(t\akt)\|_{(1)}\leq c\exp(-C/t).
    \end{equation}
\end{lemma}

From (\ref{F+I(akt)}) and (\ref{I(akt)}), we see that we only need to establish the expansion (\ref{asymptotic higher akt}) for $\str\big[g\widetilde{F}_t(\akt)\big]$ in place of $\str\big[g\exp(-\akt)\big]$.

For $s\in S$, $x,x'\in X_s$, let $\ft(\akt)(x,x')$ be the smooth kernel of the operator $\ft(-\akt)$ with respect to $dv_X/(2\pi)^{\ell}$. Then,
\begin{equation}\label{integral of supertrace F}
    \str\big[g\ft(\akt)\big]=\int_{X}\str\big[g\ft(\akt)(g^{-1}x,x)\big]\frac{dv_X}{(2\pi)^{\ell}}.
\end{equation}

For $\varepsilon>0$, let $B^X(x,\varepsilon)$ be the open ball in $X_s$ with center $x$ and radius $\varepsilon$. Using the finite propagation speed  for solutions of hyperbolic equations (cf. \cite[Appendix.D.2]{MaGeorge}), we find that given any $x\in X_s$, $\ft(\akt)(x,\cdot)$ vanishes on the complement of $B^X(x,\alpha)$ in $X_s$, and depends only on the restriction of the operator $\akt$ to the ball $B^X(x,\alpha)$. Therefore, we have shown that the proof of (\ref{asymptotic higher akt}) can be made local on $X_s$.

By the above, it follows that $g\ft(\akt)(g^{-1}x,x)$ vanishes if $d^{X_s}(g^{-1}x,x)\geq \alpha$. Here $d^{X_s}$ denote the distance on $(X_s, g^{TX})$.

Now we explain our choice of $\alpha$. Given $\varepsilon>0$, let $\uneighbor$ be the $\varepsilon$-neighborhood of $X_{s,g}$ in $N_{X_g/X,\rr}$. Recall that $N_{X_g/X,\rr}$ is identified with the orthogonal bundle to $T_{\rr}W_g$ in $T_{\rr}W|_{W_g}$. There exists $\varepsilon_0\in \,]0,a_X/32]$ such that if $\varepsilon\in \,]0,16\varepsilon_0]$, the map $(x,Z)\in N_{X_g/x,\rr}\rightarrow \exp_x(Z)$ is a diffeomorphism of $\uneighbor$ on the turbular neighborhood $\vneighbor$ of $X_{s,g}$ in $X_s$. In the sequel, we identify $\uneighbor$ and $\vneighbor$, and this identification is $g$-equivariant.

We will assume that $\alpha\in \,]0,\varepsilon_0]$ is small enough so that for any $s\in S$, $x\in X_s$, if $d^{X_s}(g^{-1}x,x)<\alpha$, then $x\in \mathcal{V}_{\varepsilon_0}$.

Let $k(x,Z)$ be the smooth function on $\mathcal{U}_{\varepsilon_0}$ such that 
\begin{equation}\label{def of k}
    dv_{X}(x,Z)=k(x,Z)dv_{X_g}(x)dv_{N_{X_g/X}}(Z).
\end{equation}
In particualr, $k|_{X_g}=1$.

By (\ref{integral of supertrace F}) and the above considerations, it follows that for $\beta>0$ small enough, our proof of (\ref{asymptotic higher akt}) is localized on the $\varepsilon_0$-neighborhood $\mathcal{V}_{\varepsilon_0}$ of $X_g$. And we have
\begin{equation}\label{calculate str gftakt}
    \begin{split}
   & \str\big[g\ft(\akt)\big]=\int_{\mathcal{U}_{\varepsilon_0}}\str\big[g\ft(\akt)(g^{-1}x,x)\big]\frac{dv_X}{(2\pi)^{\ell}}\\ 
=&\int_{(x,Z)\in \mathcal{U}_{\varepsilon_0/\sqrt{t}}}t^{\mathrm{dim}N_{X_g/X}}\str\big[g\ft(\akt)(g^{-1}(x,\sqrt{t}Z),(x,\sqrt{t}Z))\big]\\ 
&\qquad\qquad\qquad\qquad\qquad\qquad \cdot k(x,\sqrt{t}Z)\frac{dv_{X_g}(x)dv_{N_{X_g/X}}(Z)}{(2\pi)^{\ell}}.    
    \end{split}
\end{equation}

By using (\ref{F+I(akt)}), (\ref{I(akt)}) and (\ref{calculate str gftakt}), the expansion (\ref{asymptotic higher akt}) is implied by the following theorem.

\begin{theorem}\label{thm ft akt expansion}
    There exist $\beta>0$, $C_k>0$, for any $k\in \nn$, $K\in \mathfrak{z}(g)\otimes_{\mathbb{R}}\mathbb{C}$, $|K|<\beta$, $t\in \,]0,1]$, there are smooth forms $a_{r,g,K}$ with $r\in \mathbb{N}$ on $W_g$ such that 
    \begin{equation}
        \begin{split}
            &\Big|\frac{t^{\mathrm{dim}N_{X_g/X}}}{(2\pi)^{\dim X_g}}\int_{Z\in N_{X_g/X,\rr}\atop |Z|\leq \varepsilon_0/\sqrt{t}}\str\big[g\ft(\akt)(g^{-1}(x,\sqrt{t}Z),(x,\sqrt{t}Z))\big]\\ 
&\qquad\qquad\qquad\qquad\quad \cdot k(x,\sqrt{t}Z)\frac{dv_{N_{X_g/X}}(Z)}{(2\pi)^{\mathrm{dim}N_{X_g/X}}}-\sum_{r=0}^k t^r a_{r,g,K}\Big|\leq C_k t^{k+1}.
        \end{split}
    \end{equation}
\end{theorem}
\begin{proof}
    Sections 6.4-6.12 are devoted to the proof of this theorem.
\end{proof}

\subsection{Lichnerowicz formula}
Let $e_1,.., e_{2\ell}$ be the local orthogonal frame of $T_{\rr}X$. If $(F,\nabla^F)$ be a vector bundle with connection on $X$, we use the notation 
\begin{equation}
    (\nabla_{e_i}^F)^2=\sum_{i=1}^{2\ell}(\nabla_{e_i}^F)^2-\nabla_{\nabla_{e_i}^{TX}e_i}^F.
\end{equation}

Let $H$ be the scalar curvature of $X$. Let $f_{1},...,f_{2n}$ be a local basis of $T_{\rr}S$, with dual basis $f^1,..., f^{2n}$. By omitting the $dt$ part of \cite[Proposition 5.18]{LiuMa2022}, we obtain the following result.

\begin{proposition}
    The following identity holds,
    \begin{equation}\label{lichnerowicz formula}
        \begin{split}
            \akt=-\frac{t}{2}\Big(\nabla_{e_i}^{\Lambda(T^{*(0,1)}X)\otimes \xi}+\frac{1}{\sqrt{2t}}\product{S(e_i)e_j, f_p^H}c(e_j)f^p\wedge\qquad\qquad\qquad\qquad\qquad\\
             +\frac{1}{2t}\product{S(e_i)f_p^H,f_q^H}f^p\wedge f^q\wedge -\frac{\product{K^X,e_i}}{2t}\Big)^2+\frac{t}{8}H\\ 
            +\frac{t}{4}\bigl(R^{\xi}+\frac{1}{2}\trace[R^{TX}]\bigr)(e_i,e_j)c(e_i)c(e_j)
            +\sqrt{\frac{t}{2}}\bigl(R^{\xi}+\frac{1}{2}\trace[R^{TX}]\bigr)(e_i,f_p^H)c(e_i)f^p\\ 
            +\frac{1}{2}\bigl(R^{\xi}+\frac{1}{2}\trace[R^{TX}]\bigr)(f_p^H,f_q^H)f^p\wedge f^q\wedge-(m^{\xi}(K)+\frac{1}{2}\trace[m^{TX}(K)]).
        \end{split}
    \end{equation}
\end{proposition}

\subsection{A local coordinate system near $X_g$.}\label{ef}
Take $x\in W_g$,  the fiberwise exponential map $Z\in T_xX\rightarrow \exp_x^X(Z)\in X$ identifies $B^{T_xX}(0,16\varepsilon_0)$ with $B^X(x,16\varepsilon_0)$. With this identification, there exists a smooth function $k_x'(Z)$, $Z\in B^{T_xX}(0,a_X/2)$ such that
\begin{equation}\label{ea}
    dv_X(Z)=k_x'(Z)dv_{TX}(Z),\qquad \mathrm{with}\kongge k_x'(0)=1.
\end{equation}
We may and we will assume that $\varepsilon_0$ is small enough so that if $Z\in T_xX, |Z|\leq 4\varepsilon_0$,
\begin{equation}\label{g uniform}
    \frac{1}{2}g_x^{TX}\leq g_Z^{TX}\leq \frac{3}{2}g_x^{TX}.
\end{equation}

Recall that $K^{\prime X}$ is the one form dual to $K^X$ with respect to $g^{TX}$.
\begin{definition}
    Let $^1\nabla^{\mathbb{E},t}$ be the connection on $\pi^*\Lambda(T_{\rr}^*S)\,\widehat{\otimes }\,\Lambda(T^{*(0,1)}X)\otimes \xi$ along the fibers,
    \begin{equation}\label{1nablat}
        \begin{split}
            ^1\nabla^{\mathbb{E},t}_{\cdot}=\nabla_{\cdot}^{\mathbb{E}}+\frac{1}{\sqrt{2t}}\product{S(\cdot)e_j, f_p^H}c(e_j)f^p\wedge
            +\frac{1}{2t}\product{S(\cdot)f_p^H,f_q^H}f^p\wedge f^q\wedge -\frac{K^{\prime X}}{2t}.
        \end{split}
    \end{equation}
\end{definition}
In the sequel, we will trivialize $\pi^*\Lambda(T_{\rr}^*S)\,\widehat{\otimes }\,\Lambda(T^{*(0,1)}X)\otimes \xi$ by parallel transport along $s\in [0,1]\rightarrow sZ$ with respect to the connection $^1\nabla^{\mathbb{E},t}$. Note that it is often more convenient first to trivialize with respect to $^1\nabla^{\Lambda(T^{*(0,1)}X)\otimes \xi,1}$, and then replace $f^p$, $f^q$ and $K$ by $\frac{f^p}{\sqrt{t}}$, $\frac{f^q}{\sqrt{t}}$ and $\frac{K}{t}$. Observe that the above connections are $g$-invairant.

By omitting the $dt$ part in \cite[(5.80)]{LiuMa2022} and by using (\ref{action on exterior algebra}), we calculate that
\begin{equation}\label{curvature 1nablat}
    \begin{split}
       ( ^1\nabla^{\mathbb{E},1})^2(e_i,e_j)=\frac{1}{4}\product{R^{TX}(e_k,e_{\ell})e_i,e_j}c(e_k)c(e_{\ell})\hspace{10 em}\\ 
       +\frac{1}{\sqrt{2}}\product{R^{TX}(e_k,f_p^H)e_i,e_j}c(e_k)f^p\wedge
       +\frac{1}{2}\product{R^{TX}(f_p^H,f_q^H)e_i,e_j}f^p\wedge f^q\wedge\\ +\big(R^{\xi}+\frac{1}{2}\trace[R^{TX}]\big)(e_i,e_j)-\product{m^{TX}(K)e_i,e_j}.
    \end{split}
\end{equation}
When $K=0$, \eqref{curvature 1nablat} is the result of \cite[Theorem 4.14]{B1986} or \cite[Theorem 11.8]{B1997}.

Let $A,B$ be smooth sections of $T_{\rr}X$. By (\ref{1nablat}),
\begin{equation}\label{1nablat c}
    ^1\nabla^{\mathbb{E},1}_A c(B)=c(\nabla_A^{TX}B)+\sqrt{2}\product{S(A)B,f_p^H}f^p\wedge.
\end{equation}
Let $c^1(T_{\rr}X)\simeq T_{\rr}X$ be the set of the elements of length $1$  in $c(T_{\rr}X)$. It follows from (\ref{1nablat c}) that parallel transport along the curve $s\in [0,1]\rightarrow sZ$ with respect to $ ^1\nabla^{\Lambda(T^{*(0,1)}X)\otimes \xi,1}$ maps $c^1(T_{\rr}X)$ into $c^1(T_{\rr}X)\oplus T_{\rr}^*S$, while leaving $\Lambda(T_{\rr}^*S)$ invariant.

\subsection{Replacing $X$ by $(T_{\rr}X)_x$}\label{ei}
Let $\gamma(s)$ be a smooth even function from $\rr$ into $[0,1]$ such that
\begin{equation}\label{dt}
    \begin{split}
        \gamma(s)=&1\kongge\mathrm{for}\kongge |s|\leq \frac{1}{2},\\ 
=&0 \kongge\mathrm{for}\kongge |s|\geq  {1}.
    \end{split}
\end{equation}
If $Z\in (T_{\rr}X)_x$, put
\begin{equation}\label{du}
    \rho(Z)=\gamma\Big(\frac{|Z|}{4\varepsilon_0}\Big).
\end{equation}

For $x\in W_g$, let $\mathbf{H}_x$ be the vector space of smooth sections of $\big(\pi^*\Lambda(T_{\rr}^*S)\widehat{\otimes} \Lambda(T^{*(0,1)}X)\otimes\xi \big)_x$ over $(T_{\rr}X)_x$. Let $\Delta^{TX}$ be the standard Laplacian on the fibers of $T_{\rr}X$.

Let $L_{x,K}^{1,t}$ be the differential operator acting on $\mathbf{H}_x$,
\begin{equation}\label{lxk1t}
    L_{x,K}^{1,t}=-(1-\rho^2(Z))\frac{t}{2}\Delta^{TX}+\rho^2(Z)\akt.
\end{equation}

Let $\ft(L_{x,K}^{1,t})(Z,Z')$ be the smooth kernel of $\ft(L_{x,K}^{1,t})$ with respect to $dv_{TX}(Z')/(2\pi)^{\ell}$. Using the finite propagation speed for solutions of hyperbolic equations (cf. \cite[Appendix D.2]{MaGeorge}), we find that if $Z\in (N_{X_g/X,\rr})_x$, $|Z|\leq \varepsilon_0$, then
\begin{equation}\label{finite propagation ftakt=ftl}
    \ft(\akt)(g^{-1}Z,Z)k'_x(Z)=\ft(L_{x,K}^{1,t})(g^{-1}Z,Z).
\end{equation}
Thus in our proof of Theorem \ref{thm ft akt expansion}, we can replace $\akt$ by $L_{x,t}^{1,t}$.

\subsection{The Getzler rescaling.}\label{sec getzler rescaling}
Let $\mathrm{Op}_x$ be the set of scalar differential operators acting on $\mathbf{H}_x$. Then
\begin{equation}\label{lxk1t belong}
    L_{x,K}^{1,t}\in \big(\pi^*\Lambda(T_{\rr}^*S)\big)_x\,\widehat{\otimes }\, \ENDO\big((\Lambda(T^{*(0,1)}X)\otimes\xi)_x\big)\otimes \mathrm{Op}_x.
\end{equation}

For $t>0$, let $H_t: \mathbf{H}_x\rightarrow \mathbf{H}_x$ be the linear map
\begin{equation}\label{def of ht}
    H_th(Z)=h(Z/\sqrt{t}).
\end{equation}
For $t>0$, let $L_{x,K}^{2,t}$ be the differential operator acting on $\mathbf{H}_{x}$ defined by 
\begin{equation}\label{lxk2t}
    L_{x,K}^{2,t}=H_t^{-1}L_{x,K}^{1,t}H_t.
\end{equation}

By (\ref{lxk1t belong}), and since $\ENDO(\Lambda(T^{*(0,1)}X))=c(T_{\rr}X)\otimes_{\rr}\cc$,
\begin{equation}\label{lxk2t belong}
    L_{x,K}^{2,t}\in \bigl(\pi^*\Lambda(T_{\rr}^*S)_x \,\widehat{\otimes}\,  c(T_{\rr}X)\otimes \ENDO(\xi)\bigr)\otimes \mathrm{Op}_x.
\end{equation}

Recall that $\dim X=\ell$, $\dim X_g=\ell'$. Let $e_1,...,e_{2\ell'}$ be an orthonormal oriented basis of $(T_{\rr}X_g)_x$, let $e_{2\ell'+1},...,e_{2\ell }$ be an orthonormal oriented basis of $(N_{X_g/X,\rr})_x$, so that $e_1,...,e_{2\ell}$ is an orthonormal oriented basis of $(T_{\rr}X)_x$. We denote with an upper script the corresponding dual basis of $(T_{\rr}^*X)_x$.

For $1\leq j\leq 2\ell'$, the operators $e^j\wedge$ and $i_{e_j}$ act as odd operators on $\Lambda(T_{\rr}^*X_g)_x$.
\begin{definition}\label{def getzler rescaling}
    For $t>0$, put
    \begin{equation}\label{getzler rescaling}
        c_t(e_j)=\sqrt{\frac{2}{t}}e^j\wedge-\sqrt{\frac{t}{2}}i_{e_j},\kongge 1\leq j\leq 2l'.
    \end{equation}
\end{definition}

Let $L_{x,K}^{3,t}$ be th operator obtained from $L_{x,K}^{2,t}$ by replacing $c(e_j)$ by $c_t(e_j)$ for $1\leq j\leq 2\ell'$ while leaving the $c(e_j)$'s unchanged for $2\ell'+1\leq j\leq 2\ell$. 

From (\ref{lxk2t belong}), (\ref{getzler rescaling}),
\begin{equation}
    L_{x,K}^{3,t}\in \Bigl(\pi^*\Lambda(T_{\rr}^*S)\,\widehat{\otimes }\,\ENDO\big(\Lambda(T_{\rr}^*X_g)\,\widehat{\otimes }\, \Lambda(N_{X_g/X,\rr}^{*}) \otimes\xi\big)\Bigr)_x\otimes \mathrm{Op}_x.
\end{equation}

Let $\ft(L_{x,K}^{3,t})(Z,Z')$ be the smooth kenel associated to $\ft(L_{x,K}^{3,t})$ with respect to $dv_{TX}(Z')/(2\pi)^{\ell}$. 
We  write $\ft(L_{x,K}^{3,t})(Z,Z')$ in the form
\begin{equation}
    \begin{split}
    \ft(L_{x,K}^{3,t})(Z,Z')=\sum e^{i_1}\wedge...\wedge\, e^{i_p}i_{e_{j_1}}...i_{e_{j_q}}\widetilde{F}_{t,i_1,...,i_p}^{j_1,...,j_q}(Z,Z'),\qquad\qquad\quad\\ 
        1\leq i_1<...<i_p\leq 2\ell',\quad 1\leq j_1<...<j_q\leq 2\ell'.
    \end{split}
\end{equation}
where $\widetilde{F}_{t,i_1,...,i_p}^{j_1,...,j_q}(Z,Z')\in \pi^*\Lambda(T_{\rr}^*S)\,\widehat{\otimes }\, \bigl(c(N_{X_g/X,\rr})\otimes\ENDO(\xi)\bigr)_x$. Then we put
\begin{equation}
    [\ft(L_{x,K}^{3,t})(Z,Z')]^{\mathrm{max}}=\widetilde{F}_{t,1,2,...,2\ell'}.
\end{equation}
As an analogue of \cite[Proposition 7.25]{BismutGoette2000}, we have the following result.
\begin{proposition}
    For $Z\in (T_{\rr}X)_x$, $|Z|\leq \varepsilon_0/\sqrt{t}$, 
    \begin{equation}\label{getzler rescaling str}
        \begin{split}
        t^{\mathrm{dim}N_{X_g/X}}\str\big[g\ft(\akt)(g^{-1}(\sqrt{t}Z),\sqrt{t}Z)\big]k_x'(\sqrt{t}Z)=\qquad\qquad\qquad\\ 
        (-i)^{\ell'}\str^{\Lambda(N_{X_g/X}^*)\otimes \xi}\big[g\ft(L_{x,K}^{3,t})(g^{-1}Z,Z)\big]^{\mathrm{max}}.
        \end{split}
    \end{equation}
\end{proposition}

For $A\in\bigl(\pi^*\Lambda(T_{\rr}^*S)\,\widehat{\otimes }\, c(TX)\otimes \ENDO(\xi) \bigr)_x\otimes \mathrm{Op}_x$, we denote by $[A]_t^{(3)}$ the differential operator obtained from $A$ by using Getzler rescaling  in Definition \ref{def getzler rescaling}.

Let $\tau e_j(Z)$ be the parallel transport of $e_j$ along the curve $t\in [0,1]\rightarrow tZ$ with respect to the connection $\nabla^{TX}$. By (\ref{1nablat c}), in the trivialization associated with $ ^1\nabla^{\mathbb{E},t}$,
\begin{equation}\label{ctauej}
    c(\tau e_j(Z))=c(e_j)+\sqrt{\frac{2}{t}}\bigl(\product{S(Z)e_j,f_p^H}+\mathcal{O}(|Z|^2)\bigr)f^p\wedge.
\end{equation}
From (\ref{ctauej}), for $1\leq j\leq 2\ell'$,
\begin{equation}\label{ct explicit 1}
    \big[c(\tau e_j(\sqrt{t}Z))\big]_t^{(3)}=c_t(e_j)+\sqrt{{2}}\bigl(\product{S(Z)e_j,f_p^H}+\mathcal{O}(\sqrt{t}|Z|^2)\bigr)f^p\wedge,
\end{equation}
for $2\ell'+1\leq j\leq 2\ell$,
\begin{equation}\label{ct explicit 2}
    \big[c(\tau e_j(\sqrt{t}Z))\big]_t^{(3)}=c(e_j)+\sqrt{{2}}\bigl(\product{S(Z)e_j,f_p^H}+\mathcal{O}(\sqrt{t}|Z|^2)\bigr)f^p\wedge.
\end{equation}

From (\ref{lichnerowicz formula}), (\ref{lxk1t}), (\ref{lxk2t}) and (\ref{getzler rescaling}), we calculate that
\begin{equation}\label{explicit Lxk2t}
    \begin{split}
        L_{x,K}^{3,t}=-\frac{1}{2}(1-\rho^2(\sqrt{t}Z))\Delta^{TX}+\rho^2(\sqrt{t}Z)\Big\{-\frac{g^{ij}(\sqrt{t}Z)}{2}\bigl(\nabla_{e_i}'\nabla_{e_j}'-\Gamma_{ij}^k(\sqrt{t}Z)\sqrt{t}\nabla_{e_k}' \bigr)\\ +\frac{t}{8}H_{\sqrt{t}Z} 
        +\frac{t}{4}\bigl(R_{\sqrt{t}Z}^{\xi}+\frac{1}{2}\trace[R_{\sqrt{t}Z}^{TX}]\bigr)(\tau e_i,\tau e_j)\big[c(\tau e_i(\sqrt{t}Z))c(\tau e_j(\sqrt{t}Z))\big]_t^{(3)}\\ 
        +\sqrt{\frac{t}{2}}\bigl(R_{\sqrt{t}Z}^{\xi}+\frac{1}{2}\trace[R_{\sqrt{t}Z}^{TX}]\bigr)(\tau e_i,f_p^H)\big[c(\tau e_i(\sqrt{t}Z))\big]_t^{(3)}f^p\wedge\qquad\qquad \\ 
        +\frac{1}{2}\bigl(R_{\sqrt{t}Z}^{\xi}+\frac{1}{2}\trace[R_{\sqrt{t}Z}^{TX}]\bigr)(f_p^H,f_q^H)f^p\wedge f^q\wedge 
        -(m_{\sqrt{t}Z}^{\xi}(K)+\frac{1}{2}\trace[m_{\sqrt{t}Z}^{TX}(K)])\Big\},
    \end{split}
\end{equation}
where $(g^{ij}(Z))$ is the inverse matrix of $(g_{ij}(Z))=(\product{e_i,e_j}_Z$), $(\nabla_{e_i}^{TX}e_j)_Z=\Gamma_{ij}^k(Z)e_k$ and $\nabla_{e_i}'=[^1\nabla_{e_i}^{\Lambda(T^{*(0,1)}X)\otimes \xi,t}]_t^{(3)}$ is computed by
\begin{equation}\label{explicit nabla ei'}
    \begin{split}
        \nabla_{e_i}'=\nabla_{e_i}+\frac{t}{8}\Bigl(\product{R_{x}^{TX}(e_k,e_{\ell})Z,e_i}+\mathcal{O}(\sqrt{t}|Z|^2)\Bigr)\big[c(
        \tau e_k(\sqrt{t}Z))c(\tau e_{\ell}(\sqrt{t}Z))\big]_t^{(3)}\\
        +\frac{\sqrt{t}}{2\sqrt{2}}\Bigl(\product{R_{x}^{TX}(e_j,f_p^{H})Z,e_i}+\mathcal{O}(\sqrt{t}|Z|^2)\Bigr)\big[c(
            \tau e_j(\sqrt{t}Z))\big]_t^{(3)}f^p\wedge\\ 
            +\frac{1}{4}\Bigl(\product{R_{x}^{TX}(f_p^H,f_q^H)Z,e_i}+\mathcal{O}(\sqrt{t}|Z|^2)\Bigr)f^p\wedge f^q\wedge\qquad\qquad\qquad\\ 
            +\frac{1}{2}\bigl(R^{\xi}(e_i,e_j)+\frac{1}{2}\trace[R^{TX}]+\mathcal{O}(\sqrt{t}|Z|^2)\bigr)-\frac{1}{2}\product{m^{TX}(K)e_i,e_j}+\frac{1}{\sqrt{t}}h_i(K,\sqrt{t}Z).
    \end{split}
\end{equation}
Here the smooth function $h_i(K,Z)$ depends linearly on $K\in\mathfrak{g}$ and $h_i(K,Z)=\mathcal{O}(|Z|^2)$ for $|K|\leq 1$. And the terms  $\mathcal{O}(\sqrt{t}|Z|^2)$ in (\ref{explicit nabla ei'}) are of the form $\frac{1}{\sqrt{t}}f(\sqrt{t}Z)$ for a smooth function $f(Z)=\mathcal{O}(|Z|^2)$.

Now we introduce a few notations to study explicitly the expansion of $L_{x,K}^{3,t}$ at $t=0$.
We denote by $\cc[\mathfrak{g}]$  the space of polynomials on the Lie algebra $\mathfrak{g}$ with complex coefficients. Let $\nabla_{e_i}$ be the ordinary differentiation operator on $TX$ in the direction $e_i$.

\begin{definition}\label{dr}
    For $r\in \mathbb{N}$,
we denote by $\Op_{r}$ the space of operators in the form 
\begin{align}\label{e4}
\mathcal{A}_{i,j,r}\nabla_{e_{i}}\nabla_{e_{j}}+\mathcal{B}_{\ell,r}\nabla_{e_{\ell}}+\mathcal{C}_{r}, 
\end{align}
where $\mathcal{A}_{i,j,r}$ is a homogeneous polynomial in $Z$ of degree $r$,  $\mathcal{B}_{\ell, r}$ (resp. $\mathcal{C}_{r}$) is a linear combination of 
\begin{equation}\label{Op}
    \begin{split}
    & Z^{\alpha}f^{k_1}\wedge...\wedge f^{k_s}\wedge e^{i_1}\wedge...\wedge e^{i_p}\, i_{e_{j_1}}...i_{e_{j_q}}\, c(e_{\ell_1})...c(e_{\ell_u}) \otimes \Omega\otimes \phi\, \\
     &\quad\in \pi^*\Lambda(T_{\rr}^*S)_x\,\widehat{\otimes }\,\ENDO(\Lambda(T_{\rr}^*X_g))_x\,\widehat{\otimes }\, \bigl(c(N_{X_g/X,\rr})\otimes\ENDO(\xi)\bigr)_x\otimes \cc[\mathfrak{g}]\\
     &   1\leq i_1<...<i_p\leq 2\ell',\quad 1\leq j_1<...<j_q\leq 2\ell',\quad 2\ell'+1\leq \ell_1<...<\ell_u\leq 2\ell,
    \end{split}
\end{equation}
$\Omega\in \ENDO(\xi)_{x}$, $\phi\in \mathbb{C}[\mathfrak{g}]$, such that $|\alpha|\leq r+1$, (resp. $\mathcal{C}_r$ is $\leq r+2$), $s+p+q+u\leq 2$ (resp. $\mathcal{C}_r$ is $\leq 4$) and $|\alpha|+u$ has the same parity as $r-1$ (resp. $r$).

\end{definition}

Let $\jmath: W_g\rightarrow W$ be the nature embedding.
\begin{definition}
    Let $L_{x,K}^{3,0}$ be the operator in $\bigl(\pi^*\Lambda(T_{\rr}^*S)\,\widehat{\otimes }\,\Lambda(T_{\rr}^*X_g)\,\widehat{\otimes }\, c(N_{X_g/X,\rr}) \otimes \ENDO(\xi)\bigr)_x\otimes \mathrm{Op}_x$ given by
    \begin{equation}\label{def of lxk30}
        \begin{split}
        L_{x,K}^{3,0}=-\frac{1}{2}\Bigl(\nabla_{e_i}+\frac{1}{2}\product{(\jmath^*R^{TX}-m^{TX}(K))Z,e_i}\Bigr)^2+\jmath^*R_x^{\xi}-m^{\xi}(K)_x\\ 
        +\frac{1}{2}\trace\big[\jmath^*R^{TX}-m^{TX}(K)\big].
        \end{split}
    \end{equation}
\end{definition}

\begin{theorem}\label{thm 1xk3t expansion}
    There exist $\mathcal{L}_r\in \Op_r$ $(r\in \mathbb{N})$ such that for any $m\in\nn$,
\begin{equation}\label{lxk3t expansion}
    L_{x,K}^{3,t}=\sum_{r=0}^{m}\frac{1}{r!}t^{r/2}\mathcal{L}_r(K)+R_{m+1}(t,Z)(K),
\end{equation}
with $\mathcal{L}_0(K)=L_{x,K}^{3,0}$.
 Let $t=u^2$, for any $k\in\nn$,  $|K|\leq 1$, $u\in \,]0,1]$, the coefficients of $\frac{\partial^k}{\partial u^k}R_{m+1}(t,Z)(K)$ together with their derivatives in $Z$ of order $\leq p$ are uniformly bounded by $C_{k,m,p} t^{\max\{m+1-k,0\}/2}(1+|Z|)^{m+3+k}$.
\end{theorem}
\begin{proof}

    By using (\ref{ctauej})-(\ref{explicit nabla ei'}), we see that $L_{x,K}^{3,t}$ consists of two sorts of terms,
    \begin{equation}\label{ti/2drho}
        t^{i/2}d(\sqrt{t}Z)\rho^2(\sqrt{t}Z)\omega(K),\qquad (1-\rho^2(\sqrt{t}Z))\Delta^{TX},
    \end{equation} 
    where $\omega\in\Op_i$, and $d(Z)\rho^2(Z)$ is a smooth function which is uniformly bounded together with its derivatives in $Z$. 
    
   By using  the  Taylor expansion of  $d(\sqrt{t}Z)\rho^2(\sqrt{t}Z)$, we can write
\begin{equation}\label{lxk3t expansion first}
    t^{i/2}d(\sqrt{t}Z)\rho^2(\sqrt{t}Z)\omega=\sum_{r=0}^{m}t^{r/2}D_r+r_{m+1}(t,Z),
\end{equation}
    where $D_r\in \Op_r$. Let $t=u^2$, for any $k\in\nn$, the coefficients of $\frac{\partial^k}{\partial u^k}r_{m+1}(t,Z)$  together with their derivatives in $Z$ of order $\leq p$ are uniformly bounded by  $C_{k,m,p}'t^{\max\{m+1-k,0\}/2}(1+|Z|)^{m+3+k}$ for $|K|\leq 1$, $u\in \,]0,1]$.

From (\ref{du}), we see that for any $n\in\nn$, the function $(1-\rho^{2}(\sqrt{t}Z))$ is uniformly bounded by $(\sqrt{t}|Z|/2\varepsilon_0)^n$. By  (\ref{ti/2drho}) and (\ref{lxk3t expansion first}), we obtain (\ref{lxk3t expansion}).
\end{proof}

\subsection{A family of norms}
For $x\in W_g$, let $\mathbf{I}_x$ be the vector space of smooth sections of $\bigl(\pi^*\Lambda(T_{\rr}^*S)\,\widehat{\otimes }\,\Lambda(T_{\rr}^*X_g)\,\widehat{\otimes }\, \Lambda(N_{X_g/X}^{*(0,1)}) \otimes  \xi\bigr)_x$ on $(T_{\rr}X)_x$, let $\mathbf{I}_{(r,q),x}$ be the  space of smooth sections of $\bigl(\pi^*\Lambda^r(T_{\rr}^*S)\,\widehat{\otimes }\, \Lambda^q(T_{\rr}^*X_g)\,\widehat{\otimes }\, \Lambda(N_{X_g/X}^{*(0,1)}) \otimes \xi\bigr)_x$. We denote by $\mathbf{I}_x^0=\oplus_{r,q}\mathbf{I}_{(r,q),x}^0$ the corresponding vector space of square integrable sections. 
\begin{definition}
    If $s\in \mathbf{I}_{(r,q),x}$ has compact support, put
    \begin{equation}\label{a family of norms}
        |s|_{t,x,0}^2=\int_{(T_{\rr}X)_x}|s|^2\Bigl(1+|Z|\rho\big(\frac{\sqrt{t}Z}{2}\big)\Bigr)^{2(2\ell'+2\mathrm{dim}S-r-q)}dv_{TX}(Z).
    \end{equation}
    Then $(\mathbf{I}_{x}^0, |\cdot|_{t,x,0})$ form a Hilbert spcae, and let $\product{\kongge,\kongge}_{t,x,0}$ be the corresponding Hermitian product.
\end{definition}

\begin{definition}
    If $s\in \mathbf{I}_x$ has compact support, put
    \begin{equation}
        \begin{split}
        &|s|_{t,x,1}^2=|s|_{t,x,0}^2+\sum_{1}^{2\ell}|\nabla_{e_i}s|_{t,x,0}^2,\\ 
          &  |s|_{t,x,-1}=\sup_{0\neq s'\in\mathbf{I}_x}\frac{|\product{s,s'}_{t,x,0}|}{|s'|_{t,x,1}}.
        \end{split}
    \end{equation}
\end{definition}

Let $(\mathbf{I}_x^1,|\cdot|_{t,x,1})$ be the Hilbert closure of the above vector spaces with respect to $|\cdot|_{t,x,1}$. Let $(\mathbf{I}_{x}^{-1},|\cdot|_{t,x,-1})$ be the antidual of $(\mathbf{I}_{x}^{1},|\cdot|_{t,x,1})$. Then $(\mathbf{I}_{x}^1,|\cdot|_{t,x,1})$ and $(\mathbf{I}_{x}^{0},|\cdot|_{t,x,0})$ are densely embedded in $(\mathbf{I}_{x}^{0},|\cdot|_{t,x,0})$ and $(\mathbf{I}_{x}^{-1},|\cdot|_{t,x,-1})$ with norms smaller than $1$ respectively.

We recall the following estimate in \cite[Theorem 7.31]{BismutGoette2000}, \cite[Lemma 5.27]{LiuMa2022}.
\begin{lemma}\label{lemma elliptic estimate lxk3t}
    There exist constants $C_i>0$, $i=1,2,3,4$, such that if $t\in \,]0,1]$, $K\in\mathfrak{g}\otimes \mathbb{C}$, $|K|\leq 1$, if $n\in \nn$,  if the support of $s,s'\in\mathbf{I}_{x}$ is included in $\{Z\in (T_{\rr}X)_{x}:|Z|\leq n\}$, then
\begin{equation}
    \begin{split}
        &\mathrm{Re}\product{L_{x,K}^{3,t}s,s}_{t,x,0}\geq C_1|s|_{t,x,1}^2-C_2(1+|nK|^2)|s|_{t,x,0}^2,\\ 
        &|\mathrm{Im}\product{L_{x,K}^{3,t}s,s}_{t,x,0}|\leq C_3\bigl((1+|nK|)|s|_{t,x,1}|s|_{t,x,0}+|nK|^2|s|_{t,x,0}^2\bigr),\\ 
        &|\product{L_{x,K}^{3,t}s,s'}_{t,x,0}|\leq C_4(1+|nK|^2)|s|_{t,x,1}|s'|_{t,x,1}.
    \end{split}
\end{equation}
\end{lemma}

\subsection{The kernel $\ft(L_{x,K}^{3,t})$ as an infinite sum}
Let $h:\rr\rightarrow [0,1]$ be a smooth even function such that
\begin{equation}
    \begin{split}
       h(s)=&1\kongge\mathrm{for}\kongge |s|\leq \frac{1}{2},\\ 
=&0 \kongge\mathrm{for}\kongge |s|\geq  {1}.
    \end{split}
\end{equation}

For $n\in\nn$, put 
\begin{equation}
    h_n(u)=h(u+n/2)+h(u-n/2).
\end{equation}
Then $h_n$ is a smooth even function whose support is included in $[-\frac{n}{2}-1,-\frac{n}{2}+1]\cup [\frac{n}{2}-1,\frac{n}{2}+1]$. 
Set
\begin{equation}
    H(u)=\sum_{n\in\nn}h_n(u).
\end{equation}
The above sum is locally finite, and $H(u)$ is a bounded smooth even function which takes positive values and has a positive values and has a positive lower bound on $\rr$. Put
\begin{equation}
    k_n(u)=\frac{h_u}{H}(u).
\end{equation}
Then the $k_n$ are bounded even smooth functions with bounded derivatives, and moreover 
\begin{equation}\label{sum kn}
    \sum_{n\in\nn}k_n=1.
\end{equation}
\begin{definition}
    For $t\in [0,1]$, $n\in\nn$, $a\in\cc$, put
    \begin{equation}\label{def of ftn}
        F_{t,n}(a)=\int_{-\infty}^{+\infty}\exp(is\sqrt{2}a)\exp(-\frac{s^2}{2})f(\sqrt{t}s)k_n(s)\frac{ds}{\sqrt{2\pi}}.
    \end{equation}
\end{definition}
By (\ref{sum kn}),
\begin{equation}\label{sum ftn}
    F_t(a)=\sum_{n\in\nn}F_{t,n}(a).
\end{equation}
Also given $m,m'\in\nn$, there exist $C>0$, $C'>0$, $C''>0$ such that for any $n\in\nn$, $c>0$,
\begin{equation}\label{ftn estimate}
    \sup_{|\mathrm{Im}(a)|\leq c}|a|^{m}|F_{t,n}^{(m')}(a)|\leq C\exp(-C'n^2+C'' c^2).
\end{equation}
Let $\widetilde{F}_{t,n}(a)$ be the unique holomorphic function such that 
\begin{equation}
    F_{t,n}(a)=\widetilde{F}_{t,n}(a^2).
\end{equation}
Recall that $V_c$ is defined in (\ref{U,V,Gamma}), we note that $V_c$ are the images of $\{\lambda\in\cc:|\mathrm{Im}(\lambda)|\leq c\}$ by the map $\lambda\rightarrow\lambda^2$. So by (\ref{ftn estimate}),  given $m,m'\in\nn$, there exist $C>0$, $C'>0$, $C''>0$ such that for any $n\in\nn$, $c>0$, if $\lambda\in V_c$,
\begin{equation}\label{estimate ftn}
    |\lambda|^{m}|\widetilde{F}_{t,n}^{(m')}(\lambda)|\leq C\exp(-C'n^2+C'' c^2).
\end{equation}
By (\ref{sum ftn}),
\begin{equation}\label{sum widetilde ftn}
    \widetilde{F}_t(a)=\sum_{n\in\nn}\widetilde{F}_{t,n}(a).
\end{equation}

Using (\ref{sum widetilde ftn}), we get
\begin{equation}\label{as infinite sum}
    \widetilde{F}_t(L_{x,K}^{3,t})=\sum_{n\in\nn}\widetilde{F}_{t,n}(L_{x,K}^{3,t}).
\end{equation}
More precisely, by (\ref{estimate ftn}) and using the standard elliptic estimates, given $t\in \,]0,1]$, we have the identity
\begin{equation}\label{sum ftn kernel}
    \widetilde{F}_t(L_{x,K}^{3,t})(Z,Z')=\sum_{n\in\nn}\widetilde{F}_{t,n}(L_{x,K}^{3,t})(Z,Z')
\end{equation}
and the series in the right-hand side of (\ref{sum ftn kernel}) converges uniformly together with its derivatives on the compact sets in $T_xX\times T_xX$.
\begin{definition}
 For $t\geq 0$, put
 \begin{equation}\label{lxkn3t}
    L_{x,K,n}^{3,t}=-\Bigl(1-\gamma^2\Bigl(\frac{|Z|}{2(n+2)}\Bigr)\Bigr)\frac{\Delta^{TX}}{2}+\gamma^2\Bigl(\frac{|Z|}{2(n+2)}\Bigr)L_{x,K}^{3,t}.
 \end{equation}
\end{definition}

Observe that if $k_n(s)\neq 0$, then $|s|\leq \frac{n}{2}+1$. Using finite propagation speed and (\ref{g uniform}), we find there is $C>0$ such that if $Z\in (T_{\rr}X)_x$, the support of $\widetilde{F}_{t,n}(L_{x,K}^{3,t})(Z,\cdot)$ is included in $\{Z'\in (T_{\rr}X)_x, |Z'-Z|\leq 2n+2\}$. Therefore, given $p\in\nn$, if $Z\in (T_{\rr}X)_x$, $|Z|\leq p$, the support of $\widetilde{F}_{t,n}(L_{x,K}^{3,t})(Z,\cdot)$ is included in $\{Z'\in (T_{\rr}X)_x, |Z'|\leq n+p+2\}$.

If $|Z|\leq n+p+2$, then $\gamma(|Z|/2(n+p+2))=1$. Using \eqref{e2}, \eqref{e3} and the finite propagation speed for solutions of hyperbolic equations, we see that if $Z\in (T_{\rr}X)_{x}$, $|Z|\leq p$, then 
\begin{equation}\label{finite propagation n+p}
    \widetilde{F}_{t,n}(L_{x,K}^{3,t})(Z,Z')=\widetilde{F}_{t,n}(L_{x,K,n+p}^{3,t})(Z,Z').
\end{equation}

\subsection{Estimates on the resolvent of $L_{x,K,n}^{3,t}$.}
From Lemma \ref{lemma elliptic estimate lxk3t}, we see that if $s, s'\in \mathbf{I}_x$ have compact support,
\begin{equation}\label{elliptic estimate lxkn3t}
    \begin{split}
        &\mathrm{Re}\product{L_{x,K,n}^{3,t}s,s}_{t,x,0}\geq C_1|s|_{t,x,1}^2-C_2(1+|nK|^2)|s|_{t,x,0}^2,\\ 
        &|\mathrm{Im}\product{L_{x,K,n}^{3,t}s,s}_{t,x,0}|\leq C_3\bigl((1+|nK|)|s|_{t,x,1}|s|_{t,x,0}+|nK|^2|s|_{t,x,0}^2\bigr),\\ 
        &|\product{L_{x,K,n}^{3,t}s,s'}_{t,x,0}|\leq C_4(1+|nK|^2)|s|_{t,x,1}|s'|_{t,x,1}.
    \end{split}
\end{equation}

Let $\mathcal{Q}_x$ be the family of operators 
\begin{equation}
    \mathcal{Q}_x=\{\nabla_{e_i}, 1\leq i\leq 2\ell\}.
\end{equation}
For $j\in\nn$, let $\mathcal{Q}_x^j$ be the set of operators  $Q_1\dots Q_j$ with $Q_i\in\mathcal{Q}_x$, $1\leq i\leq j.$
If $s\in \mathbf{I}_x$ has compact support, put 
\begin{equation}
    \|s\|_{t,x,k}^2=\sum_{j=0}^{k}\sum_{Q\in \mathcal{Q}_x^j}|Qs|_{t,x,0}^2.
\end{equation}
Let $(\mathbf{I}_x^k,\|\cdot\|_{t,x,k})$ be the corresponding closure with respect to the norm $\|\cdot\|_{t,x,k}$. If $k,k'\in\nn$ and if $A\in \mathcal{L}(\mathbf{I}_x^{k},\mathbf{I}_x^{k'})$, let $\|A\|_{t,x}^{k,k'}$ be the norm of $A$ with respect to the norms $\|\cdot\|_{t,x,k}$, $\|\cdot\|_{t,x,k'}$.

By using (\ref{elliptic estimate lxkn3t}), we follow exactly in the same way as in proofs of \cite[Theorems 7.34, 7.37]{BismutGoette2000}, we obtain the following result.
\begin{theorem}\label{thm estimate of resolvent first}
    Given $\eta>0$, there exist $c_{\eta}\in \,]0,1]$, $C>0$, $d>0$ such that if $t\in \,]0,1]$, $K\in \mathfrak{g}\otimes_{\mathbb{R}}\mathbb{C}$, $|K|\leq c_{\eta}$, $n\in\nn$, $x\in W_g$, $\lambda\in U_{\eta n+d}$ , the resolvent $(\lambda-L_{x,K,n}^{3,t})^{-1}$ exists, and moreover
    \begin{equation}\label{resolvent lxkn3t estimate first}
        \begin{split}
        &\|(\lambda-L_{x,K,n}^{3,t})^{-1}\|_{t,x}^{0,0}\leq \frac{4}{(\eta n+d)^2},\\ 
         &   \|(\lambda-L_{x,K,n}^{3,t})^{-1}\|_{t,x}^{-1,1}\leq C(1+|n|^2)(1+|\lambda|^2).
        \end{split}
    \end{equation} 
\end{theorem}

\begin{theorem}
    Given $\eta>0$, $k\in\nn$, there exist $m_k\in\nn$, $C_k>0$ such that such if $t\in \,]0,1]$, $K\in \mathfrak{g}\otimes_{\mathbb{R}}\mathbb{C}$, $|K|\leq c_{\eta}$, $n\in\nn$, $x\in W_g$, $\lambda\in U_{\eta n+d}$ , the resolvent $(\lambda-L_{x,K,n}^{3,t})^{-1}$ maps $\mathbf{I}_x^{k}$ into $\mathbf{I}_x^{k+1}$, and moreover
    \begin{equation}\label{resolvent lxkn3t estimate high}
        \begin{split}
        &\|(\lambda-L_{x,K,n}^{3,t})^{-1}\|_{t,x}^{k,k+1}\leq C_k\bigl((1+n^2)(1+|\lambda|)\bigr)^{m_k}.
        \end{split}
    \end{equation} 
\end{theorem}

For $t=u^2$, we now proceed as in \cite[\S 4.1.4]{MaGeorge} to obtain the estimates of $\frac{\partial^{r}}{\partial u^{r}}(\lambda-L_{x,K,n}^{3,t})^{-1}$.
For $k,r\in \nn^*$, let
\begin{equation}
    I_{{k},{r}}=\Big\{(\mathbf{k},\mathbf{r})=(k_i,r_i):\sum_{i=0}^j k_i=k+j,\quad \sum_{i=1}^j r_i=r,\quad j,k_i,r_i\in\nn^*\Big\}.
\end{equation}

For $(\mathbf{k},\mathbf{r})\in I_{k,r}$, $\lambda\in U_{\eta n +d}$, $t=u^2\in \,]0,1]$, set 
\begin{equation}\label{arnklambdat}
    A_{\mathbf{r},n}^{\bf{k}}(\lambda,t)=(\lambda-L_{x,K,n}^{3,t})^{-k_0}\frac{\partial^{r_1}L_{x,K,n}^{3,t}}{\partial u^{r_1}}(\lambda-L_{x,K,n}^{3,t})^{-k_1}...\frac{\partial^{r_j}L_{x,K,n}^{3,t}}{\partial u^{r_j}}(\lambda-L_{x,K,n}^{3,t})^{-k_j}.
\end{equation}
Then there exist $a_{\bf{r}}^{\bf{k}}\in\rr$ such that
\begin{equation}\label{sum of arnklambdat}
    \frac{\partial^{r}}{\partial u^{r}}(\lambda-L_{x,K,n}^{3,t})^{-k}=\sum_{(\bf{k},\bf{r})\in I_{k,r}}a_{\bf{r}}^{\bf{k}} A_{\mathbf{r},n}^{\bf{k}}(\lambda,t).
\end{equation}

Then we have the following analogue of \cite[Theorem 4.1.13]{MaGeorge}.
\begin{theorem}\label{thm QArknQ'}
    For any $m>0$, $k\geq m+2r$, $(\bf{k},\bf{r})\in I_{k,r}$, there exist $C_k'>0$, $m_k'>0$ such that for $\lambda\in U_{\eta n+d}$, $t\in \,]0,1]$, $K\in \mathfrak{g}\otimes_{\mathbb{R}}\mathbb{C}$, $|K|\leq c_{\eta}$, $n\in\nn$, $Q\in \mathcal{Q}_x^m$, 
\begin{equation}\label{e7}
    \|QA_{\mathbf{r},n}^{\bf{k}}(\lambda,t)\|_{t,x}^{0,0}\leq C_k'\bigl((1+n^2)(1+|\lambda|)\bigr)^{m_k'}.
\end{equation} 
\end{theorem}

\begin{proof}
    From (\ref{a family of norms}), we see that the following operators acting on $(\mathbf{I}_x^0,|\cdot|_{t,x,0})$ are uniformly bounded :
    \begin{equation}\label{3operators bound}
       e^j\wedge, \kongge\kongge  1_{|Z|\leq n} \frac{i_{e_j}}{1+n}, \kongge\kongge f^p\wedge,\kongge \kongge\mathrm{for}\kongge 1\leq j\leq 2\ell', \, n\in\nn.
    \end{equation}

    By \eqref{e4}, \eqref{Op}, \eqref{lxk3t expansion} and \eqref{lxkn3t}, $\frac{\partial^{r}}{\partial u^{r}}L_{x,K,n}^{3,t}$ is a second differential operator where the 
operators $e^j\wedge$, $i_{e_j}$, $f^p\wedge$ appear at most $4$ times in each term. Moreover by Theorem \ref{thm 1xk3t expansion}, for  $t\in \,]0,1]$, $|K|\leq 1$, the coefficients in $\frac{\partial^{r}}{\partial u^{r}}L_{x,K,n}^{3,t}$ are uniformly bounded by $C(1+n)^{3+r}$. So by \eqref{3operators bound}, we have
\begin{equation}\label{e5}
    \Big\| \frac{\partial^{r}}{\partial u^{r}}L_{x,K,n}^{3,t}\Big\|_{t,x}^{1,-1}\leq C_{m,r}(1+n)^{7+r}.
\end{equation}

By \eqref{e4}, \eqref{Op}, \eqref{lxk3t expansion} and \eqref{lxkn3t}, we see that for $1\leq p\leq k$, the operator 
\begin{align}\label{e6}
  \Big[\nabla_{e_{i_1}},\Big[\nabla_{e_{i_2}},\dots,\Big[\nabla_{e_{i_p}},\frac{\partial^{r}}{\partial u^{r}}L_{x,K,n}^{3,t}\Big]\dots\Big]\Big]
\end{align}
has the same structure as $\frac{\partial^{r}}{\partial u^{r}}L_{x,K,n}^{3,t}$. Therefore by Theorem \ref{thm 1xk3t expansion} and (\ref{e6}), we deduce that there are constants $C_{m,r}>0$ such that for $n\in\nn$, $t\in \,]0,1]$, $|K|\leq 1$, 
\begin{equation}\label{partial r lxkn3t high estimate}
    \Big\| \frac{\partial^{r}}{\partial u^{r}}L_{x,K,n}^{3,t}\Big\|_{t,x}^{m+1,m-1}\leq C_{m,r}(1+n)^{7+r}.
\end{equation}
From  (\ref{resolvent lxkn3t estimate high}), \eqref{arnklambdat} and (\ref{partial r lxkn3t high estimate}), we get \eqref{e7}.
\end{proof}

Observe that we can take $t=0$ in (\ref{a family of norms}). And the estimates in this section still hold when making $t=0$. Also
\begin{equation}\label{comparision norm t=0}
    |s|_{t,x,0}\leq |s|_{0,x,0}.
\end{equation}
\begin{theorem}\label{thm two estimates dur-t-0}
    Let $t=u^2$, for any $r\geq 0$, $k>0$, there exist $C>0$, $N\in\nn$ such that for $\lambda\in U_{\eta n+d}$,  $u\in \,]0,1]$, $|K|\leq c_{\eta}$, $n\in\nn$, 
    \begin{equation}\label{two estimates dur-t-0}
        \begin{split}
            &\Big|\Bigl( \frac{\partial^{r}L_{x,K,n}^{3,t}}{\partial u^{r}}- \frac{\partial^{r}L_{x,K,n}^{3,t}}{\partial u^{r}}\Big|_{u=0}\Bigr)s\Big|_{t,x,-1}\leq C(1+n^2)^N\sqrt{t}|s|_{0,x,1},\\ 
            \Big|\Bigl(  \frac{\partial^{r}}{\partial u^{r}}(\lambda-&L_{x,K,n}^{3,t})^{-k}-\sum_{(\bf{k},\bf{r})\in I_{k,r}}a_{\bf{r}}^{\bf{k}} A_{\mathbf{r},n}^{\bf{k}}(\lambda,0)\Bigr)s\Big|_{t,x,0}\leq C\sqrt{t}\bigl((1+n^2)(1+|\lambda|)\bigr)^{N}|s|_{t,x,0}.
        \end{split}
    \end{equation}
\end{theorem}
\begin{proof}
    From (\ref{lxk3t expansion}), (\ref{lxkn3t}), we have
    \begin{equation}\label{dur lxkn3t-t=0}
        \frac{\partial^{r}L_{x,K,n}^{3,t}}{\partial u^{r}}-\frac{\partial^{r}L_{x,K,n}^{3,t}}{\partial u^{r}}\Big|_{u=0}=\gamma^2\Bigl(\frac{|Z|}{2(n+2)}\Bigr)\frac{\partial^{r}}{\partial u^{r}}R_{r+1}(t,Z)(K).
    \end{equation}
    By Theorem \ref{thm 1xk3t expansion} and (\ref{dur lxkn3t-t=0}), we use the same arguments before \eqref{e5} and get
    \begin{equation}\label{norm dur lxkn3t-t=0}
        \Big\| \frac{\partial^{r}L_{x,K,n}^{3,t}}{\partial u^{r}}-\frac{\partial^{r}L_{x,K,n}^{3,t}}{\partial u^{r}}\Big|_{u=0}\Big\|_{t,x}^{1,-1}\leq C_r\sqrt{t}(1+n^{2r+7}).
    \end{equation}
From (\ref{comparision norm t=0}), (\ref{norm dur lxkn3t-t=0}), we obtain the first estimate in (\ref{two estimates dur-t-0}). Note that
\begin{equation}\label{difference of resolvent}
    (\lambda-L_{x,K,n}^{3,t})^{-1}-(\lambda-L_{x,K,n}^{3,0})^{-1}=(\lambda-L_{x,K,n}^{3,t})^{-1}(L_{x,K,n}^{3,t}-L_{x,K,n}^{3,0})(\lambda-L_{x,K,n}^{3,0})^{-1}.
\end{equation}
   If $s\in\mathbf{I}_x$ is smooth with compact support, from (\ref{comparision norm t=0}),
   \begin{equation}
    |s|_{0,x,-1}\leq |s|_{t,x,-1}\leq |s|_{t,x,0}.
   \end{equation}
   Using now the estimates in (\ref{resolvent lxkn3t estimate first}) (also in the case $t=0$) and (\ref{norm dur lxkn3t-t=0}), (\ref{difference of resolvent}), we obtain 
   \begin{equation}\label{norm difference of resolvent}
    \Big|\Bigl((\lambda-L_{x,K,n}^{3,t})^{-1}-(\lambda-L_{x,K,n}^{3,0})^{-1}\Bigr)s\Big|_{t,x,0}\leq C\sqrt{t}\bigl((1+n^2)(1+|\lambda|)\bigr)^{N}|s|_{t,x,0}.
   \end{equation}
   Observe from (\ref{arnklambdat}), (\ref{sum of arnklambdat}) that $\frac{\partial^{r}}{\partial u^{r}}(\lambda-L_{x,K,n}^{3,t})^{-k}$ is a combination of  $\frac{\partial^{r}L_{x,K,n}^{3,t}}{\partial u^{r}}$ and $(\lambda-L_{x,K,n}^{3,t})^{-1}$, so the second inequality of (\ref{two estimates dur-t-0}) follows from (\ref{resolvent lxkn3t estimate first}),  (\ref{norm dur lxkn3t-t=0}) and (\ref{norm difference of resolvent}).
\end{proof}

\subsection{Uniform estimate on the kernels}
For $h\in\nn$, let
$\widetilde{F}_{t,n}^h(a)$ be the holomorphic function defined by 
\begin{equation}\label{def of ftnh}
    \widetilde{F}_{t,n}^h(a^2)=\int_{-\infty}^{+\infty}\exp(is\sqrt{2}a)\exp(-\frac{s^2}{2})f(\sqrt{t}s)k_n(s)(1-\gamma^2(2s/h))\frac{ds}{\sqrt{2\pi}}.
\end{equation}
Then as in (\ref{estimate ftn}), we see that given  $k,m,m'\in\nn$, there exist $C>0$, $C'>0$, $C''>0$, $C'''>0$ such that for any $n\in\nn$, $c>0$, if $\lambda\in V_c$,
\begin{equation}\label{estimate derivative ftn}
    |\lambda|^{m}|\frac{\partial^k}{\partial u^k}\widetilde{F}_{t,n}^{h(m')}(\lambda)|\leq C\exp(-C'n^2+C'' c^2-C'''h^2).
\end{equation}

By using \eqref{e2} and the  finite propagation speed as in (\ref{finite propagation n+p}), we deduce that if $|Z|, |Z'|\in (T_{\rr}X)_x$, $|Z-Z'|\geq h$, then for any $p\in\nn$,
\begin{equation}\label{ftnh propagation}
    \widetilde{F}_{t,n}(L_{x,K,n+p}^{3,t})(Z,Z')=\widetilde{F}_{t,n}^h(L_{x,K,n+p}^{3,t})(Z,Z').
\end{equation}

\begin{theorem}\label{thm uniform estimate frn lxkn2t}
   Let $t=u^2$. There exist $C'>0$, $C''>0$, $C'''>0$ such that for $\eta>0$ small enough, for any $m\in\nn$, $r\in\nn$, there are $C>0$, $N\in\nn$, for any $u\in \,]0,1]$, $|K|\leq c_{\eta}$, $n,p,h\in\nn$, 
   \begin{equation}\label{estimate ftnn+p kernel}
    \begin{split}
    \sup_{|\alpha|,|\alpha'|\leq m\atop |Z|,|Z'|\leq p+1}\Big|\frac{\partial^{|\alpha|+|\alpha'|}}{\partial Z^{\alpha}\partial Z'^{\alpha'}}\frac{\partial^r}{\partial u^r}\widetilde{F}_{t,n}^h(L_{x,K,n+p}^{3,t})(Z,Z')\Big|\qquad\qquad\qquad\qquad\\ 
    \leq C(1+p)^N\exp(-C'n^2/4+2C''\eta^2p^2-C'''h^2).
    \end{split}
\end{equation}
\end{theorem}
\begin{proof}
   We take the constants such that (\ref{estimate derivative ftn}) is valid for $k\leq r$. We fix $\eta>0$ such that 
    \begin{equation}\label{take small eta}
        C'-2C''\eta^2\geq \frac{C'}{2}.
    \end{equation}
    By  (\ref{estimate derivative ftn}), (\ref{take small eta}), given $m$, $m'\in\nn$, $d>0$, $k\leq r$, there exists $C_0>0$ such that if  $\lambda\in V_{\eta(n+p)+d}$,
    \begin{equation}\label{estimate ftn eta}
        |\lambda|^m\Big|\frac{\partial^k}{\partial u^k}\widetilde{F}_{t,n}^{h(m')}(\lambda)\Big|\leq C_0\exp(-\frac{C'}{2}n^2+2C''\eta^2p^2-C'''h^2).
    \end{equation}
    In view of (\ref{estimate ftn eta}), given $q\in\nn$, there is a unique holomorphic function $\widetilde{F}_{t,n,q}^h(\lambda)$ defined on a neighborhood of $V_{\eta(n+p)+d}$ such that
    \begin{equation}\label{def of ftnq}
        \begin{split}
            \widetilde{F}_{t,n,q}^{h}(\lambda)\rightarrow 0\quad\mathrm{as}\kongge \lambda\rightarrow +\infty,\quad
            \frac{\widetilde{F}_{t,n,q}^{h,(q-1)}(\lambda)}{(q-1)!}=\widetilde{F}_{t,n}^{h}(\lambda).
        \end{split}
    \end{equation}
    Then $\widetilde{F}_{t,n,q}^h(\lambda)$ and its derivatives in $u$ verifies a similar bound on $V_{\eta(n+p)+d}$ as  in (\ref{estimate ftn eta}). 

    By Theorem \ref{thm estimate of resolvent first}, (\ref{estimate ftn eta}) and (\ref{def of ftnq}), if $\eta>0$ is taken small enough, for $c_{\eta}$ and $d$ as in Theorem \ref{thm estimate of resolvent first} and for $|K|\leq c_{\eta}$,
    \begin{equation}\label{spectral thm}
        \widetilde{F}_{t,n}^h(L_{x,K,n+p}^{3,t})=\frac{1}{2i\pi}\int_{\Gamma_{\eta(n+p)+d}}\widetilde{F}_{t,n,q}^h(\lambda)(\lambda-L_{x,K,n+p}^{3,t})^{-q}d\lambda.
    \end{equation} 
    By using Theorem \ref{thm QArknQ'}, we see that for $m\in\nn$, $k\leq r$, $q\geq 2(m+r)$, for $Q\in \mathcal{Q}_x^m$, $\lambda\in \Gamma_{\eta(n+p)+d}$, $t\in \,]0,1]$, $|K|\leq c_{\eta}$,
    \begin{equation}\label{qresolventq' estimate}
        \Big\|Q\frac{\partial^k}{\partial u^k}(\lambda-L_{x,K,n+p}^{3,t})^{-q}\Big\|_{t,x}^{0,0}\leq C_q'\bigl((1+(n+p)^2)(1+|\lambda|^2)\bigr)^{m_q'}.
    \end{equation}
    From (\ref{estimate ftn eta}), (\ref{spectral thm}) and (\ref{qresolventq' estimate}), we conclude that
    \begin{equation}\label{qftnq}
        \Big\|Q\frac{\partial^r}{\partial u^r}\widetilde{F}_{t,n}^h(L_{x,K,n+p}^{3,t})\Big\|_{t,x}^{0,0}\leq C(1+p)^{2m_q'}\exp(-C'\frac{n^2}{4}+2C''\eta^2p^2-C'''h^2).
    \end{equation}

    Let $(J_{x,p}^{0},|\cdot|_0)$ be the vector space of square integrable sections of $\bigl( \pi^*\Lambda(T_{\rr}^*S)\,\widehat{\otimes} \,\Lambda(T_{\rr}^*X_g)$ $\widehat{\otimes} \Lambda(N_{X_g/X}^{*(0,1)}) \, \otimes\xi\bigr)_x$ over $\{Z\in(T_{\rr}X)_x: |Z|\leq p+1\}$, equipped with the normal Hermitian norm. By (\ref{a family of norms}), for $s\in J_{x,p}^0$,
    \begin{equation}\label{comparision normal special norm}
        |s|_0\leq |s|_{t,x,0}\leq C(1+p)^{2\ell'}|s|_0.
    \end{equation}

    If $A\in \mathcal{L}(J_{x,p}^0)$, let $\|A\|_{\infty,p}$ be the norm of $A$ with respect to $|\cdot|_0$. From (\ref{qftnq}), (\ref{comparision normal special norm}), we find that 
    \begin{equation}\label{qfynq normal norm}
        \begin{split}
        \Big\|Q\frac{\partial^r}{\partial u^r}\widetilde{F}_{t,n}^h(L_{x,K,n+p}^{3,t})\Big\|_{\infty,p}\qquad\qquad\qquad\qquad\qquad\qquad\qquad\qquad\\ 
        \leq C(1+p)^{2(\ell'+m_q')}\exp(-C'\frac{n^2}{4}+2C''\eta^2p^2-C'''h^2).
        \end{split}
    \end{equation}

We denote by $(L_{x,K,n}^{3,t})^*$ the formal adjoint of $L_{x,K,n}^{3,t}$ with respect to the usual Hermitian product on $\mathbf{I}_x^0$. Note that $(L_{x,K,n}^{3,t})^*$ has the same structure as $L_{x,K,n}^{3,t}$, except that we exchanges $e^j\wedge$ and $i_{e_j}$ for $j=1,...,2\ell'$.  By using the norm 
   \begin{equation}
        |s|_{t,x,0}'^2=\int_{(T_{\rr}X)_x}|s|^2\Bigl(1+|Z|\rho\big(\frac{\sqrt{t}Z}{2}\big)\Bigr)^{2(-2\ell'+2\mathrm{dim}S-r+q)}dv_{TX}(Z).
    \end{equation}
    instead of \eqref{a family of norms}, we can use the same arguments which leads to  (\ref{qfynq normal norm})  and obtain
 \begin{equation}\label{qfynq normal norm'}
        \begin{split}
        \Big\|\frac{\partial^r}{\partial u^r}\widetilde{F}_{t,n}^h(L_{x,K,n+p}^{3,t})Q\Big\|_{\infty,p}\qquad\qquad\qquad\qquad\qquad\qquad\qquad\qquad\\ 
        \leq C(1+p)^{2(\ell'+m_q')}\exp(-C'\frac{n^2}{4}+2C''\eta^2p^2-C'''h^2).
        \end{split}
    \end{equation}

  By using \eqref{qfynq normal norm}, \eqref{qfynq normal norm'} and the Sobolev inequalities, we obtain (\ref{estimate ftnn+p kernel}), thus we complete the proof of Theorem \ref{thm uniform estimate frn lxkn2t}.
\end{proof}

\subsection{A proof of Theorem \ref{thm ft akt expansion}}\label{ee}

Recall that the function $\widetilde{F}_{t,n}^h(\lambda)$ is defined in (\ref{def of ftnh}). For $q\in\nn$, we set $\widetilde{F}_{t,n,q}^h(\lambda)$ as in (\ref{def of ftnq}). By proceeding as in (\ref{estimate derivative ftn}), we find that for $r\in\nn$, for $m,m'\in\nn$, there exist $C>0$, $C'>0$, $C'''>0$, $C''''>0$ such that for $c>0$, $\lambda\in V_c$,
\begin{equation}\label{estimate derivative ftnq difference}
    \begin{split}
    &|\lambda|^{m}\Big|\frac{\partial^r}{\partial u^r}\widetilde{F}_{t,n,q}^{h, (m')}(\lambda)-\frac{\partial^r}{\partial u^r}\Big|_{u=0}\widetilde{F}_{t,n,q}^{h, (m')}(\lambda)\Big|\\ 
    &\qquad\qquad\qquad\qquad\qquad\leq C\exp(-C'n^2+C'' c^2-C'''h^2-C''''/t).
    \end{split}
\end{equation} 
 
 Let $r\in\nn$ and let $q$ be large enough, $\eta$ be small enough, for $n,p,h\in\nn$, set
\begin{equation}\label{def of fnprh}
    \begin{split}
        &\mathscr{F}_{n,p}^{r,h}=\frac{1}{2i\pi r!}\int_{\Gamma_{\eta(n+p)+d}}\widetilde{F}_{0,n,q}^h(\lambda)\sum_{(\bf{k},\bf{r})\in I_{q,r}}a_{\bf{r}}^{\bf{k}} A_{\mathbf{r},n+p}^{\bf{k}}(\lambda,0)\, d\lambda,\\ 
        &\mathscr{F}_{t,n,p}^{r,h}=\frac{1}{r!}\frac{\partial^r}{\partial u^r}\widetilde{F}_{t,n}^h(L_{x,K,n+p}^{3,t})-\mathscr{F}_{n,p}^{r,h}.
    \end{split}
\end{equation}

Let $\mathscr{F}_{n,p}^{r,h}(Z,Z' )$ and $\mathscr{F}_{t,n,p}^{r,h}(Z,Z' )$ be the Schwartz kernel associated to $\mathscr{F}_{n,p}^{r,h}$ and $\mathscr{F}_{t,n,p}^{r,h}$ with respect to $dv/(2\pi)^{\ell}$.
It is clear that the uniform estimate in (\ref{estimate ftnn+p kernel}) holds for $\mathscr{F}_{n,p}^{r,h}(Z,Z')$ and $\mathscr{F}_{t,n,p}^{r,h}(Z,Z')$, that is to say, under the condition of Theorem \ref{thm uniform estimate frn lxkn2t}, we have
\begin{equation}\label{uniform estimate frnp}
    \begin{split}
    \sup_{|\alpha|,|\alpha'|\leq m\atop |Z|,|Z'|\leq p+1}\Big|\frac{\partial^{|\alpha|+|\alpha'|}}{\partial Z^{\alpha}\partial Z'^{\alpha'}}\mathscr{F}_{t,n,p}^{r,h}(Z,Z')\Big|
    \leq C(1+p)^N\exp(-C'n^2/4+2C''\eta^2p^2-C'''h^2).
    \end{split}
\end{equation}

\begin{theorem}\label{thm t-0 estimate ftnprh}
   Let $t=u^2$.  There exist  $C'>0$, $C'' >0$, $C'''>0$, such that  for $\eta>0$ small enough, $r\in\nn$, there exist $C>0$, $N\in\nn$, $c_{\eta}\in \,]0,1]$, for any $u\in \,]0,1]$, $|K|\leq c_{\eta}$, $n,p,h\in \nn$,
    \begin{equation}\label{t-0 estimate ftnprh}
        \begin{split}
        \sup_{|Z|, |Z'|\leq p}\big|\mathscr{F}_{t,n,p}^{r,h}(Z,Z')\big|\leq Ct^{1/4(2\ell+1)}(1+p)^N\exp(-C'n^2/4+2C''\eta^2 p^2-C'''h^2).
        \end{split}
    \end{equation}
\end{theorem}
\begin{proof}
   From (\ref{spectral thm}), (\ref{def of fnprh}), we see that $\mathscr{F}_{t,n,p}^{r,h}$   is the combination of three sorts of terms, for $r_1, r_2\in \mathbb{N}$, $r_1+r_2=r$,
   \begin{equation}\label{two operator in ftnprh}
    \begin{split}
    \int_{\Gamma_{\eta(n+p)+d}}\Bigl(\frac{\partial^{r_1}}{\partial u^{r_1}}\widetilde{F}_{t,n,q}^{h}(\lambda)-\frac{\partial^{r_{1}}}{\partial u^{r_{1}}}\Big|_{u=0}\widetilde{F}_{t,n,q}^{h}(\lambda)\Bigr)\frac{\partial^{r_{2}}}{du^{{r_{2}}}}(\lambda-L_{x,K,n+p}^{3,0})^{-q}d\lambda,\\
    \int_{\Gamma_{\eta(n+p)+d}}\frac{\partial^{r_1}}{\partial u^{r_1}}\Big|_{u=0}\widetilde{F}_{t,n,q}^h(\lambda)\Bigl(  \frac{\partial^{r_2}}{\partial u^{r_2}}(\lambda-L_{x,K,n+p}^{3,t})^{-q}-\sum_{(\bf{k},\bf{r})\in I_{q,r_2}}a_{\bf{r}}^{\bf{k}} A_{\mathbf{r},n+p}^{\bf{k}}(\lambda,0)\Bigr)d\lambda,
    \end{split}
   \end{equation}
   and for $r_1\geq 1$, $r_2\in \mathbb{N}$, $r_1+r_2=r$, 
   \begin{align}\label{e8}
\int_{\Gamma_{\eta(n+p)+d}}\frac{\partial^{r_1}}{\partial u^{r_1}}\Big|_{u=0}\widetilde{F}_{t,n,q}^h(\lambda)\sum_{(\bf{k},\bf{r})\in I_{q,r_2}}a_{\bf{r}}^{\bf{k}} A_{\mathbf{r},n+p}^{\bf{k}}(\lambda,0)d\lambda.                                     
   \end{align}
  From the $t=0$ version of (\ref{resolvent lxkn3t estimate first}) and (\ref{estimate derivative ftnq difference}), the $\|\cdot\|_{\infty,p}$ norm of the first operator in (\ref{two operator in ftnprh}) is bounded by $C(1+p)^N\exp(-C'n^2/4+2C''\eta^2 p^2-C'''h^2-C''''/t)$.  From (\ref{two estimates dur-t-0}), (\ref{estimate ftn eta}) and (\ref{comparision normal special norm}), we see that the $\|\cdot\|_{\infty,p}$ norm of the second operator in (\ref{two operator in ftnprh}) is bounded by $C\sqrt{t}(1+p)^N\exp(-C'n^2/4+2C''\eta^2 p^2-C'''h^2)$. As for \eqref{e8}, by (\ref{def of f}), we know that for any $r_1\geq 1$, $f^{(r_1)}(0)=0$. By \eqref{def of ftnh}, \eqref{def of ftnq}, we see that the term in \eqref{e8} vanish.

To sum up, for $t\in \,]0,1]$, $|K|\leq c_{\eta}$, $n,p,h\in\nn$, we have
   \begin{equation}\label{infty p norm ftnprh}
    \|\mathscr{F}_{t,n,p}^{r,h}\|_{\infty,p}\leq C\sqrt{t}(1+p)^N\exp(-C'n^2/4+2C''\eta^2 p^2-C'''h^2).
   \end{equation}

   Let $\phi:\rr^{2\ell}\rightarrow [0,1]$ be a smooth function with compact support, which equals to $1$ near $0$, such that $\int_{\rr^{2\ell}}\phi(Z)dv_{TX}(Z)=1$. For $\nu\in \,]0,1]$ small enough, by (\ref{uniform estimate frnp}), there exists $C>0$, such that if $|Z|, |Z'|\leq p$, $U,U'\in \bigl(\pi^*\Lambda(T_{\rr}^*S)\,\widehat{\otimes }\,\Lambda(T_{\rr}^*X_g)\,\widehat{\otimes }\,\Lambda(N_{X_g/X}^{*(0,1)}) \otimes  \xi\bigr)_x$,
   \begin{equation}\label{take phi 1}
    \begin{split}
    \Big|\product{\mathscr{F}_{t,n,p}^{r,h}(Z,Z')U,U'}-\hspace*{22em}\\ 
    \int_{|Z|,|Z'|\leq p+1}\product{\mathscr{F}_{t,n,p}^{r,h}(Z-W,Z'-W')U,U'}
\frac{1}{\nu^{4\ell}}\phi(W/\nu)\phi(W'/\nu)dv_{TX}(W)dv_{TX}(W')\Big|\\ 
   \leq C\nu(1+p)^N \exp(-C'n^2/4+2C''\eta^2 p^2-C'''h^2)|U||U'|.
    \end{split}
   \end{equation}
   On the other hand, by (\ref{infty p norm ftnprh}),
   \begin{equation}\label{take phi 2}
    \begin{split}
    &\Big|\int_{|Z|,|Z'|\leq p+1}\product{\mathscr{F}_{t,n,p}^{r,h}(Z-W,Z'-W')U,U'}
\frac{1}{\nu^{4\ell}}\phi(W/\nu)\phi(W'/\nu)\\ 
&\quad dv_{TX}(W)dv_{TX}(W')\Big|\leq C\frac{\sqrt{t}}{\nu^{2\ell}}(1+p)^N\exp(-C'n^2/4+2C''\eta^2 p^2-C'''h^2).
    \end{split}
   \end{equation}
   By taking $\nu=t^{1/2(2\ell+1)}$, we deduce from (\ref{take phi 1}), (\ref{take phi 2}) that for $t$ small enough, (\ref{t-0 estimate ftnprh}) holds. By using (\ref{uniform estimate frnp}) again, we complete the proof of Theorem \ref{thm t-0 estimate ftnprh}.
\end{proof}

Let $t=u^2$. We deduce from Theorem \ref{thm t-0 estimate ftnprh} the following essential fact,
\begin{equation}\label{derivative of ftnh at t=0}
    \frac{1}{r!}\frac{\partial^r}{\partial u^r}\Big|_{u=0}\widetilde{F}_{t,n}^h(L_{x,K,n+p}^{3,t})(Z,Z')=\mathscr{F}_{n,p}^{r,h}(Z,Z'), \kongge\kongge \mathrm{for}\kongge |Z|,|Z'|\leq p.
\end{equation}

From (\ref{finite propagation n+p}), (\ref{ftnh propagation}), we see that if $|Z|,|Z'|\leq p$, $|Z-Z'|\geq h$, the kernel
\begin{equation}\label{combine propagation}
    \widetilde{F}_{t,n}^h(L_{x,K,n+p}^{3,t})(Z,Z')=\widetilde{F}_{t,n}(L_{x,K}^{3,t})(Z,Z')
\end{equation}
does not depend on $p$ and $h$. So by (\ref{derivative of ftnh at t=0}), we find that  if $|Z|,|Z'|\leq p$, $|Z-Z'|\geq h$, the kernel $\mathscr{F}_{n,p}^{r,h}(Z,Z')$ does not depend on $p$ and $h$. Therefore for $\eta>0$ small enough and $t\in \,]0,1]$, $|K|\leq c_{\eta}$, by (\ref{uniform estimate frnp}), we may define the kernel $\mathscr{F}^r(Z,Z')$   by the formula,
\begin{equation}\label{def of fr}
    \mathscr{F}^{r}(Z,Z')=\sum_{n\in\nn}\mathscr{F}_{n,p}^{r,h}(Z,Z'),\qquad \,|Z|,|Z'|\leq p,\, |Z-Z'|\geq h,
\end{equation}
and the series in (\ref{def of fr}) converges uniformly on compact subsets of $(T_{\rr}X)_x$ together with its derivatives. We note that $\mathscr{F}^0(Z,Z')$ coincides with the kernel $\exp(-L_{x,K}^{3,0})(Z,Z')$ defined in \cite[(7.212)]{BismutGoette2000}.

\begin{theorem}
    Let $t=u^2$. There exist $C'' >0$, $C'''>0$ such that for  $\eta>0$ small enough, $k\in\nn$, there exist $C>0$,   $N\in\nn$, $c_{\eta}\in \,]0,1]$ such that if $u\in \,]0,1]$, $|K|\leq c_{\eta}$, 
    \begin{equation}\label{expansion ftlxk}
        \begin{split}
        &\big|\widetilde{F}_{t}(L_{x,K}^{3,t})(Z,Z')-\sum_{r=0}^k  t^{r/2}\mathscr{F}^r(Z,Z')\big|\\ 
        &\leq C t^{\frac{k+1}{2}}(1+|Z|+|Z'|)^N\exp(2C''\eta^2 \sup(|Z|^2,|Z'|^2)-C'''|Z-Z'|^2).
        \end{split}
    \end{equation}
\end{theorem}
\begin{proof}
From (\ref{estimate ftnn+p kernel}), (\ref{derivative of ftnh at t=0}), and the Taylor expansion
\begin{equation}
    G(u)-\sum_{r=0}^k \frac{1}{r!}\frac{\partial^r G}{\partial u^r}(0)u^r=\frac{1}{k!}\int_{0}^{u}(u-u_0)^k\frac{\partial^{k+1}G}{\partial u^{k+1}}(u_0)du_0,
\end{equation}
we find that if $|Z|, |Z'|\leq p$,
\begin{equation}\label{expansion ftnlxkn+p}
        \begin{split}
        &\big|\widetilde{F}_{t,n}(L_{x,K,n+p}^{3,t})(Z,Z')-\sum_{r=0}^k  t^{r/2}\mathscr{F}_{n,p}^{r,h}(Z,Z')\big|\\ 
        &\qquad\qquad\leq Ct^{(k+1)/2}(1+p)^N\exp(-C'n^2/4+2C''\eta^2 p^2-C'''h^2).
        \end{split}
    \end{equation}

   Thus we can  sum up (\ref{expansion ftnlxkn+p}) for $n\in\nn$, by using (\ref{as infinite sum}), (\ref{combine propagation}) and (\ref{def of fr}), we obtain (\ref{expansion ftlxk}). Then our theorem is proved.
\end{proof}

\begin{proof}[Proof of Theorem \ref{thm ft akt expansion}]
    Now there is $c>0$ such that if $Z\in N_{X_g/X,\rr,x}$, then
    \begin{equation}\label{distance gz z}
        |g^{-1}Z-Z|\geq c|Z|.
    \end{equation}
    We use the notations in (\ref{expansion ftlxk}), and we take $\eta>0$ small enough such that
    \begin{equation}\label{fix small eta}
        2C''\eta^2\leq c^2C'''/2.
    \end{equation}
    Therefore by (\ref{expansion ftlxk}), (\ref{distance gz z}) and (\ref{fix small eta}), if $Z\in N_{X_g/X,\rr,x}$,
    \begin{equation}\label{expansion ftlxk g}    \begin{split}
        \big|\widetilde{F}_{t}(L_{x,K}^{3,t})(g^{-1}Z,Z)-\sum_{r=0}^k  t^{r/2}\mathscr{F}^r(g^{-1}Z,Z)\big|
        \leq C t^{(k+1)/2}\exp(-\frac{C'''}{4}|Z|^2).
        \end{split}
    \end{equation}

    By using (\ref{estimate ftnn+p kernel}) and the same arguments leading to (\ref{expansion ftlxk g}), we obtain
    \begin{equation}
        \big|\widetilde{F}_{t}(L_{x,K}^{3,t})(g^{-1}Z,Z)\big|
        \leq C \exp(-\frac{C'''}{4}|Z|^2).
    \end{equation}

    From \eqref{ea}, (\ref{getzler rescaling str}), we get
    \begin{equation}\label{e9}
        \begin{split}
       \frac{t^{\mathrm{dim}N_{X_g/X}}}{(2\pi)^{\ell'}} \int_{Z\in N_{X_g/X,\rr}\atop |Z|\leq \varepsilon_0/\sqrt{t}}\str\big[g\ft(\akt)(g^{-1}(x,\sqrt{t}Z),(x,\sqrt{t}Z))\big] 
        k(x,\sqrt{t}Z)\frac{dv_{N_{X_g/X}}(Z)}{(2\pi)^{\mathrm{dim}N_{X_g/X}}}\\ 
        =  \Big(\frac{1}{2i\pi}\Big)^{\ell'}\int_{Z\in N_{X_g/X,\rr}\atop |Z|\leq \varepsilon_0/\sqrt{t}}\str^{\Lambda(N_{X_g/X}^*)\otimes \xi}\big[g\ft(L_{x,K}^{3,t})(g^{-1}Z,Z)\big]^{\mathrm{max}}\widetilde{k}_x(\sqrt{t}Z)\frac{dv_{N_{X_g/X}}(Z)}{(2\pi)^{\mathrm{dim}N_{X_g/X}}},
        \end{split}
    \end{equation}
    where $\widetilde{k}_x(Z)=k(x,Z)k^{'-1}_x(Z)$ is a smooth function on $\{Z\in N_{X_g/X,\rr}:|Z|\leq \varepsilon_0\}$.

    By using (\ref{expansion ftlxk g})-(\ref{e9}), we see that for $k\in\nn$, there is $C>0$ such that for $K\in\mathfrak{z}(g)\otimes_{\mathbb{R}}\mathbb{C}$, $|K|\leq c_{\eta}$, $t\in \,]0,1]$,
    \begin{equation}
        \begin{split}
            &\Big|\frac{t^{\mathrm{dim}N_{X_g/X}}}{(2\pi)^{\ell'}}\int_{Z\in N_{X_g/X,\rr}\atop |Z|\leq \varepsilon_0/\sqrt{t}}\str\big[g\ft(\akt)(g^{-1}(x,\sqrt{t}Z),(x,\sqrt{t}Z))\big]\\ 
&\qquad\qquad\qquad\quad \cdot k(x,\sqrt{t}Z)\frac{dv_{N_{X_g/X}}(Z)}{(2\pi)^{\mathrm{dim}N_{X_g/X}}}-\sum_{r=0}^k t^{r/2} a_{r/2,g,K}\Big|\leq C t^{(k+1)/2}.
        \end{split}
        \end{equation}

        Finally, to complete the proof of Theorem \ref{thm ft akt expansion}, it suffices to show that for odd $r$, $a_{r/2,g,K}$ vanishes. This is done if we show that $\str^{\Lambda(N_{X_g/X}^{*(0,1)})\otimes \xi_x}\big[g\mathscr{F}^r(g^{-1}Z,Z)\big]$ is an odd function in $Z$ for $r$ odd. By \eqref{lxk3t expansion}, \eqref{arnklambdat}, \eqref{def of fnprh} and \eqref{def of fr}, $\mathscr{F}^r$ consists of integrals of 
        \begin{equation}\label{eb}
        \begin{split}
       (\lambda-L_{x,K,n}^{3,0})^{-k_0}\gamma^2\Bigl(\frac{|Z|}{2(n+2)}\Bigr)\mathcal{L}_{r_1}(K)(\lambda-L_{x,K,n}^{3,0})^{-k_1} \qquad\qquad\\
        \dots(\lambda-L_{x,K,n}^{3,0})^{-k_{j-1}}\gamma^2\Bigl(\frac{|Z|}{2(n+2)}\Bigr)\mathcal{L}_{r_j}(K)(\lambda-L_{x,K,n}^{3,0})^{-k_j},
        \end{split}
     \end{equation}
     here $\mathcal{L}_{r_i}\in \Op_{r_{i}}$ is of the form \eqref{e4}, and satisfies $r_1+\dots+r_j=r$. By \eqref{def of lxk30}, \eqref{lxkn3t}, the operators $(\lambda-L_{x,K,n}^{3,0})^{-k_i}$ are invariant under the change of variable $Z\to -Z$.  Note that $\gamma$ is an even function in $Z$, by using the parity condition in Definition \ref{dr},  if $r$ is odd, then under the change of variable $Z \to -Z$, the terms in \eqref{eb} of odd Clifford degree remain invariant, whereas those of even Clifford degree change sign. However, the odd Clifford degree terms do not contribute to supertraces. So we deduce that the supertrace $\str^{\Lambda(N_{X_g/X}^{*(0,1)})\otimes \xi_x}\big[g\mathscr{F}^r(g^{-1}Z,Z)\big]$ is an odd function in $Z$, so we finishes the proof of Theorem \ref{thm ft akt expansion}.
\end{proof}

\subsection{Proof of Theorem \ref{thm asymptotic-gK} for $u\in \,]0,1]$.}\label{section establish analytic torsion}
Recall that 
\begin{equation}\label{ec}
    \gamma_{t,g,K}=\Phi\str\big[gN_{K,t}\exp(-\akt)\big].
\end{equation}

Let $e_1,\dots,e_{2\ell}$ be an oriented orthogonal basis of $T_{\rr}X$ and let $J^{TX}$ be the complex structure on $T_{\rr}X$. By \cite[Proposition 2.4]{BGS2} and by (\ref{definition nkt}),
\begin{equation}\label{nkt def using c}
    N_{K,t}=\frac{i}{4}\sum_{i=1}^{2\ell}c(J^{TX}e_i)c(e_i)+\frac{\ell}{2}+\frac{i\omega^H}{t}-i\frac{\product{\mu,K}}{t}.
\end{equation}
Let $[N_{K,t}]_t^{(3)}$ be the operator obtained from $N_{K,t}$ by   using Getzler rescaling  in Definition \ref{def getzler rescaling}. Recall that $\jmath: W_g\rightarrow W$ is the nature embedding. From (\ref{nkt def using c}), one verifies easily that as $t\rightarrow 0$,
\begin{equation}\label{nkt rescaling}
    t[N_{K,t}]_t^{(3)}\rightarrow i\, \jmath^*\omega-i\,\product{\mu,K}.
\end{equation}

Now we work on $\gamma_{t,g,K}$ in \eqref{ec} and wish to get a similar expansion as Theorem \ref{thm asymptotic higher akt}. Indeed we can follow same steps in Sections \ref{ed}-\ref{ee}, in particular equations \eqref{getzler rescaling str} and \eqref{expansion ftlxk g}. Now we write down the first two terms of the expansion of $\gamma_{t,g,K}$: there exist $\beta>0$, $C>0$ such that for $K\in \mathfrak{z}(g)\otimes _{\mathbb{R}}\mathbb{C}$, $|K|\leq \beta$, there are differential forms $C_{-1,g,K}', C_{0,g,K}'\in P^S$, for any $t\in \,]0,1]$,
\begin{equation}\label{expansion gammatgk}
    \big|\gamma_{t,g,K}-\frac{C_{-1,g,K}'}{t}-C_{0,g,K}'\big|\leq Ct.
\end{equation}
Here the expansion initials from $t^{-1}$ because of the singularity of $[N_{K,t}]_t^{(3)}$ as $t\to 0$, which is calculated in \eqref{nkt rescaling}.

Let $\varphi:\Lambda(T_{\rr}^*W_g)\rightarrow \Lambda(T_{\rr}^*W_g) $ be the map $\alpha\rightarrow (2i\pi)^{-\mathrm{deg}\alpha/2}\alpha$. 
Recall that the  operator $\exp(-L_{x,K}^{3,0})=\mathscr{F}^0$ is defined in (\ref{def of fr}). By using \eqref{expansion ftlxk g}, (\ref{nkt rescaling}),  and by computing as in (\ref{getzler rescaling str}), we obtain
\begin{equation}\label{compute c-1gk}
    \begin{split}
    &C_{-1,g,K}'=\int_{X_g}\varphi\big(i(\jmath^*\omega)-i\product{\mu,K}\big)\cdot\,\\ 
    &\qquad\varphi\int_{Z\in N_{X_g/X},\rr}\str^{\Lambda(N_{X_g/X}^*)\otimes \xi}\big[g\exp(-L_{x,K}^{3,0})(g^{-1}Z,Z)\big]\frac{dv_{N_{X_g/X}}(Z)}{(2\pi)^{\mathrm{dim}N_{X_g/X}}}.
    \end{split}
\end{equation}

By using the same arguments as in \cite[(7.225)]{BismutGoette2000}, we get
\begin{equation}\label{compute c-1gk 2}
    C_{-1,g,K}'=C_{-1,g,K}=\int_{X_g}\Big(\frac{\omega}{2\pi}-i\product{\mu,K}\Big)\td_{g,K}(TX,h^{TX})\ch_{g,K}(\xi,h^{\xi}).
\end{equation}

It remains to compute $C_{0,g,K}'\in P^S$. Here we proceed as in \cite[\S 7.14]{BismutGoette2000}.

Recall that $\bisconn_{K,t}''$ and $\bisconn_{K,t}'$ are defined in (\ref{bku}), which satisfy the identities in (\ref{bku-algebraic identity}).  
Let $da$ be the canonical generator of $\mathbb{C}^{*}$. Then $da\dabar$ span $\rr^{2*}\otimes_{\rr}\cc$. If $\sigma\in\Lambda(\rr^{2*})\,\widehat{\otimes}\,\Lambda(T_{\rr}^*S)$, we write $\sigma$ in the form 
 \begin{equation}
     \sigma=\sigma_0+da \sigma^{da}+ \dabar \sigma^{\dabar}+da\dabar \sigma^{da\dabar}, \kongge \sigma_0,\sigma^{da},\sigma^{\dabar},\sigma^{da\dabar}\in\Lambda(T_{\rr}^*S).
 \end{equation}
 \begin{proposition}
    The following identity holds,
    \begin{equation}\label{derivative gamma}
        \begin{split}
               \frac{\partial}{\partial t}\gamma_{t,g,K}=\Phi\str\Big[g\exp\Big(-\lk-\bisconn_{K,t}^2-da\sqrt{2}t\frac{\partial}{\partial t}(\bisconn_{K,t}''-\bisconn_{K,t}')
               -\dabar\sqrt{2}t\frac{\partial}{\partial t}(\bisconn_{K,t}''\\ +\bisconn_{K,t}')
               +N_Vda\dabar\Big)\Big]^{da\dabar}-\overline{\partial}\partial \Phi\frac{1}{2}\frac{\partial^2}{\partial c^2}\str\big[g\exp(-\lk-\bisconn_{K,t}^2+cN_{K,t})\big]_{c=0}. 
        \end{split}
    \end{equation}
 \end{proposition}
 \begin{proof}
    By using the identities in (\ref{bku-algebraic identity}). The proof is totally the same as the proof of \cite[Theorem 2.14]{BGS2}.
 \end{proof}

 \begin{theorem}
    The following identity holds,
    \begin{equation}\label{gamma lichnerowicz formula}
        \begin{split}
            \lk+\bisconn_{K,t}^2+da\sqrt{2}t\frac{\partial}{\partial t}(\bisconn_{K,t}''-\bisconn_{K,t}')
               +\dabar\sqrt{2}t\frac{\partial}{\partial t}(\bisconn_{K,t}'' +\bisconn_{K,t}')
               -N_Vda\dabar\\
            =-\frac{t}{2}\Big(\nabla_{e_i}^{\Lambda(T^{*(0,1)}X)\otimes \xi}+\frac{1}{\sqrt{2t}}\product{S(e_i)e_j, f_p^H}c(e_j)f^p\wedge
             +\frac{1}{2t}\product{S(e_i)f_p^H,f_q^H}f^p\wedge f^q\wedge\\ -\frac{\product{K^X,e_i}}{2t}+\frac{da}{2\sqrt{t}}ic(J^{TX}e_i)-\frac{\dabar}{2\sqrt{t}}c(e_i)\Big)^2+\frac{t}{8}H\\ 
            +\frac{t}{4}\bigl(R^{\xi}+\frac{1}{2}\trace[R^{TX}]\bigr)(e_i,e_j)c(e_i)c(e_j)
            +\sqrt{\frac{t}{2}}\bigl(R^{\xi}+\frac{1}{2}\trace[R^{TX}]\bigr)(e_i,f_p^H)c(e_i)f^p\\ 
            +\frac{1}{2}\bigl(R^{\xi}+\frac{1}{2}\trace[R^{TX}]\bigr)(f_p^H,f_q^H)f^p\wedge f^q\wedge-(m^{\xi}(K)+\frac{1}{2}\trace[m^{TX}(K)]).
        \end{split}
    \end{equation}
 \end{theorem}
 \begin{proof}
    When making $da=0$, $\dabar=0$, (\ref{gamma lichnerowicz formula}) is just (\ref{lichnerowicz formula}). The terms with $da$, $\dabar$ can be dealt with in the same manner as   \cite[Theorem 2.15]{BGS2}.
 \end{proof}

As in Section \ref{ef}, we take $x\in W_g$, and we identify $B^{T_xX}(0,16\varepsilon_0)$ with $B^{X}(x,16\varepsilon_0)$ through the exponential map centered at $x$. We trivialize $\Lambda(\rr^{2*})\,\widehat{\otimes }\, \pi^*\Lambda(T_{\rr}^*S)\,\widehat{\otimes }\, \Lambda(T^{*(0,1)}X)\otimes \xi $ by parallel transport along $s\in [0,1]\rightarrow sZ$ with respect to the connection
\begin{equation}\label{2nablat}
    \begin{split}
        ^2\nabla^{\mathbb{E},t}_{\cdot}=\nabla_{\cdot}^{\mathbb{E}}+\frac{1}{\sqrt{2t}}\product{S(\cdot)e_j, f_p^H}c(e_j)f^p\wedge +\frac{1}{2t}\product{S(\cdot)f_p^H,f_q^H}f^p\wedge f^q\wedge\qquad\\
        -\frac{K^{\prime X}}{2t}+\frac{da}{2\sqrt{t}}ic(J^{TX}\cdot)-\frac{\dabar}{2\sqrt{t}}c(\cdot).
    \end{split}
\end{equation}

From (\ref{curvature 1nablat}), (\ref{2nablat}), we calculate that
\begin{equation}\label{curvature 2nabla}
    \begin{split}
        ( ^2\nabla^{\mathbb{E},1})^2(e_i,e_j)=\frac{1}{4}\product{R^{TX}(e_k,e_{\ell})e_i,e_j}c(e_k)c(e_{\ell})\hspace{12 em}\\ 
        +\frac{1}{\sqrt{2}}\product{R^{TX}(e_k,f_p^H)e_i,e_j}c(e_k)f^p\wedge
        +\frac{1}{2}\product{R^{TX}(f_p^H,f_q^H)e_i,e_j}f^p\wedge f^q\wedge\\ +\big(R^{\xi}+\frac{1}{2}\trace[R^{TX}]\big)(e_i,e_j)-\product{m^{TX}(K)e_i,e_j}-i{da\dabar}\product{e_i,J^{TX}e_j}.
    \end{split}
\end{equation}

 If $s$ is a smooth section of $\bigl(\Lambda^r(\rr^{2*})\,\widehat{\otimes }\,\pi^*\Lambda^q(T_{\rr}^*S) \,\widehat{\otimes}\, \Lambda^p(T_{\rr}^*X_g)\,\widehat{\otimes }\, \Lambda(N_{X_g/X}^{*(0,1)}) \otimes  \xi\bigr)_x$, and if $s$ has compact support, we set 
 \begin{equation}
    |s|_{t,x,0}'^2= \int_{(T_{\rr}X)_x}|s|^2\Bigl(1+|Z|\rho\big(\frac{\sqrt{t}Z}{2}\big)\Bigr)^{2(2\ell'+2\mathrm{dim}S+1-p-q-r)}dv_{TX}(Z).
 \end{equation}

 The techniques in Sections \ref{ei}-\ref{ee} work well for the operator in (\ref{gamma lichnerowicz formula}), with the  trivialization associated to $ ^2\nabla^{\mathbb{E},t}$ and the new norm $|\cdot|_{t,x,0}'$. 
Thus we can establish the analogue of Theorem \ref{thm asymptotic higher akt} by following the lines of the previous sections. In particular, we find that for $K\in\mathfrak{z}(g)\otimes_{\mathbb{R}}\mathbb{C}$ with $|K|$ small enough, as $t\rightarrow 0$,
 \begin{equation}\label{process limit gamma}
    \begin{split}
       \Phi\str\Big[g\exp\Big(-\lk-\bisconn_{K,t}^2-da\sqrt{2}t\frac{\partial}{\partial t}(\bisconn_{K,t}''-\bisconn_{K,t}')\qquad\qquad\qquad\qquad\quad\\
               -\dabar\sqrt{2}t\frac{\partial}{\partial t}(\bisconn_{K,t}'' +\bisconn_{K,t}')
               +N_Vda\dabar\Big)\Big]^{da\dabar}\rightarrow C_{0,g,K}.
    \end{split}
 \end{equation}

 Once again, by using (\ref{nkt rescaling}) and by proceeding as in the previous sections, we obtain (\ref{etatstar}), that is 
    \begin{equation}\label{etatstar2}
        \Phi\frac{1}{2}\frac{\partial^2}{\partial c^2}\str\big[g\exp(-\lk-\bisconn_{K,t}^2+cN_{K,t})\big]_{c=0}=\frac{L_{-2}}{t^2}+\frac{L_{-1}}{t}+L_0+\mathcal{O}(t).
    \end{equation}

   By using (\ref{derivative gamma}) and by comparing the coefficients of  (\ref{expansion gammatgk}), (\ref{process limit gamma}) and (\ref{etatstar2}), we find that
   \begin{equation}
    C_{0,g,K}'=C_{0,g,K}-\overline{\partial}\partial L_0,
   \end{equation}
   and this completes the proof of Theorem \ref{thm asymptotic-gK}.

\subsection{Proof of Theorem \ref{main lemma1}}
We adapt the  notations  in Section \ref{ej}. Set 
\begin{equation}
   {A}_{K,v,t}'=\lk+\cvt^2.
\end{equation}
Since $ v\geq t\geq t_0>0$, we have
\begin{equation}
    0\leq \frac{1}{t}-\frac{1}{v}<\frac{1}{t}.
\end{equation}
Therefore all of the results in Section 6.1 hold true if we  replace $\akt$ by $ {A}_{K,v,t}'$. And we get Theorem \ref{main lemma1}.

\subsection{Proof of Theorem \ref{h,t1infty}}
We only prove the theorem for $\pi_{v,t}^{0\prime}$, the other cases are similar. We keep the notations in Sections \ref{ej} and \ref{section establish analytic torsion}.

For $1\leq t\leq v< +\infty$, set
\begin{equation}\label{definition ckvt}
    C_{K,v,t}'=\lk+\cvt^2+\frac{2t}{v}\dabar \derivative{t}\cvt'=A_{K,v ,t}'+\frac{2t}{v}\dabar \derivative{t}\cvt'.
\end{equation}
From (\ref{definition cvt}), we have
\begin{equation}\label{explicit cvt'}
    2\derivative{t}\cvt'=\frac{1}{\sqrt{t}}\dbarstar+\frac{c(T^{(0,1)})}{2\sqrt{2}t^{3/2}}-\frac{c(K^{X(0,1)})}{2\sqrt{2t}}\Big(\frac{1}{t}+\frac{1}{v}\Big).
\end{equation}

Since $\dbarstar$ is a first order differential operator along the fiber which vanish on $\mathrm{Ker}(D)$, by \eqref{sobolev infty 01}, we see that $\sqrt{t}\,\dbarstar$ is a bounded operator from $\mathbb{E}^1$ to  $\mathbb{E}^0$.  Therefore from (\ref{explicit cvt'}), there is $C>0$, such that for $1\leq t\leq v< +\infty$, $|K|\leq \beta$, $s\in \mathbb{E}$,
\begin{equation}\label{estimate partialt cvt'}
    \Big|\Big(\frac{2t}{v}\derivative{t}\cvt'\Big)s\Big|_{t,0}\leq \frac{C}{\sqrt{v}}|s|_{t,1}.
\end{equation}

Let $C_{K,v,t}^{\prime(0)}$ (resp.$C_{K,v,t}^{\prime(>0)})$ be the zero degree piece (resp. positive degree piece) of $C_{K,v,t}'$ in $\Lambda(\rr^{2*})\,\widehat{\otimes }\, \Lambda(T_{\rr}^*S)$.
From (\ref{estimate partialt cvt'}) and proceed as in Lemmas \ref{elliptic estimate infty}, \ref{resolvent infty}, we see that there are $c,C>0$ such that for $t\geq 1$,  $|K|\leq \beta$, $\lambda\in U_c$, the resolvent $(\lambda-C_{K,v,t}')^{-1}$ exists, and for $s\in E$,
\begin{equation}\label{resolvent ckvt}
    |(\lambda-C_{K,v ,t})^{-1}s|_{t,1}\leq C(1+|\lambda|)^{p}|s|_{t,-1}.
\end{equation}
And we have the following analogue of Proposition \ref{prop akt-bt},
\begin{proposition}\label{prop ckvt-akvt}
    There exists $C>0$, such that for $1\leq t\leq v<+\infty$, $K\in\mathfrak{z}(g)\otimes_{\mathbb{R}}\mathbb{C}$,  $|K|\leq \beta$, $s\in \mathbb{E}$,
    \begin{equation}\label{estimate exp ckvt akvt}
        \|\big(\exp(-C_{K,v,t}')-\exp(-A_{K,v,t}')\big)s\|_0\leq \frac{C}{\sqrt{v}}\|s\|_0.
    \end{equation}
\end{proposition}
\begin{proof}
    By using (\ref{definition ckvt}), (\ref{estimate partialt cvt'}), (\ref{resolvent ckvt}) and the same arguments as in the proof of Proposition \ref{prop akt-bt}, we get (\ref{estimate exp ckvt akvt}).
\end{proof}

Let $\exp(-C_{K,v,t}')(x,x')$ be the smooth kernel of $\exp(-C_{K,v,t}')$ with respect to $dv_X/(2\pi)^{\ell}$.
By using \eqref{explicit cvt'}, we have the following analogue of Proposition \ref{prop uniform estimate akt}.
\begin{proposition}\label{prop uniform estimate ckvt}
    For any $k\in\nn$, there exists $C>0$ such that for  $\alpha, \beta\in\nn^{2\ell}$, $|\alpha|+|\beta|\leq k$,  $t\geq 1$, $K\in\mathfrak{z}(g)\otimes _{\mathbb{R}}\mathbb{C}$,  $|K|\leq \beta$,  $s\in S$, $x,x'\in X_s$,  we have
    \begin{equation}
        \Big|\frac{\partial^{|\alpha|+|\beta|}}{\partial x^{\alpha}\partial x'^{\beta}}\exp(-C_{K,v,t}')(x,x')\Big|\leq C.
    \end{equation}
\end{proposition}

From Propositions \ref{prop ckvt-akvt}, \ref{prop uniform estimate ckvt} and using again the same arguments in  \cite[\S 11 p)]{BL}, we know that there are $C>0$, $\delta>0$, such that for  $1\leq t\leq v<+\infty$, $K\in\mathfrak{z}(g)\otimes_{\mathbb{R}}\mathbb{C}$,  $|K|\leq \beta$,  $s\in S$, $x,x'\in X_s$,
\begin{equation}\label{kernel ckvt-akvt}
    \big|\exp(-C_{K,v,t}')(x,x')-\exp(-A_{K,v,t}')(x,x')\big|\leq \frac{C}{v^{\delta}}.
\end{equation}

From (\ref{pivt}), (\ref{definition ckvt}), we know
\begin{equation}\label{dabar wedge pi}
    \dabar\wedge \pi_{v,t}^{0\prime}=\str\big[i\product{\mu,K}g\bigl(\exp(-C_{K,v,t}')-\exp(-A_{K,v,t}')\bigr)\big].
\end{equation}
By (\ref{kernel ckvt-akvt}), (\ref{dabar wedge pi}), we get Theorem \ref{h,t1infty} for $h_{v,t}=\pi_{v,t}^{0\prime}$.

\section{The analysis of the two parameter operator $\exp(-\lk-\cvt^2)$}\label{ek}
The aim of this section is to prove Theorems \ref{main lemma 1}, \ref{h,t01} and  \ref{thm,h2,t01v01} by using the  techniques in \cite[\S 8]{BismutGoette2000} and \cite[\S 6]{LiuMa2022}. This section is organized as follows. In Section \ref{dv}, we establish a Lichnerowicz formula. In Section \ref{sec compute to ev}, we prove \eqref{limit of main}. In Sections \ref{sec localization of problem mu1}-\ref{sec mu2 control}, we prove \eqref{t01v01}. In Section \ref{dw}, we prove \eqref{t01v1}. In Section \ref{dx}, we prove Theorem \ref{h,t01}. In Section \ref{dy}, we prove Theorem  \ref{thm,h2,t01v01}. In this section, we use the assumptions and the notations in Section \ref{dz}.

Recall that
\begin{equation}\label{redef cvt}
    \cvt=\bisconn_{K,t}-\frac{\sqrt{t}c(K^{X})}{2\sqrt{2}v}.
\end{equation}

\subsection{A Lichnerowicz formula}\label{dv}

Let $K\in\mathfrak{g}$, $v>0$.
Let $L$ be a trivial line bundle over $W$, equipped with the curvature $\nabla^L_{v}$ which is given by
\begin{equation}\label{def of nabla l}
    \nabla^L_{v}=d-\frac{K^{\prime X}}{2v}.
\end{equation}
Since $K^{X}\in T_{\mathbb{R}}X$, for $U\in T_{\rr}S$,
\begin{equation}\label{L horizontal connection}
    \nabla^{\Lambda(T^{*(0,1)}X)\otimes \xi\otimes L}_{U^H}=\nabla^{\Lambda(T^{*(0,1)}X)\otimes \xi}_{U^H}
\end{equation}
Observe that $\sqrt{2}(D^X-\frac{c(K^{X})}{2\sqrt{2}v})$ is the standard Dirac operator on $\Lambda(T^{*(0,1)}X)\otimes \xi\otimes L$ associated to $\nabla^{TX}$, $\nabla^{\xi}$ and $\nabla^L_{v}$. From (\ref{L horizontal connection}), the new Bismut superconnection associated with $\Lambda(T^{*(0,1)}X)\otimes \xi\otimes L$ is
\begin{equation}\label{def bisconntv}
    \bisconn_{t}^v=\bisconn_{t}-\frac{\sqrt{t}c(K^{X})}{2\sqrt{2}v}.
\end{equation}

In the sense of (\ref{redef cvt}) and (\ref{def bisconntv}), the following result is obtained from  (\ref{lichnerowicz formula}) with $\xi$ replaced by $\xi\otimes L$.  We refer to \cite[Theorem 6.1]{LiuMa2022} for more details of the proof.
\begin{theorem}\label{thm lichnerowicz avt}
    The following identity holds,
    \begin{equation}\label{lichnerowicz avt}
        \begin{split}
            \lk+\cvt^2=-\frac{t}{2}\Big(\nabla_{e_i}^{\Lambda(T^{*(0,1)}X)\otimes \xi}+\frac{1}{\sqrt{2t}}\product{S(e_i)e_j, f_p^H}c(e_j)f^p\wedge\qquad\qquad\\
             +\frac{1}{2t}\product{S(e_i)f_p^H,f_q^H}f^p\wedge f^q\wedge -\frac{\product{K^X,e_i}}{2}\Bigl(\frac{1}{t}+\frac{1}{v}\Bigr)\Big)^2\\ 
            +\frac{t}{8}H+\frac{t}{4}\bigl(\bigl(R^{\xi}+\frac{1}{2}\trace[R^{TX}]\bigr)(e_i,e_j)-\frac{1}{v}\product{\nabla_{e_i}^{TX}K^X,e_j}\bigr)c(e_i)c(e_j)\\ 
            +\sqrt{\frac{t}{2}}\bigl(\bigl(R^{\xi}+\frac{1}{2}\trace[R^{TX}]\bigr)(e_i,f_p^H)-\frac{1}{v}\product{T(e_i,f_p^H),K^X}\bigr)c(e_i)f^p\\ 
            +\frac{1}{2}\bigl(\bigl(R^{\xi}+\frac{1}{2}\trace[R^{TX}]\bigr)(f_p^H,f_q^H)-\frac{1}{2v}\product{T(f_p^H,f_q^H),K^X}\bigr)f^p\wedge f^q\wedge\\ 
            -m^{\xi}(K)-\frac{1}{2}\trace[m^{TX}(K)]+\frac{1}{2v}|K^X|^2.
        \end{split}
    \end{equation}
\end{theorem}

\subsection{A proof of the convergence of the supertrace as $t\rightarrow 0$.}\label{sec compute to ev}
From (\ref{def of nabla l}), 
 the corresponding curvature associated to $\nabla^L_{v}$ is given by
\begin{equation}
    R^L=(\nabla^L_{v})^2=-\frac{dK^{\prime X}}{2v}.
\end{equation}
Since $G$ acts on $L$ trivially, then the corresponding $m^L(K)$ in the sense of (\ref{moment}) is given by
\begin{equation}
    m^L(K)=-\frac{|K^X|^2}{2v}.
\end{equation}
Then by (\ref{genus-gK}), (\ref{dkkprime}),
\begin{equation}\label{chern l}
    \ch_{g,K}(L,\nabla^L_{v})=\exp(-\frac{|K^X|^2}{2v}+\frac{dK^{\prime X}}{4i\pi v})=\exp\Bigl(\frac{\overline{\partial}_K\partial_K}{2\pi v}\product{\mu,K}\Bigr).
\end{equation}

Recall that $\widetilde{e}_v$ is defined in (\ref{ev formula}). By using Theorem \ref{thm lichnerowicz avt} and by proceeding as in Section \ref{e0}, especially as in (\ref{compute c-1gk}) and (\ref{compute c-1gk 2}), we see that for $K\in\mathfrak{z}(g)$ with $|K|$ sufficiently small, as $t\rightarrow 0$,
\begin{equation}\label{compute tend to ev}
    \begin{split}
    &\Phi\str\big[i\frac{\product{\mu,K}}{v}g\exp(-\lk-\cvt^2)\big]\\ 
    &\quad\qquad\rightarrow\int_{X_g}i\frac{\product{\mu,K}}{v}\td_{g,K}(TX,h^{TX})\ch_{g,K}(\xi,h^{\xi})\ch_{g,K}(L,\nabla^L_{v})=-\widetilde{e}_v.
    \end{split}
\end{equation}

\subsection{Localization of the problem}\label{sec localization of problem mu1}
For $t>0$, $v>0$, set
\begin{equation}\label{eo}
    \mathcal{B}_{K,t,v}=\lk+\cvt^2.
\end{equation}

As in \cite[\S 8.3]{BismutGoette2000}, the proof of (\ref{t01v01}) is split into two parts. Recall that $\product{\mu,K}$ is split into $\mu_1+\mu_2$ in Section 5.1. For $v>0$, set
\begin{equation}\label{en}
    \begin{split}
        \widetilde{e}_v^k=-\int_{X_g}\frac{\mu_k}{v}\td_{g,K}(TX,h^{TX})\ch_{g,K}(\xi,h^{\xi})\exp\Big(\frac{\bar{\partial}_K\partial_K}{2\pi v}\product{\mu,K}\Big), \qquad {k=1,2}.
    \end{split}
\end{equation}  

It is clear that $\widetilde{e}_v=\widetilde{e}_v^1+\widetilde{e}_v^2$. We will establish the following two estimates, for $K_0\in \mathfrak{z}(g)$, $K=zK_0\neq 0$ with $|z|$ small enough, $0<t\leq v\leq 1$,
\begin{equation}\label{split of mu bktv}
    \begin{split}
      &\Big|\Phi\str\Big[i\frac{\mu_1}{v}g\exp\Big(-\bktv\Big)\Big]+\widetilde{e}_v^1\Big|\leq C\Bigl(\frac{t}{v}\Bigr)^{\gamma},\\ 
      &\Big|\Phi\str\Big[i\frac{\mu_2}{v}g\exp\Big(-\bktv\Big)\Big]+\widetilde{e}_v^2-\Phi\partial \xi_{v,t}'-\Phi\bar{\partial}\xi_{v,t}''\Big|\leq C\Bigl(\frac{t}{v}\Bigr)^{\gamma}.
    \end{split}
\end{equation}

By omitting the $dt$ part in \cite[Theorem 6.2]{LiuMa2022}, we have the following analogue of \cite[Theorem 8.6]{BismutGoette2000}.
\begin{lemma}
    There exist $\beta>0$, $c,C>0$ such that for $K\in\mathfrak{g}$, $|K|\leq \beta$, $t\in \,]0,1]$, $v\in [t,1]$, 
    \begin{equation}\label{I(bktv)}
        \|\widetilde{I}_t(t\mathcal{B}_{K,t,v})\|_{(1)}\leq c\exp(-C/t).
    \end{equation}
\end{lemma}
So from (\ref{F+I(akt)}) and (\ref{I(bktv)}), we may replace $\exp(-\bktv)$ by $\widetilde{F}_{t}(\bktv)$ in the proof of (\ref{split of mu bktv}). Thus the analysis of $\widetilde{F}_t(\bktv)$  can be localized near $X_g$.

\subsection{A rescaling of the normal coordinate to $X_{g,K}$ in $X_g$}\label{eg}
In the sequel, we fix $g\in G$, $K_0\neq 0\in \mathfrak{z}(g)$, and take $K=zK_0$ with $z\neq 0\in \rr$.

Recall that $X_g$ and $X_{g,K}=X_g\cap X_K$ are totally geodesic in $X$. Given $\varepsilon>0$, let $\uprime'$ be the $\varepsilon$-neighborhood of $X_{g,K}$ in $N_{X_{g,K}/X_g,\rr}$. By zooming out $\varepsilon_0\in \,]0,a_X/32]$ in Section 6.3, we may assume that the map $(y_0,Z_0)\in N_{X_{g,K}/X_g,\rr}\rightarrow \exp_{y_0}^{X_g}(Z_0)\in X_g$ is a diffeomorphism from $\uprime'$ into the turbular neighborhood $\mathcal{V}_{\varepsilon}''$ of $X_{g,K}$ in $X_g$ for any $0<\varepsilon\leq 16\varepsilon_0$.

Since $X_g$ is totally geodesic in $X$, the connection $\nabla^{TX}$ induces the holomorphic Hermitian connection $\nabla^{N_{X_g/X}}$ on $N_{X_g/X}$.

If $(y_0,Z_0)$ is taken as before, we identity $N_{X_g/X,(y_0,Z_0)}$ with $N_{X_g/X,y_0}$ by parallel transport along the geodesic $s\in [0,1]\rightarrow sZ_0$ with respect to $\nabla^{N_{X_g/X}}$. If $y_0\in X_{g,K}$, $Z_0\in N_{X_{g,K}/X_g,\rr,y_0}$, $Z\in N_{X_g/X,\rr,y_0}$, we identify $(y_0,Z_0,Z)$ with $\exp_{\exp_{y_0}^{X_g}(Z_0)}^X(Z)\in X$. Therefore $(y_0, Z_0,Z)$ defines a coordinate system on $X$ near $X_{g,K}$.

Recall that the function $k$ was defined in (\ref{def of k}), and $\widetilde{d}_{v}^1$ defined in (\ref{def of dvt}). Set $\ell''=\mathrm{dim}X_{g,K}$.  Put 
\begin{equation}
    \beta_v^1=(2\pi)^{\mathrm{dim}X_g}\big[\td_{g,K}(TX,h^{TX})\ch_{g,K}(\xi,h^{\xi})\widetilde{d}_v^1\big]^{\mathrm{max}}.
\end{equation}

The following result is an important step to the first inequality in (\ref{split of mu bktv}).

\begin{theorem}\label{thm first ineq of mu}
    Take $K_{0}\neq 0\in \mathfrak{z}(g)$. There exist $\beta\in \,]0,1]$, $\gamma\in \,]0,1]$ such that for $p\in\nn$, there is $C>0$ such that if $z\in\rr$, $0<|z|\leq \beta$,  $0<t\leq v\leq 1$, $y_0\in X_{g,K}$, $Z_0\in N_{X_{g,K}/X_g,\rr,y_0}$, $|Z_0|\leq \varepsilon_0/\sqrt{v}$, then
    \begin{equation}\label{first ineq of mu}
        \begin{split}
        \Big|v^{\mathrm{dim}N_{X_{g,K}/X_g}}\Bigl(\Phi\int_{Z\in N_{X_g/X,\rr,y_0}\atop |Z|\leq \varepsilon_0}\str\Big[i\frac{\mu_1}{v}g\widetilde{F}_t(\mathcal{B}_{zK_{0},t,v})(g^{-1}(y_0,\sqrt{v}Z_0,Z), \qquad\\ (y_0,\sqrt{v}Z_0,Z))\Big]k(y_0,\sqrt{v}Z_0,Z)\frac{dv_{N_{X_g/X}(Z)}}{(2\pi)^{\mathrm{dim}N_{X_g/X}}}+\beta_{v}^1(y_0,\sqrt{v}Z_0)\Bigr)\Big|\\ 
        \leq C\frac{(1+|Z_0|)^{2\ell'+1}}{(1+|zZ_0|)^p}\Bigl(\frac{t}{v}\Bigr)^{\gamma}.
        \end{split}
    \end{equation}
\end{theorem}
\begin{proof}
    Sections \ref{sec A new trivialization and Getzler rescaling}-\ref{e1} are devoted to the proof of Theorem \ref{thm first ineq of mu}.
\end{proof}

\subsection{A trivialization and the Getzler rescaling}\label{sec getzler rescaling v}\label{sec A new trivialization and Getzler rescaling}
We use the notations in Sections \ref{ef}-\ref{sec getzler rescaling}.
Compared with $ ^1\nabla^{\mathbb{E},t}_{\cdot}$ in (\ref{1nablat}), we set
\begin{equation}
    \begin{split}
        ^3\nabla^{\mathbb{E},t}_{e_i}=\nabla_{e_i}^{\mathbb{E}}+\frac{1}{\sqrt{2t}}\product{S(e_i)e_j, f_p^H}c(e_j)f^p\wedge\qquad\qquad\qquad\\
        +\frac{1}{2t}\product{S(e_i)f_p^H,f_q^H}f^p\wedge f^q\wedge -\frac{\product{K^X,e_i}}{2}\Bigl(\frac{1}{t}+\frac{1}{v}\Bigr).
    \end{split}
\end{equation}

We use the coordinate system $(y_0, Z_0, Z)$ of $X$ near $X_{g,K}$ in the previous section.
Since $g$ preserves geodesics and parallel transports,  we have
\begin{equation}
    g(y_0,Z_0,Z)=(y_0,Z_0,gZ).
\end{equation}

Fix $y_0\in X_{g,K}$, if $Z_0\in N_{X_{g,K}/X_g,\rr,y_0}$, $|Z_0|\leq \varepsilon_0$, we identify $T_{Z_0}X$ with $T_{y_0}X$ by parallel transport along the geodesic $s\in [0,1]\rightarrow sZ_0\in X_g$ with respect to $\nabla^{TX}$. By an abuse of notations, if $Z\in (T_{\rr}X)_{y_0}$, we will write $Z_0+Z$ as $\exp_{Z_0}^{X}(Z).$

Now we fix $Z_0\in N_{X_{g,K}/X_g,\rr,y_0}$, $|Z_0|\leq \varepsilon_0$, and we take $Z\in (T_{\rr}X)_{y_0}$, $|Z|\leq 4\varepsilon_0.$ The curve $s\in [0,1]\rightarrow \exp_{Z_0}^X(sZ)$ lies in $B_{y_0}^X(0,5\varepsilon)$. Moreover we identify $TX_{Z_0+Z}$, $\bigl(\pi^*\Lambda(T_{\rr}^*S)\,\widehat{\otimes }\,\Lambda(T^{*(0,1)}X)\otimes \xi\bigr)_{Z_0+Z}$ with $TX_{Z_0}$, $\bigl(\pi^*\Lambda(T_{\rr}^*S)\,\widehat{\otimes }\,\Lambda(T^{*(0,1)}X)\otimes \xi\bigr)_{Z_0}$  by parallel transport with respect to the connections $\nabla^{TX}$, $^3\nabla^{\mathbb{E}}$ along  the curve.

When $Z_0\in N_{X_{g,K}/X_g,\rr,y_0}$ is allowed to vary, we identify $TX_{Z_0}$, $\big(\Lambda(T^{*(0,1)}X)\big)_{Z_0}$,  $\xi_{Z_0}$ with $TX_{y_0}$, $\big(\Lambda(T^{*(0,1)}X)\big)_{y_0}$,  $\xi_{y_0}$  by parallel transport along $s\in [0,1]\rightarrow sZ_0$ with respect to $\nabla^{TX}$, $\nabla^{\Lambda(T^{*(0,1)}X)}$, $\nabla^{\xi}$. Then $\mathbf{H}_{Z_0}$ is identified with $\mathbf{H}_{y_0}$ isometrically.

For $Z_0\in N_{X_{g,K}/X_g,\rr,y_0}$, $|Z_0|\leq \varepsilon_0$, the considered trivialization depend explicitly on $Z_0$. We denote by $(\bktv)_{Z_0}$ the action of $\bktv$ centered at $Z_0$. The operator $(\bktv)_{Z_0}$ thus acts on $\mathbf{H}_{Z_0}$, by the identification of $\mathbf{H}_{Z_0}$ with $\mathbf{H}_{y_0}$, $(\bktv)_{Z_0}$ acts on $\mathbf{H}_{y_0}$.

We may and we will assume that $\varepsilon_0$ is small enough so that if $|Z_0|\leq \varepsilon_0,$ $|Z|\leq 4\varepsilon_0$, then
\begin{equation}
    \frac{1}{2}g_{y_0}^{TX}\leq g_{Z_0+Z}^{TX}\leq \frac{3}{2}g_{y_0}^{TX}.
\end{equation}

\begin{definition}
    Let $L_{Z_0,K}^{1,(t,v)}$ be the differential operator acting on $H_{y_0}$,
    \begin{equation}\label{lzk1tv}
        L_{Z_0,K}^{1,(t,v)}=-(1-\rho^2(Z))\frac{t}{2}\Delta^{TX}+\rho^2(Z)(\bktv)_{Z_0}.
    \end{equation}
\end{definition}

As in (\ref{finite propagation ftakt=ftl}), we find that if $Z_0\in N_{X_{g,K}/X_g,\rr,y_0}$, $Z\in N_{X_g/X,\rr,y_0}$, $|Z_0|, |Z|\leq \varepsilon_0$,
\begin{equation}
    \widetilde{F}_t(\bktv)\big(g^{-1}(Z_0,Z), (Z_0,Z)\big)k_{(y_0,Z_0)}'(Z)=\ft(L_{Z_0,K}^{1,(t,v)})(g^{-1}Z,Z).
\end{equation} 

As in (\ref{lxk2t}), let $L_{Z_0,K}^{2,(t,v)}$ be the differential operator acting on $\mathbf{H}_{y_0}$ defined by 
\begin{equation}
    L_{Z_0,K}^{2,(t,v)}=H_t^{-1}L_{Z_0,K}^{1,(t,v)}H_t.
\end{equation}

Let $(e_1,...,e_{2\ell''})$, $(e_{2\ell''+1},...,e_{2\ell'})$ and $(e_{2\ell'+1},...,e_{2\ell})$ be orthonormal oriented bases of $(T_{\rr}X_{g,K})_{y_0}$, $N_{X_{g,K}/X_g,\rr,y_0}$, $N_{X_g/X,\rr,y_0}$ respectively.
\begin{definition}\label{def lzk3tv}
    Let $L_{Z_0,K}^{3,(t,v)}$ be differential operator acting on $H_{y_0}$ obtained from $L_{Z_0,K}^{2,(t,v)}$ by replacing $c(e_j)$ by $c_t(e_j)$ for $1\leq j\leq 2\ell''$, by $c_{t/v}(e_j)$ for $2\ell''+1\leq j\leq 2\ell'$, while leaving unchanged for $2\ell'+1\leq j\leq 2\ell$.
\end{definition}

We denote by $\ft(L_{Z_0,K}^{3,(t,v)})(Z,Z')$ the smooth kernel associated to the operator $\ft(L_{Z_0,K}^{3,(t,v)})$ with respect to $dv_{TX}(Z')/(2\pi)^{\mathrm{dim}X}$.
As an analogue of \cite[Proposition 8.11]{BismutGoette2000}, if $Z_0\in /X_g$, $|Z_0|\leq \varepsilon_0$, $Z\in N_{X_g/X,\rr,y_0}$, $ |Z|\leq \varepsilon_0/\sqrt{t}$, we have
\begin{equation}\label{str getzler rescaling v}
    \begin{split}
    t^{\mathrm{dim}N_{X_g/X}}\str\Big[i\frac{\mu_1}{v}g\ft(\bktv)\bigl(g^{-1}(Z_0,\sqrt{t}Z), (Z_0,\sqrt{t}Z)\bigr)\Big]k_{(y_0,Z_0)}'(\sqrt{t}Z)\\ 
    =\frac{(-i)^{\ell'}}{v^{\ell'-\ell''}}\str^{\Lambda(N_{X_g/X}^{*(0,1)})\otimes \xi}\Big[i\frac{\mu_1(Z_0,\sqrt{t}Z)}{v}g\ft(\lzktv)(g^{-1}Z,Z)\Big]^{\mathrm{max}}.
    \end{split}
\end{equation}

\begin{definition}
    If $x\in X_g$, let $\overline{L}_{x,K}^{3,(0,v)}$ be the operator in $\bigl(\pi^*\Lambda(T_{\rr}^*S)\,\widehat{\otimes }\,\Lambda(T_{\rr}^*X_g)\,\widehat{\otimes }\, c(N_{X_g/X,\rr}) \otimes \ENDO(\xi)\bigr)_{x}\otimes \mathrm{Op}_{x}$ given by
    \begin{equation}\label{lxk30vbar}
        \begin{split}
        \overline{L}_{x,K}^{3,(0,v)}=-\frac{1}{2}\bigl(\nabla_{e_i}+\frac{1}{2}\product{(\jmath^* R^{TX}-m^{TX}(K))Z,e_i}\bigr)^2-\frac{1}{2v}(d-i_{K^X})\jmath^* K^{'X}\\+\jmath^* R_x^{\xi}-m^{\xi}(K)_x
        +\frac{1}{2}\trace[\jmath^* R^{TX}-m^{TX}(K)].
        \end{split}
    \end{equation}
\end{definition}

Let $\psi_v\in \ENDO(\Lambda(T_{\rr}^*X_g))$ be the morphism of exterior algebras such that $\psi_v(e^j)=e^j$ for $1\leq j\leq 2\ell''$, $\psi_v(e^j)=\sqrt{v}e^j$ for $2\ell''+1\leq j\leq 2\ell'$.
\begin{definition}
    Let $L_{Z_0,K}^{3,(0,v)}$ be the operator
    \begin{equation}\label{lxk0v}
        L_{Z_0,K}^{3,(0,v)}=\psi_v\overline{L}_{Z_0,K}^{3,(0,v)}\psi_v^{-1}.
    \end{equation}
\end{definition}

By proceeding as in \cite[(6.27)-(6.29)]{LiuMa2022}, we see that as $t\rightarrow 0$,
\begin{equation}\label{limit to lxk0v}
    L_{Z_0,K}^{3,(t,v)}\rightarrow L_{Z_0,K}^{3,(0,v)}
    .
\end{equation}

\subsection{A proof of Theorem \ref{thm first ineq of mu}}\label{eh}
The following uniform estimate is a family version of \cite[Theorem 8.23]{BismutGoette2000}.
\begin{theorem}\label{thm uniform estimate ftlvzk3tv}
    There exist $C'>0$, $C'' >0$, $C'''>0$ such that for $\eta\in \,]0,1]$ small enough, there are $c_{\eta}>0$, such that for $z\in\rr$, $0<|z|\leq c_{\eta}$,  for $m,m'\in\nn$, there is $C>0$, $r\in\nn$ such that for $0<t\leq v\leq 1$, $n\in\nn$, $y_0\in W_{g,K}$, $Z_0\in N_{X_{g,K}/X_g,\rr,y_0}$, $|Z_0|\leq \varepsilon_0/\sqrt{v}$, $Z,Z'\in (T_{\rr}X)_{y_0}$, $|Z|,|Z'|\leq \varepsilon_0/\sqrt{t}$,
    \begin{equation}\label{uniform estimate ftlvzk3tv}
        \begin{split}
            (1+|zZ_0|)^m\sup_{|\alpha|,|\alpha'|\leq m'}\Big|\frac{\partial^{|\alpha|+|\alpha'|}}{\partial Z^{\alpha}\partial Z^{'\alpha'}}\widetilde{F}_{t,n}(L_{\sqrt{v}Z_0,zK_0}^{3,(t,v)})(Z,Z')\Big|\qquad\qquad\\ 
            \leq C(1+|Z|+|Z'|)^r(1+|Z_0|)^{2\ell''+2\mathrm{dim}S}\hspace{8 em}\\ 
            \exp(-C'n^2/4+2C''\eta^2\sup(|Z|^2,|Z'|^2)-C'''|Z-Z'|^2).
        \end{split}
    \end{equation}
\end{theorem}
\begin{proof}
    The proof of Theorem \ref{thm uniform estimate ftlvzk3tv} is explained in \cite[\S 6.7]{LiuMa2022}.
\end{proof}

We fix $\eta>0$ small enough as in (\ref{distance gz z}), (\ref{fix small eta}). Then from (\ref{as infinite sum}), (\ref{fix small eta}) and (\ref{uniform estimate ftlvzk3tv}), we see that for $Z_0\in N_{X_{g,K}/X_g,\rr,y_0}$, $|Z_0|\leq \varepsilon_0/\sqrt{v}$, $ Z\in N_{X_g/X,\rr,y_0}$, $|Z|\leq \varepsilon_0/\sqrt{t}$,
\begin{equation}\label{uniform estimate ftlvzk3tv g}
    \begin{split}
 \big|\ft(L_{\sqrt{v}Z_0,zK_0}^{3,(t,v)})(g^{-1}Z,Z)\big|
    \leq C\frac{(1+|Z_0|)^{2\ell''+2\mathrm{dim}S}}{ (1+|zZ_0|)^m}
    \exp(-\frac{C'''}{4}|Z|^2).
    \end{split}
\end{equation}

Let $\exp(-L_{Z_0,K}^{3,(0,v)})(Z,Z')$ be the heat kernel of the operator $L_{Z_0,K}^{3,(0,v)}$ in the sense of (\ref{def of fr}) with $r=0$.
Now we recall an important estimate in \cite[Theorem 6.13]{LiuMa2022}, which is a family version of \cite[(8.76)]{BismutGoette2000}.
\begin{theorem}
    There exist $C>0$, $C'>0$ such that given $m>0$, there exists $\beta>0$ such that if $t\in \,]0,1]$, $v\in [t,1]$, $z\in\rr$, $0<|z|\leq \beta$,  $y_0\in W_{g,K}$, $Z_0\in N_{X_{g,K}/X_g,\rr,y_0}$, $|Z_0|\leq \varepsilon_0/\sqrt{v}$, $Z\in N_{X_g/X,\rr,y_0}$, $|Z|\leq \varepsilon_0/\sqrt{t}$,
    \begin{equation}\label{difference ftlvzk3tv}
        \begin{split}
        \big|\bigl(\ft(L_{\sqrt{v}Z_0,zK_0}^{3,(t,v)})-\exp(-L_{\sqrt{v}Z_0,zK_0}^{3,(0,v)})\bigr)(g^{-1}Z,Z)\big|\qquad\qquad\qquad\qquad\\ 
           \leq C\Bigl(\frac{t}{v}\Bigr)^{1/4(2\ell+1)}\cdot\frac{(1+|Z_0|)^{2\ell''+2\mathrm{dim}S+1}}{(1+|zZ_0|)^m}\exp(-C'|Z|^2/4).
        \end{split}
    \end{equation}
\end{theorem}

From (\ref{proposition of mu1mu2}), we see that for $|z|\leq \beta$, $Z_0\in N_{X_{g,K}/X_g,\rr,y_0}$, $|Z_0|\leq \varepsilon_0/\sqrt{v}$, $Z\in N_{X_g/X,\rr,y_0}$, $|Z|\leq \varepsilon_0/\sqrt{t}$,
\begin{equation}\label{estimate mu1}
    \begin{split}
        &\Big|\frac{\mu_1(\sqrt{v}Z_0+\sqrt{t}Z)-\mu_1(\sqrt{v}Z_0)}{v}\Big|\leq C\Bigl(\frac{t}{v}|Z|^2+\Bigl(\frac{t}{v}|Z||Z_0|\Bigr)\Bigr),\\ 
        &\Big|\frac{\mu_1(\sqrt{v}Z_0)}{v}\Big|\leq C|Z_0|^2.
    \end{split}
\end{equation}

From the same arguments as in \cite[(7.223)-(7.225)]{BismutGoette2000}, and by using (\ref{lxk30vbar}) and (\ref{lxk0v}), we get
\begin{equation}\label{calculate exp -lzk30v}
    \begin{split}
        \Phi\int_{N_{X_g/X}}\str^{\Lambda(N_{X_g/X}^{*(0,1)})\otimes \xi}\Big[i\frac{\mu_1(Z_0,\sqrt{t}Z)}{v}g\exp(-L_{Z_0,zK_{0}}^{3,(0,v)})(g^{-1}Z,Z)\Big]^{\mathrm{max}}\\ \cdot\frac{dv_{N_{X_g/X}(Z)}}{(2\pi)^{\mathrm{dim}N_{X_g/X}}}
        =v^{\mathrm{dim}N_{X_{g,K}/X_g}}\beta_v^1.
    \end{split}
\end{equation}
From (\ref{str getzler rescaling v}), (\ref{uniform estimate ftlvzk3tv g})-(\ref{calculate exp -lzk30v}), we get (\ref{first ineq of mu}). The proof of Theorem \ref{thm first ineq of mu} is completed.

\subsection{The analysis of $\ft(\bktv)$ on $X_g\backslash \mathcal{U}_{\varepsilon_0}''$}\label{sec outside neighbor}
Recall that $\mathcal{U}_{\varepsilon_0}''\in N_{X_{g}/X,\rr}$ is identified to a turbular neighborhood $\mathcal{V}_{\varepsilon_0}''$ of $X_{g,K}$ in $X_g$ by sending $(x,Z)\in N_{X_g/X,\rr}$ to $\exp_x(Z)\in X$.   This subsection is devoted to the proof of the following theorem.

\begin{theorem}\label{thm outside uepsilon''}
    There exist $\beta\in \,]0,1]$, $\gamma\in \,]0,1]$ such that for $p\in\nn$, there is $C>0$ such that if $z\in\rr$, $0<|z|\leq \beta$, $t\in \,]0,1]$, $v\in [t,1]$, $x\in X_{g}\backslash \mathcal{U}_{\varepsilon_0}''$,  then
    \begin{equation}\label{outside uepsilon''}
        \begin{split}
            v^{\mathrm{dim}W_g+1} \Big|\Phi\int_{Z\in N_{X_g/X,\rr,x}\atop |Z|\leq \varepsilon_0}\str\Big[i\frac{\mu_1}{v}g\widetilde{F}_t(\mathcal{B}_{zK_{0},t,v})(g^{-1}(x,Z), (x,Z))\Big]k(x,Z) \qquad\\
        \cdot\frac{dv_{N_{X_g/X}(Z)}}{(2\pi)^{\mathrm{dim}N_{X_g/X}}} +\beta_{v}^1(x)\Big|
        \leq C\Big|\frac{v}{z^2}\Big|^p\Bigl(\frac{t}{v}\Bigr)^{\gamma}.
        \end{split}
    \end{equation}
\end{theorem}

In principal, the proof of Theorem \ref{thm outside uepsilon''} is proceeded by applying the techniques of Sections \ref{eg}-\ref{eh} to the case $X_{g,K}=\emptyset$.
Fix $x\in X_g\backslash \mathcal{U}_{\varepsilon_0}''$, the fiberwise exponential map $Z\in T_{x}X\rightarrow \exp_x^X(Z)\in X$ identifies $B^{T_{x}X}(0,16\varepsilon_0)$ with $B^X(x,16\varepsilon_0)\subset X_g$. We trivialize $\pi^*\Lambda(T_{\rr}^*S)\,\widehat{\otimes}\,\Lambda(T^{*(0,1)}X)\otimes \xi$ by parallel transport along $s\in [0,1]\rightarrow sZ$ with respect to the connection $^3\nabla^{\mathbb{E},t}$.

We set $L_{x,K}^{1,(t,v)}$ as in (\ref{lzk1tv}). Let ${L}_{x,K}^{3,(t,v)}$ be the operator obtained from $L_{x,K}^{1,(t,v)}$ by using the rescaling in Section \ref{sec getzler rescaling v} to the case $X_{g,K}=\emptyset$. In particular, the Clifford variables  $c(e_i)$ are now replaced by $c_{t/v}(e_i)$ for $1\leq i\leq 2\ell'$, while leaving unchanged for  $2\ell'+1\leq i\leq 2\ell$. Recall that ${L}_{x,K}^{3,(0,v)}$ is defined in (\ref{lxk30vbar}), it is clear to have
\begin{equation}
    \begin{split}
        {L}_{x,K}^{3,(t,v)}\rightarrow {L}_{x,K}^{3,(0,v)}
        .
    \end{split}
\end{equation}\

For $t\geq 0$, $v\geq t$, $n\in\nn$, set
\begin{equation}\label{lxkn3tv}
    {L}_{x,K,n}^{3,(t,v)}=-\Bigl(1-\gamma^2\Bigl(\frac{|Z|}{2(n+2)}\Bigr)\Bigr)\frac{\Delta^{TX}}{2}+\gamma^2\Bigl(\frac{|Z|}{2(n+2)}\Bigr){L}_{x,K}^{3,(t,v)}.
\end{equation}

Recall that $\mathbf{I}_{(r,q),x}$ is the  space of smooth sections of $\bigl(\pi^*\Lambda^r(T_{\rr}^*S)\,\widehat{\otimes }\,\Lambda^q(T_{\rr}^*X_g)\otimes \Lambda(N_{X_g/X}^{*(0,1)}) \otimes \xi\bigr)_x$ on $(T_{\rr}X)_x$.
If $s\in \mathbf{I}_{(r,q),x}$ has compact support, put 
\begin{equation}\label{sobolev norm tvnk}
    \begin{split}
   & |s|_{t,v,n,0}^2=\int_{(T_{\rr}X)_x}|s|^2\Bigl(1+\sqrt{v}|Z|\rho\big(\frac{\sqrt{t}Z}{2}\big)\Bigr)^{2(2\ell'-q)}\\ 
   &\qquad\qquad\qquad\qquad \cdot\Bigl(\frac{1}{\sqrt{v}}\big(1+|Z|\rho\big(\frac{\sqrt{t}Z}{2}\big)\big)\Bigr)^{2(2\mathrm{dim}S-r)}dv_{TX}(Z),\\ 
    &|s|_{t,v,n,1}^2=|s|_{t,v,n,0}^2+\sum_{1}^{2\ell}|\nabla_{e_i}s|_{t,v,n,0}^2\\&\qquad\qquad\qquad\qquad\qquad+\frac{1}{v}\big|\rho(\sqrt{t}Z)\gamma^2\Bigl(\frac{|Z|}{2(n+2)}\Bigr)|K^X|(\sqrt{t}Z)s\big|_{t,v,n,0}^2.
    \end{split}
\end{equation}
The corresponding Hermitian product associated to $|\cdot|_{t,v,n,0}$ is denoted by $\product{\,,\,}_{t,v,n,0}$. 

Recall that the Hermitian vector space $(J_{x,m}^0,|\cdot|_0)$ before (\ref{comparision normal special norm}). By (\ref{sobolev norm tvnk}), for $s\in J_{x,m}^0$, $v\leq 1$,
\begin{equation}\label{comparision normal special norm,v}
    |s|_0\leq |s|_{t,v,n,0}\leq \frac{C}{v^{\mathrm{dim}S}}(1+m)^{2\ell'}|s|_{0}.
\end{equation}

\begin{theorem}
    There exist constants $C_i>0$, $i=1,2,3,4$, such that if $t\in \,]0,1]$, $v\in [t,1]$, $n\in\nn$, $z\in\rr$, $|z|\leq 1$, if $n\in \nn$,  if  $s,s'\in\mathbf{I}_{x}$ have compact support, then
\begin{equation}\label{elliptic estimate lxkn3tv}
    \begin{split}
        &\mathrm{Re}\product{L_{x,zK_0,n}^{3,(t,v)}s,s}_{t,v,n,0}\geq C_1|s|_{t,v,n,1}^2-C_2(1+|nz|^2)|s|_{t,v,n,0}^2,\\ 
        &|\mathrm{Im}\product{L_{x,zK_0,n}^{3,(t,v)}s,s}_{t,v,n,0}|\leq C_3\bigl((1+|nz|)|s|_{t,v,n,1}|s|_{t,v,n,0}+|nz|^2|s|_{t,v,n,0}^2\bigr),\\ 
        &|\product{L_{x,zK_0,n}^{3,(t,v)}s,s'}_{t,v,n,0}|\leq C_4(1+|nz|^2)|s|_{t,v,n,1}|s'|_{t,v,n,1}.
    \end{split}
\end{equation}
\end{theorem}
\begin{proof}
  Note that the family of operators $\rho(\sqrt{t}Z)\frac{1}{\sqrt{v}}f^p\wedge,\,p=1,...,2\mathrm{dim}S$ are uniformly bounded acting on $(\mathbf{I}_{x}^0,|\cdot|_{t,v,n,0})$.  Then the proof of (\ref{elliptic estimate lxkn3tv}) is  the same as the proof of \cite[Theorem 6.10]{LiuMa2022}. 
\end{proof}

Since $W$ is a compact manifold, there exists a finite family of smooth functions $f_i,...,f_r: W\rightarrow [-1,1]$ which have the following properties:
\begin{itemize}
    \item $W_K=\cap_{j=1}^r\{x\in W: f_j(x)=0\}$;
    \item On $W_K$, $df_1,...,df_r$ span $N_{X_{g,K}/X_g,\rr}$.
\end{itemize}
Observe that there is $c>0$ such that for any $s\in S$, $x\in \bigl(X_{g}\backslash \mathcal{U}_{\varepsilon_0}''\bigr)_s$,
\begin{equation}\label{lower bound fj}
   \Big| \frac{z}{\sqrt{v}}f_j(x)\Big|\geq c\frac{|z|}{\sqrt{v}},\kongge 1\leq j\leq r.
\end{equation}

Let $\mathcal{Q}_{t,v,n,x}$ be the family of operators 
\begin{equation}
    \mathcal{Q}_{t,v,n,x}=\big\{\nabla_{e_i},\, 1\leq i\leq 2\ell;\; \frac{z}{\sqrt{v}}\rho(\sqrt{t}Z)\gamma^2\bigl(\frac{|Z|}{2(n+2)}\bigr)f_j(\sqrt{t}Z),\, 1\leq j\leq r\big\}.
\end{equation}

Following the arguments in \cite[\S 8.8-8.10]{BismutGoette2000}, we have the following uniform estimates, which are formally the same as (\ref{uniform estimate ftlvzk3tv g}) and (\ref{difference ftlvzk3tv}) (or \cite[(8.57), (8.76)]{BismutGoette2000}). In particular, in the following formulas, the scalar $v^{\mathrm{dim}S}$ comes from the inequality (\ref{comparision normal special norm,v}), and the term $\big|\frac{v}{z^2}\big|^m$ origins from the lower bound in (\ref{lower bound fj}).

\begin{theorem}
    There exist $C>0$, $C'>0$ such that given $m>0$, there exists $\beta>0$ such that if $t\in \,]0,1]$, $z\in\rr$, $0<|z|\leq \beta$,  $x\in X_{g}\backslash \mathcal{U}_{\varepsilon_0}''$,  $Z\in N_{X_g/X,\rr,x}$, $|Z|\leq \varepsilon_0/\sqrt{t}$,
    \begin{equation}\label{away from xgk two estimate 1}
        \begin{split}
         & v^{\mathrm{dim}S}\big|\ft(L_{x,zK_0}^{3,(t,v)})(g^{-1}Z,Z)\big|
    \leq C\Big|\frac{v}{z^2}\Big|^m
    \exp(-\frac{C'''}{4}|Z|^2),\\
      &  v^{\mathrm{dim}S}\big|\bigl(\ft(L_{x,zK_0}^{3,(t,v)})-\exp(-L_{x,zK_0}^{3,(0,v)})\bigr)(g^{-1}Z,Z)\big|\\ 
       &    \qquad\qquad\qquad\qquad\leq C\Bigl(\frac{t}{v}\Bigr)^{1/4(2\ell+1)}\cdot\Big|\frac{v}{z^2}\Big|^m\exp(-C'|Z|^2/4).
        \end{split}
    \end{equation}
\end{theorem}

Since $W$ is compact, the function $\mu_1$ is bounded. Then from (\ref{str getzler rescaling v}), (\ref{calculate exp -lzk30v}) and (\ref{away from xgk two estimate 1}), we obtain (\ref{outside uepsilon''}). The proof of Theorem \ref{thm outside uepsilon''} is completed.

\subsection{Control of the supertrace containing $\mu_1$}\label{e1}

\begin{theorem}\label{sec mu1 control}
    There exist $\beta>0$, $r\in\nn$, $C>0$, $\gamma\in \,]0,1]$ such that if $t\in \,]0,1]$, $v\in [t,1]$, if $z\in\rr$, $0<|z|\leq \beta$, $K=zK_{0}$, then
    \begin{equation}\label{mu1 to sum up}
        |z|^r\Big|\Phi\str\Big[i\frac{\mu_1}{v}g\exp(-\mathcal{L}_{K}-\cvt^2)\Big]+\widetilde{e}_v^1\Big|\leq C\Bigl(\frac{t}{v}\Bigr)^\gamma.
    \end{equation} 
\end{theorem}
\begin{proof}
    The estimate in (\ref{mu1 to sum up}) follows from (\ref{I(bktv)}), (\ref{uniform estimate ftlvzk3tv}) and (\ref{outside uepsilon''}). We refer to \cite[Theorem 8.29]{BismutGoette2000} for more details of the proof.
\end{proof}

\subsection{Control of the supertrace containing $\mu_2$}\label{sec mu2 control}
This subsection is devoted to prove the following result.
\begin{theorem}\label{thm mu2 to sum up}
    There exist $\beta>0$, $\gamma\in \,]0,1]$ such that for any $r\in\nn$, there are $r\in\nn$, $C>0$,  such that if $t\in\,]0,1]$, $v\in [t,1]$, if $z\in\rr$, $0<|z|\leq \beta$, $K=zK_{0}$, then
    \begin{equation}\label{mu2 to sum up}
        |z|^r\Big|\Phi\str\Big[i\frac{\mu_2}{v}g\exp(-\lk-\cvt^2)\Big]+\widetilde{e}_v^2-\Phi\partial\xi_{v,t}'-\Phi\bar{\partial}\xi_{v,t}''\Big|\leq C\Big|\frac{v}{z^2}\Big|^r\Bigl(\frac{t}{v}\Bigr)^\gamma.
    \end{equation} 
\end{theorem}

From the same arguments as in Section \ref{sec compute to ev}, we see that for $K\in \mathfrak{z}(g)$ with $|K|\neq  0$ small enough, as $t\rightarrow 0$,
\begin{equation}\label{tend to ev2}
    \Phi\str\Big[i\frac{\mu_2}{v}g\exp(-\lk-\cvt^2)\Big]\rightarrow -\widetilde{e}_v^2.
\end{equation}

Recall that the differential forms $\pi_{v,t}^{2'}$, $\pi_{v,t}^{2''}$ are set in (\ref{pivt}). We take the odd Grassmannian variables $da$, $\dabar$ as in Section \ref{sec intermediate right}.  Now we state an analogue of \cite[Proposition 8.32]{BismutGoette2000}.
\begin{proposition}
The following identity holds,
\begin{equation}\label{tderivative t str mu2}
    \begin{split}
    t\derivative{t}\str\Big[\frac{i\mu_2}{v}g\exp(-\lk-\cvt^2)\Big]=\frac{1}{2}(\bar{\partial}\pi_{v,t}^{2''}+\partial\pi_{v,t}^{2'})\qquad\qquad\\+\trace_s\Big[\frac{\sqrt{t}c(K_2^{X(1,0)})}{\sqrt{2}v}g\exp(-\lk-\cvt^2-\dabar t\derivative{t}\cvt')\Big]^{\dabar}\\
    -\trace_s\Big[\frac{\sqrt{t}c(K_2^{X(0,1)})}{\sqrt{2}v}g\exp(-\lk-\cvt^2-da t\derivative{t}\cvt'')\Big]^{da}.
    \end{split}
\end{equation}
\end{proposition}
\begin{proof}
    By repacing $\product{\mu,K}$ by $\mu_2$, and $K^X$ by $K_2^X$ in (\ref{calculation derivative t}), we get (\ref{tderivative t str mu2}).
\end{proof}

\begin{theorem}
The following identify holds,
\begin{equation}\label{lichnerowicz mu2}
    \begin{split}
        \lk+\cvt^2+\dabar t\derivative{t}\cvt'=-\frac{t}{2}\Big(\nabla_{e_i}^{\Lambda(T^{*(0,1)}X)\otimes \xi}+\frac{1}{\sqrt{2t}}\product{S(e_i)e_j, f_p^H}c(e_j)f^p\wedge\qquad\qquad\\
        +\frac{1}{2t}\product{S(e_i)f_p^H,f_q^H}f^p\wedge f^q\wedge -\frac{\product{K^X,e_i}}{2}\Bigl(\frac{1}{t}+\frac{1}{v}\Bigr)-\frac{\dabar c(e_i^{(0,1)})}{2\sqrt{2t}}\Big)^2\\ 
       +\frac{t}{8}H+\frac{t}{4}\bigl(\bigl(R^{\xi}+\frac{1}{2}\trace[R^{TX}]\bigr)(e_i,e_j)-\frac{1}{v}\product{\nabla_{e_i}^{TX}K^X,e_j}\bigr)c(e_i)c(e_j)\\ 
       +\sqrt{\frac{t}{2}}\bigl(\bigl(R^{\xi}+\frac{1}{2}\trace[R^{TX}]\bigr)(e_i,f_p^H)-\frac{1}{v}\product{T(e_i,f_p^H),K^X}\bigr)c(e_i)f^p\\ 
       +\frac{1}{2}\bigl(\bigl(R^{\xi}+\frac{1}{2}\trace[R^{TX}]\bigr)(f_p^H,f_q^H)-\frac{1}{2v}\product{T(f_p^H,f_q^H),K^X}\bigr)f^p\wedge f^q\wedge\\ 
       -m^{\xi}(K)-\frac{1}{2}\trace[m^{TX}(K)]+\frac{1}{2v}|K^X|^2.
    \end{split}
\end{equation}
\end{theorem}

Let $^4\nabla^{\mathbb{E},t}$ be the connection on $\Lambda(T^*\rr)\,\widehat{\otimes }\, \pi^*\Lambda(T_{\rr}^*S)\,\widehat{\otimes }\,\Lambda(T^{*(0,1)}X)\otimes \xi$ defined by
\begin{equation}
    \begin{split}
        ^4\nabla^{\mathbb{E},t}_{e_i}=\nabla_{e_i}^{\mathbb{E}}+\frac{1}{\sqrt{2t}}\product{S(e_i)e_j, f_p^H}c(e_j)f^p\wedge\hspace{12 em}\\
        +\frac{1}{2t}\product{S(e_i)f_p^H,f_q^H}f^p\wedge f^q\wedge -\frac{\product{K^X,e_i}}{2}\Bigl(\frac{1}{t}+\frac{1}{v}\Bigr)-\frac{\dabar c(e_i^{(0,1)})}{2\sqrt{2t}}.
    \end{split}
\end{equation}

From (\ref{lichnerowicz mu2}), and proceeding as in \cite [Theorem 2.16]{BGS2}, we find that as $t\rightarrow 0$,
\begin{equation}\label{mu2 transgression tend to 0}
    \trace_s\Big[\frac{\sqrt{t}c(K_2^{X(1,0)})}{\sqrt{2}v}g\exp(-\lk-\cvt^2-\dabar t\derivative{t}\cvt')\Big]^{\dabar}\rightarrow 0.
\end{equation}
In fact, the relevant computations leading to (\ref{mu2 transgression tend to 0}) is essentially the same as in Section \ref{sec compute to ev}, by using the key fact that the $\dabar$-part of $\big(^4\nabla^{\mathbb{E},t}\big)^2$ vanishes.

Since $c(K_2^{X})$ vanishes near $X_{K}$, we can proceed as in Section \ref{sec outside neighbor}, by  using the trivialization associated to $ ^4\nabla^{\mathbb{E},t}$. Also from (\ref{mu2 transgression tend to 0}), we find that for any $p\in\nn$, $t\in \,]0,1]$, $v\in [t,1]$,
\begin{equation}\label{ft transgression mu2}
   v^{\mathrm{dim}W_g+\frac{1}{2}}\Big|\trace_s\Big[\frac{\sqrt{t}c(K_2^{X(1,0)})}{\sqrt{2}v}g\ft(-\lk-\cvt^2-\dabar t\derivative{t}\cvt')\Big]^{\dabar}\Big|\leq C\big|\frac{v}{z^2}\big|^p\Bigl(\frac{t}{v}\Bigr)^{\gamma}.
\end{equation}

By proceeding as in Section \ref{sec localization of problem mu1} and by (\ref{ft transgression mu2}), we get
\begin{equation}\label{exp transgression mu2}
    |z|^r\Big|\trace_s\Big[\frac{\sqrt{t}c(K_2^{X(1,0)})}{\sqrt{2}v}g\exp(-\lk-\cvt^2-\dabar t\derivative{t}\cvt')\Big]^{\dabar}\Big|\leq C\big|\frac{v}{z^2}\big|^p\Bigl(\frac{t}{v}\Bigr)^{\gamma}.
\end{equation}

The analogue of (\ref{exp transgression mu2}) holds for $\trace_s\Big[\frac{\sqrt{t}c(K_2^{X(0,1)})}{\sqrt{2}v} g\exp(-\lk-\cvt^2-da t\derivative{t}\cvt'')\Big]^{da}$. By (\ref{hvt}), (\ref{tend to ev2}), (\ref{exp transgression mu2}), and by integration of (\ref{tderivative t str mu2}) in the variable $t$, we get (\ref{mu2 to sum up}), the proof of Theorem \ref{thm mu2 to sum up} is completed.

\subsection{A proof of the last part of Theorem \ref{main lemma 1}}\label{dw}
For $v\geq 1$, $1/v$ remains bounded. By using Theorem \ref{thm lichnerowicz avt} and the trivialization associated to the connection $ ^2\nabla^{\mathbb{E},t}$, we can proceed as in Section \ref{e0} to obtain an analogue of Theorem \ref{thm asymptotic higher akt}. In particular, we see that for $|K|$ small enough, $0<t\leq 1\leq v$,
\begin{equation}
    \Big|\Phi\str\big[i\product{\mu,K}g\exp(-\lk-\cvt^2)\big]\Big|\leq C,
\end{equation}
which is equivalent to (\ref{t01v1}). Then we complete the proof of Theorem \ref{main lemma 1}.

\subsection{A proof of Proposition \ref{i2}, Theorem \ref{h,t01}}\label{dx}
The proof of (\ref{h,t01v1}) is essentially the same as the proof of (\ref{t01v1}). 

By using the Lichnerowicz formula (\ref{lichnerowicz mu2}), and by proceeding as in \cite[Theorem 2.16]{BGS2}, we find that as $t\rightarrow 0$,
\begin{equation}
    \pi_{v,t}^{k\prime}\rightarrow 0, \qquad k=0,1,2,
\end{equation}
as we explained in Section \ref{sec mu2 control}, this is because $\big(^4\nabla^{\mathbb{E},t}\big)^2$ does not have $\dabar$-part. Now we apply the methods in Section \ref{sec localization of problem mu1}-\ref{sec mu1 control} to the operator $\exp(-\lk-\cvt^2-\dabar t\derivative{t}\cvt')$, with the trivialization associated to $^4\nabla^{\mathbb{E},t}$, we obtain the following result corresponding to (\ref{mu1 to sum up}),
\begin{equation}\label{estimate pivt1'}
    |\pi_{v,t}^{1\prime}|\leq C\Bigl(\frac{t}{v}\Bigr)^{\gamma},\kongge 0< t\leq v\leq 1.
\end{equation}
Thus we obtain the first formula in (\ref{h,t01v01}) when $h_{v,t}=\pi_{v,t}^{1\prime}$.

The techniques in Section \ref{e0} concerning about the derivatives in $u$ fit well with the methods of this section. In particular, from the analogues of (\ref{uniform estimate ftlvzk3tv g}), (\ref{away from xgk two estimate 1}) for the derivatives of kernels, one sees that if $u=\sqrt{t}$,
\begin{equation}\label{estimate derivative pivt1'}
    \big|\derivative{u}\pi_{v,t}^{1'}\big|\leq C\sqrt{t},\qquad 0< t\leq v\leq 1.
\end{equation} 
From (\ref{estimate pivt1'}), (\ref{estimate derivative pivt1'}), we obtain the second formula in (\ref{h,t01v01}) when $h_{v,t}=\pi_{v,t}^{1\prime}$.

Now there is no difficulty to get (\ref{h,t01v01}) for $h=\pi_{v,t}^{1,\prime/\prime\prime}$, $v\pi_{v,t}^{2,\prime/\prime\prime}$. We should note that the scalar $v$ before $\pi_{v,t}^{2,\prime/\prime\prime}$ is necessary, technically because (\ref{estimate mu1}) does not hold when we replace $\mu_1$ by $\mu_2$, but holds when $\mu_1$ replaced by  $v\mu_2$.

The only difference for $h=\theta_{v,t}^1$, $v\theta_{v,t}^{2}$ is that we need a Lichnerowicz formula (\ref{lichnerowicz mu2}) for the operator $\lk+\cvt^2+2dat\derivative{t}\cvt''+2\dabar t\derivative{t}\cvt'-da\dabar N_V$, which is the combination of (\ref{gamma lichnerowicz formula}) and (\ref{lichnerowicz mu2}). Thus we complete the proof of Theorem \ref{h,t01}.

\subsection{A proof of Theorem \ref{thm,h2,t01v01}}\label{dy}
The most important fact in the proof of Theorem \ref{thm,h2,t01v01} is that $K_2^X$ vanishes near $X_K$. Once we get a limit as in (\ref{mu2 transgression tend to 0}), since $K_2^X$ vanishes near $X_K$, one immediately get an estimate as in (\ref{exp transgression mu2}) using the techniques in Section \ref{sec outside neighbor}.

From  (\ref{nkt rescaling}), (\ref{chern l}), we have an analogue of (\ref{compute tend to ev}), as $t\rightarrow 0$,
\begin{equation}\label{tchi tend to}
    \begin{split}
    &\Phi^{\mathrm{odd}}\derivative{b}\str\Big[\frac{{\sqrt{t}c(K_2^{X(0,1)})}}{2\sqrt{2}v} g\exp(-\lk-\cvt^2-btN_{K,t})\Big]_{b=0}
        \rightarrow \\ 
      &  \int_{X_g}\frac{K_2^{\prime X(1,0)}}{2v}\big(i\product{\mu,K}-\frac{\omega}{2\pi}\big)\td_{g,K}(TX,h^{TX})\ch_{g,K}(\xi,h^{\xi})\exp\Bigl(\frac{\bar{\partial}_K\partial_K}{2\pi v}\product{\mu,K}\Bigr),
    \end{split}
\end{equation}
where $\Phi^{\mathrm{odd}}:\Lambda^{\mathrm{odd}}(T_{\rr}^*S)\rightarrow \Lambda^{\mathrm{odd}}(T_{\rr}^*S)$ sends $\alpha$ to $(2i\pi)^{-(\mathrm{deg}\alpha-1)/2}\alpha$. From  (\ref{dkkprime}), (\ref{pivt}), (\ref{tchi tend to}), we see that the limit
$ \lim_{t\rightarrow 0}(t\chi_{v,t}^{2\prime})=\mathcal{O}(e^{-1/v})$ exists. Then we have the following analogue of (\ref{exp transgression mu2}),
\begin{equation}
    |z|^r\big|t\chi_{v,t}^{2\prime}-\lim_{t\rightarrow 0}(t\chi_{v,t}^{2\prime})\big|\leq C\Big|\frac{v}{z^2}\Big|^p\Bigl(\frac{t}{v}\Bigr)^{\gamma},\qquad 0<t\leq v\leq 1.
\end{equation}
Thus we obtain (\ref{h2,t01v01}) when $h_{v,t}=\chi_{v,t}^{2\prime}$. The other cases are similar, as long as  we notice that the following limits exist and decay exponentially as $v\rightarrow 0$,
\begin{equation}
    \lim_{t\rightarrow 0}(t\chi_{v,t}^{2,\prime/\prime\prime}),\kongge\kongge  \lim_{t\rightarrow 0}(t\rho_{v,t}^{2,\prime/\prime\prime}),\kongge\kongge  \lim_{t\rightarrow 0}(t^2\sigma_{v,t}^{\prime/\prime\prime}).
\end{equation}

The proof of (\ref{psi,t01v01}) is  easier, once we observe that $t\psi_{v,t}$ has totally the same structure as the supertrace in (\ref{exp transgression mu2}).
\section{The analysis of the two parameter operator $\exp(-\lk-\ctvt^2)$}\label{el}

The aim of this section is to  prove Theorems \ref{main lemma 2}, \ref{thm h,vt}. The main point is to study the asymptotic behavior of $\exp(-\lk-\ctvt^2)$ as $t\to 0$, for a fixed $v\in [1,+\infty[\,$. Recall that in Section \ref{ek}, we split the proof into two parts with the terms $\mu_1$ and $\mu_2$ respectively. In this Section, we will continue to use this approach.  In particular, the terms with $\mu_2$ is relatively easy, whose estimates can be derived directly from Theorem \ref{thm mu2 to sum up}. To study the terms with $\mu_1$, we will use the techniques of Bismut and Goette in \cite[\S 9]{BismutGoette2000}. We will establish more detailed bounds on the norms of the relevant operators, these will be used in the proof of (\ref{h,vt}).

This section is organized as follows. In Section \ref{ep}, we split Theorem \ref{main lemma 2} into two parts and prove one of them. Then we translate the remaining part to a local question in a neighborhood of $X_{g,K}$. In Sections \ref{sec A change of coordinates and a Getzler rescaling: a first simple approach}, \ref{sec a new trivialization}, we describe the trivialization of \cite[\S 9.4, 9.5]{BismutGoette2000} near $X_{g,K}$. In Section \ref{sec A proof of Theorem main lemma 2}, we complete the proof of  Theorem \ref{main lemma 2}. Finally in Section \ref{er}, we prove Theorem \ref{thm h,vt}. 

In this section, we use the notations in Sections \ref{dz}, \ref{ek}. In particular, we  fix $g\in G$, $ K_0\neq 0\in \mathfrak{z}(g)$, and take $K=zK_0$ with $z\neq 0 \in\rr$.

\subsection{Localization of the proof of  Theorem \ref{main lemma 2} to a neighborhood of $X_{g,K}$}\label{ep}
As in Section \ref{sec localization of problem mu1}, we split the proof of Theorem \ref{main lemma 2} into two parts. Recall that the operator $\bktv$ is defined in \eqref{eo}. Then Theorem \ref{main lemma 2}  is implied by the following result.

\begin{theorem}\label{es}
For $z\in \rr$ with $|z|\neq  0$ small enough, for $v\in [1,+\infty[\,$, as $t\rightarrow 0$,
\begin{equation}\label{main lemma 2 into 2}
    \begin{split}
       & \Phi\str\Big[i\frac{\mu_1}{tv}g\exp(-\mathcal{B}_{zK_{0},t,tv})\Big]=\widetilde{D}_{0,g,zK_{0}}^v+o(1),\\ 
       & \Phi\str\Big[i\frac{\mu_2}{tv}g\exp(-\mathcal{B}_{zK_{0},t,tv})\Big]-\Phi\overline{\partial}\xi_{tv,t}'-\Phi\partial\xi_{tv,t}''=\frac{\widetilde{D}_{-1,g,zK_{0}}}{tv}\\ &\qquad\qquad\qquad\qquad\qquad\qquad\qquad\qquad\qquad+E_{0,g,zK_{0}}'-E_{0,g,zK_{0}}+o(1).
    \end{split}
\end{equation}
\end{theorem}

The second identity in (\ref{main lemma 2 into 2}) is relatively easy. Recall that we split $\widetilde{e}_v$ in    into $\widetilde{e}_v^1+\widetilde{e}_v^2$ in \eqref{en}, by using \eqref{aymptopic-bt}, \eqref{obo},  we see that $\widetilde{e}_v^2$ has an asymptotic  expansion similar to (\ref{expansion ev}), as $v\rightarrow 0$,
\begin{equation}\label{expansion ev2}
    \widetilde{e}_v^2=-\frac{\widetilde{D}_{-1,g,zK_{0}}}{v}-E_{0,g,zK_{0}}'+E_{0,g,zK_{0}}+\mathcal{O}(v).
\end{equation}
The second identity in (\ref{main lemma 2 into 2}) follows  from Theorem \ref{thm mu2 to sum up} and (\ref{expansion ev2}).

Now we deal with the first identity of (\ref{main lemma 2 into 2}). From (\ref{I(bktv)}), we see that we can replace $\exp(-\bkttv)$ by the operator $\ft(\bkttv)$. Since  $\ft(\bkttv)(g^{-1}x,x)$ vanishes for $x\notin \mathcal{U}_{\varepsilon_0}$, our problem is now localized in a neighborhood of $X_g$.

Recall that $\mathcal{U}_{\varepsilon_0}''$ in Section \ref{eg} is a neighborhood of $X_{g,K}$ in $X_g$. From (\ref{str getzler rescaling v}) and (\ref{away from xgk two estimate 1}), we find that for $z\in\rr$, $0<|z|\leq \beta$, for any $m\in\nn$, there exist $C>0$ such that if $t\in \,]0,1]$, $v\in [1,1/t]$, 
\begin{equation}\label{estimate to main lemma 2 mu2}
    \begin{split}
    \Big|\int_{x\in X_g\backslash \mathcal{U}_{\varepsilon_0}''}\int_{Z\in N_{X_g/X,\rr,x}\atop |Z|\leq \varepsilon_0}\str\Big[i\frac{\mu_1}{tv}g\ft(\mathcal{B}_{zK_{0},t,tv})(g^{-1}(x,Z),(x,Z))\Big]\qquad\qquad\\ \cdot k(x,Z)\frac{dv_{N_{X_g/X}(Z)}}{(2\pi)^{\mathrm{dim}N_{X_g/X}}}\frac{dv_{X_g}}{(2\pi)^{\mathrm{dim}X_g}}\Big|\leq C(vt)^m.
    \end{split}
\end{equation}

It's clear that our problem is now localized in a neighborhood of $X_{g,K}$. By possibly shrinking $\varepsilon_0$, we may assume that if $\varepsilon\in \,]0,16\varepsilon_0]$, the map $(y_0,Z)\in N_{X_{g,K}/X,\rr}\rightarrow \exp_{y_0}^X(Z)\in X$ is an isomorphism from the $\varepsilon$-neighborhood $\overline{\mathcal{U}}_{\varepsilon}$ of $X_{g,K}$ in $N_{X_{g,K}/X}$ onto the tubular neighborhood $\overline{\mathcal{V}}_{\varepsilon}$ of $X_{g,K}$ in $X$.

From (\ref{lichnerowicz avt}), we have
\begin{equation}\label{lichnerowicz bkttv}
    \begin{split}
        \bkttv=-\frac{t}{2}\Big(\nabla_{e_i}^{\Lambda(T^{*(0,1)}X)\otimes \xi}+\frac{1}{\sqrt{2t}}\product{S(e_i)e_j, f_p^H}c(e_j)f^p\wedge\qquad\qquad\\
         +\frac{1}{2t}\product{S(e_i)f_p^H,f_q^H}f^p\wedge f^q\wedge -\frac{\product{K^X,e_i}}{2t}\Bigl(1+\frac{1}{v}\Bigr)\Big)^2\\ 
        +\frac{t}{8}H+\frac{t}{4}\Bigl(\bigl(R^{\xi}+\frac{1}{2}\trace[R^{TX}]\bigr)(e_i,e_j)-\frac{1}{vt}\product{\nabla_{e_i}^{TX}K^X,e_j}\Bigr)c(e_i)c(e_j)\\ 
        +\sqrt{\frac{t}{2}}\Bigl(\bigl(R^{\xi}+\frac{1}{2}\trace[R^{TX}]\bigr)(e_i,f_p^H)-\frac{1}{vt}\product{T(e_i,f_p^H),K^X}\Bigr)c(e_i)f^p\\ 
        +\frac{1}{2}\Bigl(\bigl(R^{\xi}+\frac{1}{2}\trace[R^{TX}]\bigr)(f_p^H,f_q^H)-\frac{1}{2vt}\product{T(f_p^H,f_q^H),K^X}\Bigr)f^p\wedge f^q\wedge\\ 
        -m^{\xi}(K)-\frac{1}{2}\trace[m^{TX}(K)]+\frac{1}{2vt}|K^X|^2.
    \end{split}
\end{equation}

\begin{definition}
    For $t>0$, put
    \begin{equation}\label{5nablat}
        \begin{split}
        ^5\nabla^{\mathbb{E},t}=\nabla^{\Lambda(T^{*(0,1)}X)\otimes \xi}+\frac{1}{\sqrt{2t}}\product{S(\cdot)e_j, f_p^H}c(e_j)f^p\wedge\qquad\qquad\qquad\\
        +\frac{1}{2t}\product{S(\cdot)f_p^H,f_q^H}f^p\wedge f^q\wedge -\frac{K^{'X}}{2t}\Bigl(1+\frac{1}{v}\Bigr).
        \end{split}
    \end{equation}
\end{definition}
Let $A,B$ be smooth sections of $T_{\rr}X$. By (\ref{5nablat}),
\begin{equation}\label{5nablat c}
    ^5\nabla^{\mathbb{E},t}_A c(B)=c(\nabla_A^{TX}B)+\sqrt{\frac{2}{t}}\product{S(A)B,f_p^H}f^p\wedge.
\end{equation}

\subsection{A change of coordinates and a Getzler rescaling: a first simple approach}\label{sec A change of coordinates and a Getzler rescaling: a first simple approach}

Here we follow the strategy of \cite[\S 9.4]{BismutGoette2000} and \cite[\S 6.10]{LiuMa2022}.

Take $y_0\in X_{g,K}$. If $Z\in (T_{\rr}X)_{y_0}$, $|Z|\leq 4\varepsilon_0$, we identify $Z$ with $\exp_{y_0}^X(Z)\in X$, and we identify $(\pi^*\Lambda(T_{\rr}^*S)\,\widehat{\otimes }\,\Lambda(T^{*(0,1)}X)\otimes \xi )_{Z}$ with $(\pi^*\Lambda(T_{\rr}^*S)\,\widehat{\otimes }\,\Lambda(T^{*(0,1)}X)\otimes \xi )_{y_0}$ by parallel transport  with respect to the connection $^5\nabla^{\mathbb{E},t}$ along the curve $s\in [0,1]\rightarrow sZ$.

Recall that $\rho$ is taken as a cut-off function in \eqref{du}, put
\begin{equation}\label{ly0k1tv}
    L_{y_0,K}^{1,(t,v)}=(1-\rho^2(Z))(-\frac{t}{2}\Delta^{TX})+\rho^2(Z)\bkttv.
\end{equation}
We still define $H_t$ as in (\ref{def of ht}). Let $L_{y_0,K}^{2,(t,v)}$ be obtained from $L_{y_0,K}^{1,(t,v)}$ as in (\ref{lxk2t}).

Let $(e_1,...,e_{2\ell})$ be an orthonormal oriented basis of $(T_{\rr}X)_x$ taken as before Definition \ref{def lzk3tv}. 
\begin{definition}\label{et}
    Let $L_{y_0,K}^{3,(t,v)}$ be the differential operator  obtained from $L_{Z_0,K}^{2,(t,v)}$ by replacing $c(e_j)$ by $c_t(e_j)$ for $1\leq j\leq 2\ell''$,  while leaving unchanged for $2\ell''+1\leq j\leq 2\ell$.
\end{definition}

Let $\jmath':W_{g,K}\rightarrow W$ be the obvious embedding. Put
\begin{equation}\label{lyk30v}
    \begin{split}
       L_{y_0,K}^{3,(0,v)}= -\frac{1}{2}\Bigl(\nabla_{e_i}+\frac{1}{2}\big\langle\big(\jmath^{\prime *}R^{TX}-\bigl(1+\frac{1}{v}\bigr)m^{TX}(K)\big)Z,e_i\big\rangle\Bigr)^2\qquad\qquad\\+\jmath^{\prime *}R_{y_0}^{\xi}-m^{\xi}(K)_{y_0}
        +\frac{1}{2}\trace[\jmath^{\prime *}R^{TX}-m^{TX}(K)]_{y_0}+\frac{1}{2v}\big|m_{y_0}^{TX}(K)Z\big|^2\quad\\ 
        -\frac{1}{4v}\product{m^{TX}(K)e_j,e_k}_{y_0}c(e_j)c(e_k)+\frac{1}{2v}\product{\jmath^{\prime *} R^{TX}m^{TX}(K)Z,Z}_{y_0}.
    \end{split}
\end{equation}

\begin{theorem}\label{thm limit of lyk3tv}
    For $v\geq 1$, as $t\rightarrow 0$,
    \begin{equation}\label{limit of lyk3tv}
        L_{y_0,K}^{3,(t,v)}\rightarrow L_{y_0,K}^{3,(0,v)}.
    \end{equation}
\end{theorem}
\begin{proof}
    For $A\in\bigl(\pi^*\Lambda(T_{\rr}^*S)\,\widehat{\otimes }\, c(TX)\otimes \ENDO(\xi) \bigr)_{y_0}\otimes \mathrm{Op}_{y_0}$, we denote by $[A]_t^{(3)}$ the differential operator obtained from $A$ by using the Getzler rescaling in Definition \ref{et}.
Let $\tau e_j(Z)$ be the parallel transport of $e_j$ along the curve $t\in [0,1]\rightarrow tZ$ with respect to the connection $\nabla^{TX}$. By (\ref{5nablat c}), in the trivialization associated with $ ^5\nabla^{\mathbb{E},t}$, we have
\begin{equation}\label{5 ctauej}
    c(\tau e_j(Z))=c(e_j)+\sqrt{\frac{2}{t}}\product{S(Z)e_j,f_p^H}f^p+\mathcal{O}(t^{-1/2}|Z|^2).
\end{equation}
From (\ref{ctauej}), for $1\leq j\leq 2\ell''$,
\begin{equation}\label{5 ct explicit 1}
    \Big[\sqrt{t}c_t(\tau e_j(\sqrt{t}Z))\Big]_t^{(3)}=\sqrt{2}c(e_j)+\mathcal{O}(\sqrt{t}|Z|^2),
\end{equation}
for $2\ell''+1\leq j\leq 2\ell$,
\begin{equation}\label{5 ct explicit 2}
    \Big[c(\tau e_j(\sqrt{t}Z))\Big]_t^{(3)}=c(e_j)+\sqrt{{2}}(\product{S(Z)e_j,f_p^H}f^p+\mathcal{O}(\sqrt{t}|Z|^2).
\end{equation}

From \cite[(3.3)-(3.5)]{LiuMa2022},
\begin{equation}\label{dk'}
    \begin{split}
    dK^{\prime X}=\product{\nabla_{e_i}^{TX}K^X,e_j}e^i\wedge e^j+2\product{T(e_i,f_p^H),K^X}e^i\wedge f^p\qquad\\ 
    +\frac{1}{2}\product{T(f_p^H,f_q^H),K^X}f^p\wedge f^q.
    \end{split}
\end{equation}
From (\ref{expansion of dk}), (\ref{dk'}), we see that if $1\leq i,j\leq 2\ell''$,
\begin{equation}\label{expansion T 1}
    \begin{split}
       & \product{\nabla_{\tau e_i}^{TX}K^X,\tau e_j}_Z=-\frac{1}{2}\product{R^{TX}_{y_0}(e_i,e_j)(m^{TX}(K)_{y_0}Z),Z}+\mathcal{O}(|Z|^3),\\
        &\product{T(\tau e_i,f_p^H),K^X}_Z=-\frac{1}{2}\product{R^{TX}_{y_0}(e_i,f_p^H)(m^{TX}(K)_{y_0}Z),Z}+\mathcal{O}(|Z|^3),\\ 
      &  \product{T(f_p^H,f_q^H),K^X}_Z=-\product{R^{TX}_{y_0}(f_p^H,f_q^H)(m^{TX}(K)_{y_0}Z),Z}+\mathcal{O}(|Z|^3),
    \end{split}
\end{equation}
 and if $2\ell''+1\leq i\leq 2\ell$,
 \begin{equation}\label{expansion T 2}
    \begin{split}
        & \product{\nabla_{\tau e_i}^{TX}K^X,\tau e_j}_Z=\product{m^{TX}(K)_{y_0}e_i,e_j}+\mathcal{O}(|Z|),\\
        &\product{T(\tau e_i,f_p^H),K^X}_Z=\mathcal{O}(|Z|^2).
    \end{split}
 \end{equation}

By using (\ref{S and T}) and (\ref{5 ct explicit 1})-(\ref{expansion T 2}), we find that as $t\rightarrow 0$,
 \begin{equation}\label{explicit 53t}
    \begin{split}
        \frac{1}{4}  \Big[\product{\nabla_{\tau e_i}^{TX}K^X,\tau e_j}c(\tau e_i)c(\tau e_j)(\sqrt{t}Z)\Big]_t^{(3)}+\frac{1}{\sqrt{2t}}\Big[\product{T(\tau e_i,f_p^H),K^X}\qquad\qquad\\ \cdot c(\tau e_i)f^p(\sqrt{t}Z)\Big]_t^{(3)}
        +\frac{1}{4t}\Big[\product{T(f_p^H,f_q^H),K^X}f^p\wedge f^q(\sqrt{t}Z)\Big]_t^{(3)}
       \\  \rightarrow \frac{1}{4}\sum_{2\ell''+1\leq i,j\leq 2\ell}\product{m^{TX}(K)e_i,e_j}c(e_i)c(e_j)
        -\frac{1}{2}\product{\jmath^{\prime *}R^{TX}m^{TX}(K)Z,Z}_{y_0}.
    \end{split}
 \end{equation}

Now from  (\ref{limit to lxk0v}), (\ref{lichnerowicz bkttv}) and (\ref{explicit 53t}), we complete the proof of Theorem \ref{thm limit of lyk3tv}.
\end{proof}

\subsection{A new trivialization}\label{sec a new trivialization}
Now we introduce another trivialization as in \cite[\S 9.5]{BismutGoette2000} 
to make the estimates easier. 
Recall that in Section \ref{sec A new trivialization and Getzler rescaling}, we use the notation $Z_0+Z$ to denote $\exp_{Z_0}^X Z$.

Take $Z_0\in N_{X_{g,K}/X_g,\rr,y_0}$, $|Z_0|\leq \varepsilon_0$. If $Z\in (T_{\rr}X)_{y_0}$, $|Z|\leq 4\varepsilon_0$, we identify $(\pi^*\Lambda(T_{\rr}^*S)\,\widehat{\otimes }$ $\Lambda(T^{*(0,1)}X)\otimes \xi)_{Z_0+Z}$ with $(\pi^*\Lambda(T_{\rr}^*S)\,\widehat{\otimes }\,\Lambda(T^{*(0,1)}X)\otimes \xi )_{Z_0}$ by parallel transport  along the curve $s\in [0,1]\rightarrow Z_0+sZ$ with respect to the connection $^5\nabla^{\mathbb{E},t}$. Also we identify $(\pi^*\Lambda(T_{\rr}^*S)\,\widehat{\otimes }\,\Lambda(T^{*(0,1)}X)\otimes \xi)_{Z_0}$ with $(\pi^*\Lambda(T_{\rr}^*S)\,\widehat{\otimes }\,\Lambda(T^{*(0,1)}X)\otimes \xi )_{y_0}$ by parallel transport  along the curve $s\in [0,1]\rightarrow sZ_0$ with respect to the connection $\nabla^{\Lambda(T^{*(0,1)}X)\otimes \xi}$.

We may and we will assume that $\varepsilon_0$ is small enough so that if $|Z_0|\leq \varepsilon_0$, $|Z|\leq 4\varepsilon_0$,
\begin{equation}
    \frac{1}{2}g_{y_0}^{TX}\leq g_{Z_0+Z}^{TX}\leq \frac{3}{2}g^{TX}_{y_0}.
\end{equation}
Let $L_{Z_0,K}^{\prime 1,(t,v)}$ be the operator 
\begin{equation}
    L_{y_0,Z_0,K}^{\prime 1,(t,v)}=(1-\rho^2(Z))(-\frac{t}{2}\Delta^{TX})+\rho^2(Z)(\bkttv)_{Z_0}.
\end{equation}

We still obtain $L_{y_0,Z_0,K}^{\prime 2,(t,v)}$, $L_{y_0,Z_0,K}^{\prime 3,(t,v)}$ from $L_{Z_0,K}^{\prime 1,(t,v)}$ as in Section \ref{sec A change of coordinates and a Getzler rescaling: a first simple approach}. As in (\ref{str getzler rescaling v}), if $Z_0\in N_{X_{g,K}/X_g,\rr,y_0}$, $|Z_0|\leq \varepsilon_0$, $Z\in N_{X_g/X,\rr,y_0}$, $ |Z|\leq \varepsilon_0/\sqrt{t}$,
\begin{equation}\label{str getzler rescaling tv}
    \begin{split}
    t^{\mathrm{dim}N_{X_{g,K}/X}}\str\Big[i\frac{\mu_1}{tv}g\ft(\bkttv)\bigl(g^{-1}(Z_0,\sqrt{t}Z),(Z_0,\sqrt{t}Z)\bigr)\Big] k_{(y_0,Z_0)}'(\sqrt{t}Z)\\ 
        =(-i)^{\mathrm{dim}X_{g,K}}\str^{\Lambda(N_{X_{g,K}/X}^{*(0,1)})}\Big[i\frac{\mu_1(Z_0,\sqrt{t}Z)}{tv}g\ft(L_{y_0,Z_0,K}^{\prime 3,(t,v)})(g^{-1}Z,Z)\Big]^{\mathrm{max}}.
    \end{split}
\end{equation}
\begin{definition}
    Let $L_{y_0,Z_0,K}^{\prime 3,(0,v)}$ be the operator
    \begin{equation}
        \begin{split}
            L_{y_0,Z_0,K}^{\prime 3,(0,v)}= -\frac{1}{2}\Bigl(\nabla_{e_i}+\frac{1}{2}\big\langle\big(\jmath^{\prime *}R^{TX}-\bigl(1+\frac{1}{v}\bigr)m^{TX}(K)\big)Z,e_i\big\rangle\Bigr)^2\qquad\qquad\qquad\\+\jmath^{\prime *}R_{y_0}^{\xi}-m^{\xi}(K)_{y_0}
            +\frac{1}{2}\trace[\jmath^{\prime *}R^{TX}-m^{TX}(K)]_{y_0}+\frac{1}{2v}\big|m_{y_0}^{TX}(K)(Z+Z_0)\big|^2\\ 
            -\frac{1}{4v}\product{m^{TX}(K)e_j,e_k}_{y_0}c(e_j)c(e_k)+\frac{1}{2v}\product{\jmath^{\prime *} R^{TX}m^{TX}(K)Z+Z_0,Z+Z_0}_{y_0}.
        \end{split}
    \end{equation}
\end{definition}
From (\ref{lyk30v}), we observe that
\begin{equation}\label{two l0 conjugation}
    \begin{split}
        L_{y_0,Z_0,K}^{\prime 3,(0,v)}=\exp\Big(\frac{1}{2}\big\langle\big(\jmath^{\prime *}R^{TX}-\bigl(1+\frac{1}{v}\bigr)m^{TX}(K)\big)Z,e_i\big\rangle\Big)\Big[L_{y_0,K}^{3,(0,v)}\Big]_{Z+Z_0}\\ 
        \cdot\exp\Big(-\frac{1}{2}\big\langle\big(\jmath^{\prime *}R^{TX}-\bigl(1+\frac{1}{v}\bigr)m^{TX}(K)\big)Z,e_i\big\rangle\Big).
    \end{split}
\end{equation}

The following result follows from the same arguments as the proof of Theorem \ref{thm limit of lyk3tv}.
\begin{theorem}\label{thm limit of lzkprime3tv}
    For $v\geq 1$, as $t\rightarrow 0$,
    \begin{equation}\label{limit of lzkprime3tv}
        L_{y_0,\sqrt{t}Z_0,K}^{\prime 3,(t,v)}\rightarrow L_{y_0,Z_0,K}^{\prime 3,(0,v)}.
    \end{equation}
\end{theorem}

Let $L_{y_0,Z_0,K,n}^{\prime 3,(t,v)}$ be the operators obtained from $L_{y_0,\sqrt{t}Z_0,K}^{\prime 3,(t,v)}$ as in (\ref{lxkn3t}) and (\ref{lxkn3tv}). By proceeding as in (\ref{finite propagation n+p}), if $Z,Z'\in (T_{\rr}X)_{y_0}$, $|Z|\leq p$, then
\begin{equation}
    \widetilde{F}_{t,n}(L_{y_0,Z_0,K}^{3,t})(Z,Z')=\widetilde{F}_{t,n}(L_{y_0,Z_0,K,n+p}^{3,t})(Z,Z').
\end{equation}

Let  $\mathbf{I}_{(p,q),y_0}'$ be the  space of smooth sections of $\bigl( \pi^*\Lambda^q(T_{\rr}^*S)\,\widehat{\otimes }\,\Lambda^p(T_{\rr}^*X_{g,K})\otimes \Lambda(N_{X_{g,K}/X}^{*(0,1)}) \otimes \xi\bigr)_{y_0}$ on $(T_{\rr}X)_{y_0}$.
If $s\in \mathbf{I}_{(p,q),y_0}'$ has compact support, put 
\begin{equation}\label{sobolev norm tznk}
    \begin{split}
   & |s|_{t,Z_0,n,0}^2=\int_{(T_{\rr}X)_x}|s|^2\Bigl(1+(|Z|+|Z_0|)\rho\Big(\frac{\sqrt{t}Z}{2}\Big)\Bigr)^{2(2\ell''+2\mathrm{dim}S-p-q)}dv_{TX}(Z),\\ 
    &|s|_{t,Z_0,n,1}^2=|s|_{t,Z_0,n,0}^2+\sum_{1}^{2\ell}|\nabla_{e_i}s|_{t,Z_0,n,0}^2\\&\qquad\qquad\qquad\qquad\qquad+\frac{1}{tv}\big|\rho(\sqrt{t}Z)\gamma^2\Bigl(\frac{|Z|}{2(n+2)}\Bigr)|K^X|(\sqrt{t}Z_0+\sqrt{t}Z)s\big|_{t,Z_0,n,0}^2.
    \end{split}
\end{equation}
The corresponding Hermitian product associated to $|\cdot|_{t,Z_0,n,0}$ is denoted by $\product{\,,\,}_{t,Z_0,n,0}$. 
\begin{theorem}\label{thm elliptic estimate lykn3tv}
    There exist constants $C_i>0$, $i=1,2,3,4$, such that if $t\in \,]0,1]$, $v\in [1,+\infty[\,$, $n\in\nn$, $z\in\rr$, $0<|z|\leq 1$, if $n\in \nn$,  $y_0\in W_{g,K}$, $Z_0\in N_{X_{g,K}/X_g,\rr,y_0}$, $|Z_0|\leq \varepsilon_0/\sqrt{t}$,   if  $s,s'\in\mathbf{I}_{y_0}$ have compact support, then
    \begin{equation}\label{elliptic estimate lykn3tv}
        \begin{split}
            &\mathrm{Re}\product{L_{y_0,\sqrt{t}Z_0,zK_{0},n}^{\prime 3,(t,v)}s,s}_{t,Z_0,n,0}\geq C_1|s|_{t,Z_0,n,1}^2-C_2(1+|nz|^2)|s|_{t,Z_0,n,0}^2,\\ 
            &|\mathrm{Im}\product{L_{y_0,\sqrt{t}Z_0,zK_{0},n}^{\prime 3,(t,v)}s,s}_{t,Z_0,n,0}|\leq C_3\bigl((1+|nz|)|s|_{t,Z_0,n,1}|s|_{t,Z_0,n,0}+|nz|^2|s|_{t,Z_0,n,0}^2\bigr),\\ 
            &|\product{L_{y_0,\sqrt{t}Z_0,zK_{0},n}^{\prime 3,(t,v)}s,s'}_{t,Z_0,n,0}|\leq C_4(1+|nz|^2)|s|_{t,Z_0,n,1}|s'|_{t,Z_0,n,1}.
        \end{split}
    \end{equation}
\end{theorem}
\begin{proof}
    The proof of (\ref{elliptic estimate lykn3tv}) is the same as the proof of \cite[Theorem 6.10]{LiuMa2022}.
\end{proof}

Let $|\cdot|_{t,Z_0,n,-1}$ be the dual norm of $|\cdot|_{t,Z_0,n,1}$ with respect to the product $\product{\,,\,}_{t,Z_0,n,0}$. For $k=-1,0,1$,
let $(\mathbf{I}_{y_0}^{\prime k},\|\cdot\|_{t,Z_0,n,k})$ be the closure of $\mathbf{I}_{y_0}$ with respect to the norm $\|\cdot\|_{t,Z_0,n,k}$. If $k_0,k_1\in\nn$ and if $A\in \mathcal{L}(\mathbf{I}_{y_0}^{\prime k_0},\mathbf{I}_{y_0}^{\prime k_1})$, let $\|A\|_{t,Z_0,n}^{k_0,k_1}$ be the norm of $A$ with respect to the norms $\|\cdot\|_{t,Z_0,n,k_0}$, $\|\cdot\|_{t,Z_0,n,k_1}$.

By using (\ref{elliptic estimate lxkn3t}), exactly in the same way as the proofs of \cite[Theorems 7.34, 7.37]{BismutGoette2000}, we obtain the following analogue of Theorem \ref{thm estimate of resolvent first}.
\begin{theorem}\label{thm estimate of resolvent first vt}
    Given $\eta>0$, there exist $c_{\eta}\in \,]0,1]$, $C>0$, $d>0$ such that if $t\in \,]0,1]$, $z\in\rr$, $0<|z|\leq c_{\eta}$, $n\in\nn$, $y_0\in W_{g,K}$, $Z_0\in N_{X_{g,K}/X_g,\rr,y_0}$, $|Z_0|\leq \varepsilon_0/\sqrt{t}$, $\lambda\in U_{\eta n+d}$ , the resolvent $(\lambda-L_{y_0,Z_0,zK_{0},n}^{\prime 3,(t,v)})^{-1}$ exists, and moreover
    \begin{equation}\label{resolvent lxkn3t estimate first vt}
        \begin{split}
        &\|(\lambda-L_{y_0,Z_0,zK_{0},n}^{\prime 3,(t,v)})^{-1}\|_{t,Z_0,n}^{0,0}\leq \frac{4}{(\eta n+d)^2},\\ 
         &   \|(\lambda-L_{y_0,Z_0,zK_{0},n}^{\prime 3,(t,v)})^{-1}\|_{t,Z_0,n}^{-1,1}\leq C(1+|n|^2)(1+|\lambda|^2).
        \end{split}
    \end{equation} 
\end{theorem}

\subsection{A proof of Theorem \ref{main lemma 2}}\label{sec A proof of Theorem main lemma 2}
\quad

\begin{theorem}\label{thm uniform estimate ftlvzk3tv tv}
    There exist $C'>0$, $C'' >0$, $C'''>0$ such that for $m, m'\in\nn$, there are $C>0$, $r\in\nn$ satisfying that for $\eta\in \,]0,1]$ small enough, there is $c_{\eta}>0$ such that: for $z\in\rr$, $0<|z|\leq c_{\eta}$, $t\in \,]0,1]$, $v\in [1,1/t]$, $n\in\nn$, $y_0\in W_{g,K}$, $Z_0\in N_{X_{g,K}/X_g,\rr,y_0}$, $|Z_0|\leq \varepsilon_0/\sqrt{t}$, $Z,Z'\in (T_{\rr}X)_{y_0}$, $|Z|,|Z'|\leq \varepsilon_0/\sqrt{t}$,
    \begin{equation}\label{uniform estimate ftlvzk3tv tv}
        \begin{split}
            &\Big(1+\Big|\frac{zZ_0}{\sqrt{v}}\Big|\Big)^m\sup_{|\alpha|,|\alpha'|\leq m'}\Big|\frac{\partial^{|\alpha|+|\alpha'|}}{\partial Z^{\alpha}\partial Z^{'\alpha'}}\widetilde{F}_{t,n}(L_{y_0,\sqrt{t}Z_0,zK_{0}}^{\prime 3,(t,v)})(Z,Z')\Big|\qquad\qquad\\ 
           &\qquad \leq C(1+|Z|+|Z'|)^r(1+|Z_0|)^{2\ell''+2\mathrm{dim}S}\\
          &\qquad\qquad\qquad\qquad  \exp(-C'n^2/4+2C''\eta^2\sup(|Z|^2,|Z'|^2)-C'''|Z-Z'|^2).
        \end{split}
    \end{equation}
\end{theorem}
\begin{proof}
    First when $v=1$, by taking $v=t$ in Theorem \ref{thm uniform estimate ftlvzk3tv}  and using (\ref{as infinite sum}), we obtain (\ref{uniform estimate ftlvzk3tv tv}). The remaining cases are similar.
\end{proof}

\begin{theorem}\label{i1}
        There exist $C>0$, $C'>0$ such that for $m>0$, there exists $\beta>0$ such that if $t\in \,]0,1]$, $v\in [1,1/t]$, $z\in\rr$, $0<|z|\leq \beta$,  $y_0\in W_{g,K}$, $Z_0\in N_{X_{g,K}/X_g,\rr,y_0}$, $|Z_0|\leq \varepsilon_0/\sqrt{t}$, $Z\in N_{X_g/X,\rr,y_0}$, $|Z|\leq \varepsilon_0/\sqrt{t}$,
        \begin{equation}\label{difference ftlvzk3tv tv}
            \begin{split}
            \big|\bigl(\ft(L_{y_0,\sqrt{t}Z_0,zK_{0}}^{\prime 3,(t,v)})-\exp(-L_{y_0,Z_0,zK_{0}}^{\prime 3,(0,v)})\bigr)(g^{-1}Z,Z)\big|\qquad\qquad\qquad\qquad\\ 
               \leq Ct^{1/4(2\ell+1)}\cdot\frac{(1+|Z_0|)^{2\ell''+2\mathrm{dim}S+3}}{(1+|zZ_0/\sqrt{v}|)^m}\exp(-C'|Z|^2/4),
            \end{split}
        \end{equation}
        here the heat kernel $\exp(-L_{y_0,Z_0,zK_{0}}^{\prime 3,(0,v)})(g^{-1}Z,Z)$ is calculated as in (\ref{def of fr}) with $r=0$.
\end{theorem}

\begin{proof}
    From (\ref{expansion of K}), we observe that the norms in (\ref{sobolev norm tznk}) extend to $t=0$. Moreover the estimates in Theorems \ref{thm elliptic estimate lykn3tv}-\ref{thm uniform estimate ftlvzk3tv tv} still hold when making $t=0$.

From the proof of Theorem \ref{thm limit of lzkprime3tv}, we  find that   there is $C>0$, $\beta>0$, such that for $t\in \,]0,1]$, $v\in [1,+\infty[\,$, $z\in\rr$, $0<|z|\leq \beta$, $n\in\nn$, $Z_0\in N_{X_{g,K}/X_g,\rr,y_0}$, $|Z_0|\leq \varepsilon_0/\sqrt{t}$, if $s,s'\in \mathbf{I}_{y_0}$ has compact support,
\begin{equation}\label{difference lyzkn3tv}
    \begin{split}
    \big|\langle \bigl(L_{y_0,\sqrt{t}Z_0,zK_{0},n}^{\prime 3,(t,v)}-L_{y_0,Z_0,zK_{0},n}^{\prime 3,(0,v)}\bigr)s,s'\rangle_{t,Z_0,n,0}\big|\qquad\qquad\qquad\qquad\qquad\qquad\\ 
    \leq C\sqrt{t}(1+n^4)(1+|Z_0|)^3|s|_{0,Z_0,n,1}|s'|_{t,Z,n,1}.
    \end{split}
\end{equation}

Let $C', C'', C'''$ be the constants in \eqref{uniform estimate ftlvzk3tv tv}.
From (\ref{resolvent lxkn3t estimate first vt}), (\ref{uniform estimate ftlvzk3tv tv}), (\ref{difference lyzkn3tv}), we  proceed as in  the proof of Theorem \ref{thm t-0 estimate ftnprh} to obtain that for any $m\in\nn$, there exist $C>0$ such that if $Z_0\in N_{X_{g,K}/X_g,\rr,y_0}$, $|Z_0|\leq \varepsilon_0/\sqrt{t}$, $Z,Z'\in N_{X_g/X,\rr,y_0}$, $|Z|, |Z'|\leq \varepsilon_0/\sqrt{t}$,
\begin{equation}\label{difference ftn-f0n tv}
    \begin{split}
   & \big|\bigl(\widetilde{F}_{t,n}(L_{y_0,\sqrt{t}Z_0,zK_{0}}^{\prime 3,(t,v)})-\widetilde{F}_{0,n}(L_{y_0,Z_0,zK_{0}}^{\prime 3,(0,v)})\bigr)(Z,Z')\big|
    \leq Ct^{1/4(2\ell+1)}\frac{(1+|Z_0|)^{2\ell''+2\mathrm{dim}S+3}}{(1+|zZ_0/\sqrt{v}|)^m}\\ 
    &\qquad\qquad\qquad\cdot\exp(-C'n^2/4+2C''\eta^2\sup(|Z|^2,|Z'|^2)-C'''|Z-Z'|^2).
    \end{split}
\end{equation}
We fix $\eta>0$ as in (\ref{fix small eta}) and sum up (\ref{difference ftn-f0n tv}) for $n\in\nn$, thus we obtain (\ref{difference ftlvzk3tv tv}). So we complete the proof of Theorem \ref{i1}.
\end{proof}

\begin{theorem}
    There exists $\beta>0$, $\gamma>0$,  $r>0$  such that for $z\in\rr$, $0<|z|\leq \beta$, there exist  $C>0$, such that if $t\in\,]0,1]$, $v\in [1,1/t]$,
    \begin{equation}\label{estimate to main lemma 2 mu1}
        \big|\Phi\str\Big[i\frac{\mu_1}{tv}g\widetilde{F}_t(-\mathcal{B}_{zK_{0},t,tv})\Big]-\widetilde{D}_{0,g,zK_{0}}^v\big|\leq Cv^r t^{\gamma}.
    \end{equation}
\end{theorem} 
\begin{proof}
 Let $\overline{k}_{y_0}(Z_0)$ be the smooth function such that for $(y_0,Z_0)\in N_{X_{g,K}/X_g,\rr}$,
 \begin{equation}\label{def of kbar}
    dv_{X_g}(y_0,Z_0)=\overline{k}_{y_0}(Z_0) dv_{X_{g,K}}(y_0) dv_{N_{X_{g,K}/X_g}}(Z_0).
 \end{equation}
Clearly, $\overline{k}_{y_0}(0)=1$.
 From (\ref{def of k}), (\ref{str getzler rescaling tv}),  (\ref{def of kbar}),  we see that
 \begin{equation}\label{calculate bkttv by l}
    \begin{split}
        \int_{(Z_0,Z)\in N_{X_{g,K}/X_g}\times N_{X_g/X}\atop |Z|,|Z'|\leq \varepsilon_0}\str\Big[i\frac{\mu_1}{tv}g\ft(\mathcal{B}_{zK_{0},t,tv})\bigl(g^{-1}(y_0,Z_0,Z), (y_0,Z_0,Z)\bigr)\Big]\frac{dv_{X}(y_0,Z_0,Z)}{(2\pi)^{\mathrm{dim}X}}\\ 
        =\int_{X_{g,K}}\int_{|Z_0|,|Z|\leq \varepsilon_0/\sqrt{t}}(-i)^{\mathrm{dim}X_{g,K}}\str^{\Lambda(N_{X_{g,K}/X}^{*(0,1)})}\Big[i\frac{\mu_1(\sqrt{t}Z_0,\sqrt{t}Z)}{tv}g\ft(L_{y_0,Z_0,zK_{0}}^{\prime 3,(t,v)}(g^{-1}Z,Z))\Big]^{\mathrm{max}}\\ 
        \cdot\frac{k(y_0,\sqrt{t}Z_0,\sqrt{t}Z)\overline{k}_{y_0}(\sqrt{t}Z_0)}{k_{(y_0,\sqrt{t}Z_0)}'(\sqrt{t}Z)}\frac{dv_{N_{X_{g,K}/X_g}}(Z_0)}{(2\pi)^{\mathrm{dim}N_{X_{g,K}/X_g}}}\frac{dv_{N_{X_g/X}}(Z)}{(2\pi)^{\mathrm{dim}N_{X_g/X}}}\frac{dv_{X_{g,K}}(y_0)}{(2\pi)^{\mathrm{dim}X_{g,K}}}.
    \end{split}
 \end{equation}

   We apply the constructions in Section \ref{em} to $B=W_{g,K}$, $M=TX|_{W_{g,K}}$. Therefore $\varphi\str\Big[i\frac{\widetilde{\mu}(K)}{v}g\exp\Big(-\lk-\mathscr{B}_{K,v}^2\Big)\Big]$ in \eqref{basic generalized supertraces} is a smooth form on $W_{g,K}$. From (\ref{basic generalized supertraces}), (\ref{lyk30v}), (\ref{two l0 conjugation}), one verifies that
 \begin{equation}\label{calculate l by scrkbv}
    \begin{split}
      \int_{X_{g,K}}  \int_{(Z_{0},Z)\in N_{X_{g,K}/X,\rr,y_0}}(-i)^{\mathrm{dim}X_{g,K}}\Phi\str^{\Lambda(N_{X_{g,K}/X}^{*(0,1)})}\Big[i\frac{\widetilde{\mu}(K)(Z_0+Z)}{v}\cdot\qquad\qquad\\ g\ft(L_{y_0,Z_0,zK_{0}}^{\prime 0,(t,v)}(g^{-1}Z,Z))\Big]^{\mathrm{max}} 
      \frac{dv_{N_{X_{g,K}/X_g}}(Z_0)}{(2\pi)^{\mathrm{dim}N_{X_{g,K}/X_g}}}\frac{dv_{N_{X_g/X}}(Z)}{(2\pi)^{\mathrm{dim}N_{X_g/X}}}\frac{dv_{X_{g,K}}}{(2\pi)^{\mathrm{dim}X_{g,K}}}\\ 
      =\int_{X_{g,K}}\varphi\str\Big[i\frac{\widetilde{\mu}(K)}{v}g\exp(-\lk-\mathscr{B}_{zK_{0},v}^2)\Big]\ch_{ge^K}(\xi,g^{\xi})=\widetilde{D}_{0,g,zK_{0}}^v.
    \end{split}
 \end{equation}

 From (\ref{dmuk}), (\ref{km}), (\ref{def of tilde mu}) and (\ref{proposition of mu1mu2}), we find that for $t\in \,]0,1]$, $v\in [1,+\infty[\,$, $Z_0\in N_{X_{g,K}/X_g,\rr,y_0}$, $|Z_0|\leq \varepsilon_0/\sqrt{t}$, $Z\in N_{X_g/X,\rr,y_0}$, $|Z|\leq \varepsilon_0/\sqrt{t}$,
 \begin{equation}\label{difference mu}
    \begin{split}
   & \Big|\frac{\mu_1(\sqrt{t}Z_0,\sqrt{t}Z)}{tv}-\frac{\widetilde{\mu}(K)(Z_0+Z)}{v}\Big|\leq C\frac{\sqrt{t}}{v}\bigl(|Z_0|^3+|Z|^3\bigr).
    \end{split}
 \end{equation}
 From (\ref{estimate to main lemma 2 mu2}),  (\ref{uniform estimate ftlvzk3tv tv}), (\ref{difference ftlvzk3tv tv}), (\ref{calculate bkttv by l})-(\ref{difference mu}), we obtain (\ref{estimate to main lemma 2 mu1}).
\end{proof}

From (\ref{I(bktv)}), (\ref{estimate to main lemma 2 mu1}), we  obtain the first identity in (\ref{main lemma 2 into 2}).  Thus we finish the proof of Theorem \ref{main lemma 2}.

\subsection{A proof of Theorem \ref{thm h,vt}}\label{er}

First we assume that $h_{v,t}=\theta_{v,t}^{1}$. From (\ref{hvt}),
\begin{equation}\label{split of theta}
    \begin{split}
        &\theta_{tv,t}^1=\str\Big[\frac{i\mu_1}{tv}g\exp\big(-\lk-\ctvt^2-2dat\derivative{t}\ctvt'' -2\dabar t\derivative{t}\ctvt'\\ 
        &\qquad\qquad+da\dabar \big(N_V-\frac{\mathrm{dim}X}{2}\big)\big)\Big]^{da\dabar}+\frac{\mathrm{dim}X}{2}\str\Big[i\frac{\mu_1}{tv}g\exp(-\bkttv)\Big].
    \end{split}
\end{equation}

From (\ref{nkt def using c}), (\ref{lichnerowicz mu2}), we find that the following identity holds.
\begin{equation}
    \begin{split}
        \lk+\ctvt^2+2dat\derivative{t}\ctvt'' +2\dabar t\derivative{t}\ctvt'-da\dabar \big(N_V-\frac{\mathrm{dim}X}{2}\big)\qquad\qquad\\
        =-\frac{t}{2}\Big(\nabla_{e_i}^{\Lambda(T^{*(0,1)}X)\otimes \xi}+\frac{1}{\sqrt{2t}}\product{S(e_i)e_j, f_p^H}c(e_j)f^p\wedge
        +\frac{1}{2t}\product{S(e_i)f_p^H,f_q^H}\\  f^p\wedge f^q\wedge  -\frac{\product{K^X,e_i}}{2t}\Bigl(1+\frac{1}{v}\Bigr)-\frac{da c(e_i^{(1,0)})}{\sqrt{2t}}-\frac{\dabar c(e_i^{(0,1)})}{\sqrt{2t}}\Big)^2\qquad\quad\\ 
       +\frac{t}{8}H+\frac{t}{4}\bigl(\bigl(R^{\xi}+\frac{1}{2}\trace[R^{TX}]\bigr)(e_i,e_j)-\frac{1}{vt}\product{\nabla_{e_i}^{TX}K^X,e_j}\bigr)c(e_i)c(e_j)\\ 
       +\sqrt{\frac{t}{2}}\bigl(\bigl(R^{\xi}+\frac{1}{2}\trace[R^{TX}]\bigr)(e_i,f_p^H)-\frac{1}{vt}\product{T(e_i,f_p^H),K^X}\bigr)c(e_i)f^p\\ 
       +\frac{1}{2}\bigl(\bigl(R^{\xi}+\frac{1}{2}\trace[R^{TX}]\bigr)(f_p^H,f_q^H)-\frac{1}{2vt}\product{T(f_p^H,f_q^H),K^X}\bigr)f^p\wedge f^q\wedge\\ 
       -m^{\xi}(K)-\frac{1}{2}\trace[m^{TX}(K)]+\frac{1}{2v}|K^X|^2.
    \end{split}
\end{equation}

For $t>0$, put   
\begin{equation}\label{6nablat}
    \begin{split}
    ^6\nabla^{\mathbb{E},t}=\nabla^{\Lambda(T^{*(0,1)}X)\otimes \xi}+\frac{1}{\sqrt{2t}}\product{S(\cdot)e_j, f_p^H}c(e_j)f^p\wedge +\frac{1}{2t}\product{S(\cdot)f_p^H,f_q^H}f^p\wedge f^q\wedge\\
    -\frac{K^{'X}}{2t}\Bigl(1+\frac{1}{v}\Bigr)-\frac{da c(\cdot^{(1,0)})}{\sqrt{2t}}-\frac{\dabar c(\cdot^{(0,1)})}{\sqrt{2t}}.
    \end{split}
\end{equation}
As in (\ref{curvature 2nabla}), we have
\begin{equation}\label{curvature 6nabla}
    \begin{split}
        ( ^6\nabla^{\mathbb{E},1})^2(e_i,e_j)=\frac{1}{4}\product{R^{TX}(e_k,e_{\ell})e_i,e_j}c(e_k)c(e_{\ell})\hspace{12 em}\\ 
        +\frac{1}{\sqrt{2}}\product{R^{TX}(e_k,f_p^H)e_i,e_j}c(e_k)f^p\wedge
        +\frac{1}{2}\product{R^{TX}(f_p^H,f_q^H)e_i,e_j}f^p\wedge f^q\wedge\\ +\big(R^{\xi}+\frac{1}{2}\trace[R^{TX}]\big)(e_i,e_j)-\Bigl(1+\frac{1}{v}\Bigr)\product{m^{TX}(K)e_i,e_j}-i{da\dabar}\product{e_i,J^{TX}e_j}.
    \end{split}
\end{equation}
Let $A,B$ be smooth sections of $T_{\rr}X$. By (\ref{6nablat}),
\begin{equation}\label{6nablat c}
    \begin{split}
^6\nabla^{\mathbb{E},t}_A c(B)=c(\nabla_A^{TX}B)+\sqrt{\frac{2}{t}}\product{S(A)B,f_p^H}f^p\wedge\hspace{8 em}\\ +\sqrt{\frac{2}{t}}da\product{A^{(1,0)},B}+\sqrt{\frac{2}{t}}\dabar\product{A^{(0,1)},B}.
    \end{split}
\end{equation}

As in Section \ref{sec A change of coordinates and a Getzler rescaling: a first simple approach}, we identify $B^{TX}(y_0,4\varepsilon_0)$ with $B^X(y_0,4\varepsilon_0)$ by sending $Z\in (T_{\rr}X)_{y_0}$ to $\exp_{y_0}^X(Z)\in X$. Then we identify $(\pi^*\Lambda(T_{\rr}^*S)\,\widehat{\otimes }\,\Lambda(T^{*(0,1)}X)\otimes \xi )_{Z}$ with $(\pi^*\Lambda(T_{\rr}^*S)\,\widehat{\otimes }\,\Lambda(T^{*(0,1)}X)\otimes \xi)_{y_0}$ by parallel transport  with respect to the connection $^6\nabla^{\mathbb{E},t}$ along the curve $s\in [0,1]\rightarrow sZ$. 

Let $L_{y_0,K,\theta}^{ 1,(t,v)}$ be the operator obtained from $  \lk+\ctvt^2+2dat\derivative{t}\ctvt'' +2\dabar t\derivative{t}\ctvt'-da\dabar \big(N_V-\frac{\mathrm{dim}X}{2}\big)$ as in (\ref{ly0k1tv}). Let $L_{y_0,K,\theta}^{3,(t,v)}$ be the operator obtained from $L_{y_0,K,\theta}^{ 1,(t,v)}$ by using the Getzler rescaling of Section \ref{sec A change of coordinates and a Getzler rescaling: a first simple approach}. 
\begin{definition}
    For $t>0$, $v\geq 1$, set
    \begin{equation}\label{ly0ktheta 30v} \begin{split}
        L_{y_0,K,\theta}^{3,(0,v)}= -\frac{1}{2}\Bigl(\nabla_{e_i}+\frac{1}{2}\big\langle\big(\jmath^{\prime *}R^{TX}-\bigl(1+\frac{1}{v}\bigr)m^{TX}(K)-iJ^{TX}da\dabar\big)Z,e_i\big\rangle\Bigr)^2\qquad\qquad\quad\\+\jmath^{\prime *}R_{y_0}^{\xi}-m^{\xi}(K)_{y_0}
         +\frac{1}{2}\trace[\jmath^{\prime *}R^{TX}-m^{TX}(K)]_{y_0}+\frac{1}{2v}\big|m_{y_0}^{TX}(K)Z\big|^2-\frac{da}{\sqrt{2}v}c(m^{TX}(K)z)\\ -\frac{\dabar}{\sqrt{2}v}c(m^{TX}(K)\bar{z}) 
         -\frac{1}{4v}\product{m^{TX}(K)e_j,e_k}_{y_0}c(e_j)c(e_k)+\frac{1}{2v}\product{\jmath^{\prime *} R^{TX}m^{TX}(K)Z,Z}_{y_0}.
     \end{split}
 \end{equation}
\end{definition}
\begin{theorem}
    For $v\geq 1$, as $t\rightarrow 0$,
    \begin{equation}\label{limit of ly0ktheta}
        L_{y_0,K,\theta}^{3,(t,v)}\rightarrow L_{y_0,K,\theta}^{3,(0,v)}.
    \end{equation}
\end{theorem}
\begin{proof}
    For $Z\in (T_{\rr}X)_{y_0}$, we split $Z$ into $z+\bar{z}$ where $z\in (T^{(1,0)}X)_{y_0}$.
    For $A\in\bigl(\Lambda(\mathbb{R}^{2*})\,\widehat{\otimes }\,\pi^*\Lambda(T_{\rr}^*S)\,\widehat{\otimes }\, c(TX)\otimes \ENDO(\xi) \bigr)_{y_0}\otimes \mathrm{Op}_{y_0}$, we denote by $[A]_t^{(3)}$ the differential operator obtained from $A$ by using the Getzler rescaling in Definition \ref{et}. By (\ref{6nablat c}), in the trivialization associated to $ ^6\nabla^{\mathbb{E},t}$,
    for $1\leq j\leq 2\ell''$,
    \begin{equation}\label{6 ct explicit 1}
        \Big[\sqrt{t}c_t(\tau e_j(\sqrt{t}Z))\Big]_t^{(3)}=\sqrt{2}c(e_j)+\mathcal{O}(\sqrt{t}|Z|^2),
    \end{equation}
    for $2\ell''+1\leq j\leq 2\ell$,
    \begin{equation}\label{6 ct explicit 2}
        \begin{split}
        \Big[c(\tau e_j(\sqrt{t}Z))\Big]_t^{(3)}=c(e_j)+\sqrt{{2}}\bigl(\product{S(Z)e_j,f_p^H}\bigr)f^p\wedge\qquad\qquad\quad\\
        +\sqrt{2}da\product{z,e_j}+\sqrt{2}\dabar\product{\bar{z},e_j}+\mathcal{O}(\sqrt{t}|Z|^2).
        \end{split}
    \end{equation}
    From (\ref{dk'})-(\ref{expansion T 2}),  (\ref{6 ct explicit 1}), (\ref{6 ct explicit 2}),   we find that in the trivialization associated with $ ^6\nabla^{\mathbb{E},t}$, as $t\rightarrow 0$, 
    \begin{equation}\label{explicit 63t}
        \begin{split}
            \frac{1}{4}  \Big[\product{\nabla_{\tau e_i}^{TX}K^X,\tau e_j}c(\tau e_i)c(\tau e_j)(\sqrt{t}Z)\Big]_t^{(3)}+\frac{1}{\sqrt{2t}}\Big[\product{T(\tau e_i,f_p^H),K^X}\\ \cdot c(\tau e_i)f^p(\sqrt{t}Z)\Big]_t^{(3)}
            +\frac{1}{4t}\Big[\product{T(f_p^H,f_q^H),K^X}f^p\wedge f^q(\sqrt{t}Z)\Big]_t^{(3)}
           \\  \rightarrow \frac{1}{4}\sum_{2\ell''+1\leq i,j\leq 2\ell}\product{m^{TX}(K)e_i,e_j}c(e_i)c(e_j)
            -\frac{1}{2}\product{\jmath^{\prime *}R^{TX}m^{TX}(K)Z,Z}_{y_0}\\
            +\frac{da}{\sqrt{2}}c(m^M(K)z) +\frac{\dabar}{\sqrt{2}}c(m^M(K)\bar{z})-\frac{i}{2}da \overline{da}\langle J^{TX}m^{TX}(K)Z,Z  \rangle.
        \end{split}
     \end{equation}
     Then (\ref{limit of ly0ktheta}) follows from (\ref{curvature 6nabla}) and (\ref{explicit 63t}).
\end{proof}

Now  we proceed as in Sections \ref{sec a new trivialization}-\ref{sec A proof of Theorem main lemma 2} for the operator $ \lk+\ctvt^2+2dat\derivative{t}\ctvt'' +2\dabar t\derivative{t}\ctvt'-da\dabar \big(N_V-\frac{\mathrm{dim}X}{2}\big)$. In particular, by using (\ref{many generalized supertraces}), (\ref{def hv}), (\ref{split of theta}), (\ref{ly0ktheta 30v}), we have the following analogue of (\ref{estimate to main lemma 2 mu1}),
\begin{equation}
    |\theta_{tv,t}^1-\theta_{v}^1|\leq Cv^rt^{\gamma},\quad t\in \,]0,1], v\in[1,1/t].
\end{equation}

Now we complete the proof of Theorem \ref{thm h,vt} when $h_{v,t}=\theta_{v,t}^1$. For the cases $h_{v,t}=\pi_{v,t}^{1,\prime/\prime\prime}$ or $\lambda_{v,t}^{\prime/\prime\prime}$, one can follow the steps in this subsection by making $da$ or $\dabar=0$. For the case $h_{v,t}=\tau_{v,t}$, one only needs to use (\ref{difference mu}) again.

\end{document}